\documentclass[11pt,a4paper]{article}
\usepackage[autolanguage]{numprint}
\usepackage[english]{babel}
\usepackage{amsmath}
\usepackage{preview}
\usepackage{amsfonts}
\usepackage{amssymb}
\usepackage{amsthm}
\usepackage{mathrsfs}
\usepackage{dsfont}
\usepackage{paralist}
\usepackage{natbib}
\usepackage{bm}
\usepackage{color}
\usepackage[colorlinks=true,linkcolor=blue,citecolor=blue,pdfborder={0 0 0}]{hyperref}
\usepackage{graphicx}
\usepackage{tabularx}
\usepackage[left=3.5cm,right=3.5cm,top=2cm,bottom=2cm,includeheadfoot]{geometry}
\usepackage{comment}
\usepackage{tikz}
\usetikzlibrary{decorations.pathreplacing}
\usepackage{resizegather}
\usepackage{grffile}
\usepackage{graphicx}
\usepackage{caption}
\usepackage{subcaption}
\usepackage{placeins}
\DeclareFontFamily{OT1}{pzc}{}
\DeclareFontShape{OT1}{pzc}{m}{it}{<-> s * [1.10] pzcmi7t}{}
\DeclareMathAlphabet{\mathpzc}{OT1}{pzc}{m}{it}

\usepackage{xspace}
\newcommand{\lebesgue}{\ensuremath{\lambda\!\!\!\!\>\lambda}\xspace}
\usepackage{enumitem}
\SetLabelAlign{LeftAlignWithIndent}{\hspace*{2.0ex}\makebox[1.5em][l]{#1}}

\newcount\Comments  % 0 suppresses notes to selves in text
\newcommand{\kibitz}[2]{\ifnum\Comments=1\textcolor{#1}{#2}\fi}

\newcommand\Item[1][]{%
  \ifx\relax#1\relax  \item \else \item[#1] \fi
  \abovedisplayskip=0pt\abovedisplayshortskip=0pt~\vspace*{-\baselineskip}}

\newtheoremstyle{normal}% name
{2ex}               % Space above, empty = `usual value'
{3ex}               % Space below
{}                  % Body font
{}                  % Indent amount (empty = no indent, \parindent = para indent)
{\bfseries} % Thm head font
{}                  % Punctuation after thm head
{2pt}   % Space after thm head.
{\thmname{#1}\thmnumber{ #2.} \thmnote{(#3)}}% Thm head spec

\newtheoremstyle{italic}% name
{2ex}%      Space above, empty = `usual value'
{3ex}%      Space below
{\itshape}% Body font
{}%         Indent amount (empty = no indent, \parindent = para indent)
{\bfseries} % Thm head font
{}%        Punctuation after thm head
{2pt}%     Space after thm head.
{\thmname{#1}\thmnumber{ #2.} \thmnote{(#3)}}% Thm head spec

\theoremstyle{normal}
\newtheorem{definition}{Definition}[section]
\newtheorem{remark}[definition]{Remark}
\newtheorem{example}[definition]{Example}
\newtheorem{condition}[definition]{Condition}

\theoremstyle{italic}
\newtheorem{theorem}[definition]{Theorem}
\newtheorem{lemma}[definition]{Lemma}
\newtheorem{proposition}[definition]{Proposition}
\newtheorem{corollary}[definition]{Corollary}

\newcommand\R{\mathbb{R}}

\DeclareMathOperator{\Var}{Var}

\DeclareMathOperator{\Cov}{Cov}

\allowdisplaybreaks[1]

\begin{document}

\title{The null hypothesis of common jumps in case of irregular and asynchronous observations}

\author{Ole Martin and Mathias Vetter\thanks{Christian-Albrechts-Universit\"at zu Kiel, Mathematisches Seminar, Ludewig-Meyn-Str.\ 4, 24118 Kiel, Germany.
{E-mail:} martin@math.uni-kiel.de/vetter@math.uni-kiel.de} \bigskip \\
{Christian-Albrechts-Universit\"at zu Kiel}
}

\maketitle

\begin{abstract}
This paper proposes novel tests for the absence of jumps in a univariate semimartingale and for the absence of common jumps in a bivariate semimartingale. Our methods rely on ratio statistics of power variations based on irregular observations, sampled at different frequencies. We develop central limit theorems for the statistics under the respective null hypotheses and apply bootstrap procedures to assess the limiting distributions. Further we define corrected statistics to improve the finite sample performance. Simulations show that the test based on our corrected statistic yields good results and even outperforms existing tests in the case of regular observations.
\end{abstract}

%\medskip

 \textit{Keywords and Phrases:} Asynchronous observations; common jumps; \\high-frequency statistics; It\^o semimartingale; stable convergence

\smallskip

 \textit{AMS Subject Classification:} 62G10, 62M05 (primary); 60J60, 60J75 (secondary)

%\tableofcontents

\section{Introduction}
\def\theequation{1.\arabic{equation}}
\setcounter{equation}{0}

A key issue of the statistical analysis in continuous time is to understand the fine structure of the underlying processes based on discrete observations. Tools are needed in order to answer fundamental questions about suitable models for these processes, and probably the most basic ones are concerned with whether the processes evolve continuously in time or exhibit jumps. Within the class of semimartingales several statistical tests on this topic have been proposed in recent years, both in one and in several (usually two) dimensions and under various null hypotheses. For an overview see for example Sections 11.4 and 15.3 in \cite{JacPro12} or Chapters 10 and 14 in \cite{AitJac14}. 

In this work we are interested in the null hypothesis that common jumps of a bivariate process exist. In fact this question was tackled and solved prominently in \cite{JacTod09}, but only under the additional assumption that observations come in at equidistant and synchronous times. Both assumptions are rarely justified in practice, which is why practitioners usually ignore a lot of observations (and then work with one-minute observations, for example) in order to artificially construct a sampling scheme which is close enough to justify the assumptions of regularity. Both from a theoretical and a practical point of view this is unsatisfactory, and one would like to understand what happens when the underlying observations come in randomly at irregular and asynchronous times. This is precisely what we aim at in this paper, and our study complements the previous work \cite{MarVet17} in which a test for the presence of common jumps (so under the opposite null hypothesis of no common jumps) was constructed. 

It has turned out that the construction of a test in two dimensions is difficult unless one properly understands the univariate situation. For this reason we provide two statistical procedures: First, we want to decide whether or not jumps are present in a realized path of a stochastic process in one dimension, based on discrete irregular observations of the process and under the null hypothesis that jumps exist. Second, we want to decide whether common jumps are present or not in two realized processes, based on irregular and asynchronous observations and under the null hypothesis that there exist common jumps. Thus, this work can not only be understood as a generalization of the results in \cite{JacTod09}, but it also complements the univariate discussion from \cite{aitjac2009}. 

In the following let $\widetilde{X}$ be a one-dimensional process observed at times $\tilde{t}_{i,n}$ and let $X=(X^{(1)},X^{(2)})$ be a bivariate process where $X^{(l)}$ is observed at times $t_{i,n}^{(l)}$, $l=1,2$, all over the time interval $[0,T]$. Then for some $k \geq 2$ our tests are based on the ratio statistics
\begin{align*}
\frac{\sum_{i\geq k:\tilde{t}_{i,n} \leq T} (\widetilde{X}_{\tilde{t}_{i,n}}-\widetilde{X}_{\tilde{t}_{i-k,n}})^4}{k\sum_{i\geq 1:\tilde{t}_{i,n} \leq T} (\widetilde{X}_{\tilde{t}_{i,n}}-\widetilde{X}_{\tilde{t}_{i-1,n}})^4}
\end{align*}
and
\begin{align*}
\frac{\sum_{i,j\geq k:t_{i,n}^{(1)} \wedge t_{j,n}^{(2)} \leq T} (X^{(1)}_{t_{i,n}^{(1)}}-X^{(1)}_{t_{i-k,n}^{(1)}})^2(X^{(2)}_{t_{j,n}^{(2)}}-X^{(2)}_{t_{j-k,n}^{(2)}})^2\mathds{1}_{\{(t_{i-k,n}^{(1)},t_{i,n}^{(1)}]\cap (t_{j-k,n}^{(2)},t_{j,n}^{(2)}]\neq \emptyset\}}}{k^2\sum_{i,j\geq 1:t_{i,n}^{(1)} \wedge t_{j,n}^{(2)} \leq T} (X^{(1)}_{t_{i,n}^{(1)}}-X^{(1)}_{t_{i-1,n}^{(1)}})^2(X^{(2)}_{t_{j,n}^{(2)}}-X^{(2)}_{t_{j-1,n}^{(2)}})^2\mathds{1}_{\{(t_{i-1,n}^{(1)},t_{i,n}^{(1)}]\cap (t_{j-1,n}^{(2)},t_{j,n}^{(2)}]\neq \emptyset\}}}.
\end{align*}
The common feature of both statistics is that they have a different asymptotic behaviour, depending on whether, for the first one, $\widetilde X$ has jumps in $[0,T]$ or not, or whether, for the second one, $X^{(1)}$ and $X^{(2)}$ have common jumps in $[0,T]$ or not. In this paper we investigate the asymptotics of these statistics both under the null hypothesis that (common) jumps are present and under the alternative that (common) jumps do not exist. We will see that the limits are similar in structure to those in the setting of equidistant observation times. Further we will develop central limit theorems using techniques from \cite{BibVet15} and \cite{MarVet17} for both statistics under the null hypothesis that (common) jumps are present. Unlike in the setting of equidistant observation times the limiting variables in the central limit theorem are not mixed normal but have a more complicated distribution. We use the bootstrap method introduced in \cite{MarVet17} to estimate the quantiles of this limit distribution and to construct a feasible testing procedure.

Finally we conduct a simulation study to check the finite sample performance of our tests. As in \cite{MarVet17} we use the models from \cite{JacTod09} to compare our results with the results from similar tests in the setting of equidistant observations. Compared to these results we find that the finite sample performance of our tests is rather poor, especially under the null hypothesis, which is due to a rather large contribution of terms which vanish asymptotically in the central limit theorem. To improve the performance of our tests we therefore construct an estimator which (partially) corrects for those asymptotically vanishing terms. Using this estimator we get a huge improvement in the performance of our tests. In particular, our results become even better than the corresponding results for the tests from \cite{aitjac2009}, \cite{JacTod09} and \cite{AitJac14} in the setting of equidistant observation times. 

The remainder of the paper is organized as follows: As the structure of the results and the formal arguments are simpler in the setting where we test for jumps in a one-dimensional process we first derive in Section \ref{Test_J} two statistical tests for jumps within a stochastic process. The formal setting and the estimator $\Phi_{k,T,n}^{(J)}$ is introduced in Section \ref{sec:setting_J}. In Section \ref{sec:cons_J} we derive the consistency of this estimator under both hypotheses, we cover the asymptotics of the estimator in the form of a central limit theorem in Section \ref{sec:clt_J}, and in Section \ref{sec:testproc_J} we use a bootstrap method to derive two feasible tests, one using the original statistic $\Phi_{k,T,n}^{(J)}$ and one using a corrected estimator $\widetilde{\Phi}_{k,T,n}^{(J)}(\rho)$ based on the central limit theorem. Building on the results from Section \ref{Test_J} we proceed similarly in Section \ref{Test_CoJ} to derive two statistical tests for deciding whether or not two processes jump at a common time. In Section \ref{sec:simul} we examine the finite sample properties of the tests derived in Sections \ref{Test_J} and \ref{Test_CoJ} by means of a simulation study. The appendix containing the proofs is split into two parts: Section \ref{sec:proof_struc} contains the main structure of the proofs, leaving out most technical arguments. Section \ref{sec:proof_detail}, available online, contains all proofs in details and thus fills the gaps left in Section \ref{sec:proof_struc}.

\section{A test for the absence of jumps}\label{Test_J}
\def\theequation{2.\arabic{equation}}
\setcounter{equation}{0}

In this section we will derive a statistical test based on high-frequency observations which allows to decide whether an observed path of a process contains jumps or not.

\subsection{Settings and test statistic}\label{sec:setting_J}
First we specify the mathematical model for the process and the observation times. Let $X$ be a one-dimensional It\^o semimartingale on the probability space $(\Omega,\mathcal{F},\mathbb{P})$ of the form
\begin{multline}\label{ItoSemimart}
X_t= X_0 + \int \limits_0^t b_s ds + \int \limits_0^t \sigma_s dW_s + \int \limits_0^t \int \limits_{\mathbb{R}} \delta(s,z)\mathds{1}_{\{|\delta(s,z)|\leq 1\}} (\mu - \nu)(ds,dz) \\
+ \int \limits_0^t \int \limits_{\mathbb{R}} \delta(s,z) \mathds{1}_{\{|\delta(s,z)|> 1\}} \mu(ds,dz).
\end{multline}
Here, $W$ denotes a standard Brownian motion, $\mu$ is a Poisson random measure on $\mathbb{R}^+ \times \mathbb{R}$ whose predictable compensator satisfies $\nu(ds,dz)=ds \otimes \lambda(dz)$ for some $ \sigma$-finite measure $\lambda$ on $\mathbb{R}$ endowed with the Borelian $\sigma$-algebra, $b$ and $\sigma$ are adapted processes and $\delta$ is a predictable function on $\Omega \times \mathbb{R}^+ \times \mathbb{R}$. For a more detailed discussion of the components of $X$ consult Section 2.1 of \cite{JacPro12}. We write $\Delta X_s=X_s-X_{s-}$ with $X_{s-}=\lim_{t \nearrow s} X_t$ for a possible jump of $X$ in $s$.

Further we define a sequence of observation schemes $(\pi_n)_{n \in \mathbb{N}}$ via
\begin{align*}
\pi_n=\left(t_{i,n}\right)_{i \in \mathbb{N}_0} ,~n \in \mathbb{N}, 
\end{align*}
where the $\left(t_{i,n}\right)_{i \in \mathbb{N}_0}$ are increasing sequences of stopping times with $t_{0,n}=0$. By
\[
|\pi_n|_T=\sup \left\{t_{i,n} \wedge T -t_{i-1,n} \wedge T\middle|i\geq 1 \right\}
\]
we denote the mesh of the observation times up to some fixed time horizon $T \geq 0$. 

Formally we will develop a statistical test which allows to decide whether a realized path $X(\omega)$ contains jumps in a given time interval $[0,T]$ or not. Specifically we want to decide based on the observations $(X_{t_{i,n}}(\omega))_{i \in \mathbb{N}_0}$ to which of the following two subsets of $\Omega$
\begin{align*}
&\Omega^{(J)}_T=\left\{\exists t \in [0,T]:\Delta X_t \neq 0 \right\},
\\&\Omega^{(C)}_T=\big(\Omega^{(J)}_T\big)^c
\end{align*}
$\omega$ belongs. Here, $\Omega^{(J)}_T$ is the set of all $\omega$ for which the path of $X$ up to $T$ has at least one jump and $\Omega^{(C)}_T$ is the set of all $\omega$ for which the path of $X$ is continuous on $[0,T]$.

All our test statistics are based on the increments
\begin{align*}
\Delta_{i,k,n}X=X_{t_{i,n}}-X_{t_{i-k,n}},\quad \Delta_{i,n}X=\Delta_{i,1,n}X.
\end{align*}
and we denote by $\mathcal{I}_{i,k,n}=\left(t_{i-k,n}, t_{i,n}\right],~i\geq k\geq 1$, $\mathcal{I}_{i,n}=\mathcal{I}_{i,1,n}$, the corresponding observation intervals. For convenience we set $\mathcal{I}_{i,k,n}=\emptyset$ and accordingly $\Delta_{i,k,n}X=0$ for $i<k$.

Further we define for $k \in \mathbb{N}$ and a function $h:\mathbb{R} \rightarrow \mathbb{R}$ the functionals
\begin{align*}
&V(h,[k],\pi_n)_T=\sum_{i\geq k:t_{i,n} \leq T} h(\Delta_{i,k,n}X),
\\&V(h,\pi_n)_T=V(h,[1],\pi_n)_T.
\end{align*}

Considering these functionals for $g:x \mapsto x^4$ and $k \geq 2$ we build the statistic
\begin{align*}
\Phi_{k,T,n}^{(J)}=\frac{V(g,[k],\pi_n)_T}{k V(g,\pi_n)_T}
\end{align*}
whose asymptotics we use to construct a statistical tests for the absence of jumps.

\begin{remark}\label{remark_equi_which_k_1d}
In the setting of equidistant observation times $t_{i,n}=i/n$ our statistic becomes
\begin{align}\label{remark_equi_all_k}
\Phi_{k,T,n}^{(J)}=\frac{\sum_{i=k}^{\lfloor nT \rfloor} g(\Delta_{i,k,n}X)}{k\sum_{i=1}^{\lfloor T/n \rfloor} g(\Delta_{i,n}X)}.
\end{align}
On the other hand in \cite{aitjac2009} a test is constructed based on the statistic
\begin{align}\label{remark_equi_spec_k}
\Phi_{k,T,n}^{(J)}=\frac{\sum_{i=1}^{\lfloor nT/k \rfloor} g(\Delta_{ik,k,n}X)}{\sum_{i=1}^{\lfloor T/n \rfloor} g(\Delta_{i,n}X)}
\end{align}
where at the lower observation frequency $n/k$ only increments over certain observation intervals $\mathcal{I}_{ik,k,n}$ enter the estimation. Intuitively it seems that using the statistic \eqref{remark_equi_all_k} should be better than using \eqref{remark_equi_spec_k}, because in \eqref{remark_equi_all_k} we utilize the available data more exhaustively by using all increments at the lower observation frequency. This intuition is confirmed by Proposition 10.19 in \cite{AitJac14} where central limit theorems are developed for both \eqref{remark_equi_all_k} and \eqref{remark_equi_spec_k} and it is shown that \eqref{remark_equi_all_k} has a smaller asymptotic variance.
\end{remark}

\subsection{Consistency}\label{sec:cons_J}

In order to derive results on the asymptotic behaviour of $\Phi_{k,T,n}^{(J)}$ we need to impose certain structural assumptions on the It\^o semimartingale $X$ and the observation scheme. Further we introduce the notation
\begin{align*}
G_{k,n}(t)=\frac{n}{k^2}\sum_{i\geq k:t_{i,n} \leq t} \left|\mathcal{I}_{i,k,n} \right|^2
\end{align*}
and abbreviate $G_n(t)=G_{1,n}(t)$.
\begin{condition}\label{cond_cons_J}
The process $b_s$ is locally bounded and the process $\sigma_s$ is c\`adl\`ag. Furthermore, there exists a locally bounded process $\Gamma_s$ with $|\delta(\omega,s,z)| \leq \Gamma_s(\omega) \gamma(z)$ for some deterministic bounded function $\gamma$ which satisfies $\int ( \gamma(z)^2 \wedge 1)\lambda(dz) < \infty$, and the process $\sigma$ fulfills $\int_0^T |\sigma_s| ds>0$ almost surely. Additionally the following assumptions on the observation scheme hold: 
\begin{enumerate}[leftmargin=*, align=LeftAlignWithIndent, itemsep=-0.1cm, font=\normalfont, label=(\roman*)]
\item
The sequence of observation schemes $(\pi_n)_n$ is exogenous, i.e.\ independent of the process $X$ and its components, and fulfills
\[
|\pi_n|_T \overset{\mathbb{P}}{\longrightarrow} 0.
\]
\item
The functions $G_n(t)$, $G_{k,n}(t)$ convergence pointwise on $[0,T]$ in probability to strictly increasing functions $G,G_k:[0,T] \rightarrow [0,\infty)$.
\end{enumerate}
\end{condition}
The structural assumptions on $b,\sigma,\delta$ are not very restrictive and occur elsewhere in the literature in similar form. The assumption that $\sigma$ almost surely does not vanish on $[0,T]$ and Condition \ref{cond_cons_J}(ii) are only needed to derive the asymptotic behaviour of $\Phi_{k,T,n}^{(J)}$ on $\Omega_T^{(C)}$. $\Phi_{k,T,n}^{(J)}$ converges on $\Omega_T^{(J)}$ also without these assumptions.

We will see in the proof of Theorem \ref{theo_cons_J} that 
\begin{align*}
&V(g,\pi_n)_T \overset{\mathbb{P}}{\longrightarrow} B_T^{(J)}=\sum_{t \leq T}\left(\Delta X_t \right)^4,
\\&V(g,[k],\pi_n)_T \overset{\mathbb{P}}{\longrightarrow} kB_T^{(J)}
\end{align*}
which yields the asymptotic behaviour of $\Phi_{k,T,n}^{(J)}$ on $\Omega_T^{(J)}$. On $\Omega_T^{(C)}$ we have $B_T^{(J)}=0$ and we expand the fraction by $n$ to get an asymptotic result. To describe the limit in that case we define
\begin{align*}
C_{T}^{(J)}=\int_0^T 3\sigma_s^4 dG(s), \quad C_{k,T}^{(J)}=\int_0^T 3\sigma_s^4 dG_k(s).
\end{align*}

\begin{theorem}\label{theo_cons_J}

Under Condition \ref{cond_cons_J} it holds
\begin{align}
\Phi_{k,T,n}^{(J)} \overset{\mathbb{P}}{\longrightarrow} 
\begin{cases}
 1, &\text{on } \Omega_T^{(J)}, \\
\frac{k C_{k,T}^{(J)}}{C_{T}^{(J)}}, &\text{on } \Omega_T^{(C)}.
\end{cases} 
\end{align}
\end{theorem}

\begin{remark}\label{remark_limit_in_1_k}
We obtain $$G_n(t)-G_n(s) +O(|\pi_n|_t)\leq k G_{k,n}(t)-k G_{k,n}(s) \leq k G_n(t)-kG_n(s)$$ for all $t \geq s \geq 0$ from the series of elementary inequalities
\begin{align}\label{elementary_estimate}
\sum_{i=1}^k a_i^2 \leq \Big( \sum_{i=1}^k a_i\Big)^2 \leq k \sum_{i=1}^k a_i^2
\end{align}
which holds for any $a_1,\ldots, a_k \geq 0$, $k \in \mathbb{N}$. Here, the second inequality follows from the Cauchy-Schwarz inequality. Equality in \eqref{elementary_estimate} holds for $a_1\geq 0$, $a_2=\ldots=a_k=0$ respectively $a_1=\ldots=a_k>0$. The relations of $G_n$ and $G_{k,n}$ are preserved in the limit which yields $$G(t)-G(s) \leq k G_k(t)-k G_k(s) \leq kG(t)-kG(s)$$ for all $t \geq s\geq 0$. Hence we get
\begin{align*}
\frac{k C_{k,T}^{(J)}}{C_{T}^{(J)}} \in [1,k].
\end{align*}
\end{remark}

Based on the fact that $\Phi_{k,T,n}^{(J)}\overset{\mathbb{P}}{\longrightarrow} 1$ on $\Omega_T^{(J)}$ and that $\Phi_{k,T,n}^{(J)}$ converges on $\Omega_T^{(C)}$ to a random variable which is strictly greater than $1$ if ${kC_{k,T}^{(J)}>C_T^{(J)}}$ by Remark \ref{remark_limit_in_1_k}, we will construct a test with critical region
\begin{align}\label{critical_region_J}
\mathcal{C}_{k,T,n}^{(J)}=\{\Phi_{k,T,n}^{(J)}>1+\mathpzc{c}_{k,T,n}^{(J)}\}
\end{align}
for an appropriate series of decreasing random positive numbers $\mathpzc{c}_{k,T,n}^{(J)}$, $n \in \mathbb{N}$.\\

In the following we illustrate the result from Theorem \ref{theo_cons_J} by looking at two prominent observation schemes: first Poisson sampling which is truely random and asynchronous and second equidistant sampling for which results exist in the literature and which will serve as a kind of benchmark.

\begin{example}\label{example_poisson1}
Consider the observation scheme where $t_{i,n}-t_{i-1,n}$ are i.i.d.\ $Exp(n\lambda )$, $\lambda>0$, distributed. We will call this observation scheme \textit{Poisson sampling} as the observation times $t_{i,n}$ correspond to the jump times of a Poisson process with intensity $n \lambda$. In this setting Condition \ref{cond_cons_J}(ii) is fulfilled as shown in Proposition 1 from \cite{HayYos08} with 
\begin{align*}
G(t)=\frac{2}{\lambda}t,
\end{align*}
and as proved in Section \ref{sec:proof_detail_poiss} with
\begin{align}\label{lln_gamma}
G_k(t)=\frac{k+1}{k\lambda}t.
\end{align}
This yields that the limit under the alternative is $(k+1)/2$.
\end{example}

\begin{remark}
In the case of equidistant synchronous observations, i.e.\ $t_{i,n}=i/n$, it holds ${G(t)=G_{k,n}(t)= t}$ which yields
\begin{align*}
\frac{k C_{k,T}^{(J)}}{C_T^{(J)}}=k.
\end{align*}
Hence in this setting $\Phi_{k,T,n}^{(J)}$ converges to a known deterministic limit on $\Omega_T^{(C)}$ as well which also allows to construct a test using $\Phi_{k,T,n}^{(J)}$ for the null hypothesis of no jumps (compare Section 10.3 in \cite{AitJac14}). This is not immediately possible in the irregular setting, unless the law of the generating mechanism is known to the statistician.
\end{remark}

\subsection{Central limit theorem}\label{sec:clt_J}
In this section we derive a central limit theorem for $\Phi_{k,T,n}^{(J)}$ which holds on $\Omega_T^{(J)}$. Denote by $i_n(s)$ the index of the interval $\mathcal{I}_{i,n}$ associated with $s \in \mathcal{I}_{i,n} $ and define
\begin{align*}
\xi_{k,n,-}(s)=n\sum_{j=1}^{k-1} (k-j)^2|\mathcal{I}_{i_n(s)-j,1,n}|, \quad \xi_{k,n,+}(s)=n\sum_{j=1}^{k-1} (k-j)^2|\mathcal{I}_{i_n(s)+j,1,n}|.
\end{align*}
$\xi_{k,n,-}(s)+\xi_{k,n,+}(s)$ is the $\sigma(\pi_n:n \in \mathbb{N})$-conditional variance of
\begin{align*}
n\sum_{j= -(k-1),j \neq 0}^{k-1} |k-j| \Delta_{i_n(s)+j,1,n} W
=n \sum_{j=0}^{k-1} \Delta_{i_n(s)+j,k,n} W-nk \Delta_{i_n(s),1,n}W.
\end{align*}
This identity is illustrated in Figure \ref{fig_xi}.

\begin{figure}[bt]
\centering
\hspace{-0.6cm}
\begin{tikzpicture}
\draw[dashed] (0,0) -- (1,0)
			(11,0) -- (12,0)
			(5,-0.2) -- (5,0.2);
\draw (1,0) -- (11,0)
      (1,-0.2) -- (1,0.2)
      (3,-0.2) -- (3,0.2)
      (4.7,-0.2) -- (4.7,0.2)
      (6.5,-0.2) -- (6.5,0.2)
      (8.6,-0.2) -- (8.6,0.2)
      (11,-0.2) -- (11,0.2);
\draw	(1,-0.5) node{$t_{i_n(s)-3,n}$}
		(3,-0.5) node{$t_{i_n(s)-2,n}$}
		(4.7,-0.5) node{$t_{i_n(s)-1,n}$}
		(6.5,-0.5) node{$t_{i_n(s),n}$}
		(8.5,-0.5) node{$t_{i_n(s)+1,n}$}
		(11,-0.5) node{$t_{i_n(s)+2,n}$}
		(5,0.4) node{$s$};
\draw[decorate,decoration={brace,amplitude=12pt}]
	(1,0.5)--(6.5,0.5) node[midway, above,yshift=10pt,]{$|\mathcal{I}_{i_n(s),3,n}|$};
\draw[decorate,decoration={brace,amplitude=12pt}]
	(3,1.5)--(8.6,1.5) node[midway, above,yshift=10pt,]{$|\mathcal{I}_{i_n(s)+1,3,n}|$};
\draw[decorate,decoration={brace,amplitude=12pt}]
	(4.7,2.5)--(11,2.5) node[midway, above,yshift=10pt,]{$|\mathcal{I}_{i_n(s)+2,3,n}|$};
\draw[decorate,decoration={brace,amplitude=12pt}]
	(6.5,-1)--(4.7,-1) node[midway, below,yshift=-10pt,]{$3|\mathcal{I}_{i_n(s),1,n}|$};
\end{tikzpicture}
\caption{Illustrating the origin of $\xi_{k,n,-}(s),\xi_{k,n,+}(s)$ for $k=3$.}\label{fig_xi}
\end{figure}
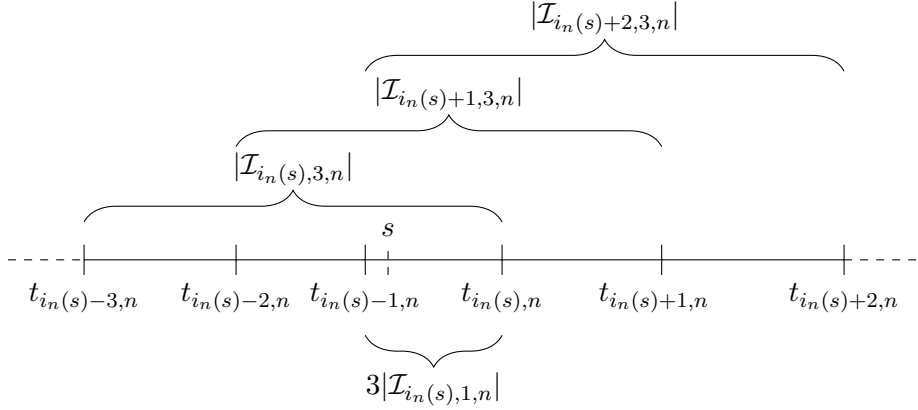

\FloatBarrier

The following condition summarizes the assumptions we need additionally to Condition \ref{cond_cons_J} to derive a central limit theorem.

\begin{condition} \label{cond_clt_J} The process $X$ and the sequence of observation schemes $(\pi_n)_n$ fulfill Condition \ref{cond_cons_J}. Further the following additional assumptions on the observation schemes hold:
\begin{enumerate}[leftmargin=*, align=LeftAlignWithIndent, itemsep=-0.1cm, font=\normalfont, label=(\roman*)]
\item
We have
\begin{align}
|\pi_n|_T=o_{\mathbb{P}}(n^{-1/2}).
\end{align}
\item
The integral
\begin{align*}
\int_{[0,T]^P} g(x_1,\dots,x_P) \mathbb{E} \Big[ \prod_{p=1}^{P} h_p \left(\xi_{k,n,-}(x_p),\xi_{k,n,+}(x_p)\right)\Big] dx_1 \dots dx_P
\end{align*}
converges for $n \rightarrow \infty$ to
\begin{align*}
\int_{[0,T]^P} g(x_1,\dots,x_P)\prod_{p=1}^{P} \int_{\mathbb{R}} h_p \left(y\right)\Gamma(x_p,dy)
 dx_1 \dots dx_P
\end{align*}
for all bounded continuous functions $g:\mathbb{R}^{P} \rightarrow \mathbb{R}$, $h_p:\mathbb{R}^2\rightarrow \mathbb{R}$, ${p=1,\ldots,P}$, and any $P \in \mathbb{N}$. Here $\Gamma(\cdot,dy)$ is a family of probability measures on $[0, T]$ with uniformly bounded first moments and $\int_0^T\Gamma(x,\{(0,0)\})dx=0$.
\end{enumerate}
\end{condition}

Part (i) of Condition \ref{cond_clt_J} guarantees that $|\pi_n|_T$ vanishes sufficiently fast, while part (ii) of \ref{cond_clt_J} yields that the $(\xi_{k,n,-}(s),\xi_{k,n,+}(s))$ converge in law in a suitable sense.

Because of the exogeneity of the observation times we may assume in the following that the probability space has the form
\begin{align}\label{prob_space}
(\Omega,\mathcal{F},\mathbb{P})=(\Omega_\mathcal{X} \times \Omega_\mathcal{S},\mathcal{X} \otimes \mathcal{S},\mathbb{P}_\mathcal{X} \otimes \mathbb{P}_\mathcal{S})
\end{align}
where $\mathcal{X}$ denotes the $\sigma$-algebra generated by $X$ and its components and $\mathcal{S}$ denotes the $\sigma$-algebra generated by the observation scheme $(\pi_n)_n$. 

To describe the limit in the upcoming central limit theorem we define
\begin{align}\label{def_FkT_J}
F_{k,T}^{(J)}=4\sum_{p:S_p \leq T} (\Delta X_{S_p})^3\sqrt{(\sigma_{S_p-})^2\xi_{k,-}(S_p)+(\sigma_{S_p})^2 \xi_{k,+}(S_p)} U_{S_p}.
\end{align}
Here, $(S_p)_{p \geq 0}$ denotes an enumeration of the jump times of $X$ and the $(\xi_{k,-}(s),\xi_{k,+}(s))$, $s \in [0,T]$, are independent random variables which are distributed according to $\Gamma(s,dy)$ and the $U_s$, $s \in [0,T]$, are i.i.d.\ standard normal distributed random variables. Both the $(\xi_{k,-}(s),\xi_{k,+}(s))$ and the $U_s$ are independent of $X$ and its components and defined on an extended probability space $(\widetilde{\Omega},\widetilde{\mathcal{F}},\widetilde{\mathbb{P}})$. Note that $F_{k,T}^{(J)}$ is well-defined because the sum in \eqref{def_FkT_J} is almost surely absolutely convergent and independent of the choice for the enumeration $(S_p)_{p \geq 0}$; compare Proposition 4.1.3 in \cite{JacPro12}.
\begin{theorem}\label{theo_clt_J}
If Condition \ref{cond_clt_J} holds, we have the $\mathcal{X}$-stable convergence
\begin{align}\label{theo_clt_J_1}
\sqrt{n} \left(\Phi_{k,T,n}^{(J)}-1\right) \overset{\mathcal{L}-s}{\longrightarrow} \frac{F_{k,T}^{(J)}}{k B_T^{(J)}}
\end{align}
on $\Omega_T^{(J)}$.
\end{theorem}

Here, the limit $F_{k,T}^{(J)}/(k B_T^{(J)})$ in \eqref{theo_clt_J_1} is defined on the extended probability space $(\widetilde{\Omega},\widetilde{\mathcal{F}},\widetilde{\mathbb{P}})$. Further the statement of the $\mathcal{X}$-stable convergence on $\Omega_T^{(J)}$ means that we have
\[
\mathbb{E} \big[g\big(\sqrt{n} (\Phi_{k,T,n}^{(J)}-1)\big) Y\mathds{1}_{\Omega_T^{(J)}}\big] \rightarrow \widetilde{\mathbb{E}}\big[g\big(F_{k,T}^{(J)}/(k B_T^{(J)}) \big) Y\mathds{1}_{\Omega_T^{(J)}}\big]
\] 
for all bounded and continuous functions $g$ and all $\mathcal{X}$-measurable bounded random variables $Y$. For more background information on stable convergence in law we refer to \cite{JacPro12}, \cite{JacShi02} and \cite{PodVet10}.

\begin{example}\label{example_poisson2}
Condition \ref{cond_clt_J} is fulfilled in the setting of Poisson sampling introduced in Example \ref{example_poisson1}. Part (i) is fulfilled by Lemma 8 from \cite{HayYos08} which states
\begin{align}\label{poiss_mesh}
\mathbb{E}\left[(|\pi_n|_T)^q\right]=o(n^{-\alpha})
\end{align}
for all $q \geq 1$ and $\alpha < q$. That \ref{cond_clt_J}(ii) is fulfilled is proved in Section \ref{sec:proof_poiss_clt}.
\end{example}

\subsection{Testing procedure}\label{sec:testproc_J}

In this section we develop a statistical test for testing the null hypothesis that $t \mapsto X_t(\omega)$ has jumps in $[0,T]$ (i.e.\ $\omega \in \Omega_T^{(J)}$) against the alternative that $t \mapsto X_t(\omega)$ is continuous on $[0,T]$ (i.e.\ $\omega \in \Omega_T^{(C)}$). To employ Theorem \ref{theo_clt_J} for this purpose we have to estimate the distribution of the limiting variable $F_{k,T}^{(J)}$ and therefore the distribution of $\xi_k(S_p)$ for the jump times $S_p$. Because the distribution of $\xi_k(S_p)$ depends on the unknown observation scheme around $S_p$, of which we observe only a single realization, we use a bootstrap method to estimate the distribution from the realization of the observation scheme close to $S_p$. For this method to work we need some sort of local homogeneity which will be guaranteed by Condition \ref{cond_testproc_CoJ}(i).

To formalize the bootstrap method let $L_n$ and $M_n$ be sequences of natural numbers which tend to infinity. Set
\begin{equation}\label{estimators_xi}
\begin{aligned}
&\hat{\xi}_{k,n,m,-}(s)=n\sum_{j=1}^{k-1} (k-j)^2|\mathcal{I}_{i_n(s)+V_{n,m}(s)-j,1,n}|,
\\&\hat{\xi}_{k,n,m,+}(s)=n\sum_{j=1}^{k-1} (k-j)^2|\mathcal{I}_{i_n(s)+V_{n,m}(s)+j,1,n}|
\end{aligned}
\end{equation}
for $m=1,\ldots,M_n$ where the random variable $V_{n,m}(s)$ attains values in $\{-L_n,\ldots,L_n\}$ with probabilities
\begin{align}\label{weights_testing}
\mathbb{P}(V_{n,m}(s) =l|\mathcal{S})=|\mathcal{I}_{i_n(s)+l,n}|\Big(\sum_{j=-L_n}^{L_n} |\mathcal{I}_{i_n(s)+j,n}| \Big)^{-1} ,~l \in \{-L_n,\ldots,L_n\}.
\end{align}
Here, $(\hat{\xi}_{k,n,m,-}(s),\hat{\xi}_{k,n,m,-}(s))$ is chosen from the $(\xi_{k,n,-}(t_{i_n(s)+i,n}),\xi_{k,n,-}(t_{i_n(s)+i,n}))$, $i=-L_n,\ldots,L_n$, which make up the $2L_n+1$ realizations of $(\hat{\xi}_{k,n,-}(t),\hat{\xi}_{k,n,-}(t))$ which lie 'closest' to $s$, with probability proportional to the interval length $|\mathcal{I}_{i_n(s)+i,n}|$. This corresponds to the probability with which a random variable which is uniformly distributed on the union of these intervals $\mathcal{I}_{i_n(s)+i,n}$, $i=-L_n,\ldots,L_n$, but otherwise independent from the observation scheme would fall into the interval $\mathcal{I}_{i_n(s)+i,n}$. Due to the structure of the predictable compensator $\nu$ the jump times $S_p$ of the It\^o semimartingale $X$ are also evenly distributed in time. This explains why we choose such a random variable $V_{n,m}(s)$ for the estimation of the law of $(\hat{\xi}_{k,n,m,-}(S_p),\hat{\xi}_{k,n,m,-}(S_p))$.

Using the estimators \eqref{estimators_xi} for realizations of $\xi_k(s)$ we build the following estimators for realizations of $F_{k,T}^{(J)}$
\begin{align}\label{est_F_J}
&\widehat{F}_{k,T,n,m}^{(J)}= 4 \sum_{i:t_{i,n} \leq T} \left(\Delta_{i,n} X \right)^3 \mathds{1}_{\{|\Delta_{i,n} X|>\beta |\mathcal{I}_{i,n}|^\varpi\}} \nonumber
\\&~~~~~~\times \sqrt{(\hat{\sigma}_n(t_{i,n},-))^2\hat{\xi}_{k,n,m,-}(t_{i,n})
+(\hat{\sigma}_n(t_{i,n},+))^2 \hat{\xi}_{k,n,m,+}(t_{i,n})}U_{n,i,m},
\end{align}
$m=1,\ldots,M_n$, where $\beta>0$ and $\varpi \in (0,1/2)$. Here an increment which is large compared to a given threshold is identified as a jump and the local volatility is estimated by
\begin{align}\label{est_sigma}
\begin{aligned}
&\left(\hat{\sigma}_n(s,-)\right)^2=\frac{1}{b_n}
\sum_{i:\mathcal{I}_{i,n} \subset [s-b_n,s)} \left(\Delta_{i,n} X\right)^2,
\\ &\left(\hat{\sigma}_n(s,+)\right)^2=\frac{1}{b_n}
\sum_{i:\mathcal{I}_{i,n} \subset [s,s+b_n]} \left(\Delta_{i,n} X\right)^2
\end{aligned}
\end{align}
for a sequence $(b_n)_n$ with $b_n \rightarrow 0 $ and $|\pi_n|_T/ b_n \rightarrow 0$. Further the $U_{n,i,m}$ are i.i.d.\ standard normal distributed random variables which are independent of $\mathcal{F}$ and defined on the extended probability space $(\widetilde{\Omega},\widetilde{\mathcal{F}},\widetilde{\mathbb{P}})$. 

Denote by
\begin{align*}
\widehat{Q}_{k,T,n}^{(J)}(\alpha)=\widehat{Q}_\alpha\big(\big\{
\widehat{F}_{k,T,n,m}^{(J)} | m=1,\ldots,M_n
\big\}\big)
\end{align*}
the $\lfloor \alpha M_n\rfloor $-th largest element of the set $\big\{
\widehat{F}_{k,T,n,m}^{(J)} | m=1,\ldots,M_n
\big\}$.

We will see that $\widehat{Q}_{k,T,n}^{(J)}(\alpha)$ converges on $\Omega_T^{(J)}$ under appropriate conditions to the $\mathcal{X}$-conditional $\alpha$-quantile $Q_k^{(J)}(\alpha)$ of $F_{k,T}^{(J)} $ which is defined via
\begin{align}
\widetilde{\mathbb{P}}\big(F_{k,T}^{(J)} \leq Q_k^{(J)}(\alpha) \big| \mathcal{X} \big)(\omega)=\alpha,~~ \omega \in \Omega_T^{(J)},
\end{align}
and we set $\big(Q_k^{(J)}(\alpha)\big)(\omega)=0$, $\omega \in (\Omega_T^{(J)})^c$. Such a random variable $Q_k^{(J)}(\alpha)$ exists if Condition \ref{cond_testproc_J} is fulfilled because the $\mathcal{X}$-conditional distribution of $F_{k,T}^{(J)}$ will be almost surely continuous on $\Omega_T^{(J)}$ under Condition \ref{cond_testproc_J}.

\begin{condition}\label{cond_testproc_J}
Assume that the process $X$ and the sequence of observation schemes $(\pi_n)_n$ satisfy Condition \ref{cond_clt_J} and that the set $\{s \in [0,T]:\sigma_s=0\}$ is almost surely a Lebesgue null set. Further, let the sequence $(b_n)_n$ fulfill $|\pi_n|_T/b_n \overset{\mathbb{P}}{\longrightarrow}0$ and suppose that $(L_n)_n$ and $(M_n)_n$ are sequences of natural numbers converging to infinity and $ L_n/n\rightarrow 0$. Additionally,
\begin{enumerate}[leftmargin=*, align=LeftAlignWithIndent, itemsep=-0.1cm, font=\normalfont, label=(\roman*)]
\Item
\begin{multline*}
\widetilde{\mathbb{P}}\Big(\big|\widetilde{\mathbb{P}}((\hat{\xi}_{k,n,1,-}(s_p),\hat{\xi}_{k,n,1,+}(s_p))\leq x_p,~p=1,\ldots,P | \mathcal{S} ) 
\\-
 \prod_{p=1}^P \widetilde{\mathbb{P}}((\xi_{k,-}(s_p),\xi_{k,+}(s_p)) \leq x_p) \big|>\varepsilon \Big) \rightarrow 0
\end{multline*}
as $n \rightarrow \infty$, for all $\varepsilon>0$ and any $x \in \mathbb{R}^{2 \times P}$, $P \in \mathbb{N}$, and $s_p \in (0,T)$, $p=1,\ldots,P$.
\item The volatility process $\sigma$ is itself an It\^o semimartingale, i.e.\ a process of the form \eqref{ItoSemimart}.
\item
On $\Omega_T^{(C)}$ we have $k C_{k,T}^{(J)} >C_T^{(J)}$ almost surely.
\end{enumerate}
\end{condition}

Part (i) of Condition \ref{cond_testproc_J} guarantees that the bootstrapped realizations $$(\hat{\xi}_{k,n,m,-}(s),\hat{\xi}_{k,n,m,+}(s))$$ consistently estimate the distribution of $(\xi_{k,-}(s),\xi_{k,+}(s))$ and thereby that $\widehat{Q}_{k,T,n}^{(J)}(\alpha)$ yields a valid estimator for $Q_k^{(J)}(\alpha)$ on $\Omega_T^{(J)}$. Part (ii) is needed for the convergence of the volatility estimators $\hat{\sigma}_{n}(S_p,-)$, $\hat{\sigma}_{n}(S_p,+)$ for jump times $S_p$, and part (iii) guarantees that $\Phi_{k,T,n}^{(J)}$ converges under the alternative to a value different from $1$, which is the limit under the null hypothesis.

\begin{theorem}\label{test_theo_J}
If Condition \ref{cond_testproc_J} is fulfilled, the test defined in \eqref{critical_region_J} with
\begin{align}
\mathpzc{c}_{k,T,n}^{(J)}= \frac{\widehat{Q}_{k,T,n}^{(J)}(1-\alpha)}{\sqrt{n}kV(g,\pi_n)_T},~\alpha \in [0,1],
\end{align}
has asymptotic level $\alpha$ in the sense that we have
\begin{align}\label{test_theo_level}
\widetilde{\mathbb{P}}\big(\Phi_{k,T,n}^{(J)} > 1+\mathpzc{c}_{k,T,n}^{(J)} \big| F^{(J)} \big) \rightarrow \alpha 
\end{align}
for all $F^{(J)} \subset \Omega_T^{(J)}$ with $\mathbb{P}(F^{(J)})>0$. 

The test is consistent in the sense that we have
\begin{align}\label{test_theo_power}
\widetilde{\mathbb{P}}\big(\Phi_{k,T,n}^{(J)} > 1+\mathpzc{c}_{k,T,n}^{(J)} \big| F^{(C)} \big) \rightarrow 1 
\end{align}
for all $F^{(C)} \subset \Omega_T^{(C)}$ with $\mathbb{P}( F^{(C)})>0$.
\end{theorem}

Note that to carry out the test introduced in Theorem \ref{test_theo_J} the unobservable variable $n$ is not explicitly needed, even though $\sqrt n$ occurs in the definition of $\mathpzc{c}_{k,T,n}^{(J)}$.  This factor actually cancels as it also enters as a linear factor in $\widehat{Q}_{k,T,n}^{(J)}(1-\alpha)$. What remains is the dependence of $b_n$ and $L_n$ on $n$, though, but for these auxiliary variables only a rough idea of the magnitude of $n$ usually is sufficient. Similar observations hold for all tests constructed later on as well. 

The simulation results in Section \ref{sec:simul_J} show that the convergence in \eqref{test_theo_level} is rather slow, because certain terms in $\sqrt{n}(\Phi_{k,T,n}^{(J)}-1)$ which vanish in the limit contribute significantly in the small sample. Our goal is to diminish this effect by including estimates for those terms in the testing procedure. The asymptotically vanishing terms stem from the continuous part which is mostly captured in the small increments. To estimate their contribution we define 
\begin{multline*}
A_{k,T,n}^{(J)}
=n\sum_{i\geq k:t_{i,n}\leq T} (\Delta_{i,k,n} X)^4
 \mathds{1}_{\{|\Delta_{i,k,n} X\leq \beta |\mathcal{I}_{i,k,n}|^\varpi\}} 
\\- k n
\sum_{i\geq 1:t_{i,n}\leq T} (\Delta_{i,n} X)^4 \mathds{1}_{\{|\Delta_{i,n} X|\leq \beta |\mathcal{I}_{i,n}^{(l)}|^\varpi\}}.
\end{multline*}
using the same $\beta,\varpi$ as in \eqref{est_F_J}. It can be shown that $A_{k,T,n}^{(J)}$ is a consistent estimator for $k^2C_{k,T}^{(J)}-kC_{1,T}^{(J)}$. We then define for $\rho \in (0,1)$ the adjusted estimator
\begin{align*}
\widetilde{\Phi}_{k,T,n}^{(J)}(\rho)=\Phi_{k,T,n}^{(J)}-\rho\frac{n^{-1} A_{k,T,n}^{(J)}}{kV(g,\pi_n)_T}
\end{align*}
where we partially correct for the contribution of the asymptotically vanishing terms.

\begin{corollary}\label{test_cor_J} 
Let $\rho \in (0,1)$. If Condition \ref{cond_testproc_J} is fulfilled, it holds with the notation from Theorem \ref{test_theo_J} 
\begin{align}\label{test_cor_level_J}
\widetilde{\mathbb{P}}\big(\widetilde{\Phi}_{k,T,n}^{(J)}(\rho) > 1+\mathpzc{c}_{k,T,n}^{(J)} \big| F^{(J)} \big) \rightarrow \alpha 
\end{align}
for all $F^{(J)} \subset \Omega_T^{(J)}$ with $\mathbb{P}(F^{(J)})>0$ and
\begin{align}\label{test_cor_power_J}
\widetilde{\mathbb{P}}\big(\widetilde{\Phi}_{k,T,n}^{(J)}(\rho) > 1+\mathpzc{c}_{k,T,n}^{(J)} \big| F^{(C)} \big) \rightarrow 1 
\end{align}
for all $F^{(C)} \subset \Omega_T^{(C)}$ with $\mathbb{P}(F^{(C)})>0$. 
\end{corollary}

The closer $\rho$ is to $1$ the faster is the convergence in \eqref{test_cor_level_J}, but also the slower is the convergence in \eqref{test_cor_power_J}. Hence an optimal $\rho$ should be chosen somewhere in between. Our simulation results in Section \ref{sec:simul_J} show that it is possible to pick a $\rho$ very close to $1$ without significantly worsening the power compared to the test from Theorem \ref{test_theo_J}.

\begin{example}\label{example_poiss_test_1d}
The assumptions on the observation scheme in Condition \ref{cond_testproc_J} are fulfilled in the Poisson setting. That part (iii) is fulfilled has been shown in Example \ref{example_poisson1} and that part (i) is fulfilled is proved in Section \ref{sec:proof_poiss_test}.
\end{example}

In fact for our testing procedure to work in the Poisson setting we do not need the weighting from \eqref{weights_testing}. All intervals could also be picked with equal probability. This is due to the fact that the interval lengths $(n|\mathcal{I}_{i_n(s)+V_{n,m}(s)+j,n}|)_{j=-(k-1),\ldots,-1,1,\ldots,k-1}$ and hence $\xi_{k,n,m,-}(s)$, $\xi_{k,n,m,+}(s)$ are (asymptotically) independent of ${n|\mathcal{I}_{i_n(s)+V_{n,m}(s),n}|}$. However the weighting is important if the interval lengths of consecutive intervals are dependent as illustrated in the following example.

\begin{example}\label{example_alpha_test_1d}
Define an observation scheme by $t_{2i,n}=2i/n$ and $t_{2i+1,n}=(2i+1+\alpha)/n$, $i \in \mathbb{N}_0$, with $\alpha \in (0,1)$ (compare Example 33 in \cite{BibVet15}). Let us consider the case $k=2$. The observation scheme is illustrated in Figure \ref{figure_example_alpha}. It can be easily checked that Condition \ref{cond_cons_J} holds with $G(t)=(1+\alpha^2)t$ and $G_2(t)=t$. Further it can be shown similarly as in \cite{BibVet15} that Condition \ref{cond_clt_J} is fulfilled for $\Gamma$ defined via
\begin{align*}
\Gamma(s,\{(1+\alpha,1+\alpha)\})=\frac{1-\alpha}{2}, \quad \Gamma(s,\{(1-\alpha,1-\alpha)\})=\frac{1+\alpha}{2}
\end{align*}
for all $s>0$. Hence in order for the distribution of $\hat{\xi}_{k,n,1}(s)$ to approximate $\Gamma(s,\cdot)$ the variable $i_n(s)+V_{n,m}(s)$ has to pick the intervals of length $(1+\alpha)/n$ with higher probability than those with length $(1-\alpha)/n$, because it holds
\begin{align*}
&n|\mathcal{I}_{i_n(s)+V_{n,m}(s),n}|=1+\alpha \Rightarrow \hat{\xi}_{2,n,m,-}(s)=\hat{\xi}_{2,n,m,+}(s)=1-\alpha,
\\&n|\mathcal{I}_{i_n(s)+V_{n,m}(s),n}|=1-\alpha \Rightarrow \hat{\xi}_{2,n,m,-}(s)=\hat{\xi}_{2,n,m,+}(s)=1+\alpha.
\end{align*}
\end{example}

\begin{figure}[bt]
\centering
\hspace{-0.6cm}
\begin{tikzpicture}
\draw[dashed] (2,-0.2) -- (2,0.2)
			(6,-0.2) -- (6,0.2)
			(10,-0.2) -- (10,0.2);
\draw (0,0) -- (12.5,0)
      (0,-0.3) -- (0,0.3)
      (4,-0.3) -- (4,0.3)
      (8,-0.3) -- (8,0.3)
      (12,-0.3) -- (12,0.3)
      (2.3,-0.3) -- (2.3,0.3)
      (6.3,-0.3) -- (6.3,0.3)
      (10.3,-0.3) -- (10.3,0.3);
      \draw	(2,0.3) node{\tiny $1/n$}
		(6,0.3) node{\tiny $3/n$}
		(10,0.3) node{\tiny $5/n$};
\draw	(0,-0.6) node{$t_{0,n}=0$}
		(2.3,-0.6) node{$t_{1,n}$}
		(4,-0.6) node{$t_{2,n}=2/n$}
		(6.3,-0.6) node{$t_{3,n}$}
		(8,-0.6) node{$t_{4,n}=4/n$}
		(10.3,-0.6) node{$t_{5,n}$}
		(12,-0.6) node{$t_{6,n}=6/n$};
\draw[decorate,decoration={brace,amplitude=12pt}]
	(0,0.35)--(2.3,0.35) node[midway, above,yshift=10pt,]{$(1+\alpha)/n$};
\draw[decorate,decoration={brace,amplitude=12pt}]
	(2.3,0.35)--(4,0.35) node[midway, above,yshift=10pt,]{$(1-\alpha)/n$};
	\draw[decorate,decoration={brace,amplitude=12pt}]
	(4,0.35)--(6.3,0.35) node[midway, above,yshift=10pt,]{$(1+\alpha)/n$};
	\draw[decorate,decoration={brace,amplitude=12pt}]
	(6.3,0.35)--(8,0.35) node[midway, above,yshift=10pt,]{$(1-\alpha)/n$};
	\draw[decorate,decoration={brace,amplitude=12pt}]
	(8,0.35)--(10.3,0.35) node[midway, above,yshift=10pt,]{$(1+\alpha)/n$};
	\draw[decorate,decoration={brace,amplitude=12pt}]
	(10.3,0.35)--(12,0.35) node[midway, above,yshift=10pt,]{$(1-\alpha)/n$};
\end{tikzpicture}
\caption{The sampling scheme from Example \ref{example_alpha_test_1d}}\label{figure_example_alpha}
\end{figure}
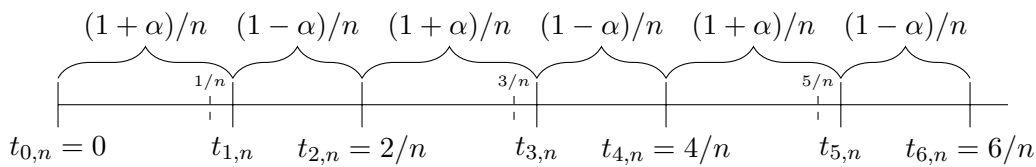

\FloatBarrier

\section{A test for the absence of common jumps}\label{Test_CoJ}
\def\theequation{3.\arabic{equation}}
\setcounter{equation}{0}

In this section we will derive a statistical test based on high-frequency observations which allows to decide whether two processes jump at a common time or not. The methods we use are similar to those in Section \ref{Test_J}. However, the form of the occuring variables and the proofs will be different because of special effects due to the asynchronicity of the data.

\subsection{Settings and test statistic}\label{sec:setting_CoJ}

As a model for the stochastic process we again consider an It\^o semimartingale $X =(X^{(1)},X^{(2)})^*$, which is now two-dimensional, defined on the probability space $(\Omega,\mathcal{F},\mathbb{P})$ and of the form
\begin{multline}
\label{ItoSemimart_2d}
X_t = X_0 + \int \limits_0^t b_s ds + \int \limits_0^t \sigma_s dW_s + \int \limits_0^t \int \limits_{\R^2} \delta(s,z)\mathds{1}_{\{\|\delta(s,z)\|\leq 1\}} (\mu - \nu)(ds,dz) \\
+ \int \limits_0^t \int \limits_{\R^2} \delta(s,z) \mathds{1}_{\{\|\delta(s,z)\|> 1\}} \mu(ds,dz),
\end{multline}
where $W=(W^{(1)},W^{(2)})^*$ is a two-dimensional standard Brownian motion and $\mu$ is a Poisson random measure on $\mathbb{R}^+ \times\mathbb{R}^2$ whose predictable compensator satisfies \mbox{$\nu(ds,dz)=ds \otimes \lambda(dz)$} for some $\sigma$-finite measure $\lambda$ on $\mathbb{R}^2$ endowed with the Borelian $\sigma$-algebra. $b$ is a two-dimensional adapted process,
\begin{align}\label{sigma_matrix}
\sigma_s=\left(\begin{array}{cc}
 \sigma_s^{(1)}& 0
 \\ \rho_s \sigma_s^{(2)} & \sqrt{1-\rho_s^2} \sigma_s^{(2)}
\end{array} \right)
\end{align}
is a $(2 \times 2)$-dimensional process and $\delta$ is a two-dimensional predictable process on $\Omega \times \mathbb{R}^+ \times \mathbb{R}^2$. $\sigma_s^{(1)},\sigma_s^{(2)} \geq 0$ and $\rho_s \in [-1,1]$ are all univariate adapted.

The observation scheme 
\begin{align}\label{obs_sch_2d}
\pi_n=\big\{(t_{i,n}^{(1)})_{i \in \mathbb{N}_0},(t_{i,n}^{(2)})_{i \in \mathbb{N}_0} \big\},~n \in \mathbb{N}, 
\end{align}
here consists of two (in general different) increasing sequences of stopping times $(t_{i,n}^{(l)})_{i \in \mathbb{N}_0}$, ${l=1,2}$, with $t_{0,n}^{(l)}=0$. By
\[
|\pi_n|_T=\sup \big\{t_{i,n}^{(l)} \wedge T -t_{i-1,n}^{(l)} \wedge T\big|i\geq 1,l=1,2 \big\}
\]
we again denote the mesh of the observation times up to $T \geq 0$.

Formally we will develop a statistical test which allows to decide to which of the following two subsets of $\Omega$ the $\omega$ which generated the observed path $t \mapsto X_t(\omega)$ belongs
\begin{align*}
&\Omega^{(CoJ)}_T=\big\{\exists t \in [0,T]:\Delta X^{(1)}_t \Delta X^{(2)}_t \neq 0 \big\},
\\&\Omega^{(nCoJ)}_T=\big(\Omega_T^{(CoJ)}\big)^c.
\end{align*}

We denote the observation intervals of $X^{(l)}$ by $\mathcal{I}_{i,k,n}^{(l)}=(t_{i-k,n}^{(l)}, t_{i,n}^{(l)}],~i\geq k\geq 1$, $l=1,2$, $\mathcal{I}^{(l)}_{i,n}=\mathcal{I}^{(l)}_{i,1,n}$, and by
\begin{align*}
\Delta_{i,k,n}^{(l)}X^{(l)}=X_{t_{i,n}}^{(l)}-X_{t_{i-k,n}}^{(l)}, \quad\Delta_{i,n}^{(l)}X^{(l)}=\Delta_{i,1,n}^{(l)}X^{(l)}
\end{align*}
the increments of $X^{(l)}$ over those intervals. As in Section \ref{sec:setting_J} we set $\mathcal{I}_{i,k,n}^{(l)}= \emptyset$ and $\Delta_{i,k,n}^{(l)}X^{(l)}=0$ for $k<i$.

Further we define the following functionals for $k \in \mathbb{N}$ and $h:\mathbb{R}^2 \rightarrow \mathbb{R}$
\begin{align*}
&V(h,[k],\pi_n)_T=\sum_{i,j\geq k:t_{i,n}^{(1)} \wedge t_{j,n}^{(2)} \leq T} h(\Delta_{i,k,n}^{(1)} X^{(1)},\Delta_{j,k,n}^{(2)} X^{(2)})\mathds{1}_{\{\mathcal{I}_{i,k,n}^{(1)} \cap \mathcal{I}_{j,k,n}^{(2)} \neq \emptyset\}},
\\&V(h,\pi_n)_T=V(h,[1],\pi_n)_T
\end{align*}
which are generalizations of the famous estimator for the quadratic covariation based on asynchronous observations from \cite{hayyos05}. Computing these functionals for $f:(x_1,x_2) \mapsto (x_1x_2)^2$ we build the statistic
\begin{align*}
\Phi_{k,T,n}^{(CoJ)}=\frac{V(f,[k],\pi_n)_T}{k^2 V(f,\pi_n)_T}
\end{align*}
whose asymptotic behaviour for $k \geq 2$ is the foundation for the upcoming statistical test. 

By
\begin{align*}
B_T^{(CoJ)}=\sum_{t \leq T}(\Delta X^{(1)}_t )^2(\Delta X^{(2)}_t )^2
\end{align*}
we denote the sum of the squared co-jumps of $X^{(1)}$ and $X^{(2)}$.

\begin{remark}\label{remark_equi_which_k_2d}
In the setting of equidistant observation times $t_{i,n}^{(l)}=i/n$, $l=1,2$, the statistic $\Phi_{k,T,n}^{(CoJ)}$ is equal to
\begin{align}\label{remark_equi_allall_k_2d}
\frac{\sum_{i,j=k}^{\lfloor nT\rfloor} (X^{(1)}_{i/n}-X^{(1)}_{(i-k)/n})^2(X^{(2)}_{j/n}-X^{(2)}_{(j-k)/n})^2\mathds{1}_{\{((i-k)/n,n]\cap((j-k)/n,j/n] \neq \emptyset\}}}{k^2\sum_{i=1}^{\lfloor nT\rfloor} (X^{(1)}_{i/n}-X^{(1)}_{(i-1)/n})^2(X^{(2)}_{i/n}-X^{(2)}_{(i-1)/n})^2}.
\end{align}
In \cite{JacTod09} a test for common jumps is constructed based on the statistic
\begin{align}\label{remark_equi_spec_k_2d}
\frac{\sum_{i=1}^{\lfloor nT/k\rfloor} (X^{(1)}_{ki/n}-X^{(1)}_{k(i-1)/n})^2(X^{(2)}_{ki/n}-X^{(2)}_{k(i-1)/n})^2}{\sum_{i=1}^{\lfloor nT\rfloor} (X^{(1)}_{i/n}-X^{(1)}_{(i-1)/n})^2(X^{(2)}_{i/n}-X^{(2)}_{(i-1)/n})^2}
\end{align}
where at the lower observation frequency $n/k$ only increments over the intervals $\mathcal{I}_{ki,k,n}^{(l)}$, $l=1,2$, enter the estimation. Further in Section 14.1 of \cite{AitJac14} a test for common jumps based on 
\begin{align}\label{remark_equi_all_k_2d}
\frac{\sum_{i=k}^{\lfloor nT\rfloor} (X^{(1)}_{i/n}-X^{(1)}_{(i-k)/n})^2(X^{(2)}_{i/n}-X^{(2)}_{(i-k)/n})^2}{k\sum_{i=1}^{\lfloor nT\rfloor} (X^{(1)}_{i/n}-X^{(1)}_{(i-1)/n})^2(X^{(2)}_{i/n}-X^{(2)}_{(i-1)/n})^2}
\end{align}
is discussed. As argued in Remark \ref{remark_equi_which_k_1d} it seems advantegeous to use the statistic \eqref{remark_equi_all_k_2d} over \eqref{remark_equi_spec_k_2d}. However, in the asynchronous setting it is a priori not clear which observation intervals should be best paired, because there is no one-to-one correspondence of observation intervals in one process to observation intervals in the other process as there is in the synchronous situation. To use the available data as exhaustively as possible we therefore decided to include products of squared increments over all overlapping observation intervals at the lower observation frequency in the numerator of $\Phi_{k,T,n}^{(CoJ)}$. 
\end{remark}

\subsection{Consistency}\label{sec:cons_CoJ}
In this section we investigate under which conditions $\Phi^{(CoJ)}_{k,T,n}$ converges to a certain limit. The following structural assumptions which are similar to those in Condition \ref{cond_cons_J} are needed to obtain an asymptotic result for $\Phi^{(CoJ)}_{k,T,n}$ on $\Omega^{(CoJ)}_T$.

\begin{condition}\label{cond_cons_CoJ}
The process $b_s$ is locally bounded and the processes $\sigma_s^{(1)},\sigma_s^{(2)},\rho_s$ are c\`adl\`ag. Furthermore, there exists a locally bounded process $\Gamma_s$ and a deterministic bounded function $\gamma$ which satisfies ${\int ( \gamma(z)^2 \wedge 1)\lambda(dz) < \infty}$ such that ${\|\delta(\omega,s,z)\| \leq \Gamma_s(\omega) \gamma(z)}$. The sequence of observation schemes $(\pi_n)_n$ fulfills
\[
|\pi_n|_T \overset{\mathbb{P}}{\longrightarrow} 0.
\]
\end{condition}

\begin{theorem}\label{theo_cons_CoJ1}
Under Condition \ref{cond_cons_CoJ} we have on $\Omega_T^{(CoJ)}$
\begin{align}\label{theo_cons_CoJ1_conv}
\Phi_{k,T,n}^{(CoJ)} \overset{\mathbb{P}}{\longrightarrow} 1.
\end{align}
\end{theorem}

As in Section \ref{sec:clt_J} the asymptotic behaviour of $\Phi^{(CoJ)}_{k,T,n}$ on $\Omega^{(nCoJ)}_T$ is more complicated than on $\Omega^{(CoJ)}_T$. First we introduce the following functions
\begin{align*}
&\widetilde{G}_{k',n}(t)=\frac{n}{(k')^3}\sum_{i,j\geq k':t_{ i,n}^{(1)} \wedge t_{ j,n}^{(2)} \leq t} |\mathcal{I}_{ i,k',n}^{(1)}\cap \mathcal{I}_{ j,k',n}^{(2)} |^2, 
\\ &H_{k',n}(t)=\frac{n}{(k')^3}\sum_{i,j\geq k':t_{ i,n}^{(1)} \wedge t_{ j,n}^{(2)} \leq t} |\mathcal{I}_{ i,k',n}^{(1)}|| \mathcal{I}_{ j,k',n}^{(2)} | 
\mathds{1}_{\{\mathcal{I}_{i,k',n}^{(1)} \cap \mathcal{I}_{j,k',n}^{(2)}\neq \emptyset\}},
\end{align*}
for $k'=1,k$. Set $\overline{W}_t=(W^{(1)}_t,\rho_t W^{(1)}_t +\sqrt{1-(\rho_t)^2} W^{(2)}_t)^*$ such that $\overline{W}^{(l)}$ is the Brownian motion driving $X^{(l)}$, $l=1,2$. Further denote by $i_n^{(l)}(s)$ the index of the interval $\mathcal{I}^{(l)}_{i,n}$ associated with $s \in \mathcal{I}_{i,n}^{(l)} $, $l=1,2$, and set
\begin{align*}
&\eta^{(l)}_{k',n,-}(s)=\sum_{j\geq k':t_{j,n}^{(3-l)} \leq T }\mathds{1}_{\{s \in \mathcal{I}_{j,k',n}^{(3-l)}\}} \sum_{i:\mathcal{I}_{i,k',n}^{(l)} \cap \mathcal{I}_{j,k',n}^{(3-l)} \neq \emptyset} (\Delta_{i,k',n}^{(l)} \overline{W}^{(l)})^2 \mathds{1}_{\{i<i_n^{(l)}(s)\}},
\\& \eta^{(l)}_{k',n,+}(s)=\sum_{j\geq k':t_{j,n}^{(3-l)} \leq T }\mathds{1}_{\{s \in \mathcal{I}_{j,k',n}^{(3-l)}\}} \sum_{i:\mathcal{I}_{i,k',n}^{(l)} \cap \mathcal{I}_{j,k',n}^{(3-l)} \neq \emptyset} (\Delta_{i,k',n}^{(l)} \overline{W}^{(l)})^2 \mathds{1}_{\{i \geq i_n^{(l)}(s)+k'\}},
\\&\Lambda_{k',n}^{(l)}(s)= \big(\eta^{(l)}_{k',n,-}(s),\Delta_{i_n^{(l)}(s)-k'+1,n}^{(l)}\overline{W}^{(l)},\ldots,\Delta_{i_n^{(l)}(s)-1,n}^{(l)}\overline{W}^{(l)},s- t^{(l)}_{i_n^{(l)}(s)-1,n} ,
\\&~~~~~~~~~~~~~~~~~~~~~~~~~~~t^{(l)}_{i_n^{(l)}(s),n}-s , \Delta_{i_n^{(l)}(s)+1,n}^{(l)}\overline{W}^{(l)},\ldots,\Delta_{i_n^{(l)}(s)+k'-1,n}^{(l)}\overline{W}^{(l)}, \eta^{(l)}_{k',n,+}(s)\big)
\end{align*}
for $k'=1,k$.

\begin{condition}\label{cond_cons_CoJ2}
The assumptions from Condition \ref{cond_cons_CoJ} hold and $\sigma$ almost surely fulfills $\int_0^T|\sigma_s^{(1)}\sigma_s^{(2)}|ds>0$. Further:
\begin{enumerate}[leftmargin=*, align=LeftAlignWithIndent, itemsep=-0.1cm, font=\normalfont, label=(\roman*)]
\item
It holds
\begin{align*}
|\pi_n|_T=o_\mathbb{P}(n^{-1/2}).
\end{align*}
\item
The functions $\widetilde{G}_{1,n}(t),\widetilde{G}_{k,n}(t), H_{1,n}(t), H_{k,n}(t)$ converge pointwise on $[0,T]$ in probability to strictly increasing functions $\widetilde{G}_1,\widetilde{G}_k,H_1,H_k:[0,\infty) \rightarrow [0,\infty)$.
\item The integral
\begin{multline*}
\int_{[0,T]^{P_1+P_2}} g(x_1,\dots,x_{P_1},x_1',\dots,x_{P_2}') \mathbb{E} \Big[ \prod_{p=1}^{P_1} h^{(1)}_p \big(n \Lambda_{1,n}^{(1)}(x_p),\frac{n}{k^2}\Lambda_{k,n}^{(1)}(x_p)\big)
\\ \times \prod_{p=1}^{P_2} h^{(2)}_p \big(n \Lambda_{1,n}^{(2)}(x_p'),\frac{n}{k^2}\Lambda_{k,n}^{(2)}(x_p')\big)\Big] dx_1 \dots dx_{P_1}dx_1' \dots dx_{P_2}'
\end{multline*}
converges for $n \rightarrow \infty$ to
\begin{multline*}
\int_{[0,T]^{P_1+P_2}} g(x_1,\dots,x_{P_1},x_1',\dots,x_{P_2}')\prod_{p=1}^{P_1} \int_{\mathbb{R}^2} h^{(1)}_p \big(y\big)\Gamma^{(1)}(x_p,dy)
\\ \times \prod_{p=1}^{P_2} \int_{\mathbb{R}^2} h^{(2)}_p \big(y'\big)\Gamma^{(2)}(x_p',dy')dx_1 \dots dx_{P_1}dx_1' \dots dx_{P_2}'
\end{multline*}
for all bounded continuous functions $g:\mathbb{R}^{P_1+P_2} \rightarrow \mathbb{R}$, $h^{(l)}_p:\mathbb{R}^{2(k+1)}\rightarrow \mathbb{R}$, $p=1,\ldots,P_l$, and any $P_l \in \mathbb{N}$, $l=1,2$. Here $\Gamma^{(l)}(\cdot,dy)$, $l=1,2,$ are families of probability measures on $[0, T]$ where $(\Lambda_1^{(l)}(x),\Lambda_k^{(l)}(x))\sim \Gamma^{(l)}(\cdot,x)$ has first moments which are uniformly bounded in $x$. Further the components of $\Lambda_k^{(l)}(x)$ which correspond to the $\Delta^{(l)}_{i_n^{(l)}(x)+j,n}\overline{W}^{(l)}$ have uniformly bounded second moments.
\end{enumerate}
\end{condition}

In order to describe the limit of $\Phi^{(CoJ)}_{k,T,n}$ on $\Omega^{(nCoJ)}_T$ we define
\begin{align*}
&C_{k',T}^{(CoJ)}=\int_0^T 2(\rho_s\sigma_s^{(1)}\sigma_s^{(2)})^2 d\widetilde{G}_{k'}(s)+\int_0^T(\sigma_s^{(1)}\sigma_s^{(2)})^2 dH_{k'}(s), \quad k'=1,k,
\\&D_{{k'},T}^{(CoJ)}=\sum_{p:S_p \leq T} \big((\Delta X^{(1)}_{S_p})^2 R_{k'}^{(2)}(S_p)+(\Delta X^{(2)}_{S_p})^2 R_{k'}^{(1)}(S_p)\big),\quad k'=1,k.
\end{align*}
Here $(S_p)_{p \geq 0}$ is an enumeration of the jump times of $X$ and $R_{k'}^{(l)}(s)$ is defined by
\begin{multline*}
R_{k'}^{(l)}(s)=(\sigma_{s-}^{(l)})^2 \eta_{{k'},-}^{(l)}(s)+
\sum_{i=1}^{k'} \Big(\sigma_{s-}^{(l)}\sum_{j=-{k'}+i}^{-1} \chi_j +\sigma_{s-}^{(l)}\sqrt{\delta_-(s)}U_-(s)
\\+\sigma_{s}^{(l)}\sqrt{\delta_+(s)}U_+(s)+\sigma_{s}^{(l)}\sum_{j=1}^{i-1}\chi_j \Big)^2
+(\sigma_{s}^{(l)})^2 \eta_{{k'},+}^{(l)}(s),~~ l=1,2,~~k'=1,k,
\end{multline*}
where 
\begin{multline*}
(\Lambda_1^{(l)}(s),\Lambda_k^{(l)}(s))=\big((\eta_{1,-}^{(l)}(s),\delta_-(s),\delta_+(s),\eta_{1,+}^{(l)}(s)),\\(\eta_{k,-}^{(l)}(s),\chi_{-k+1},\ldots, \chi_{-1},\delta_-(s),\delta_+(s),\chi_1,\ldots,\chi_{k-1},\eta_{k,+}^{(l)}(s))\big)
\end{multline*}
are random variables defined on an extended probability space $(\widetilde{\Omega},\widetilde{\mathcal{F}},\widetilde{\mathbb{P}})$ whose distribution is given by
\begin{align*}
\widetilde{\mathbb{P}}^{(\Lambda_1^{(l)}(s),\Lambda_k^{(l)}(s))}(dy) = \Gamma^{(l)}(s,dy),~~ l=1,2.
\end{align*}
The $(\Lambda_1^{(l)}(s),\Lambda_k^{(l)}(s))$, $s \in [0,T]$, are independent of each other and independent of the process $X$ and its components. The random variables $U_-(s),U_+(s)$ are i.i.d.\ standard normal and defined on $(\widetilde{\Omega},\widetilde{\mathcal{F}},\widetilde{\mathbb{P}})$ as well. They are also independent of $\mathcal{F}$ and of $(\Lambda^{(l)}(s),\Lambda_k^{(l)}(s))$.

\begin{theorem}\label{theo_cons_CoJ2} 
Under Condition \ref{cond_cons_CoJ2} we have the $\mathcal{X}$-stable convergence 
\begin{align}\label{theo_cons_CoJ2_conv}
\Phi_{k,T,n}^{(CoJ)} \overset{\mathcal{L}-s}{\longrightarrow} 
\frac{D_{k,T}^{(CoJ)}+k C_{k,T}^{(CoJ)} }{D_{1,T}^{(CoJ)}+C_{1,T}^{(CoJ)}}
\end{align}
on $\Omega_T^{(nCoJ)}$.
\end{theorem}

From Theorem \ref{theo_cons_CoJ1} we know that $\Phi_{k,T,n}^{(CoJ)}$ converges to $1$ under the null hypothesis and from Theorem \ref{theo_cons_CoJ2} we can conclude that $\Phi_{k,T,n}^{(CoJ)}$ converges under the alternative to a limit which is almost surely different from $1$ if $k C_{k,T}^{(CoJ)} \neq C_{1,T}^{(CoJ)}$, or under mild additional conditions if there is at least one jump in one of the components of $X$ on $[0,T]$. Indeed if we have $k C_{k,T}^{(CoJ)} \neq C_{1,T}^{(CoJ)}$, $\Phi_{k,T,n}^{(CoJ)}$ has to be different from $1$ because $C_{k,T}^{(CoJ)},C_{1,T}^{(CoJ)}$ are $\mathcal{F}$-conditionally constant and $D_{k,T}^{(CoJ)}$ and $D_{1,T}^{(CoJ)}$ are either zero or their $\mathcal{F}$-conditional distributions admit densities. If $k C_{k,T}^{(CoJ)} = C_{1,T}^{(CoJ)}$ holds mild conditions guarantee $D_{k,T}^{(CoJ)} \neq D_{1,T}^{(CoJ)}$ almost surely which also yields $\Phi_{k,T,n}^{(CoJ)} \neq 1$. Hence we will construct a test with critical region
\begin{align}\label{critical_region_CoJ}
\mathcal{C}_{k,T,n}^{(CoJ)}=\big\{|\Phi_{k,T,n}^{(CoJ)}-1|>\mathpzc{c}_{k,T,n}^{(CoJ)}\big\}
\end{align}
for a (possibly random) sequence $(\mathpzc{c}_{k,T,n}^{(CoJ)})_{n \in \mathbb{N}}$.\\

As in Section \ref{Test_J} we will consider the situation where the observation times are generated by Poisson processes as an example for a random and irregular sampling scheme which fulfills the conditions we need for the testing procedure to work.

\begin{example}\label{example_poisson_2d} Consider the extension of the setting from Example \ref{example_poisson1} to two dimensions where the observation times of $X^{(1)}$ and $X^{(2)}$ originate from the jump times of two independent Poisson processes, i.e.\ the $t_{i,n}^{(l)}-t_{i-1,n}^{(l)}$ are $Exp(n \lambda_l)$-distributed with $\lambda_l>0$ and independent for $i,n \geq 1$ and $l=1,2$. 

That Condition \ref{cond_cons_CoJ2}(i) is fulfilled in this setting follows from \eqref{poiss_mesh}. The convergences of $G_{1,n}$ and $H_{1,n}$ have been shown in Proposition 1 of \cite{HayYos08}, the convergences of $G_{k,n}$ and $H_{k,n}$ are proved in Lemma \ref{G_k_H_k_cons}. That Condition \ref{cond_cons_CoJ2}(iii) is fulfilled in the Poisson setting is proved in Section \ref{sec:proof_poiss_clt}.
\end{example}

\subsection{Central limit theorem}\label{sec:clt_CoJ}

Before we start to derive a central limit theorem for $\Phi_{k,T,n}^{(CoJ)}$ on $\Omega_T^{(CoJ)}$, we restrict ourselves to the case $k=2$. This simplifies notation, and also from the simulation results in Section \ref{sec:simul_J} for the test from Section \ref{Test_J} it seems to be optimal to choose $k$ as small as possible.

\begin{figure}[bt]
\centering
\hspace{-0.6cm}
\begin{tikzpicture}
\draw[dashed] (0,1) -- (12,1)
			(0,-1) -- (12,-1)
			(6,-1.2) -- (6,1.2);
\draw (1.5,1) -- (11,1)
	(1,-1) -- (10,-1)
(1.5,0.8) -- (1.5,1.2)
	(5,0.8) -- (5,1.2)
	(7,0.8) -- (7,1.2)
	(11,0.8) -- (11,1.2)
	(1,-0.8) -- (1,-1.2)
	(3.5,-0.8) -- (3.5,-1.2)
	(8,-0.8) -- (8,-1.2)
	(10,-0.8) -- (10,-1.2);
\draw	(-0.5,1) node{$X^{(1)}$}
		(-0.5,-1) node{$X^{(2)}$}
		(6.4,0) node{$s$};
\draw[decorate,decoration={brace,amplitude=12pt}]
	(5,1.25)--(7,1.25) node[midway, above,yshift=10pt,]{$\big|\mathcal{I}^{(1)}_{i_n^{(1)}(s),n}\big|$};
\draw[decorate,decoration={brace,amplitude=12pt}]
	(8,-1.25)--(3.5,-1.25) node[midway, below,yshift=-10pt,]{$\big|\mathcal{I}^{(2)}_{i_n^{(2)}(s),n}\big|$};
	
\draw[decorate,decoration={brace,amplitude=12pt}]
	(1.5,1.25)--(5,1.25) node[midway, above,yshift=10pt,]{$n^{-1}\mathcal{L}_n^{(1)}(s)$};
\draw[decorate,decoration={brace,amplitude=12pt}]
	(7,1.25)--(11,1.25) node[midway, above,yshift=10pt,]{$n^{-1}\mathcal{R}_n^{(1)}(s)$};
	
\draw[decorate,decoration={brace,amplitude=12pt}]
	(3.5,-1.25)--(1,-1.25) node[midway, below,yshift=-10pt,]{$n^{-1}\mathcal{L}_n^{(2)}(s)$};
\draw[decorate,decoration={brace,amplitude=12pt}]
	(10,-1.25)--(8,-1.25) node[midway, below,yshift=-10pt,]{$n^{-1}\mathcal{R}_n^{(2)}(s)$};
	
\draw[decorate,decoration={brace,amplitude=12pt}]
	(3.5,0.3)--(1.5,0.3) node[midway, below,yshift=-10pt,]{$n^{-1}\mathcal{L}_n(s)$};
\draw[decorate,decoration={brace,amplitude=12pt}]
	(10,0.3)--(8,0.3) node[midway, below,yshift=-10pt,]{$n^{-1}\mathcal{R}_n(s)$};
\end{tikzpicture}
\caption{Illustration of the terms $\mathcal{L}_n^{(l)}(s),\mathcal{R}_n^{(l)}(s)$, $l=1,2$, and $\mathcal{L}_n(s),\mathcal{R}_n(s)$.}\label{fig_xi2}
\end{figure}
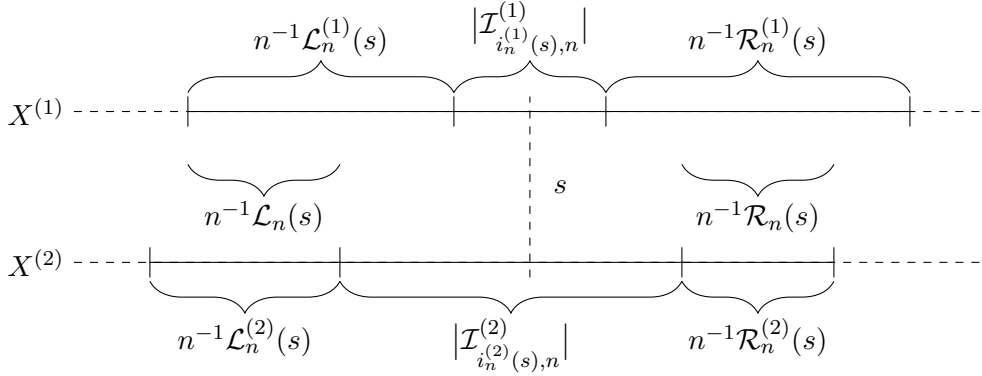

First we introduce the following notation to describe intervals around some time $s$ at which a common jump might occur
\begin{align*}
&\mathcal{L}_n^{(l)}(s)=n|\mathcal{I}^{(l)}_{i_n^{(l)}(s)-1,n}|,
~~\mathcal{R}_n^{(l)}(s)=n|\mathcal{I}^{(l)}_{i_n^{(l)}(s)+1,n}|,\quad l=1,2,
\\ &\mathcal{L}_n(s)=n|\mathcal{I}^{(1)}_{i_n^{(1)}(s)-1,n}\cap \mathcal{I}^{(2)}_{i_n^{(2)}(s)-1,n}|,~~
\mathcal{R}_n(s)=n|\mathcal{I}^{(1)}_{i_n^{(1)}(s)+1,n}\cap \mathcal{I}^{(2)}_{i_n^{(2)}(s)+1,n}|
\end{align*}
which are illustrated in Figure \ref{fig_xi2} and
\begin{align*}
Z_n(s)=\big(\mathcal{L}_n^{(1)},\mathcal{R}_n^{(1)},\mathcal{L}_n^{(2)},\mathcal{R}_n^{(2)}, \mathcal{L}_n,\mathcal{R}_n\big)(s).
\end{align*}
Limits of these variables will occur in the central limit theorem. To ensure convergence of the $Z_n(s)$ we need to impose the following assumption on the observation scheme.

\FloatBarrier

\begin{condition}\label{cond_clt_coJ}
The process $X$ and the sequence of observation schemes $(\pi_n)_n$ fulfill Condition \ref{cond_cons_CoJ}, \ref{cond_cons_CoJ2}(i)-(ii). Further the integral
\begin{align*}
\int_{[0,T]^P} g(x_1,\dots,x_P) \mathbb{E} \Big[ \prod_{p=1}^{P} h_p (Z_n(x_p)) \Big] dx_1 \dots dx_P
\end{align*}
converges for $n \rightarrow \infty$ to
\begin{align*}
\int_{[0,T]^P} g(x_1,\dots,x_P)\prod_{p=1}^{P} \int_{\mathbb{R}} h_p \big(y \big)\Gamma(x_p,dy)
 dx_1 \dots dx_P
\end{align*}
for all bounded continuous functions $g:\mathbb{R}^{P} \rightarrow \mathbb{R}$, $h_p:\mathbb{R}^6\rightarrow \mathbb{R}$ and any $P \in \mathbb{N}$. Here $\Gamma(\cdot,dy)$ is a family of probability measures on $[0, T]$ with uniformly bounded first moments and $\int_0^T\Gamma(x,\{0\}^6)dx=0$.
\end{condition}

Using the limit distribution of $Z_n(s)$ implicitly defined in Condition \ref{cond_clt_coJ} we set
\begin{align*}
&F_{2,T}^{(CoJ)}=4\sum_{p:S_p \leq T} \Delta X^{(1)}_{S_p} \Delta X^{(2)}_{S_p}
\\&~~~~~~~~~ \times\Big[\Delta X^{(2)}_{S_p}\Big( \sigma_{S_p-}^{(1)}\sqrt{ 
\mathcal{L}({S_p})}U_{S_p}^{(1),-}+\sigma_{S_p}^{(1)}
\sqrt{\mathcal{R}({S_p})} U_{S_p}^{(1),+}
\\&~~~~~~~~~~~~~~~~~~~~~~~+\sqrt{(\sigma_{S_p-}^{(1)})^2(\mathcal{L}^{(1)}-\mathcal{L})({S_p})
+(\sigma_{S_p}^{(1)})^2(\mathcal{R}^{(1)}-\mathcal{R})({S_p}) } U_{S_p}^{(2)}\Big) 
\\&~~~~~~~~~~~+\Delta X^{(1)}_{S_p} \Big(\sigma_{S_p-}^{(2)}\rho_{S_p-}\sqrt{\mathcal{L}({S_p}) }U_{S_p}^{(1),-}+\sigma_{S_p}^{(2)}\rho_{S_p}\sqrt{\mathcal{R}({S_p}) }U_{S_p}^{(1),+}
\\&~~~~~~~~~~~~~~~~~~~~~~~+\sqrt{(\sigma_{S_p-}^{(2)})^2(1-(\rho_{S_p-})^2)\mathcal{L}({S_p})+(\sigma_{S_p}^{(2)})^2(1-(\rho_{S_p})^2)\mathcal{R}({S_p}) }U_{S_p}^{(3)}
\\&~~~~~~~~~~~~~~~~~~~~~~~+\sqrt{(\sigma_{S_p-}^{(2)})^2(\mathcal{L}^{(2)}-
\mathcal{L})({S_p})+(\sigma_{S_p}^{(2)})^2(\mathcal{R}^{(2)}-
\mathcal{R})({S_p})}U_{S_p}^{(4)}\Big)\Big].
\end{align*}
Here, $(S_p)_{p \geq 0}$ is an enumeration of the common jump times of $X^{(1)}$ and $X^{(2)}$, $\big(\mathcal{L}^{(1)},\mathcal{R}^{(1)},\mathcal{L}^{(2)},\mathcal{R}^{(2)},\mathcal{L},\mathcal{R}\big)(s) $ is distributed according to $\Gamma(s,\cdot)$ and the vectors $\big(U^{(1),-}_s, U^{(1),+}_s, U^{(2)}_s, U^{(3)}_s,U^{(4)}_s\big)$ are standard normally distributed and independent for different values of $s$. Similarly as for \eqref{def_FkT_J} we obtain that the infinite sum in the definition of $F_{2,T}^{(CoJ)}$ has a well-defined limit. Using $F_{2,T}^{(CoJ)}$ we are able to state the following central limit theorem.

\begin{theorem}\label{theo_clt_CoJ}
If Condition \ref{cond_clt_coJ} is fulfilled, we have the $\mathcal{X}$-stable convergence
\begin{align}\label{theo_clt_CoJ1}
\sqrt{n}\left( \Phi_{2,T,n}^{(CoJ)}-1\right) \overset{\mathcal{L}-s}{\longrightarrow} \frac{F_{2,T}^{(CoJ)}}{4 B_{T}^{(CoJ)}}
\end{align}
on the set $\Omega_T^{(CoJ)}$.
\end{theorem}

\begin{example}\label{example_poisson_2d_2}
Condition \ref{cond_clt_coJ} is fulfilled in the Poisson setting. The convergence of the integrals follows as Condition \ref{cond_cons_CoJ2}(iii) from Lemma \ref{lemma_poiss_cond_clt} which is stated and proved in Section \ref{sec:proof_poiss_clt}. 
\end{example}

\begin{example}
Condition \ref{cond_clt_coJ} is also fulfilled in the setting of equidistant observation times $t_{i,n}^{(1)}=t_{i,n}^{(2)}=i/n$. In that case we have $$\big(\mathcal{L}^{(1)},\mathcal{R}^{(1)},\mathcal{L}^{(2)},\mathcal{R}^{(2)},\mathcal{L},\mathcal{R}\big)(s)  =(1,1,1,1,1,1)$$ for any $s \in (0,T]$. Hence we get
\begin{align*}
F_{2,T}^{(CoJ)}&=4\sum_{p:S_p \leq T} \Delta X^{(1)}_{S_p} \Delta X^{(2)}_{S_p}
\Big[\Delta X^{(2)}_{S_p}\big( \sigma_{S_p-}^{(1)}U_{S_p}^{(1),-}+\sigma_{S_p}^{(1)} U_{S_p}^{(1),+} \big)
\\&~~~~~~~~~~~+\Delta X^{(1)}_{S_p} \big(\sigma_{S_p-}^{(2)}\rho_{S_p-}U_{S_p}^{(1),-}+\sigma_{S_p}^{(2)}\rho_{S_p} U_{S_p}^{(1),+}
\\&~~~~~~~~~~~~~~~~~~~~~~~+\sqrt{(\sigma_{S_p-}^{(2)})^2(1-(\rho_{S_p-})^2)+(\sigma_{S_p}^{(2)})^2(1-(\rho_{S_p})^2) }U_{S_p}^{(3)}
\big)\Big]
\end{align*}
and $F_{2,T}^{(CoJ)}$ is $\mathcal{X}$-conditionally normally distributed with mean zero and variance
\begin{align*}
&16\sum_{p:S_p \leq T} (\Delta X^{(1)}_{S_p})^2 (\Delta X^{(2)}_{S_p})^2
\Big[ \big( \Delta X^{(2)}_{S_p}\sigma_{S_p-}^{(1)}+\Delta X^{(1)}_{S_p} \sigma_{S_p-}^{(2)}\rho_{S_p-}\big)^2
\\&~~~~~~~~~~~~~~~+ \big( \Delta X^{(2)}_{S_p}\sigma_{S_p}^{(1)}+\Delta X^{(1)}_{S_p} \sigma_{S_p}^{(2)}\rho_{S_p}\big)^2
\\&~~~~~~~~~~~~~~~+(\Delta X^{(1)}_{S_p})^2\big((\sigma_{S_p-}^{(2)})^2(1-(\rho_{S_p-})^2)+(\sigma_{S_p}^{(2)})^2(1-(\rho_{S_p})^2)\big) \Big]
\end{align*}
which is similar to the result in Theorem 4.1(a) of \cite{JacTod09}.
\end{example}

\subsection{Testing procedure}\label{sec:testproc_CoJ}
In this section we will develop a testing procedure based on Theorem \ref{theo_clt_CoJ} for which we have to estimate the law of the limiting variable $F_{2,T}^{(CoJ)}$. Therefore we need to estimate the law of 
\begin{align*}
Z(S_p)=\big(\mathcal{L}^{(1)},\mathcal{R}^{(1)},\mathcal{L}^{(2)},\mathcal{R}^{(2)},\mathcal{L},\mathcal{R}\big)(S_p).
\end{align*}
for a time $S_p$ at which a common jump occurs. As in Section \ref{sec:testproc_J} we will apply a bootstrap method for this purpose.

Let $(L_n)_n$ and $(M_n)_n$ be sequences of natural numbers which tend to infinity. Set
\begin{align}\label{test_CoJ_Ln_Rn}
\widehat{\mathcal{L}}_{n,m}^{(l)}(s)=n\big|\mathcal{I}^{(l)}_{i_n^{(l)}(s)+V^{(l)}_{n,m}(s)-1,n}\big|,~
\widehat{\mathcal{R}}_{n,m}^{(l)}(s)=n\big|\mathcal{I}^{(l)}_{i_n^{(l)}(s)+V^{(l)}_{n,m}(s)+1,n}\big|,
\end{align}
and define $\widehat{\mathcal{L}}_{n,m}(s),\widehat{\mathcal{R}}_{n,m}(s)$ as the lengths of the overlapping parts of $\widehat{\mathcal{L}}_n^{(l)}(s)$, $l=1,2$, respectively $\widehat{\mathcal{R}}_n^{(l)}(s)$, $l=1,2$, where $V_{n,m}(s)=(V^{(1)}_{n,m}(s),V^{(2)}_{n,m}(s)) \in \{-L_n,\ldots,L_n\} \times \{-L_n,\ldots,L_n\}$ and 
\begin{multline}\label{weights_testing_CoJ}
\mathbb{P}(V_{n,m}(s) =(l_1,l_2)|\mathcal{S}) \nonumber
\\=|\mathcal{I}^{(1)}_{i^{(1)}_n(s)+l_1,n}\cap \mathcal{I}^{(2)}_{i^{(2)}_n(s)+l_2,n}|\Big(\sum_{j_1,j_2=-L_n}^{L_n} |\mathcal{I}^{(1)}_{i^{(1)}_n(s)+j_1,n}\cap \mathcal{I}^{(2)}_{i^{(2)}_n(s)+j_2,n}| \Big)^{-1},
\end{multline}
$l_1,l_2 \in \{-L_n,\ldots,L_n\}$. Because we are in a two-dimensional setting, we have to simultaneously choose observation intervals of $X^{(1)}$ and $X^{(2)}$ around which the intervals used for simulating $Z(s)$ lie. Here we choose the intervals $(\mathcal{I}_{i,n}^{(1)},\mathcal{I}_{j,n}^{(2)})$ by a probability which is proportional to the length of the overlapping parts of those two intervals. As in Section \ref{sec:testproc_J} this corresponds to the probability with which a jump time $S_p$ uniformly distributed in time would fulfill $S_p \in \mathcal{I}_{i,n}^{(1)}$ and $S_p \in \mathcal{I}_{j,n}^{(2)}$.

Based on \eqref{test_CoJ_Ln_Rn} we define via
\begin{align*}
\widehat{Z}_{n,m}(s)=\big(\widehat{\mathcal{L}}_{n,m}^{(1)},\widehat{\mathcal{R}}_{n,m}^{(1)},\widehat{\mathcal{L}}_{n,m}^{(2)},\widehat{\mathcal{R}}_{n,m}^{(2)},\widehat{\mathcal{L}}_{n,m},\widehat{\mathcal{R}}_{n,m}\big)(s), \quad m=1,\ldots,M_n,
\end{align*}
estimators for realizations of $Z(s)$. Further we define using the components of $\widehat{Z}_{n,m}(s)$ and by considering increments which are larger than a certain threshold as jumps
\begin{align*}
&\widehat{F}_{T,n,m}^{(CoJ)}=4\sum_{i,j:t_{i,n}^{(1)} \wedge t_{j,n}^{(2)} \leq T} \Delta_{i,n}^{(1)} X^{(1)} \Delta_{j,n}^{(2)} X^{(2)}\mathds{1}_{\{|\Delta_{i,n}^{(1)} X^{(1)} |> \beta |\mathcal{I}_{i,n}^{(1)}|^{\varpi} \}}\mathds{1}_{\{|\Delta_{j,n}^{(2)} X^{(2)} |> \beta |\mathcal{I}_{j,n}^{(2)}|^{\varpi} \}}
\\&~~~~~~ \times\Big[\Delta_{j,n}^{(2)} X^{(2)}\Big( \hat{\sigma}^{(1)}(t_{i,n}^{(1)},-) 
\sqrt{\widehat{\mathcal{L}}_{n,m}({\tau_{i,j,n}})}U_{n,(i,j),m}^{(1,2),-}
\\&~~~~~~~~~~~~~~~~~~~~~~~~~~~~~~~~+\hat{\sigma}^{(1)}(t_{i,n}^{(1)},+)
\sqrt{\widehat{\mathcal{R}}_{n,m}({\tau_{i,j,n}})} U_{n,(i,j),m}^{(1,2),+}
\\&~~~~~~~~~~~~~~~~~~~~+\Big((\hat{\sigma}^{(1)}(t_{i,n}^{(1)},-))^2 (\widehat{\mathcal{L}}_{n,m}^{(1)}-\widehat{\mathcal{L}}_{n,m})(\tau_{i,j,n})
\\&~~~~~~~~~~~~~~~~~~~~~~~~~~~~~~~~+(\hat{\sigma}^{(1)}(t_{i,n}^{(1)},+))^2 (\widehat{\mathcal{R}}_{n,m}^{(1)}-\widehat{\mathcal{R}}_{n,m})({\tau_{i,j,n}}) \Big)^{1/2} U_{n,i,m}^{(1)}\Big)
\\&~~~~~~~~+\Delta_{i,n}^{(1)} X^{(1)} \Big(\hat{\sigma}^{(2)}(t_{j,n}^{(2)},-)\hat{\rho}(\tau_{i,j,n},-)\sqrt{\widehat{\mathcal{L}}_{n,m}({\tau_{i,j,n}}) }U_{n,(i,j),m}^{(1,2),-}
\\&~~~~~~~~~~~~~~~~~~~~~~~~~~~~~~~~+\hat{\sigma}^{(2)}(t_{j,n}^{(2)},+)\hat{\rho}(\tau_{i,j,n},+)\sqrt{\widehat{\mathcal{R}}_{n,m}(\tau_{i,j,n}) }U_{n,(i,j),m}^{(1,2),+}
\\&~~~~~~~~~~~~~~~~~~~~+\Big((\hat{\sigma}^{(2)}(t_{j,n}^{(2)},-))^2(1-(\hat{\rho}(\tau_{i,j,n},-))^2)\widehat{\mathcal{L}}_{n,m}(\tau_{i,j,n})
\\&~~~~~~~~~~~~~~~~~~~~~~~~~~~~~~~~+(\hat{\sigma}^{(2)}(t_{j,n}^{(2)},+))^2(1-(\hat{\rho}(\tau_{i,j,n},+))^2)\widehat{\mathcal{R}}_{n,m}({\tau_{i,j,n}}) \Big)^{1/2} U_{n,j,m}^{(2)}
\\&~~~~~~~~~~~~~~~~~~~~+\Big((\hat{\sigma}^{(2)}(t_{j,n}^{(2)},-))^2 (\widehat{\mathcal{L}}_{n,m}^{(2)}-
\widehat{\mathcal{L}}_{n,m})(\tau_{i,j,n})
\\&~~~~~~~~~~~~~~~~~~~~~~~~~~~~~~~~+(\hat{\sigma}^{(2)}(t_{j,n}^{(2)},+))^2 (\widehat{\mathcal{R}}_{n,m}^{(2)}-
\widehat{\mathcal{R}}_{n,m})(\tau_{i,j,n})\Big)^{1/2} U_{n,j,m}^{(3)}\Big)\Big]
\\&~~~~~~ \times\mathds{1}_{\{\mathcal{I}_{i,n}^{(1)}\cap \mathcal{I}_{j,n}^{(2)}\neq \emptyset \}}
\end{align*}
with $\tau_{i,j,n}=t_{i,n}^{(1)}\wedge t_{j,n}^{(2)}$. Here, we use \eqref{est_sigma} for the estimation of $\sigma_{s-}^{(l)}$, $\sigma_s^{(l)}$ and estimators for $\rho_{s-}$, $\rho_s$ defined by
\begin{align*}
\hat{\rho}_n(s,-)=\frac{\hat{\kappa}_n(s,-)}{\hat{\sigma}_n^{(1)}(s,-)\hat{\sigma}_n^{(2)}(s,-)}, \quad \hat{\rho}_n(s,+)=\frac{\hat{\kappa}_n(s,+)}{\hat{\sigma}_n^{(1)}(s,+)\hat{\sigma}_n^{(2)}(s,+)},
\end{align*}
where
\begin{align*}
&\hat{\kappa}_n(s,-)=\frac{1}{b_n}
\sum_{i,j:\mathcal{I}_{i,n}^{(1)} \cup \mathcal{I}_{j,n}^{(2)} \subset [s-b_n,s)} \Delta_{i,n}^{(1)} X^{(1)}\Delta_{i,n}^{(2)} X^{(2)} 
\mathds{1}_{\{\mathcal{I}_{i,n}^{(1)}\cap \mathcal{I}_{j,n}^{(2)}\neq \emptyset \}}
\\&\hat{\kappa}_n(s,+)=\frac{1}{b_n}
\sum_{i,j:\mathcal{I}_{i,n}^{(1)} \cup \mathcal{I}_{j,n}^{(2)} \subset [s,s+b_n]} \Delta_{i,n}^{(1)} X^{(1)}\Delta_{i,n}^{(2)} X^{(2)} 
\mathds{1}_{\{\mathcal{I}_{i,n}^{(1)}\cap \mathcal{I}_{j,n}^{(2)}\neq \emptyset \}},
\end{align*}
for a sequence $(b_n)_n$ with $b_n \rightarrow 0 $ and $|\pi_n|_T/ b_n \rightarrow 0$. The $U_{n,(i,j),m}^{(1,2),-}$, $U_{n,(i,j),m}^{(1,2),+}$, $U_{n,i,m}^{(1)}$, $U_{n,j,m}^{(2)}$, $U_{n,j,m}^{(3)}$ are i.i.d.\ standard normal and we have $\beta>0$ and $\varpi \in (0,1/2)$. 
Similarly as in Section \ref{sec:testproc_J} we denote by
\begin{align*}
\widehat{Q}_{T,n}^{(CoJ)}(\alpha)=\widehat{Q}_\alpha\big(\big\{
|\widehat{F}_{T,n,m}^{(CoJ)}| \big| m=1,\ldots,M_n
\big\}\big)
\end{align*}
the $\lfloor \alpha M_n \rfloor $-th largest element of the set $\{
|\widehat{F}_{T,n,m}^{(CoJ)}| | m=1,\ldots,M_n
\}$ and we will see that $\widehat{Q}_{T,n}^{(CoJ)}(\alpha)$ converges on $\Omega_T^{(CoJ)}$ under appropriate conditions to the $\mathcal{X}$-conditional $\alpha$-quantile $Q^{(CoJ)}(\alpha)$ of $|F_{2,T}^{(CoJ)} |$ which is defined via
\begin{align}
\widetilde{\mathbb{P}}\big(|F_{2,T}^{(CoJ)}| \leq Q^{(CoJ)}(\alpha) \big| \mathcal{X} \big)(\omega)=\alpha,~~ \omega \in \Omega_T^{(CoJ)},
\end{align}
and we set $Q^{(CoJ)}(\alpha)(\omega)=0$, $\omega \in (\Omega_T^{(CoJ)})^c$. Such a random variable $Q^{(CoJ)}(\alpha)$ exists because Condition \ref{cond_testproc_CoJ} will guarantee that the $\mathcal{X}$-conditional distribution of $F_{2,T}^{(CoJ)}$ is almost surely continuous on $\Omega_T^{(CoJ)}$.

\begin{comment}
\begin{remark}
Note that we could also build our estimator $\widehat{F}_{2,T,n,m}^{(CoJ)}$ from the term, where in the product $\Delta_{i,n}^{(1)}X^{(1)}\Delta_{j,n}^{(2)}X^{(2)}$ both increments have to be large. However, considering all terms where at least one process has a large increment leads to better empirical results. This is due to the fact that jumps that occur only in one component are by this choice included in the estimation. Although their contribution vanishes in the limit, they are terms of the next highest order after the common jumps. Thereby including them in our estimator reduces the approximation error from using the asymptotic distribution.
\end{remark}
\end{comment}

\begin{condition}\label{cond_testproc_CoJ}
The process $X$ and the sequence of observation schemes $(\pi_n)_n$ fulfill Condition \ref{cond_cons_CoJ2} and \ref{cond_clt_coJ}. Further the set $\{s \in [0,T]:\|\sigma_s\|=0\}$ is almost surely a Lebesgue null set. $(L_n)_n$ and $(M_n)_n$ are sequences of natural numbers converging to infinity with $ L_n/n\rightarrow 0$ and the sequence $(b_n)_n$ fulfills $|\pi_n|_T/b_n \overset{\mathbb{P}}{\longrightarrow}0$. Additionally,
\begin{enumerate}[leftmargin=*, align=LeftAlignWithIndent, itemsep=-0.1cm, font=\normalfont, label=(\roman*)]
\Item
\begin{align}\label{vert_kon1_CoJ}
\widetilde{\mathbb{P}}\Big(\big|\widetilde{\mathbb{P}}(\widehat{Z}_{n,1}(s_p)\leq x_p,~p=1,\ldots,P | \mathcal{S} ) - \prod_{p=1}^P \widetilde{\mathbb{P}}(Z(s_p) \leq x_p ) \big|>\varepsilon \Big) \rightarrow 0
\end{align}
as $n \rightarrow \infty$ for all $\varepsilon>0$ and any $x_p \in \mathbb{R}^{6}$, $s_p \in (0,T)$, $p=1,\ldots,P$, $P \in \mathbb{N}$, with $s_i\neq s_j$ for $i \neq j$.
\item It holds
\begin{align*}
\sqrt{n} \sum_{i,j:t_{i,n}^{(l)} \wedge t_{j,n}^{(3-l)}\leq T} |\mathcal{I}_{i,n}^{(l)}|^{1/2} |\mathcal{I}_{j,n}^{(3-l)}| \mathds{1}_{\{\mathcal{I}_{i,n}^{(l)} \cap \mathcal{I}_{j,n}^{(3-l)} \neq \emptyset \}}=O_\mathbb{P}(1), \quad l=1,2.
\end{align*}
\item The process $\sigma$ is a $2 \times 2$-dimensional matrix-valued It\^o semimartingale.
\item It holds $2 C_{2,T}^{(CoJ)}\neq C_{1,T}^{(CoJ)}$ or $D_{2,T}^{(CoJ)} \neq D_{1,T}^{(CoJ)}$ almost surely.
\end{enumerate}
\end{condition}

As for Condition \ref{cond_testproc_J} part (i) of Condition \ref{cond_testproc_CoJ} yields that $\widehat{Q}_{T,n}^{(CoJ)}(\alpha)$ consistently estimates $Q^{(CoJ)}(\alpha)$. Part (ii) is a technical condition needed in the proof of Theorem \ref{test_theo_CoJ}. Further part (iii) provides that $\hat{\sigma}_n^{(l)}(S_p,-)$, $\hat{\sigma}_n^{(l)}(S_p,+)$, $\hat{\rho}_n(S_p,-)$, $\hat{\rho}_n(S_p,+)$ are consistent estimators for $\sigma^{(l)}_{S_p-}$, $\sigma^{(l)}_{S_p}$, $\rho_{S_p-}$, $\rho_{S_p}$, and part (iv) guarantees that the limit under the alternative is almost surely different from $1$.

Denote by
\begin{align*}
\Omega_T^{(DisJ)}=\{\exists s \in [0,T] : \Delta X^{(1)}_s \neq 0 \vee \Delta X^{(2)}_s \neq 0\} \cap \Omega_T^{(CoJ)}
\end{align*}
the subset of $\Omega$ where either $X^{(1)}$ or $X^{(2)}$ has a jump in $[0,T]$ but there exists no common jump. This is the subset of $\Omega_T^{(nCoJ)}$ on which $D_{1,T}^{(CoJ)}$, $D_{2,T}^{(CoJ)}$ do not vanish under Condition \ref{cond_testproc_CoJ}(i)--(ii).

\begin{theorem}\label{test_theo_CoJ}
Let Condition \ref{cond_testproc_CoJ} be fulfilled. The test defined in \eqref{critical_region_CoJ} with
\begin{align}
\mathpzc{c}_{2,T,n}^{(CoJ)}= \frac{\widehat{Q}_{T,n}^{(CoJ)}(1-\alpha)}{\sqrt{n}4V(f,\pi_n)_T},~\alpha \in [0,1],
\end{align}
has asymptotic level $\alpha$ in the sense that we have
\begin{align}\label{test_theo_level_CoJ}
\widetilde{\mathbb{P}}\big(|\Phi_{2,T,n}^{(CoJ)}-1| > \mathpzc{c}_{2,T,n}^{(CoJ)} \big| F^{(CoJ)} \big) \rightarrow \alpha 
\end{align}
for all $F^{(CoJ)} \subset \Omega_T^{(CoJ)}$ with $\mathbb{P}(F^{(CoJ)})>0$. 

Further the test is consistent in the sense that we have
\begin{align}\label{test_theo_power_CoJ}
\widetilde{\mathbb{P}}\big(|\Phi_{2,T,n}^{(CoJ)}-1| > \mathpzc{c}_{2,T,n}^{(CoJ)} \big| F \big) \rightarrow 1 
\end{align}
for all $F \subset \Omega_T^{(DisJ)}$ with $\mathbb{P}(F)>0$, and even for all $F \subset \Omega_T^{(nCoJ)}$ with $\mathbb{P}(F)>0$ if we have $2 C_{2,T}^{(CoJ)} > C_T^{(CoJ)}$ almost surely. 
\end{theorem}

To improve the performance of our test in the finite sample we adjust the estimator $\Phi_{2,T,n}^{(CoJ)}$ similarly as for the test in Corollary \ref{test_cor_J} by (partially) subtracting terms that contribute in the finite sample but vanish asymptotically. Therefore we define for $\beta,\varpi$ as above the quantity
\begin{align*}
&A_{T,n}^{(CoJ)}
=n\sum_{i,j:t_{i,n}^{(1)} \wedge t_{j,n}^{(2)}\leq T} (\Delta_{i,2,n}^{(1)} X^{(1)})^2(\Delta_{j,2,n}^{(2)} X^{(2)})^2\mathds{1}_{\{\mathcal{I}_{i,2,n}^{(1)}\cap \mathcal{I}_{j,2,n}^{(2)}\neq \emptyset\}}
\\& ~~~~~~~~~~~~~~~~~~~~~~~~~~~~~~~~~~~~~~~~~~~~~~~~\times \mathds{1}_{\{|\Delta_{i,2,n}^{(1)} X^{(1)}|\leq \beta |\mathcal{I}_{i,2,n}^{(1)}|^\varpi \vee |\Delta_{j,2,n}^{(2)} X^{(2)}|\leq \beta |\mathcal{I}_{j,2,n}^{(2)}|^\varpi\}} 
\\&~~~~~~~~~~~~~~~~~~~~- 4 n\sum_{i,j:t_{i,n}^{(1)} \wedge t_{j,n}^{(2)}\leq T} (\Delta_{i,n}^{(1)} X^{(1)})^2(\Delta_{j,n}^{(2)} X^{(2)})^2\mathds{1}_{\{\mathcal{I}_{i,n}^{(1)}\cap \mathcal{I}_{j,n}^{(2)}\neq \emptyset\}}
\\& ~~~~~~~~~~~~~~~~~~~~~~~~~~~~~~~~~~~~~~~~~~~~~~~~\times \mathds{1}_{\{|\Delta_{i,n}^{(1)} X^{(1)}|\leq \beta |\mathcal{I}_{i,n}^{(1)}|^\varpi \vee |\Delta_{j,n}^{(2)} X^{(2)}|\leq \beta |\mathcal{I}_{j,n}^{(2)}|^\varpi\}} .
\end{align*}
Using this expression we then define for $\rho \in (0,1)$ by
\begin{align*}
\widetilde{\Phi}_{2,T,n}^{(CoJ)}(\rho)=\Phi_{2,T,n}^{(CoJ)}-\rho\frac{n^{-1}A_{T,n}^{(CoJ)}}{4 V(f,\pi_n)_T}
\end{align*}
the adjusted estimator.

\begin{corollary}\label{test_cor_CoJ} 
Let $\rho \in (0,1)$. If Condition \ref{cond_testproc_CoJ} is fulfilled, it holds with the notation of Theorem \ref{test_theo_CoJ}
\begin{align}\label{test_cor_level_CoJ}
\widetilde{\mathbb{P}}\big(|\widetilde{\Phi}_{2,T,n}^{(CoJ)}(\rho)-1| > \mathpzc{c}_{2,T,n}^{(CoJ)} \big| F^{(CoJ)} \big) \rightarrow \alpha 
\end{align}
for all $F^{(CoJ)} \subset \Omega_T^{(CoJ)}$ with $\mathbb{P}(F^{(CoJ)})>0$ and
\begin{align}\label{test_cor_power_CoJ}
\widetilde{\mathbb{P}}\big(|\widetilde{\Phi}_{2,T,n}^{(CoJ)}(\rho)-1| > \mathpzc{c}_{2,T,n}^{(CoJ)} \big| F \big) \rightarrow 1 
\end{align}
for all $F \subset \Omega_T^{(DisJ)}$ with $\mathbb{P}(F)>0$, and even for all $F \subset \Omega_T^{(nCoJ)}$ with $\mathbb{P}(F)>0$ if we have $2 C_{2,T}^{(CoJ)}>C_T^{(CoJ)}$ almost surely. 
\end{corollary}

\begin{example}\label{example_poisson_2d_3}
The assumptions on the observation scheme from Condition \ref{cond_testproc_CoJ} are fulfilled in the setting of Poisson sampling introduced in Example \ref{example_poisson2}. That part (i) holds is shown in Section \ref{sec:proof_poiss_test} and that Condition \ref{cond_clt_coJ} is fulfilled is proven in Section \ref{sec:proof_poiss_clt}. That part (ii) of Condition \ref{cond_testproc_CoJ} is fulfilled can be proven in the same way as $\widetilde{G}_{k,n}(t) \overset{\mathbb{P}}{\longrightarrow} \widetilde{G}_k$ and $\widetilde{H}_{k,n}(t) \overset{\mathbb{P}}{\longrightarrow} \widetilde{H}_k$ were shown; compare Lemma \ref{G_k_H_k_cons}.
\end{example}

Similarly as discussed in Section \ref{sec:testproc_J} after Example \ref{example_poiss_test_1d} we can omit the weighting in \eqref{weights_testing_CoJ} for obtaining a working testing procedure also within the two-dimensional Poisson setting. However, in the setting where both processes $X^{(1)}$ and $X^{(2)}$ are observed at the observation times introduced in Example \ref{example_alpha_test_1d} with different $\alpha_1 \neq \alpha_2$ it can be easily verified that the weighting in \eqref{weights_testing_CoJ} is necessary and leads to a correct estimation of the distribution of $F_{2,T}^{(CoJ)}$.

\section{Simulation results}\label{sec:simul}
We conduct simulation studies to verify the effectivity and to study the finite sample properties of the developed tests.

\subsection{Testing for jumps}\label{sec:simul_J}
In our simulation the observation times originate from a Poisson process with intensity $n$ which corresponds to $\lambda=1$ in Example \ref{example_poisson1} and yields on average $n$ observations in the time interval $[0,1]$. We simulate from the model
\begin{equation*}
dX_t=X_t \sigma d W_t+\alpha \int_{\mathbb{R}} X_{t-}x \mu(dt,dx)
\end{equation*}
where $X_0=1$ and the Poisson measure $\mu$ has the predictable compensator 
\[
\nu(dt,dx)=\kappa\frac{\mathds{1}_{[-h,-l]\cup [l,h]}(x)}{2(h-l)}dtdx.
\]
This model is a one-dimensional version of the model used in Section 6 of \cite{JacTod09}. 

We consider the parameter settings displayed in Table \ref{par_settings_J} with $\sigma^2= 8 \times 10^{-5}$ in all cases.
\begin{table}[h]
\centering
\begin{tabular}{lcccc}
\hline 
 & \multicolumn{4}{c}{Parameters} \\
\cline{2-5}
Case & $\alpha$ & $\kappa$ & $l$ & $h$ \\ 
\hline 
I-j & $0.01$ & $1$ & $0.05$ & $0.7484$ \\ 

II-j & $0.01$ & $5$ & $0.05$ & $0.3187$ \\ 

III-j & $0.01$ & $25$ & $0.05$ & $0.1238$ \\ 

Cont & $0.00$ & & & \\ 

\hline 
\end{tabular} 
\caption[]{Parameter settings for the simulation.}
\label{par_settings_J}
\end{table}
We choose $n=100$, $n=\numprint{1600}$ and $n=\numprint{25600}$. In a trading day of $6.5$ hours this corresponds to observing $X^{(1)}$ and $X^{(2)}$ on average every $4$ minutes, every $15$ seconds and every second. Further we set $\beta=0.03$, $\omega=0.49$ and use $b_n= 1/\sqrt{n}$ for the local interval in the estimation of $\sigma_s$ and $L_n=\lfloor \ln (n)\rfloor$, $M_n=\lfloor 10 \sqrt{n} \rfloor$ in the simulation of the $\hat{\xi}_{k,n,m,-}(s)$, $\hat{\xi}_{k,n,m,+}(s)$. For the choice of these parameters see Section 5 of \cite{MarVet17}. The cases \texttt{I-j} to \texttt{III-j} correspond to the presence of jumps of diminishing size. When there are smaller jumps we choose a situation where there are more jumps such that the overall contribution of the jumps to the quadratic variation is roughly the same in all three cases. The fraction of the quadratic variation which originates from the jumps matches the one estimated in real financial data from \cite{TodTau11}. In all three cases where the model allows for jumps we only use paths in the simulation study where jumps were realized. In the fourth case \texttt{Cont} we consider a purely continuous model. The parameter values are taken from \cite{JacTod09}.\\

We applied the two testing procedures from Theorem \ref{test_theo_J} and Corollary \ref{test_cor_J} for $k=2,3,5$ and the results are displayed in Figures \ref{figures_J} and \ref{figures_J_corr}. In Figure \ref{figures_J} the results from Theorem \ref{test_theo_J} are presented. In the left column the empirical rejection rates are plotted against the theoretical value of the test and in the right column we show estimated density plots based on the simulated values of $\Phi_{k,T,n}^{(J)}$. In Figure \ref{figures_J_corr} we present the results from the test in Corollary \ref{test_cor_J} for $\rho=\numprint{0.9}$ (for the choice of $\rho=\numprint{0.9}$ see Figure \ref{figure_jump_choice_rho}) in the same way. The density plots here show the estimated density of $\widetilde{\Phi}_{k,T,n}^{(J)}(\rho)$.

In Figure \ref{figures_J} we observe in the case \texttt{Cont} that the power of the test from Theorem \ref{test_theo_J} is very good for $n=\numprint{1600}$ and $n=\numprint{25600}$. Further the empirical rejection rates match the asymptotic values rather well for all considered values of $k$ in the cases \texttt{I-j} and \texttt{II-j} at least for the highest observation frequency corresponding to $n=\numprint{25600}$. However in the case \texttt{III-j} there is a severe over-rejection even for $n=\numprint{25600}$. In general we observe over-rejection in all cases. The empirical rejection rates match the asymptotic values better in the cases where there are on average larger jumps. Further the results are better the smaller $k$ is. Note that for $n=100$ the cases \texttt{III-j} and \texttt{Cont} are not distinguishable using our test as the rejection curves for $n=100$ in those two cases are almost identical. The density plots show the convergence of $\Phi_{k,T,n}^{(J)}$ to $1$ in the presence of jumps and to $(k+1)/2$ under the alternative as predicted from Example \ref{example_poisson1}.

In Figure \ref{figures_J_corr} we immediately see that the observed rejection rates from the test from Corollary \ref{test_cor_J} match the asymptotic values much better than those from Theorem \ref{test_theo_J}. Hence adjusting the estimator has a huge effect on the finite sample performance of the test. Here the observed rejection rates match the asymptotic values quite well in the case \texttt{III-j} at least for $n=\numprint{25600}$ and in the cases \texttt{I-j} and \texttt{II-j} we get already for $n=\numprint{1600}$ very good results. The results in the case \texttt{Cont} show that the power remains to be very good. The density plots show that under the presence of jumps $\widetilde{\Phi}^{(J)}_{k,T,n}(\rho)$ is more centered around $1$ than $\Phi_{k,T,n}^{(J)}$. Under the alternative $\widetilde{\Phi}^{(J)}_{k,T,n}(\rho)$ clusters around the value $1+(1-\rho)(k-1)/2$ which is much closer to $1$ than $(k+1)/2$, but the observed values of $\widetilde{\Phi}^{(J)}_{k,T,n}(\rho)$ still seem to be large enough such that they can be well distinguished from $1$ as can be seen from the high empirical rejection rate.

Figure \ref{figure_jump_choice_rho} illustrates how the performance of the test from Corollary \ref{test_cor_J} depends on the choice of the parameter $\rho$. For this purpose we investigate for $k=2$ the empirical rejection rates in the cases \texttt{III-j} and \texttt{Cont} with $n=\numprint{25600}$ for the test with level $\alpha=5\%$ in dependence of $\rho$. We plot for $\rho \in [0,1]$ the empirical rejection rate under the null hypothesis in the case \texttt{III-j} which serves as a proxy for the type-I error of the test together with one minus the empirical rejection rate under the alternative hypothesis in the case \texttt{Cont} which serves as a proxy for the type-II error of the test. Finally we plot the sum of both error proxies to obtain an indicator for the overall performance of the test in dependence of $\rho$. As expected we observe a decrease in the type-I error as $\rho$ increases and an increase in the type-II error. While we observe an approximately linear decrease in the type-I error, the type-II error is equal to zero until $\numprint{0.8}$, then slightly increases and starts to steeply increase at $\rho=\numprint{0.9}$. As expected from the theory the type-II error converges to $1-\alpha$ for $\rho \rightarrow 1$. In this example the overall error is minimized for a relatively large value of $\rho$ close to $\rho=\numprint{0.9}$.

Further we carried out simulations for the same four parameter settings based on equidistant observation times $t_{i,n}=i/n$. In this specific setting the test from Theorem \ref{test_theo_J} coincides with tests discussed in \cite{aitjac2009} and in Chapter 10.3 of \cite{AitJac14}. The simulation results are presented in Figures \ref{figures_equi_J}, \ref{figures_equi_J_corr} and \ref{figure_jump_choice_rho_equi} in the same fashion as in Figures \ref{figures_J}, \ref{figures_J_corr} and \ref{figure_jump_choice_rho} for the irregular observations. We observe that the results both from Theorem \ref{test_theo_J} and Corollary \ref{test_cor_J} based on irregular observations are not significantly worse than those obtained in the simpler setting of equidistant observation times. Especially we can conclude that the adjustment technique introduced for Corollary \ref{test_cor_J} cannot only be used to improve the finite sample performance of our test based on irregular observations, but also can be used to improve existing tests in the literature which are based on equidistant observations.

\begin{figure}[!htb]
\centering
\includegraphics[width=10cm]{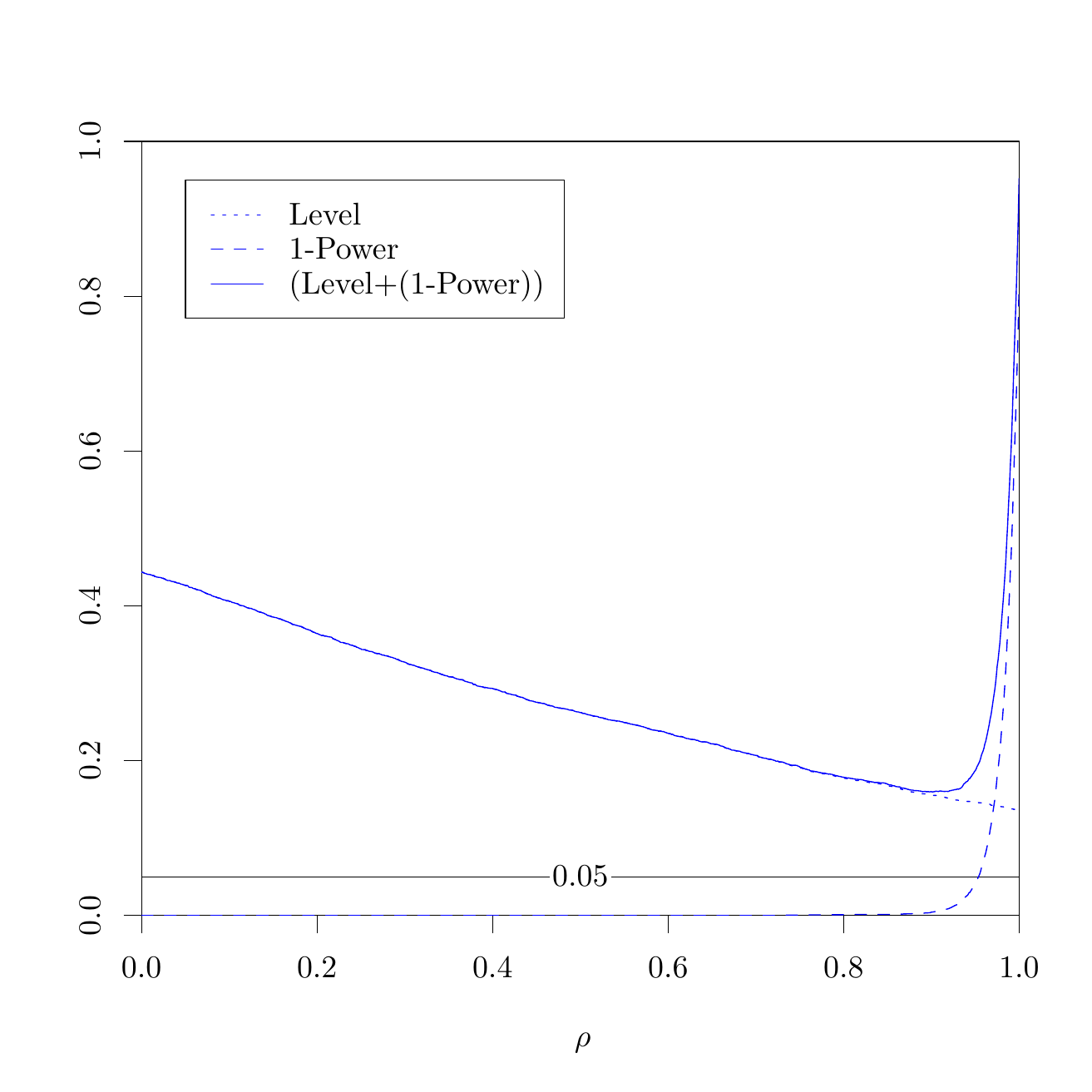}
\caption{This graphic shows for $k=2$, $\alpha=5\%$ and $n=\numprint{25600}$ the empirical rejection rate in the case \texttt{Cont} (dotted line) and $1$ minus the empirical rejection rate in the case \texttt{III-j} (dashed line) from the Monte Carlo simulation based on Corollary \ref{test_cor_J} as a function of $\rho \in [0,1]$.}\label{figure_jump_choice_rho}
\end{figure}

\begin{figure}[htb]
\centering
\begin{subfigure}[t]{0.45\textwidth}
\includegraphics[width=6.2cm]{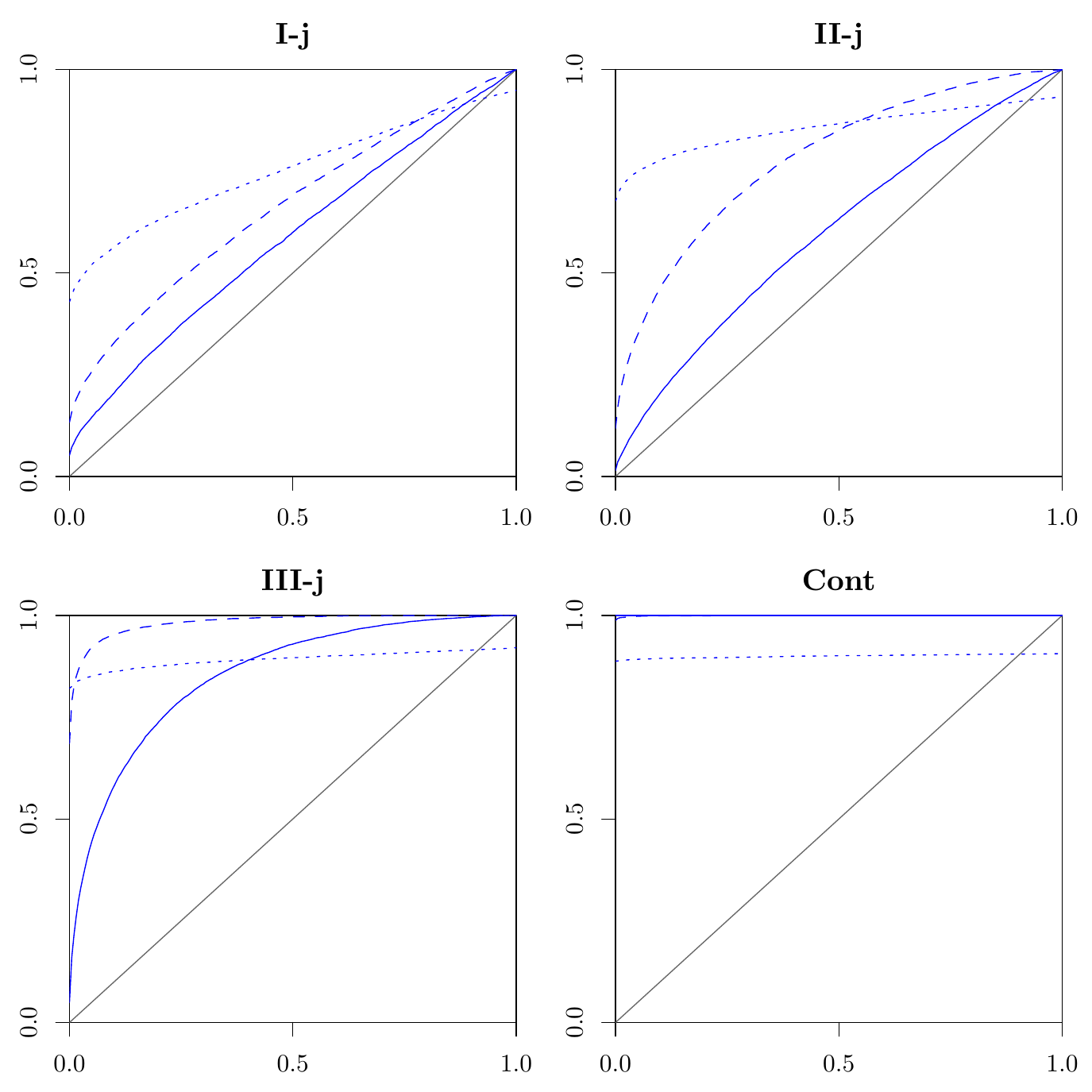}
\caption{Rejection curves from the Monte Carlo for $k=2$.}\label{fig_k2_1}
\end{subfigure}
\begin{subfigure}[t]{0.45\textwidth}
\includegraphics[width=6.2cm]{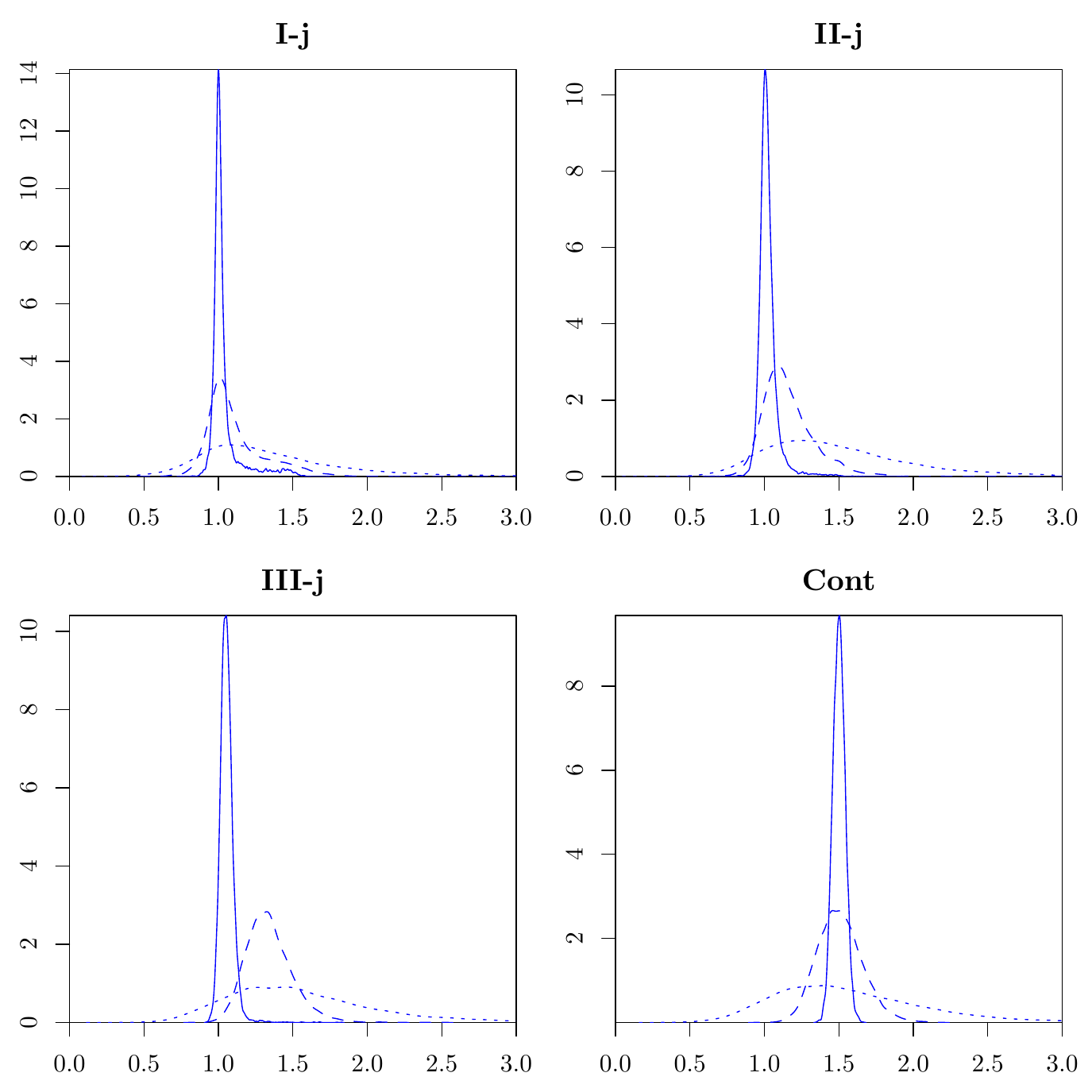}
\caption{Density estimation of $\Phi_{k,T,n}^{(J)}$ from the Monte Carlo for $k=2$.}\label{fig_k2_2}
\end{subfigure}

\begin{subfigure}[t]{0.45\textwidth}
\includegraphics[width=6.2cm]{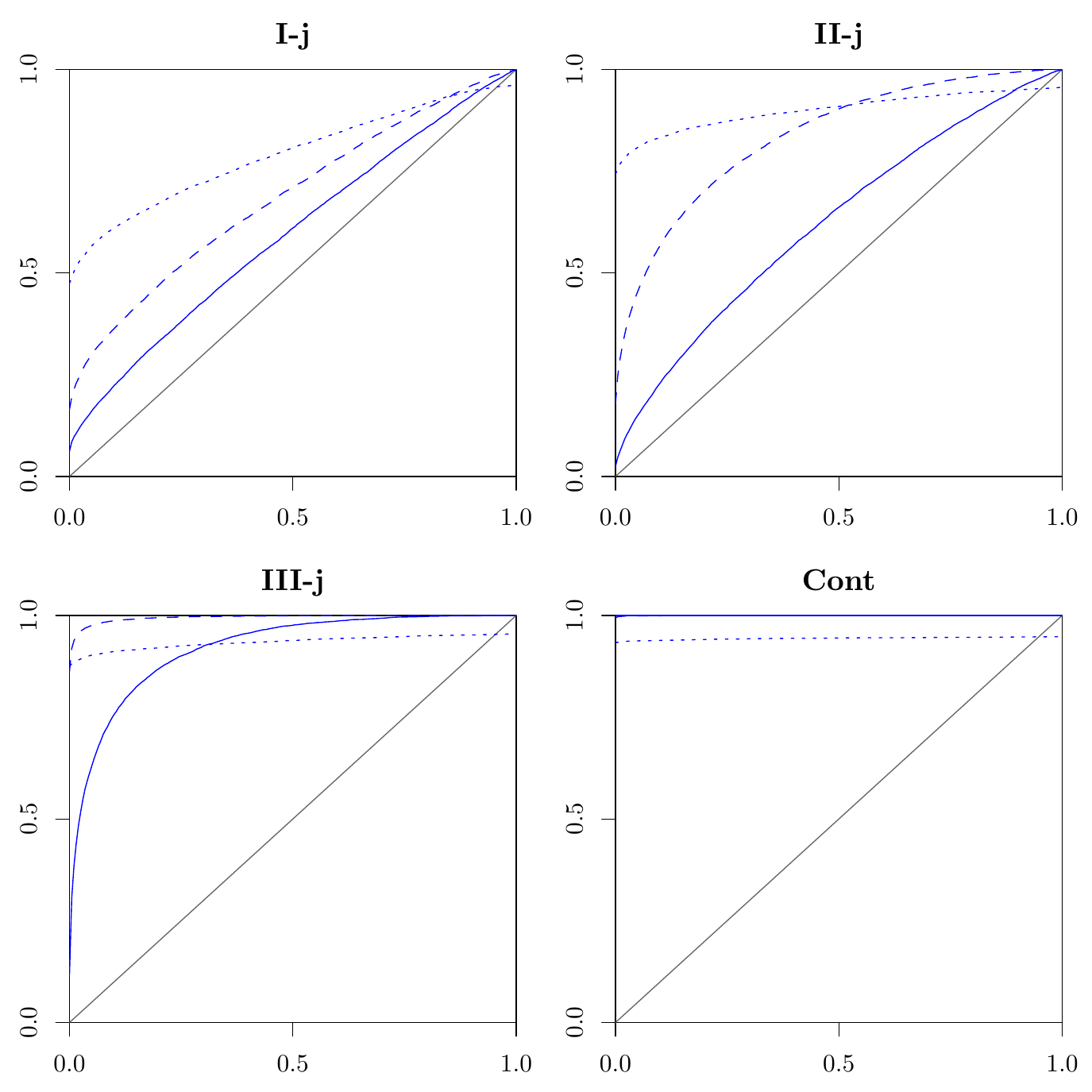}
\caption{Rejection curves from the Monte Carlo for $k=3$.}\label{fig_k3_1}
\end{subfigure}
\begin{subfigure}[t]{0.45\textwidth}
\includegraphics[width=6.2cm]{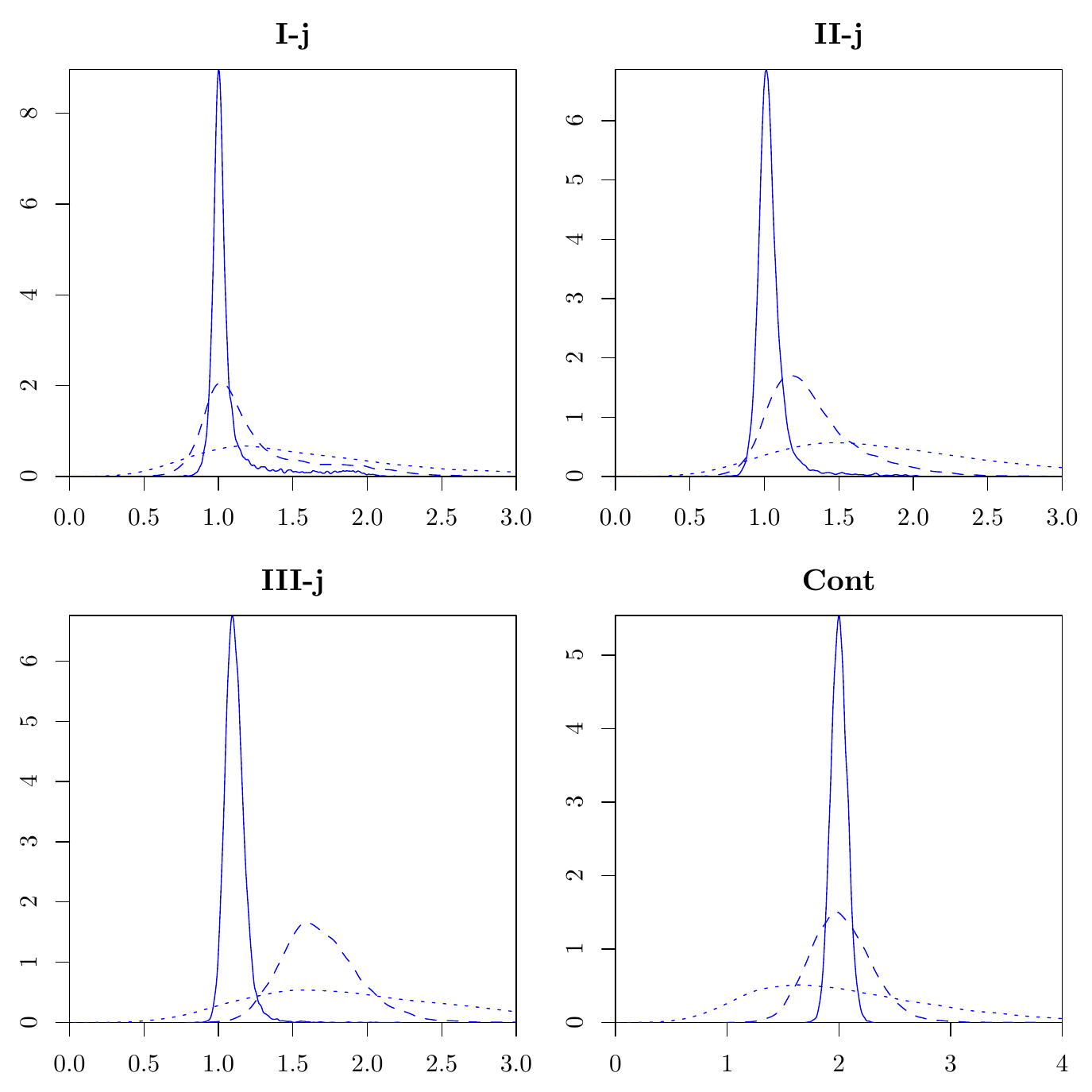} 
\caption{Density estimation of $\Phi_{k,T,n}^{(J)}$ from the Monte Carlo for $k=3$.}\label{fig_k3_2}
\end{subfigure}

\begin{subfigure}{0.45\textwidth}
\includegraphics[width=6.2cm]{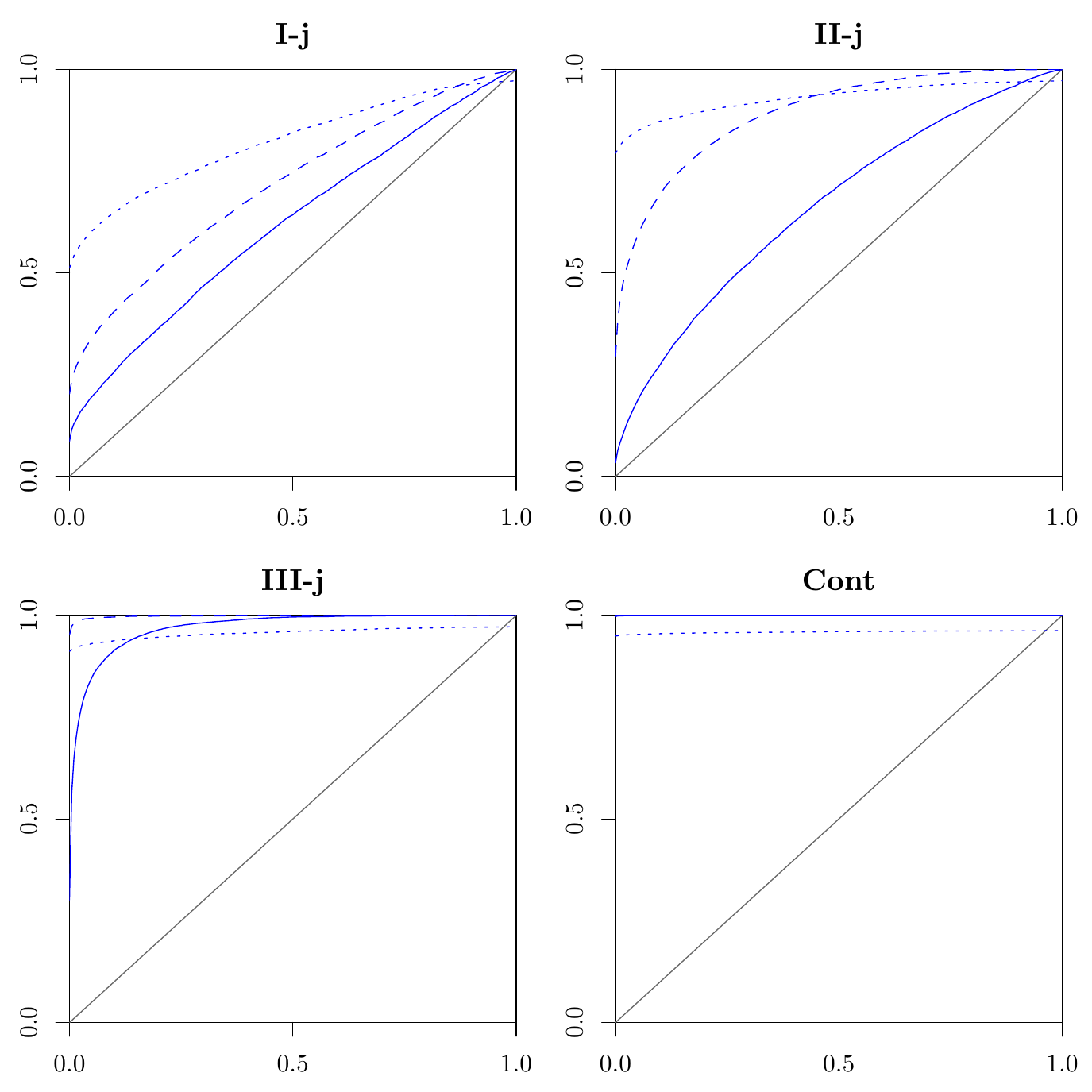}
\caption{Rejection curves from the Monte Carlo for $k=5$.}\label{fig_k5_1}
\end{subfigure}
\begin{subfigure}{0.45\textwidth}
\includegraphics[width=6.2cm]{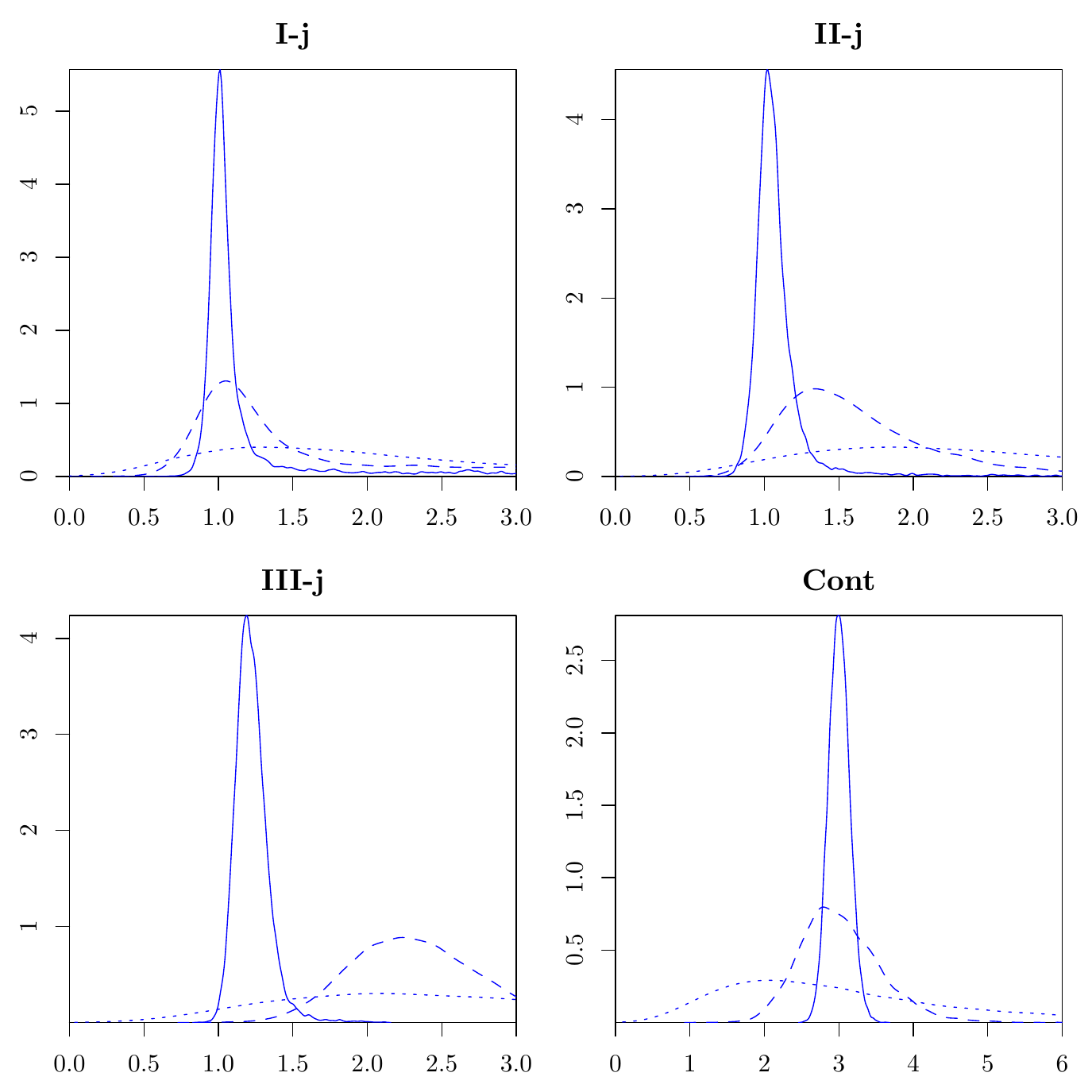}
\caption{Density estimation of $\Phi_{k,T,n}^{(J)}$ from the Monte Carlo for $k=5$.}\label{fig_k5_2}
\end{subfigure}
\caption{These graphics show the simulation results for the test from Theorem \ref{test_theo_J}. The dotted lines correspond to $n=100$, the dashed lines to $n=\numprint{1600}$ and the solid lines to $n=\numprint{25600}$. In all cases $N=\numprint{10000}$ paths were simulated.}
\label{figures_J}
\end{figure}

\begin{figure}[htb]
\centering
\begin{subfigure}[t]{0.45\textwidth}
\includegraphics[width=6.2cm]{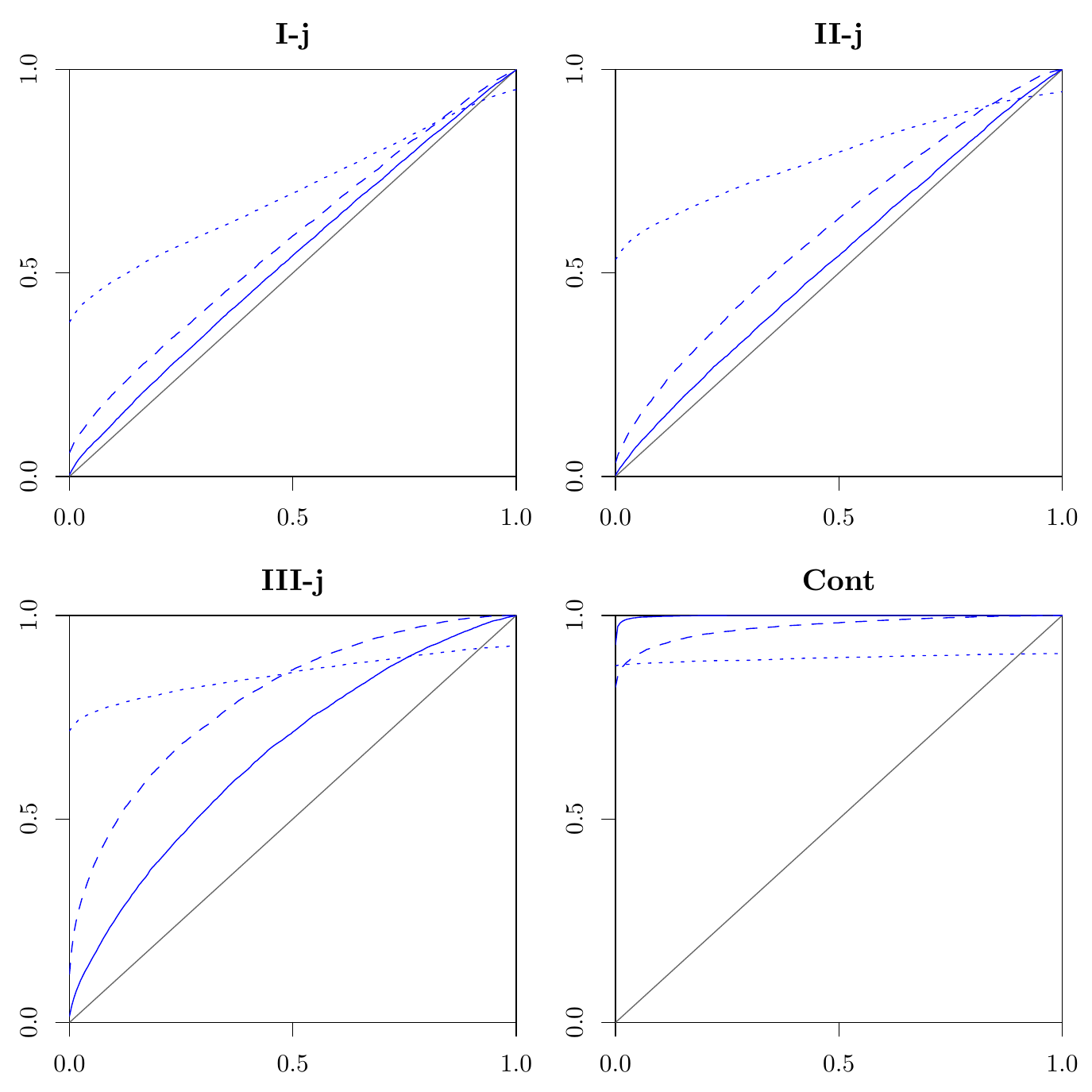} 
\caption{Rejection curves from the Monte Carlo for $k=2$.}\label{fig_k2r_1}
\end{subfigure}
\begin{subfigure}[t]{0.45\textwidth}
\includegraphics[width=6.2cm]{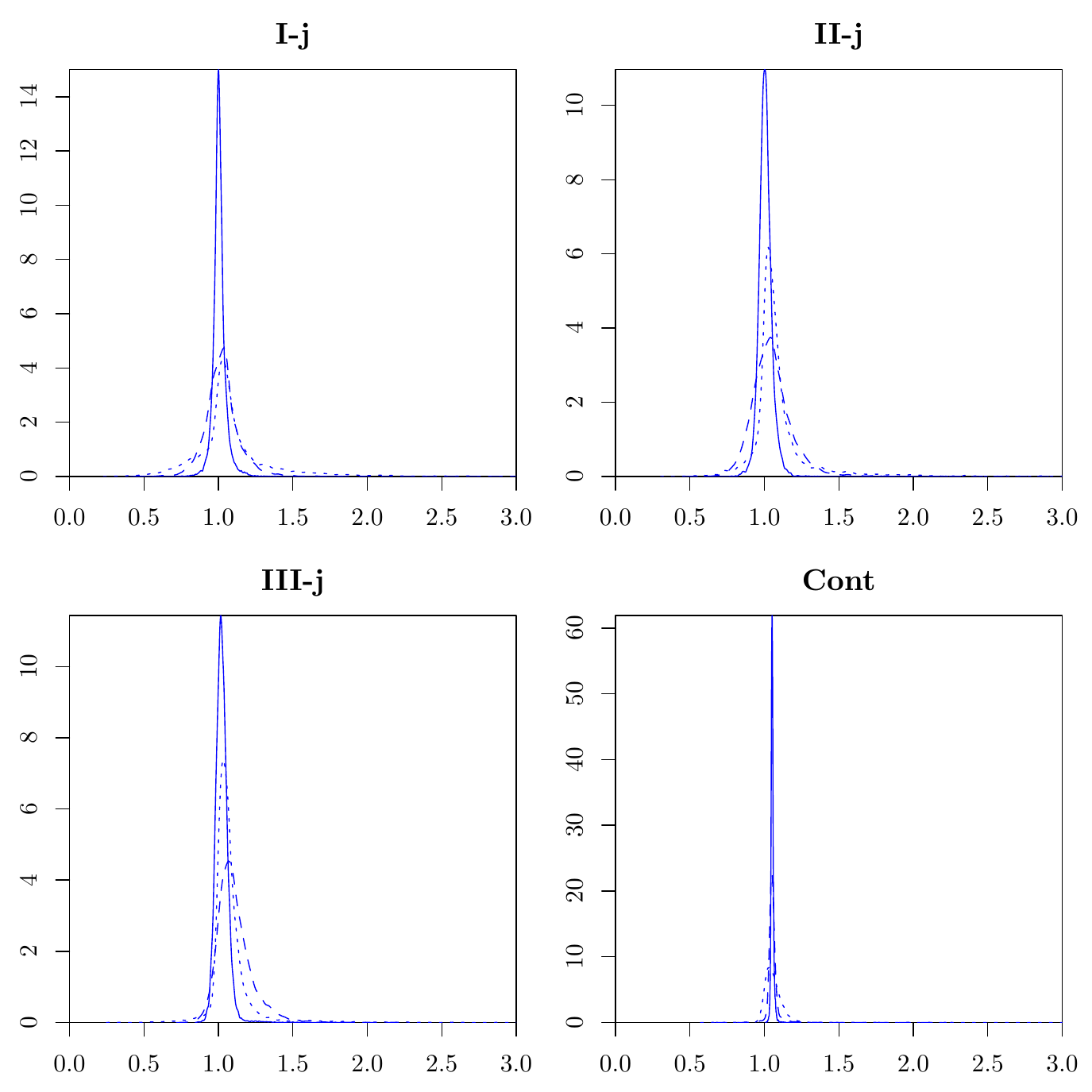}  
\caption{Density estimation of $\widetilde{\Phi}_{k,T,n}^{(J)}(0.9)$ from the Monte Carlo for $k=2$.}\label{fig_k2r_2}
\end{subfigure}

\begin{subfigure}[t]{0.45\textwidth}
\includegraphics[width=6.2cm]{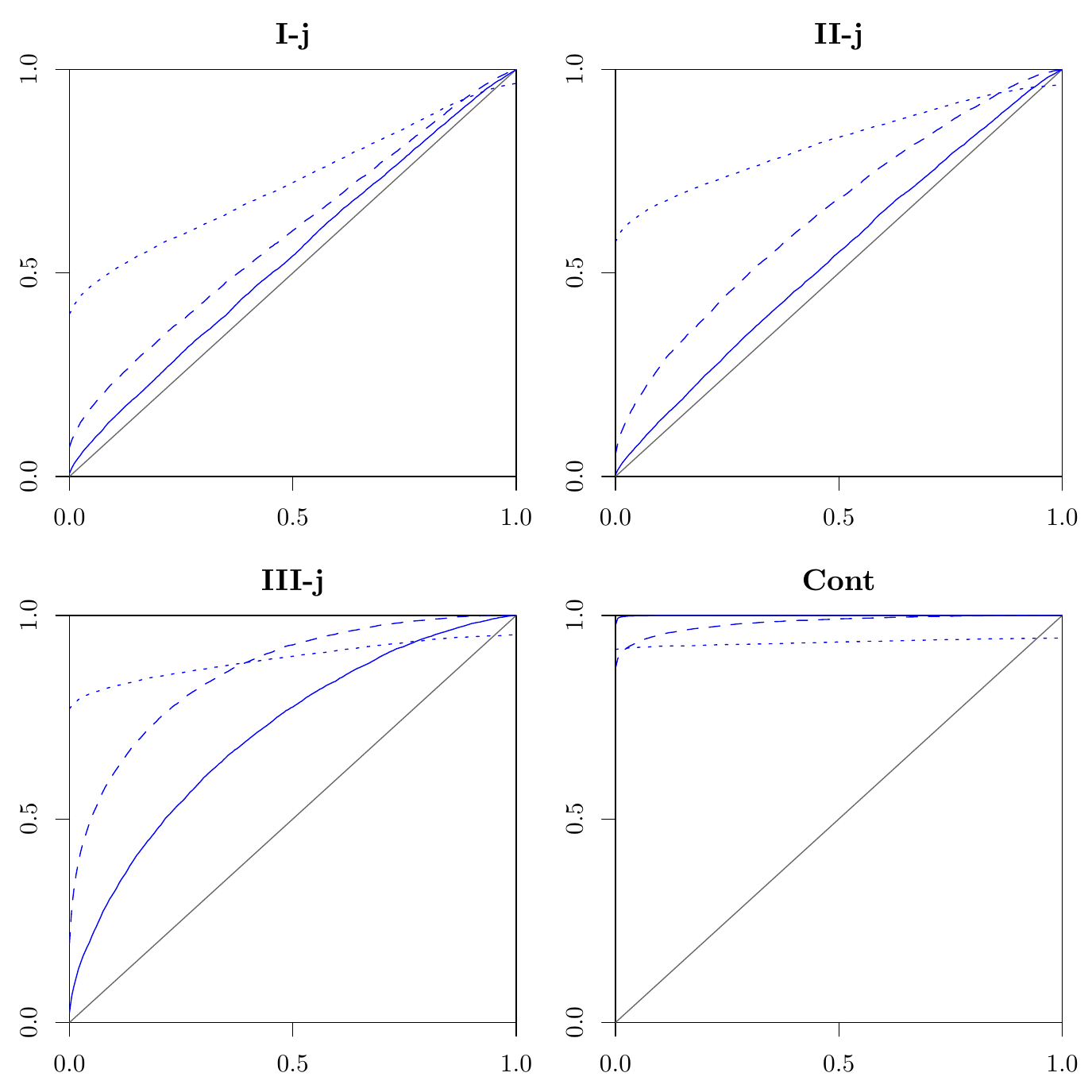}
\caption{Rejection curves from the Monte Carlo for $k=3$.}\label{fig_k3r_1}
\end{subfigure}
\begin{subfigure}[t]{0.45\textwidth}
\includegraphics[width=6.2cm]{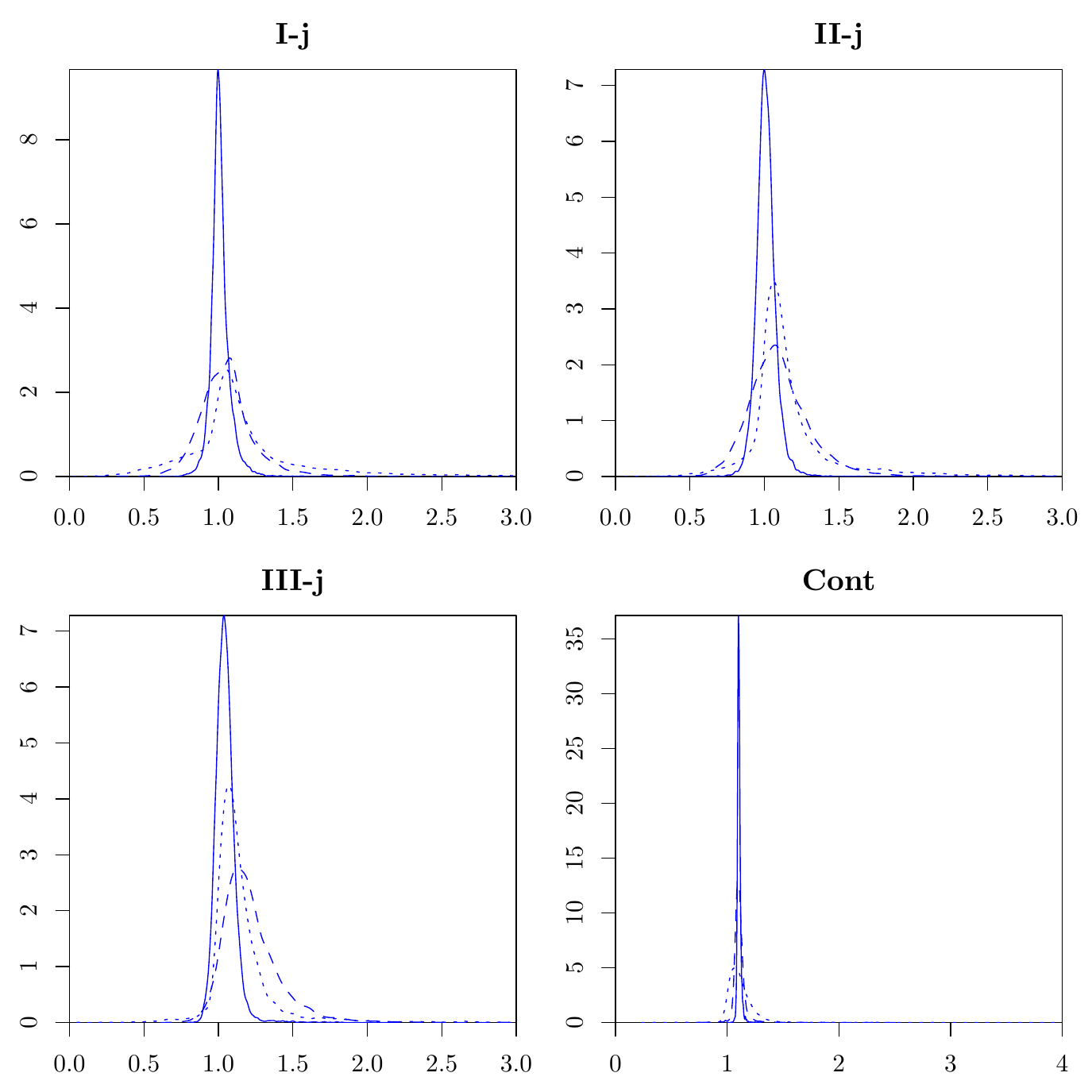} 
\caption{Density estimation of $\widetilde{\Phi}_{k,T,n}^{(J)}(0.9)$ from the Monte Carlo for $k=3$.}\label{fig_k3r_2}
\end{subfigure}

\begin{subfigure}[t]{0.45\textwidth}
\includegraphics[width=6.2cm]{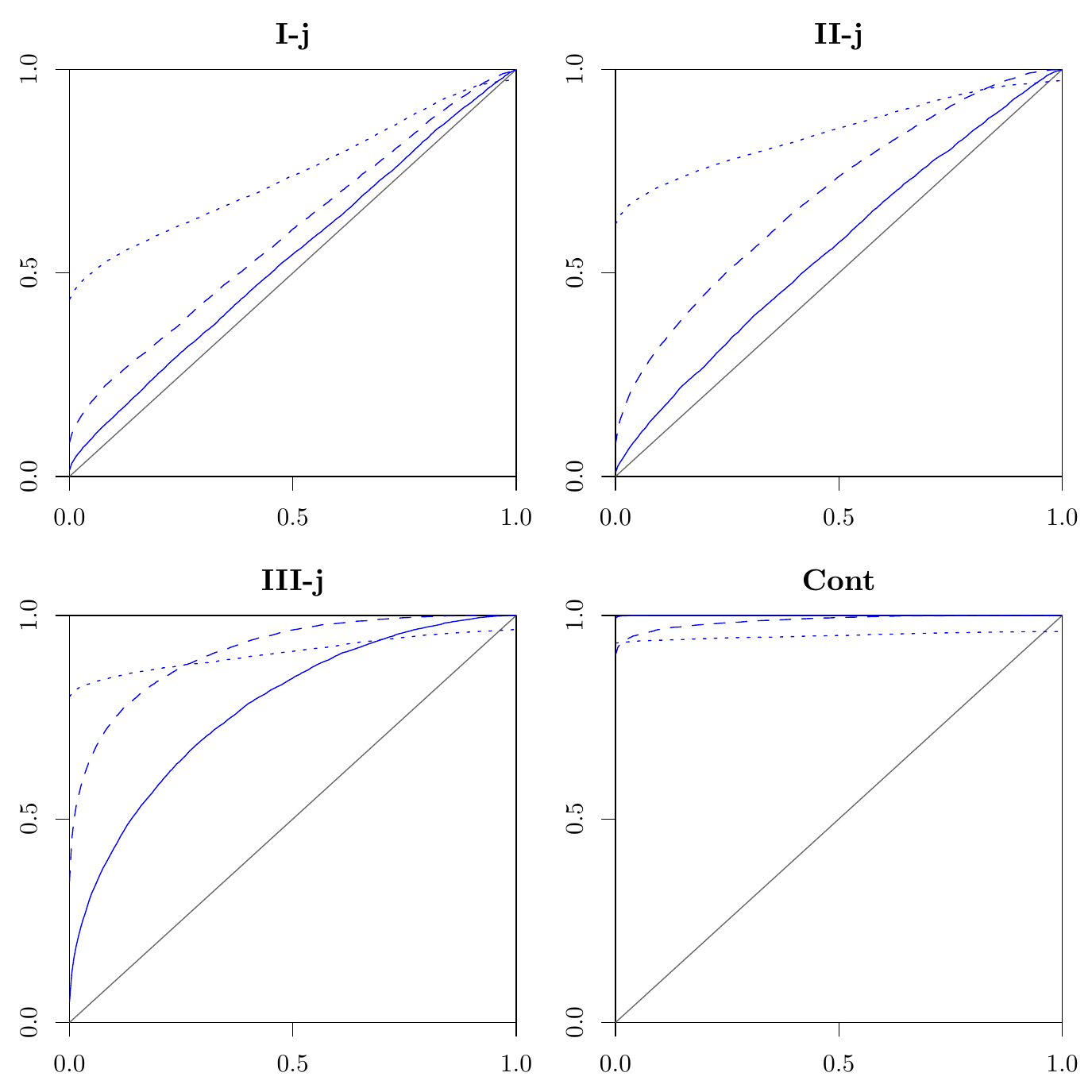}
\caption{Rejection curves from the Monte Carlo for $k=5$.}\label{fig_k5r_1}
\end{subfigure}
\begin{subfigure}[t]{0.45\textwidth}
\includegraphics[width=6.2cm]{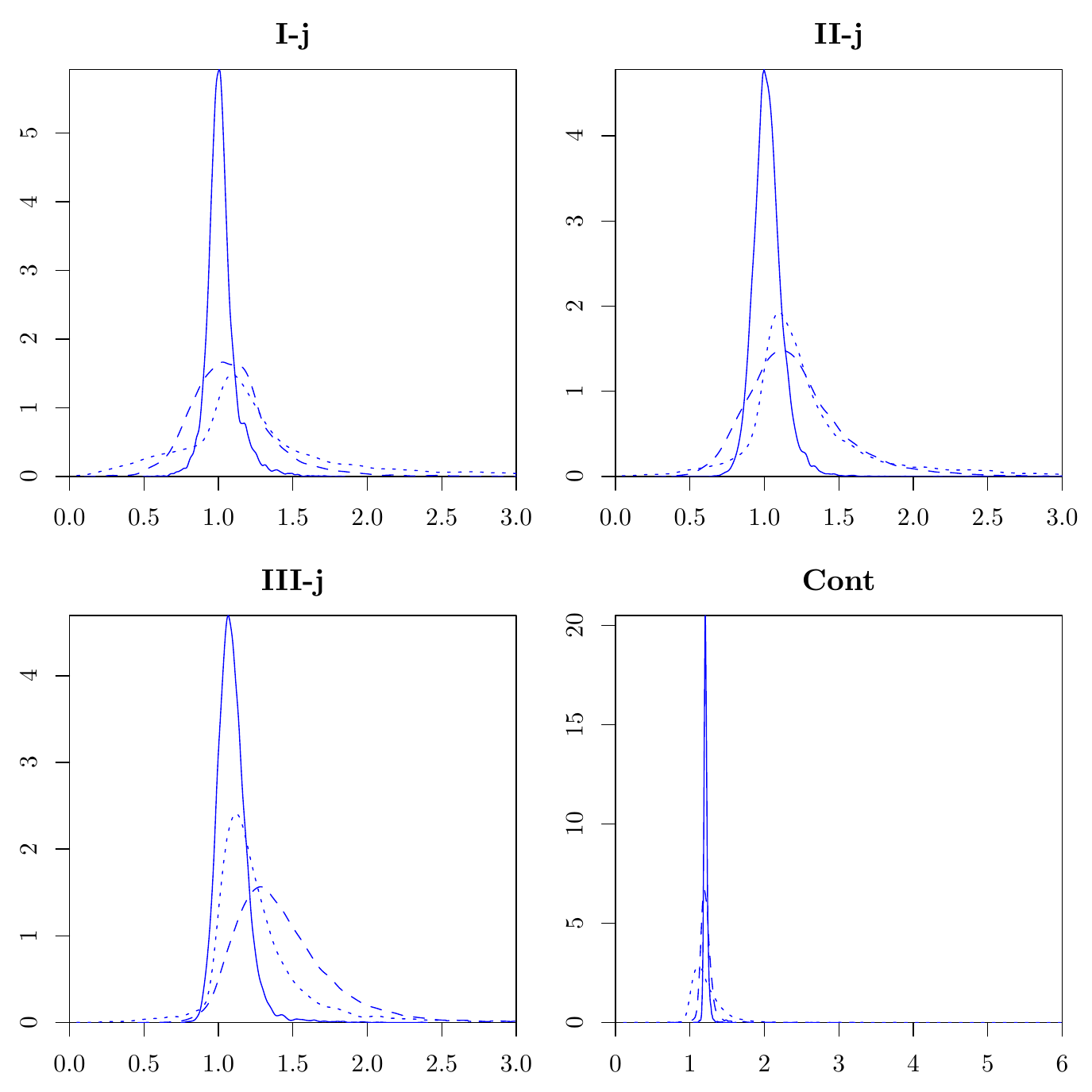} 
\caption{Density estimation of $\widetilde{\Phi}_{k,T,n}^{(J)}(0.9)$ from the Monte Carlo for $k=5$.}\label{fig_k5r_2}
\end{subfigure}
\caption{These graphics show the simulation results for the test from Corollary \ref{test_cor_J}. The dotted lines correspond to $n=100$, the dashed lines to $n=\numprint{1600}$ and the solid lines to $n=\numprint{25600}$. In all cases $N=\numprint{10000}$ paths were simulated.}
\label{figures_J_corr}
\end{figure}

\begin{figure}[htb]
\centering
\begin{subfigure}[t]{0.45\textwidth}
\includegraphics[width=6.2cm]{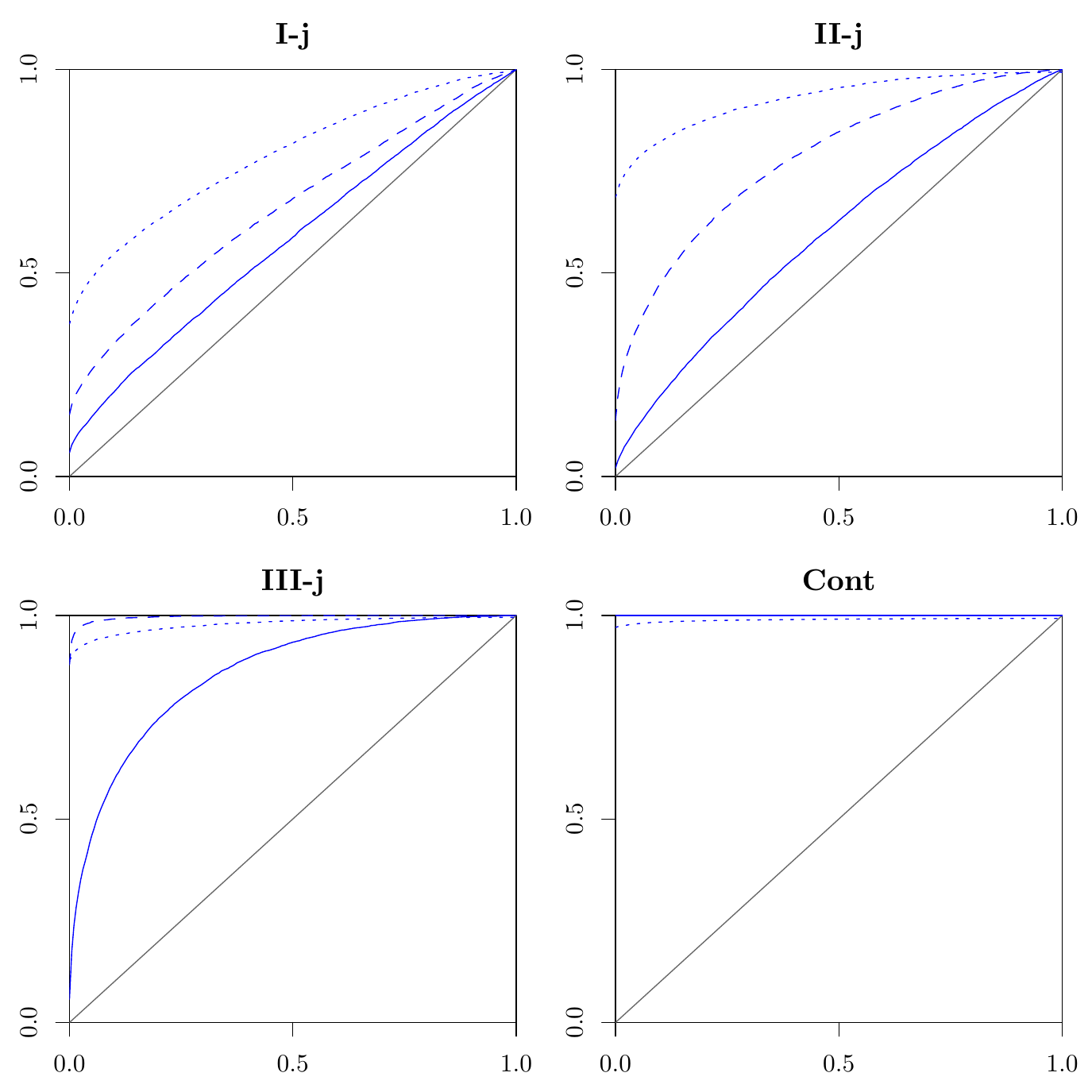}
\caption{Rejection curves for $k=2$.}\label{fig_equi_k2_1}
\end{subfigure}
\begin{subfigure}[t]{0.45\textwidth}
\includegraphics[width=6.2cm]{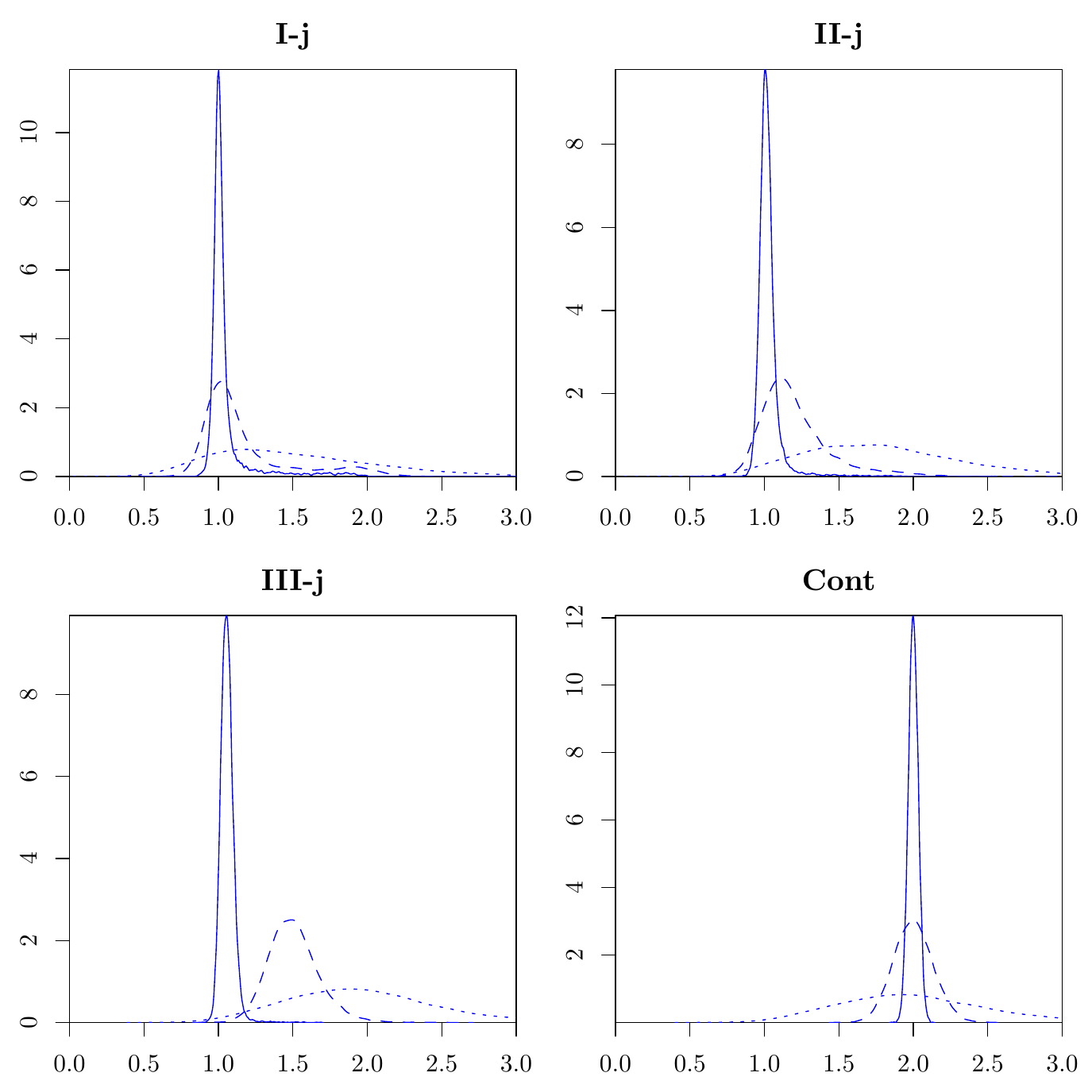} 
\caption{Density estimation of $\Phi_{2,T,n}^{(J)}$.}\label{fig_equi_k2_2}
\end{subfigure}

\begin{subfigure}[t]{0.45\textwidth}
\includegraphics[width=6.2cm]{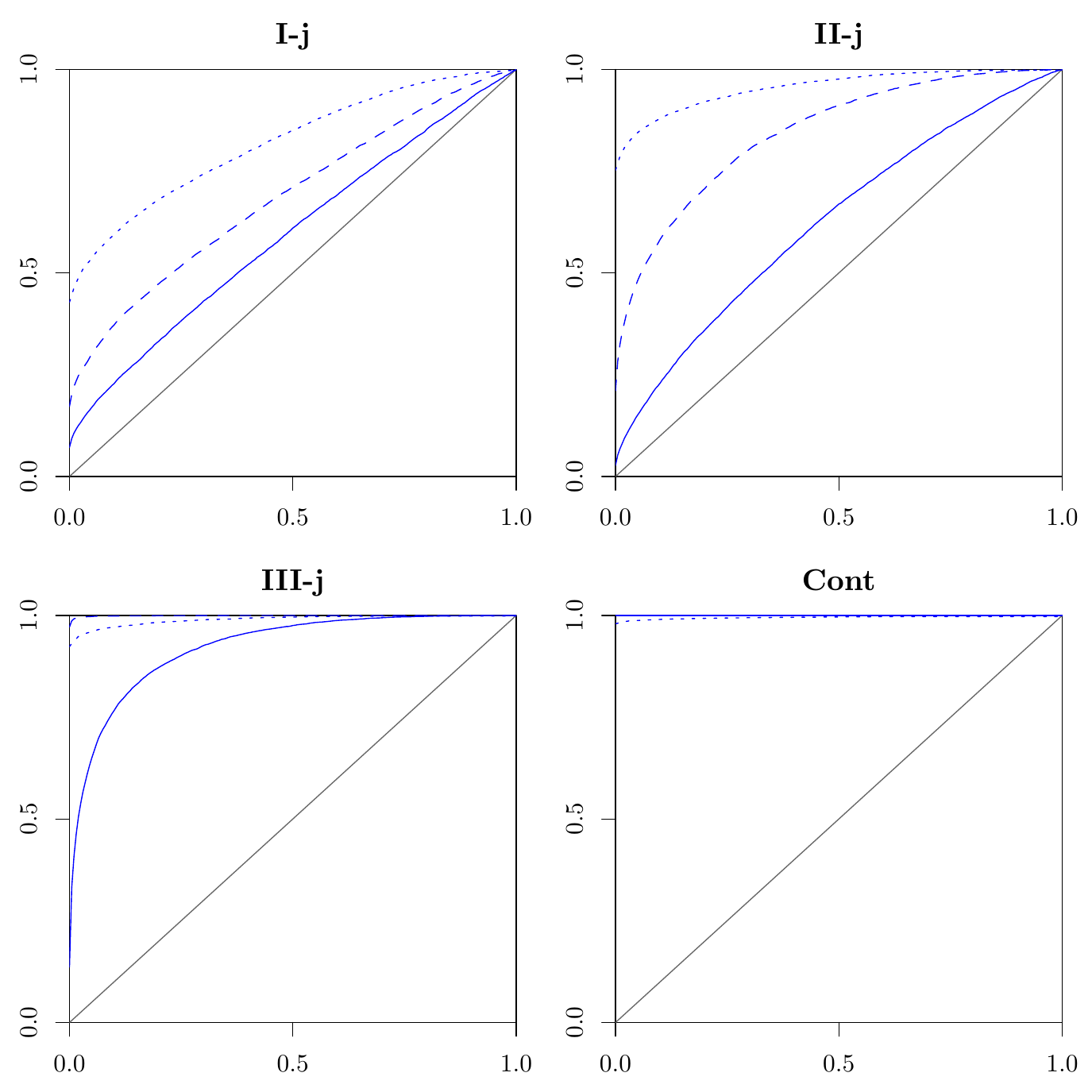} 
\caption{Rejection curves for $k=3$.}\label{fig_equi_k3_1}
\end{subfigure}
\begin{subfigure}[t]{0.45\textwidth}
\includegraphics[width=6.2cm]{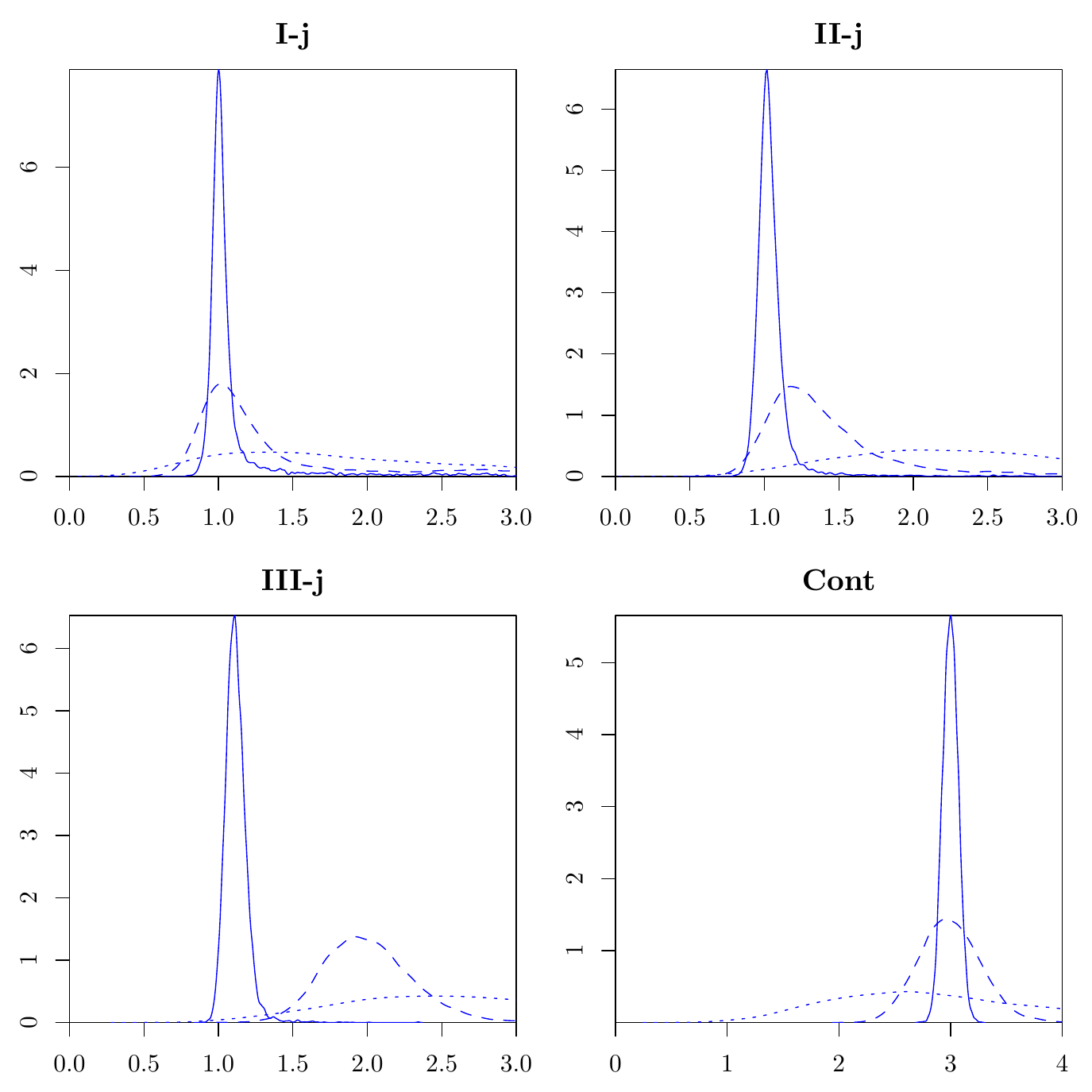}
\caption{Density estimation of $\Phi_{3,T,n}^{(J)}$.}\label{fig_equi_k3_2}
\end{subfigure}

\begin{subfigure}{0.45\textwidth}
\includegraphics[width=6.2cm]{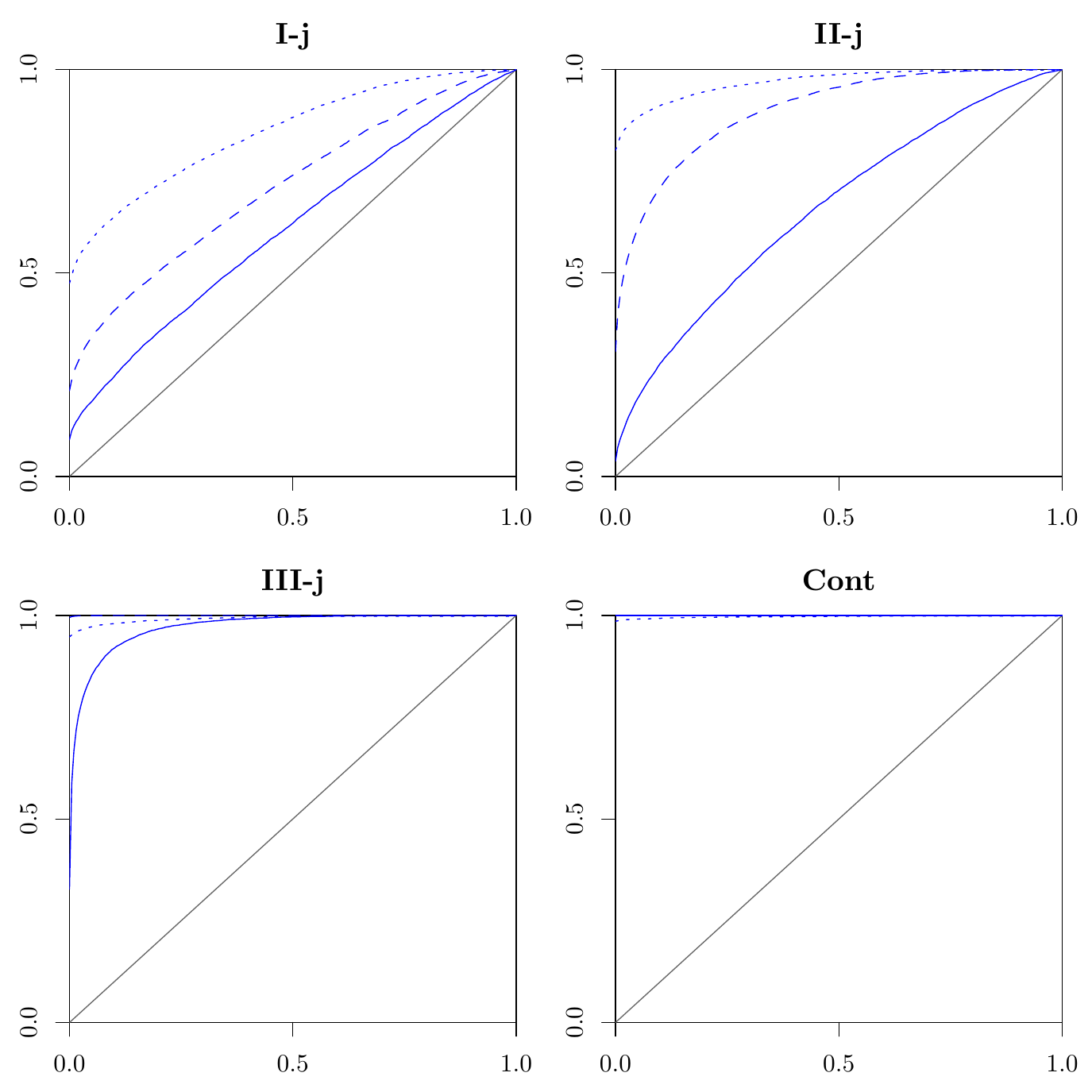}
\caption{Rejection curves for $k=5$.}\label{fig_equi_k5_1}
\end{subfigure}
\begin{subfigure}{0.45\textwidth}
\includegraphics[width=6.2cm]{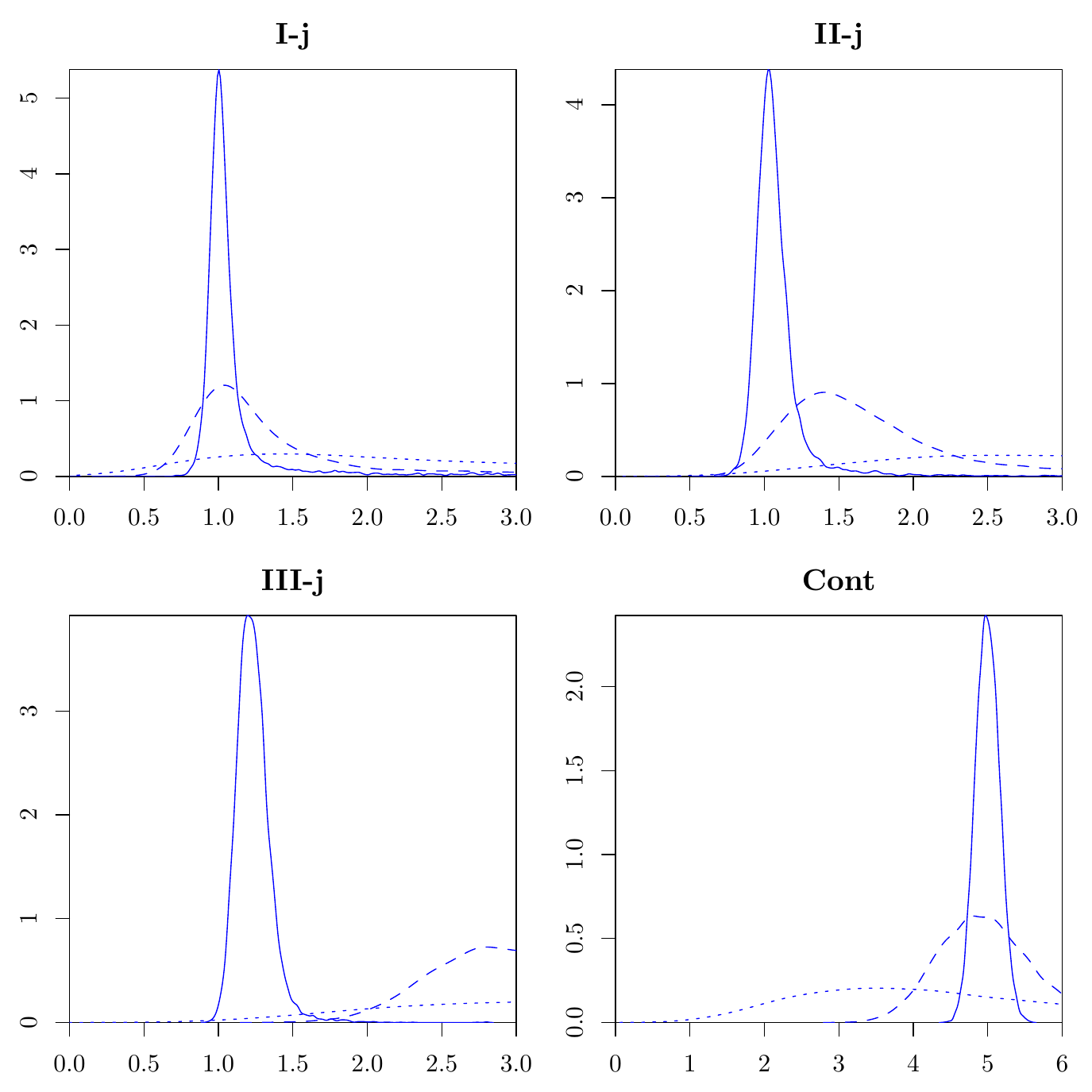}
\caption{Density estimation of $\Phi_{5,T,n}^{(J)}$.}\label{fig_equi_k5_2}
\end{subfigure}
\caption{Simulation results for the test from Theorem \ref{test_theo_J} based on \emph{equidistant observations} $t_{i,n}=i/n$. The dotted lines correspond to $n=100$, the dashed lines to $n=\numprint{1600}$ and the solid lines to $n=\numprint{25600}$. In all cases $N=\numprint{10000}$ paths were simulated.}
\label{figures_equi_J} 
\end{figure}

\begin{figure}[htb]
\centering
\begin{subfigure}[t]{0.45\textwidth}
\includegraphics[width=6.2cm]{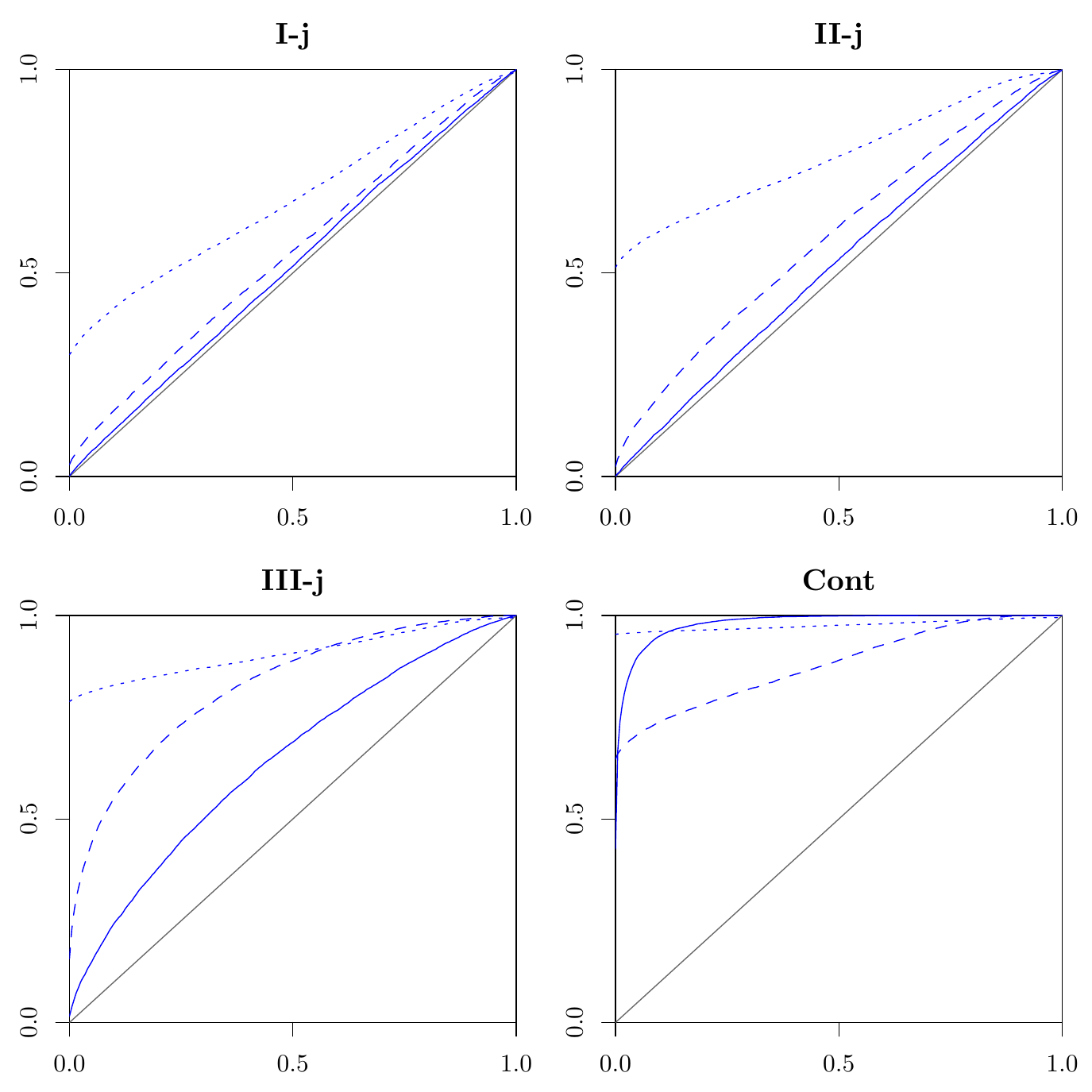}
\caption{Rejection curves from the Monte Carlo for $k=2$.}\label{fig_equi_k2r_1}
\end{subfigure}
\begin{subfigure}[t]{0.45\textwidth}
\includegraphics[width=6.2cm]{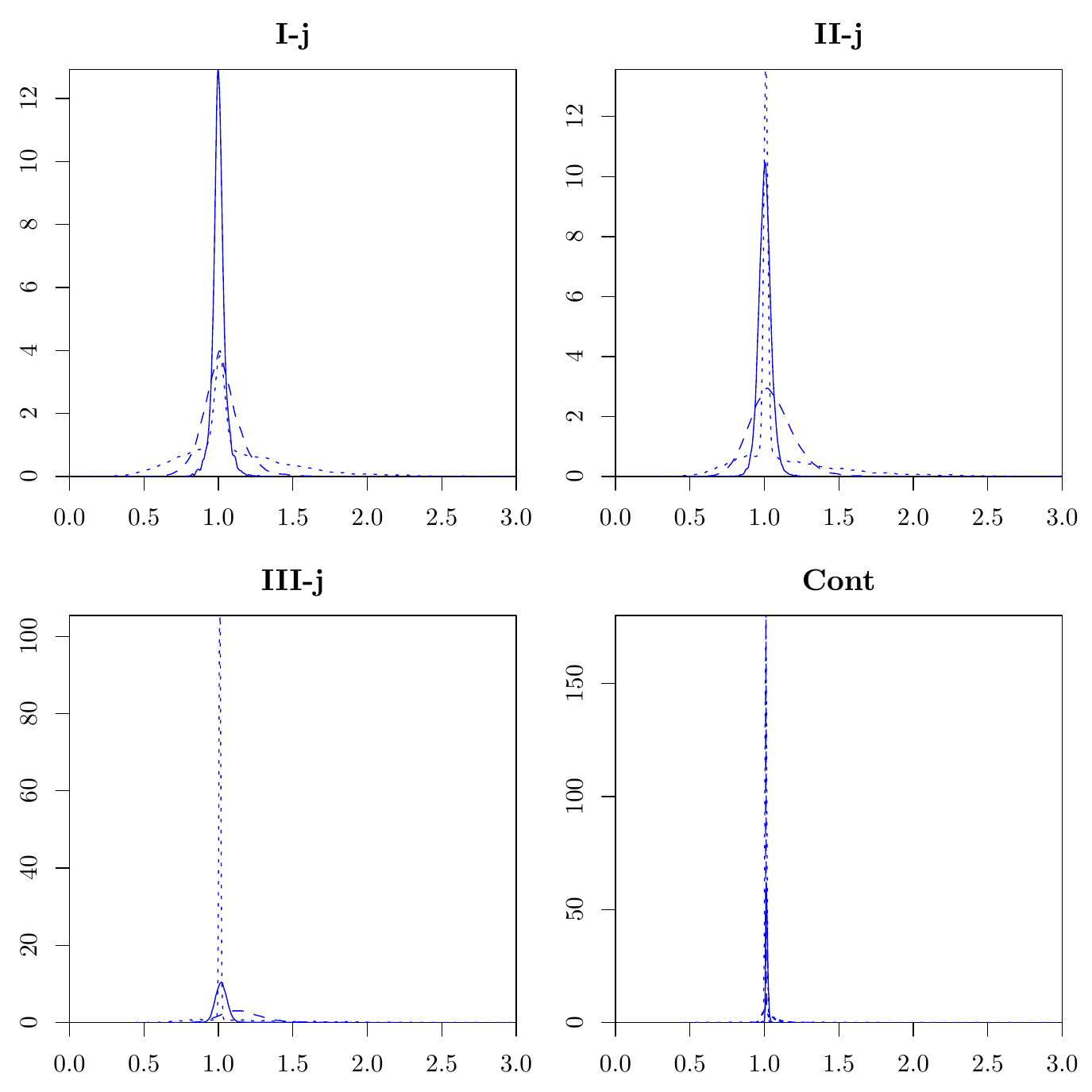}
\caption{Density estimation of $\widetilde{\Phi}_{2,T,n}^{(J)}(0.99)$.}\label{fig_equi_k2r_2}
\end{subfigure}

\begin{subfigure}[t]{0.45\textwidth}
\includegraphics[width=6.2cm]{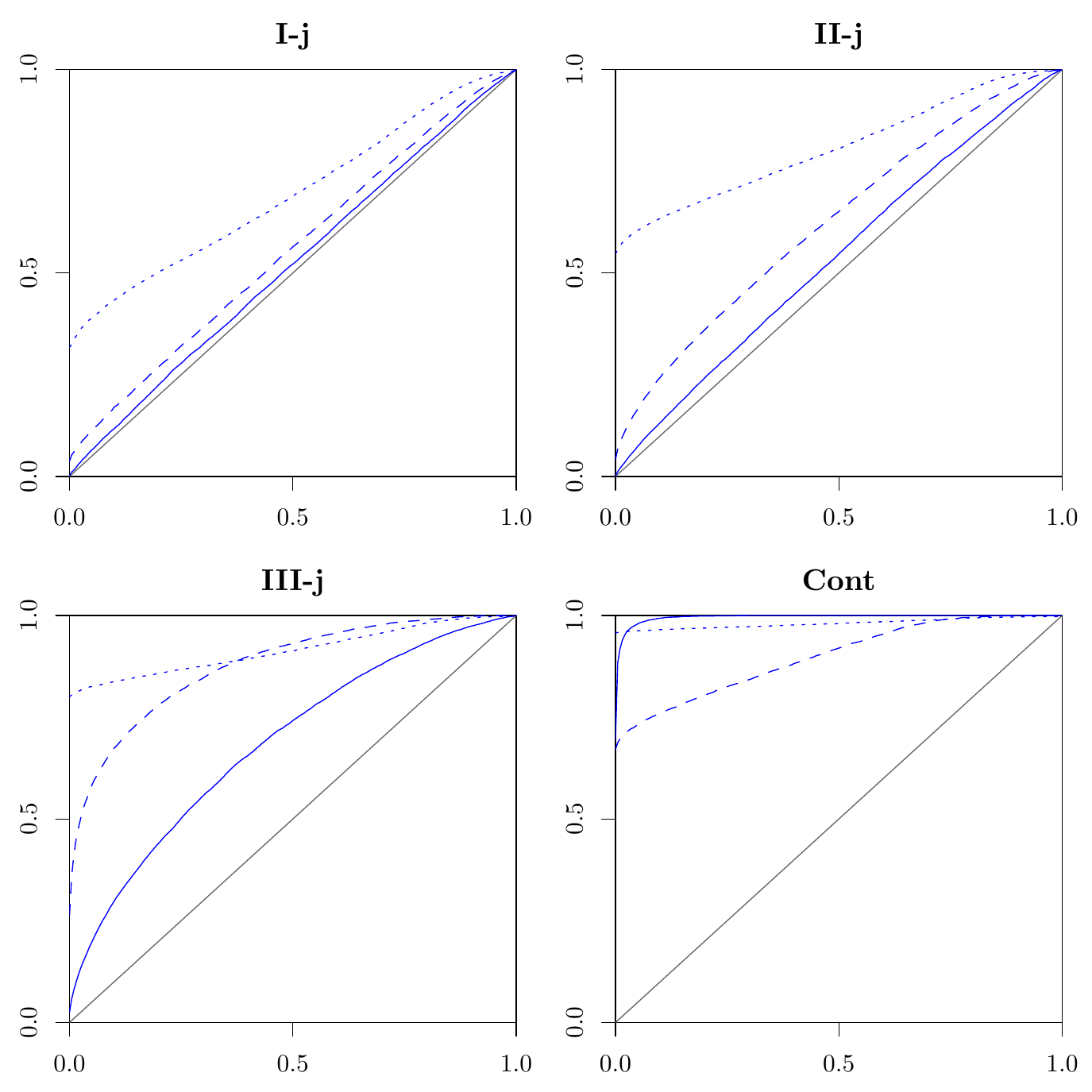}
\caption{Rejection curves from the Monte Carlo for $k=3$.}\label{fig_equi_k3r_1}
\end{subfigure}
\begin{subfigure}[t]{0.45\textwidth}
\includegraphics[width=6.2cm]{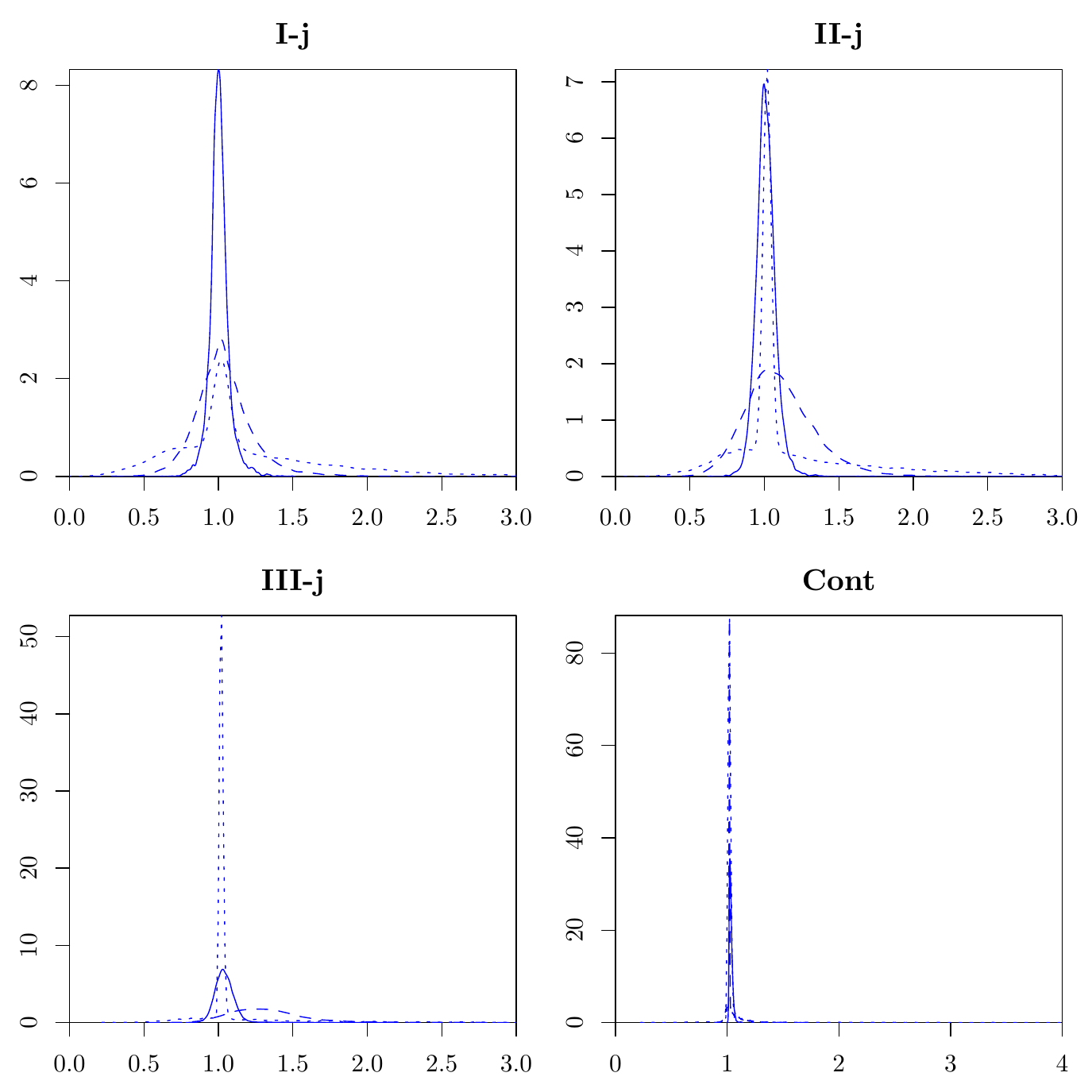}
\caption{Density estimation of $\widetilde{\Phi}_{3,T,n}^{(J)}(0.99)$.}\label{fig_equi_k3r_2}
\end{subfigure}

\begin{subfigure}[t]{0.45\textwidth}
\includegraphics[width=6.2cm]{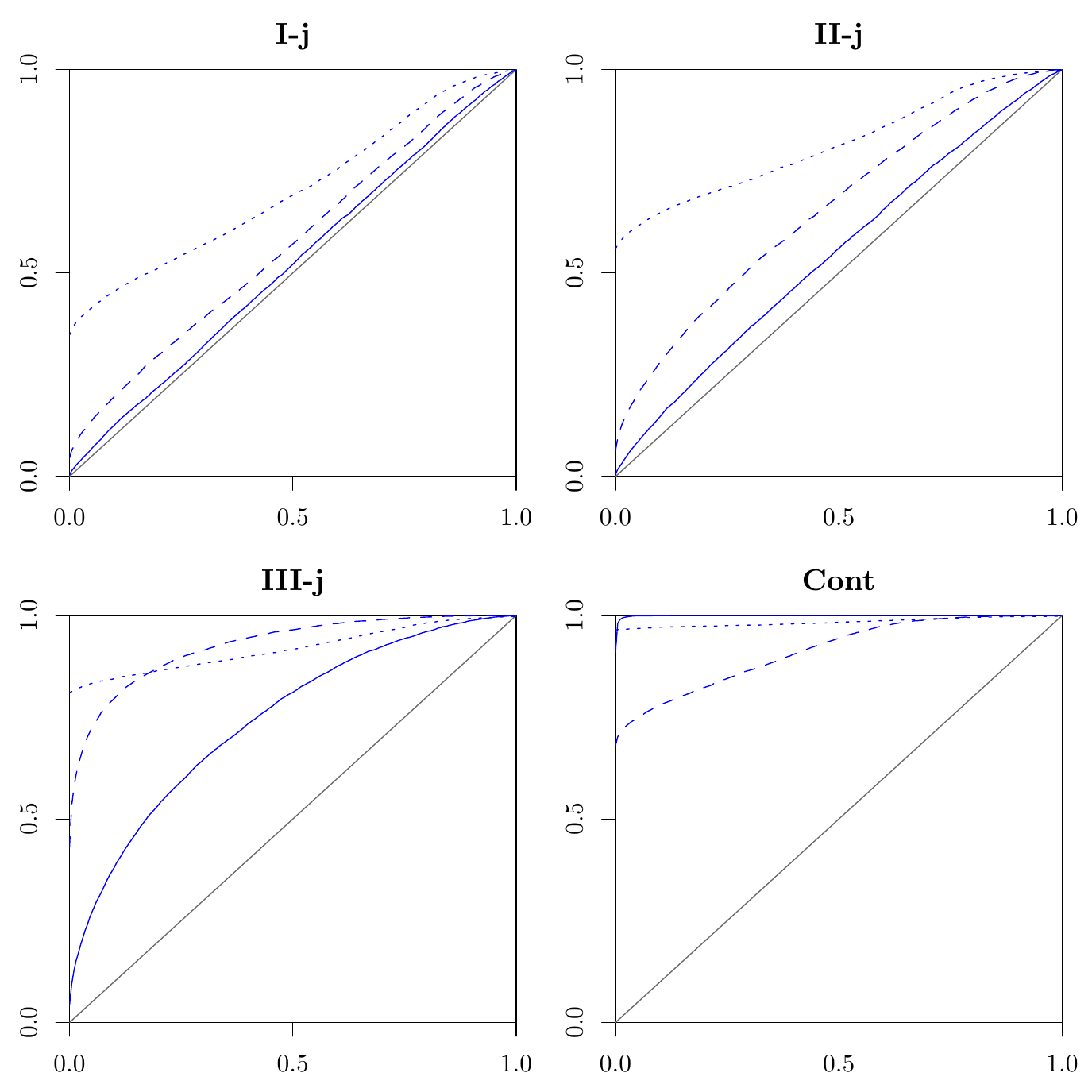}
\caption{Rejection curves from the Monte Carlo for $k=5$.}\label{fig_equi_k5r_1}
\end{subfigure}
\begin{subfigure}[t]{0.45\textwidth}
\includegraphics[width=6.2cm]{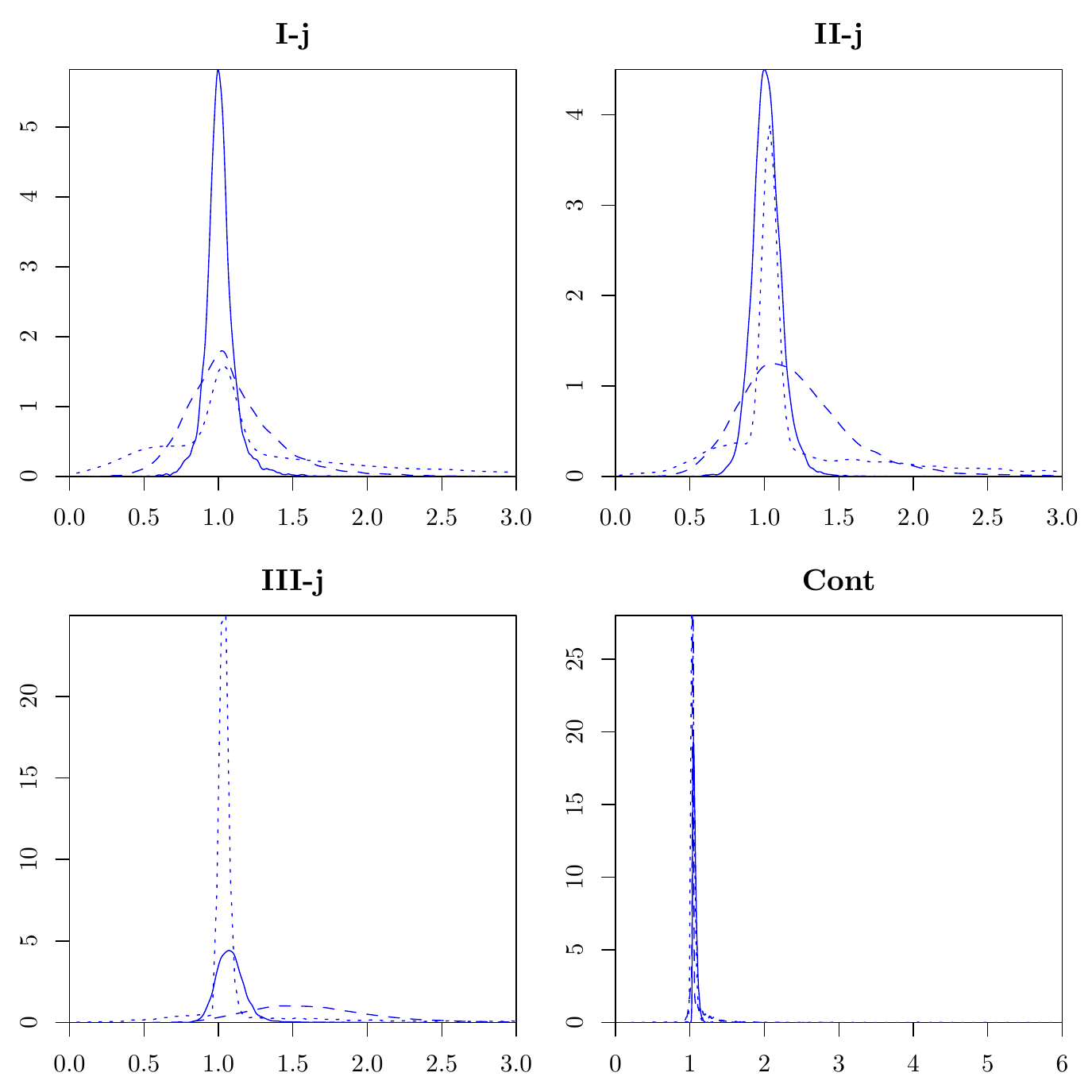}
\caption{Density estimation of $\widetilde{\Phi}_{5,T,n}^{(J)}(0.99)$.}\label{fig_equi_k5r_2}
\end{subfigure}
\caption{Simulation results for the test from Corollary \ref{test_cor_J} based on \emph{equidistant observations} $t_{i,n}=i/n$. The dotted lines correspond to $n=100$, the dashed lines to $n=\numprint{1600}$ and the solid lines to $n=\numprint{25600}$. In all cases $N=\numprint{10000}$ paths were simulated.}
\label{figures_equi_J_corr}
\end{figure}
\FloatBarrier

\begin{figure}[!htb]
\centering
\includegraphics[width=10cm]{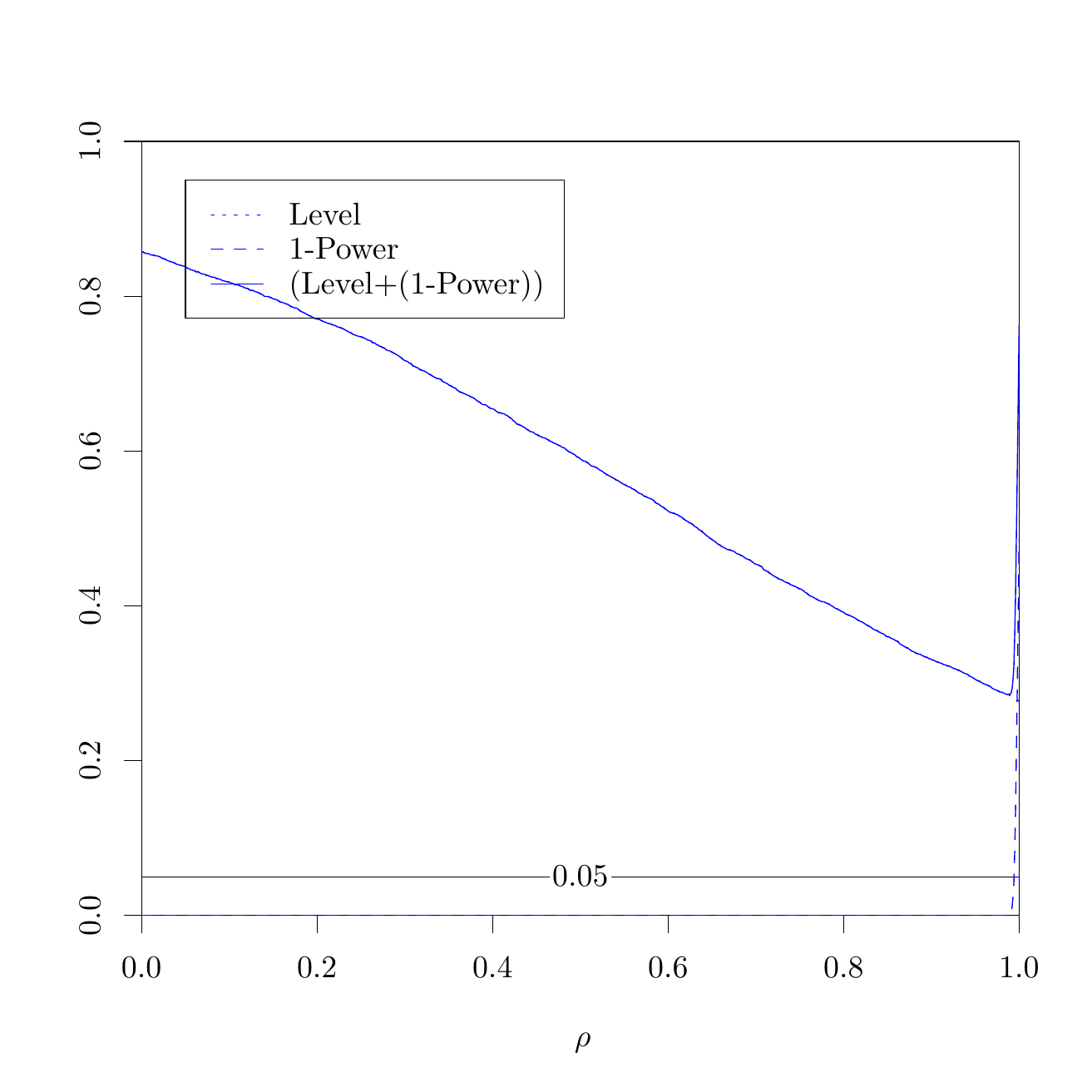}
\caption{This graphic shows for $k=2$, $\alpha=5\%$, $n=\numprint{25600}$ and \emph{equidistant observations} $t_{i,n}=i/n$ the empirical rejection rate in the case \texttt{Cont} (dotted line) and $1$ minus the empirical rejection rate in the case \texttt{III-j} (dashed line) from the Monte Carlo simulation based on Corollary \ref{test_cor_J} as a function of $\rho \in [0,1]$. We achieve a minimal overall error for approximately $\rho=0.99$.}\label{figure_jump_choice_rho_equi}
\end{figure}

\subsection{Testing for disjoint jumps}\label{sec:simul_CoJ}

We simulate according to the model used in Section 6 of \cite{JacTod09}, since by using the same configurations as in their paper we can compare the performance of our approach to the performance of their methods in the case of equidistant and synchronous observations. The model for the process $X$ is given by
\begin{equation*}
\begin{aligned}
&dX^{(1)}_t=X^{(1)}_t \sigma_1 d W^{(1)}_t+\alpha_1 \int_{\mathbb{R}} X^{(1)}_{t-}x_1 \mu_1(dt,dx_1)+\alpha_3 \int_{\mathbb{R}} X^{(1)}_{t-}x_3 \mu_3(dt,dx_3),
\\&dX^{(2)}_t=X^{(2)}_t \sigma_2 d W^{(2)}_t+\alpha_2 \int_{\mathbb{R}} X^{(2)}_{t-}x_2 \mu_2(dt,dx_2)+\alpha_3 \int_{\mathbb{R}} X^{(2)}_{t-}x_3 \mu_3(dt,dx_3),
\end{aligned}
\end{equation*}
where $[W^{(1)},W^{(2)}]_t=\rho t$ and the Poisson measures $\mu_i$ are independent of each other and have predictable compensators $\nu_i$ of the form
\[
\nu_i(dt,dx_i)=\kappa_i\frac{\mathds{1}_{[-h_i,-l_i]\cup [l_i,h_i]}(x_i)}{2(h_i-l_i)}dtdx_i
\]
where $0<l_i<h_i$ for $i=1,2,3$, and the initial values are $X_0=(1,1)^T$. We consider the same twelve parameter settings which were discussed in \cite{JacTod09} of which six allow for common jumps and six do not. In the case where common jumps are possible, we only use the simulated paths which contain common jumps. For the parameters we set $\sigma_1^2=\sigma_2^2=8 \times 10^{-5}$ in all scenarios and choose the parameters for the Poisson measures such that the contribution of the jumps to the total variation remains approximately constant and matches estimations from real financial data; see \cite{HuaTau06}. The parameter settings are summarized in Table \ref{par_settings_CoJ}; compare Table 1 in \cite{JacTod09}.

\begin{table}[b]
\centering
\resizebox{13.3cm}{!} {
\begin{tabular}{lccccccccccccc}
\hline 
 & \multicolumn{13}{c}{Parameters} \\
\cline{2-14}
Case & $\rho$ & $\alpha_1$ & $\kappa_1$ & $l_1$ & $h_1$ & $\alpha_2$ & $\kappa_2$ & $l_1$ & $h_1$ & $\alpha_3$ & $\kappa_3$ & $l_3$ & $h_3$ \\ 
\hline 
I-j & $0.0$ & $0.00$ & & & & $0.00$ & & &  & $0.01$ & $1$ & $0.05$ & $0.7484$ \\ 

II-j & $0.0$ & $0.00$ & & & & $0.00$ & & & & $0.01$ & $5$ & $0.05$ & $0.3187$ \\ 

III-j & $0.0$ & $0.00$ &  &  &  & $0.00$ &  &  &  & $0.01$ & $25$ & $0.05$ & $0.1238$ \\ 

I-m & $0.5$ & $0.01$ & $1$ & $0.05$ & $0.7484$ & $0.01$ & $1$ & $0.05$ & $0.7484$ & $0.01$ & $1$ & $0.05$ & $0.7484$ \\ 

II-m & $0.5$ & $0.01$ & $5$ & $0.05$ & $0.3187$ & $0.01$ & $5$ & $0.05$ & $0.3187$ & $0.01$ & $5$ & $0.05$ & $0.3187$ \\ 

III-m & $0.5$ & $0.01$ & $25$ & $0.05$ & $0.1238$ & $0.01$ & $25$ & $0.05$ & $0.1238$ & $0.01$ & $25$ & $0.05$ & $0.1238$ \\ 

I-d0 & $0.0$ & $0.01$ & $1$ & $0.05$ & $0.7484$ & $0.01$ & $1$ & $0.05$ & $0.7484$ &  &  &  &  \\ 

II-d0 & $0.0$ & $0.01$ & $5$ & $0.05$ & $0.3187$ & $0.01$ & $5$ & $0.05$ & $0.3187$ &  &  &  &  \\ 

III-d0 & $0.0$ & $0.01$ & $25$ & $0.05$ & $0.1238$ & $0.01$ & $25$ & $0.05$ & $0.1238$ &  &  &  &  \\ 

I-d1 & $1.0$ & $0.01$ & $1$ & $0.05$ & $0.7484$ & $0.01$ & $1$ & $0.05$ & $0.7484$ &  &  &  &  \\ 

II-d1 & $1.0$ & $0.01$ & $5$ & $0.05$ & $0.3187$ & $0.01$ & $5$ & $0.05$ & $0.3187$ &  &  &  &  \\ 

III-d1 & $1.0$ & $0.01$ & $25$ & $0.05$ & $0.1238$ & $0.01$ & $25$ & $0.05$ & $0.1238$ &  &  &  &  \\ 
\hline 
\end{tabular} }
\caption[]{Parameter settings for the simulation.}
\label{par_settings_CoJ}
\end{table}

\begin{figure}[tb]
\includegraphics[width=13.85cm]{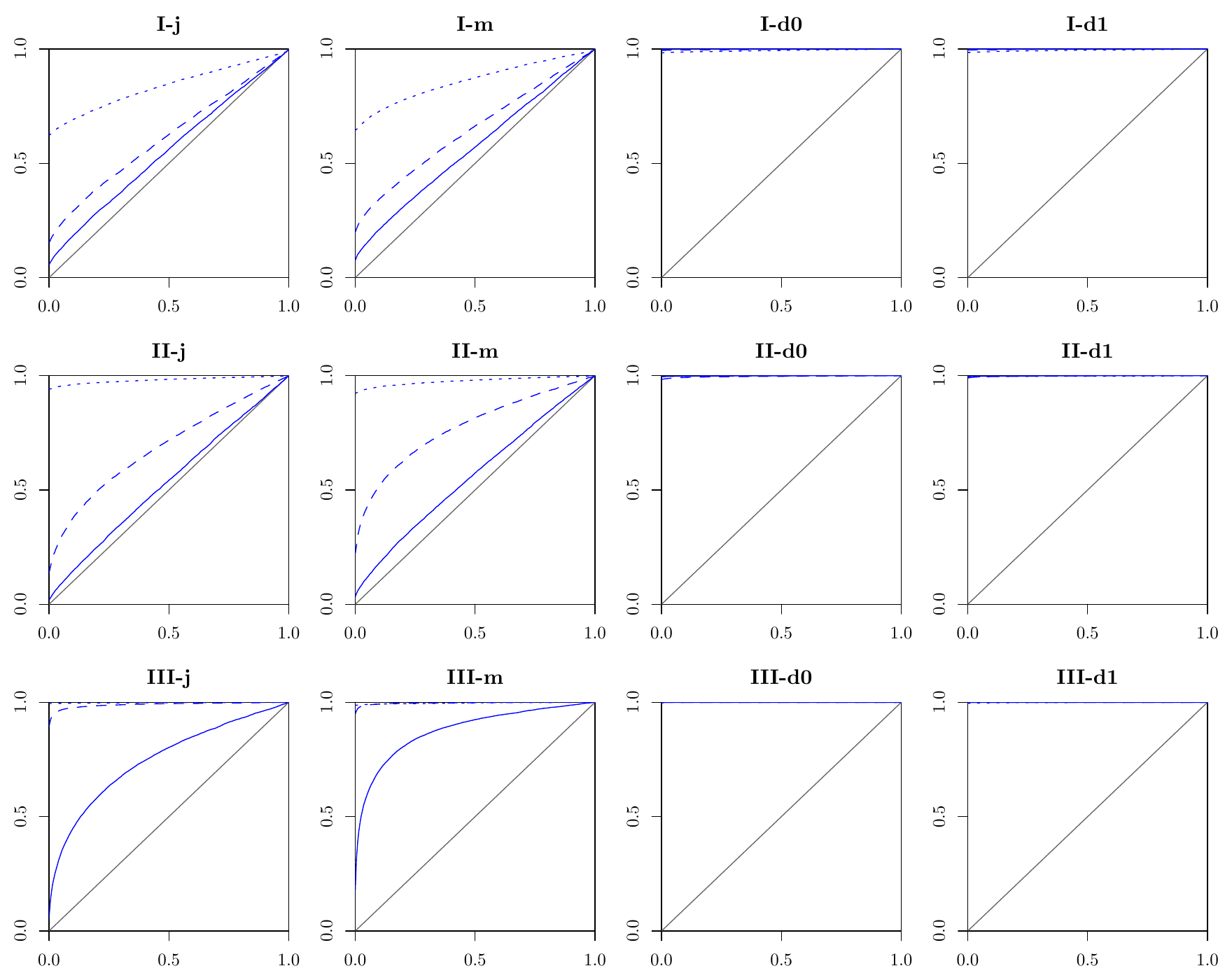}
\caption{Empirical rejection curves from the Monte Carlo simulation for the test derived from Theorem \ref{test_theo_CoJ}. The dotted lines represent the results for $n=100$, the dashed lines for $n=\numprint{1600}$ and the solid lines for $n=\numprint{25600}$. In each case $N=\numprint{10000}$ paths were simulated.}
\label{rejection_curves_CoJ}
\end{figure}

\begin{figure}[tb]
\includegraphics[width=13.85cm]{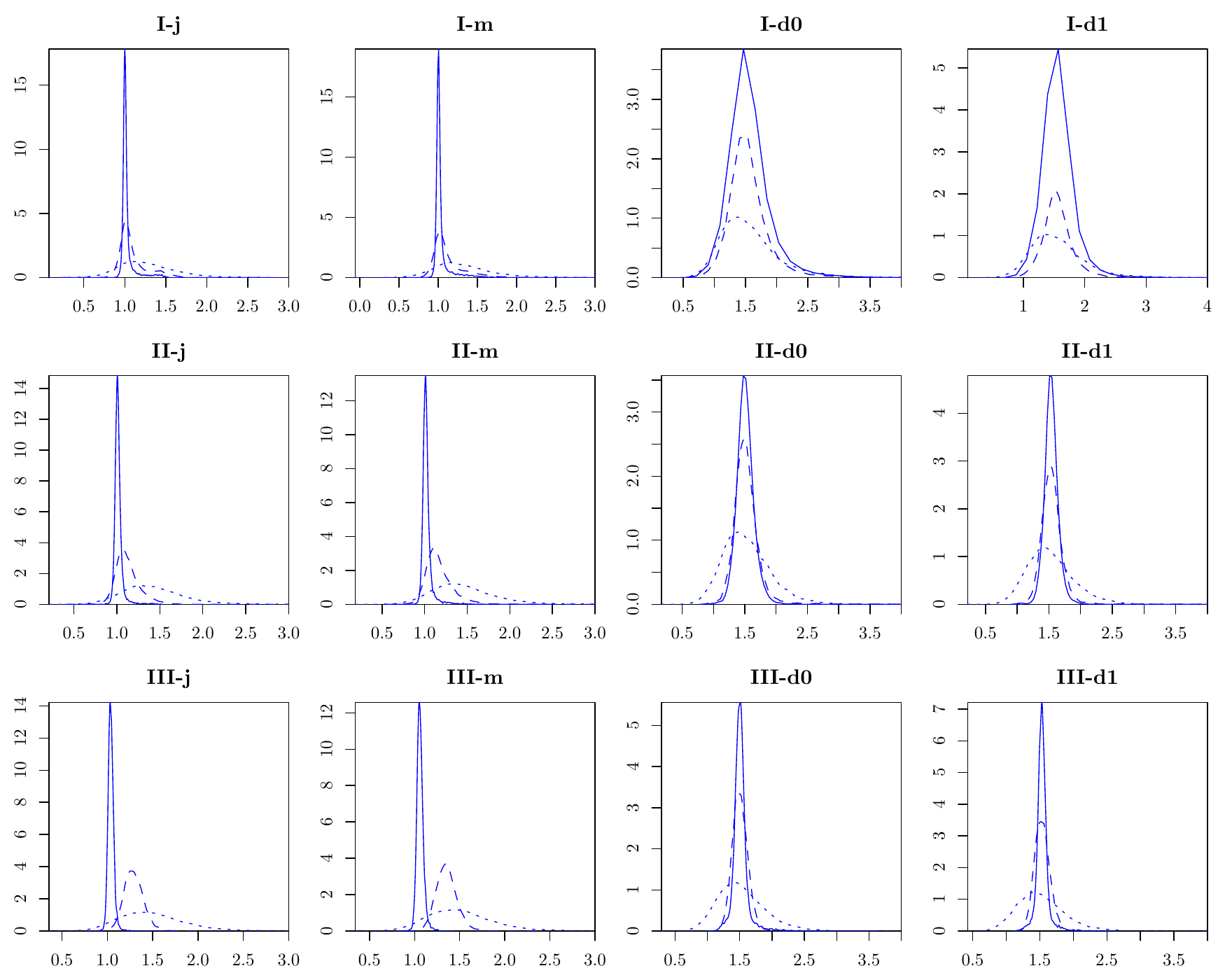}
\caption{Density estimations for $\Phi_{2,T,n}^{(CoJ)}$ from the Monte Carlo. The dotted lines correspond to $n=100$, the dashed lines to $n=\numprint{1600}$ and the solid lines to ${n=\numprint{25600}}$. In all cases $N=\numprint{10000}$ paths were simulated.}
\label{dens_est_CoJ}
\end{figure}

The first six cases in Table \ref{par_settings_CoJ} describe situations where common jumps might be present, the other six cases situations where there exist no common jumps. In the cases \texttt{I-j}, \texttt{II-j} and \texttt{III-j} there exist only common jumps and no disjoint jumps and the Brownian motions $W^{(1)}$ and $W^{(2)}$ are uncorrelated. In the cases \texttt{I-m}, \texttt{II-m} and \texttt{III-m} we have a mixed model which allows for common and disjoint jumps and also the Brownian motions are positively correlated. In the cases \texttt{I-d0}, \texttt{II-d0} and \texttt{III-d0} the Brownian motions $W^{(1)}$, $W^{(2)}$ are uncorrelated while in the cases \texttt{I-d1}, \texttt{II-d1} and \texttt{III-d1} the processes are driven by the same Brownian motion $W^{(1)}=W^{(2)}$. The prefixes \texttt{I}, \texttt{II} and \texttt{III} indicate an increasing number of jumps present in the observed paths. Since our choice of parameters is such that the overall contribution of the jumps to the quadratic variation is roughly the same in all parameter settings, this corresponds to a decreasing size of the jumps. Hence in the cases \texttt{I-*} we have few large jumps while in the cases \texttt{III-*} we have many small jumps.

As a model for the observation times we use the Poisson setting discussed in Examples \ref{example_poisson1} and \ref{example_poisson2} for $\lambda_1=\lambda_2=1$ and we set $T=1$. As in Section \ref{sec:simul_J} we choose $n=100$, $n=\numprint{1600}$ and $n=\numprint{25600}$. We set $\beta=0.03$ and $\varpi=0.49$ for all occuring truncations. We use $b_n= 1/\sqrt{n}$ for the local interval in the estimation of $\sigma_{s-}^{(l)},\sigma_s^{(l)},\rho_s$ and $L_n=\lfloor \ln (n)\rfloor$, $M_n=\lfloor 10 \sqrt{n} \rfloor$ in the simulation of the $\widehat{Z}_{n,m}(s)$. For an explanation for the choice of these parameters see again Section 5 of \cite{MarVet17}. We only use paths in the simulation where $\mu_i([0,1],\mathbb{R})\neq 0$ whenever $\alpha_i \neq 0$, $i=1,2,3$.

In Figures \ref{rejection_curves_CoJ}--\ref{dens_est_CoJ2} we display the results from the simulation for the testing procedures from Theorem \ref{test_theo_CoJ} and Corollary \ref{test_cor_CoJ}. First we plot in Figure \ref{rejection_curves_CoJ} for all twelve cases the empirical rejection rates from Theorem \ref{test_theo_CoJ} in dependence of $\alpha \in [0,1]$ as for the plots in the left column of Figure \ref{figures_J}. The six plots on the left show the results for the cases where the hypothesis of the existence of common jumps is true. Similarly as in Figure \ref{figures_J} we observe that the empirical rejection rates match the postulated asymptotic level of the test better if $n$ is larger or if common jumps are larger on average. In all six cases where common jumps are present and for all values of $n$ the test overrejects. This is due to the fact that the asymptotically negligible terms in $\Phi_{2,T,n}^{(CoJ)}$ (terms which are contained in $\Phi_{2,T,n}^{(CoJ)}$ but do not contribute in the limit) tend to be positive, hence $\Phi_{2,T,n}^{(CoJ)}$ is on average systematically larger than $1$ (see Figure \ref{dens_est_CoJ}) which yields the bias. However, at least for $n=\numprint{25600}$ the observed rejection rates match the asymptotic level quite well. Only in the cases \texttt{III-j} and \texttt{III-m} where the jumps are on average very small the empirical rejection rates are still far higher than the asymptotic level. The results are worse in the mixed model than in the model where there are only common jumps, because idiosyncratic jumps contribute to the asymptotically vanishing error. The test has very good power against the alternative of idiosyncratic jumps as can be seen in the six plots on the right hand side which correspond to the cases where there are no common jumps. 

Figure \ref{dens_est_CoJ} shows density estimations for $\Phi_{2,T,n}^{(CoJ)}$ in all twelve cases. If there are common jumps it is visible in the density plots that $\Phi_{2,T,n}^{(CoJ)}$ converges to $1$ as $n \rightarrow \infty$. However for $n=100$, $n=\numprint{1600}$ in the cases \texttt{II-j}, \texttt{II-m} and for $n=100$, $n=\numprint{1600}$, $n=\numprint{25600}$ in the cases \texttt{III-j}, \texttt{III-m} the density peaks at a value significantly larger than $1$ which corresponds to the overrejection in Figure \ref{rejection_curves_CoJ}. Under the alternative of disjoint jumps $\Phi_{2,T,n}^{(CoJ)}$ tends to cluster around $1.5$ which corresponds to the results obtained in Example \ref{example_poisson1} for the one-dimensional setting.

\begin{figure}[tb]
\includegraphics[width=13.85cm]{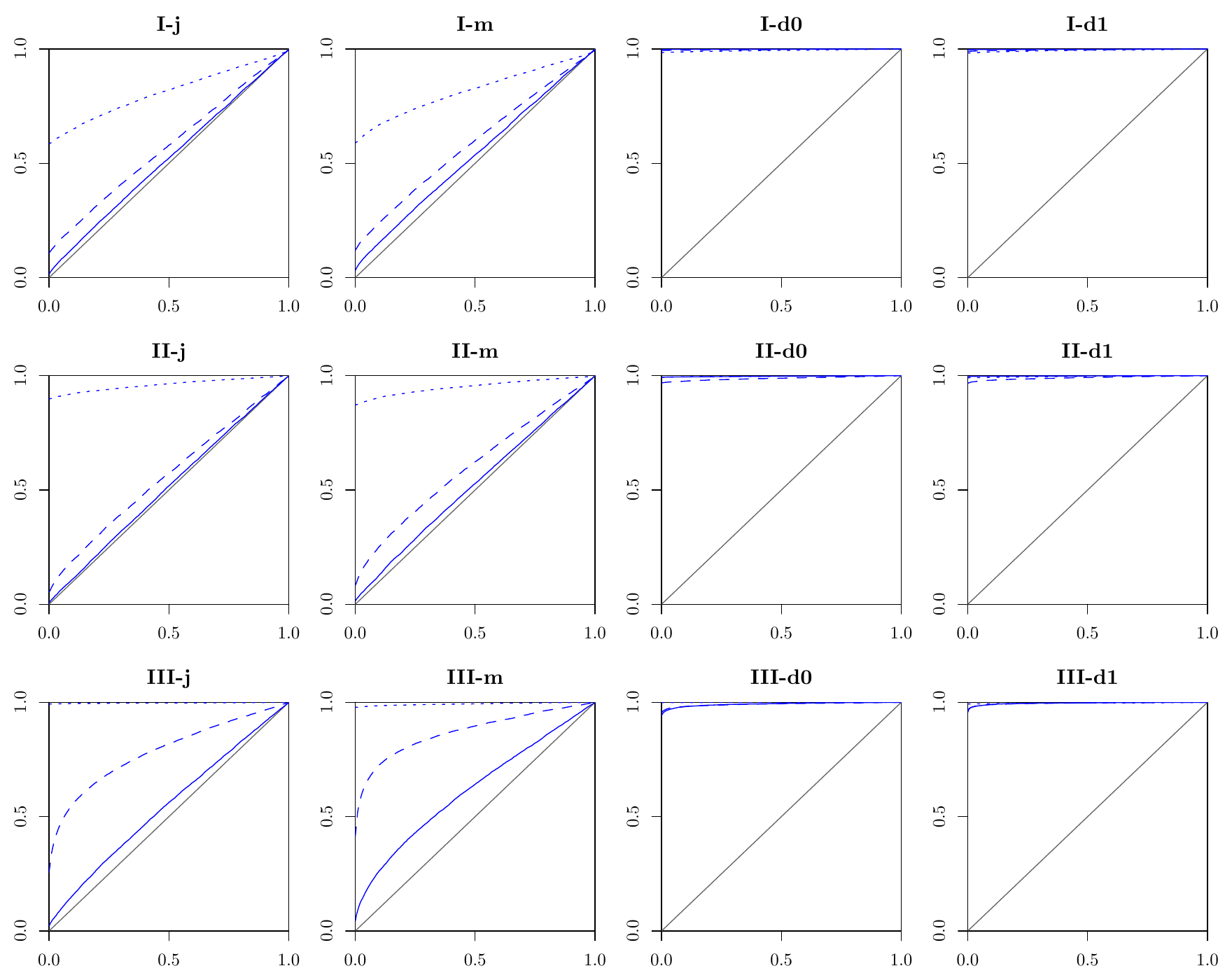}
\caption{Empirical rejection curves from the Monte Carlo simulation for the test derived from Corollary \ref{test_cor_CoJ}. The dotted lines represent the results for $n=100$, the dashed lines for $n=\numprint{1600}$ and the solid lines for $n=\numprint{25600}$. In each case $N=\numprint{10000}$ paths were simulated.}
\label{rejection_curves_CoJ2}
\end{figure}

Our simulation results from Theorem \ref{test_theo_CoJ} are worse than the results in the equidistant setting displayed in Figure 4 of \cite{JacTod09} while the power of our test is much better. This effect is partly due to the fact that, contrary to our approach, in \cite{JacTod09} idiosyncratic jumps, although their contribution is asymptotically negligible, are included in the estimation of the asymptotic variance in the central limit theorem. Hence they consistently overestimate the asymptotic variance which yields lower rejection rates. Further the asymptotically negligible terms in $\Phi_{2,T,n}^{(CoJ)}$ are larger relative to the asymptotically relevant terms in the asynchronous setting than in the setting of synchronous observation times which increases the rejection rates in the asynchronous setting.

\begin{figure}[tb]
\centering
\includegraphics[width=13.85cm]{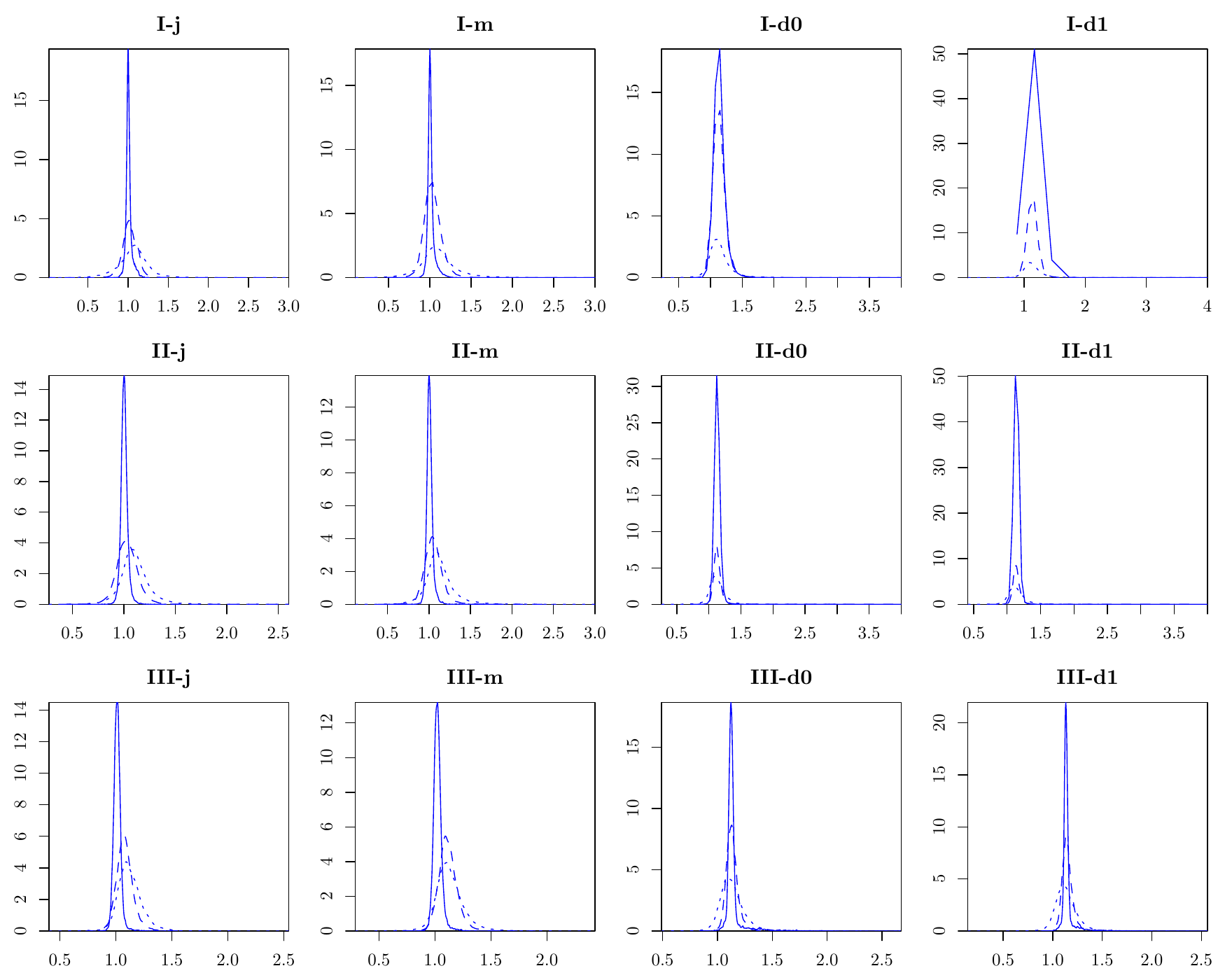}
\caption{Density estimations for $\widetilde{\Phi}_{2,T,n}^{(CoJ)}(\rho)$ from the Monte Carlo. The dotted lines correspond to $n=100$, the dashed lines to $n=\numprint{1600}$ and the solid lines to ${n=\numprint{25600}}$. In all cases $N=\numprint{10000}$ paths were simulated.}
\label{dens_est_CoJ2}
\end{figure}

The test from Corollary \ref{test_cor_CoJ} outperforms the test from Theorem \ref{test_theo_CoJ} in the simulation study. This can be seen in Figures \ref{rejection_curves_CoJ2} and \ref{dens_est_CoJ2} which show the results from the Monte Carlo simulation for the test from Corollary \ref{test_cor_CoJ} with $\rho=\numprint{0.75}$ (for the choice of $\rho=0.75$ see Figure \ref{figure_coj_choice_rho}) in the same fashion as for Theorem \ref{test_theo_CoJ} in Figures \ref{rejection_curves_CoJ} and \ref{dens_est_CoJ}. In the cases where common jumps are present we observe that the empirical rejection rates match the asymptotic level much better than in Figure \ref{rejection_curves_CoJ}. In Figure \ref{rejection_curves_CoJ2} we see that in the cases \texttt{I-j}, \texttt{I-m}, \texttt{II-j}, \texttt{II-m} we get good results already for $n=\numprint{1600}$ and in the cases \texttt{III-j}, \texttt{III-m} at least for $n=\numprint{25600}$. The power of the test from Corollary \ref{test_cor_CoJ} is practically as good as for the test from Theorem \ref{test_theo_CoJ}. Hence using the adjusted statistic $\widetilde{\Phi}_{2,T,n}^{(CoJ)}(\rho)$ instead of $\Phi_{2,T,n}^{(CoJ)}$ allows to get far better level results while the power of the test remains almost the same. 

Figure \ref{dens_est_CoJ2} shows that the adjusted estimator $\widetilde{\Phi}_{2,T,n}^{(CoJ)}(\rho)$ is much more centered around $1$ than $\Phi_{2,T,n}^{(CoJ)}$ if there exist common jumps which can be seen e.g.\ in the cases \texttt{III-j} and \texttt{III-m}. Further $\widetilde{\Phi}_{2,T,n}^{(CoJ)}(\rho)$ clusters around a value very close to $1$ also if there exist no common jumps. However, in the cases \texttt{*-d0} and \texttt{*-d1} the peak of the density still occurs at a value which is noticeably larger than $1$. 

Figure \ref{figure_coj_choice_rho} illustrates similarly as Figure \ref{figure_jump_choice_rho} for the test for jumps the performance of the test from Corollary \ref{test_cor_CoJ} in dependence of $\rho$. We choose the cases \texttt{III-j} and \texttt{III-d0} as representatives for the null hypothesis and the alternative. As expected the level as well as the power of the test decrease as $\rho$ increases. Here we get the lowest overall error for a value of $\rho$ close to $0.75$.

\begin{figure}[!htb]
\centering
\includegraphics[width=8cm]{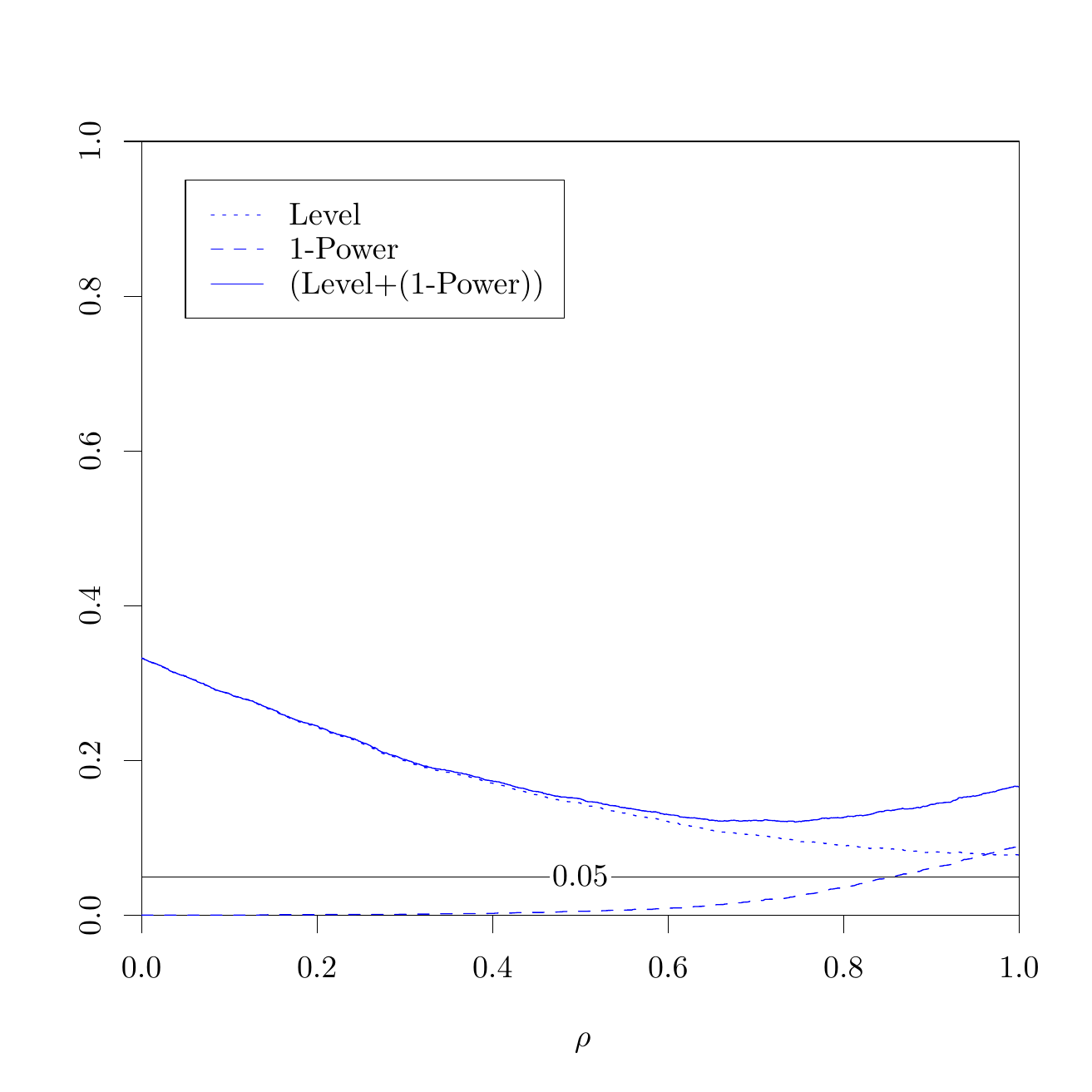}
\caption{This graphic shows for $\alpha=5\%$ and $n=\numprint{1600}$ the empirical rejection rate in the case \texttt{III-j} and $1$ minus the empirical rejection rate in the case \texttt{III-d0} from the Monte Carlo simulation based on Corollary \ref{test_cor_CoJ} as a function of $\rho \in [0,1]$.}\label{figure_coj_choice_rho}
\end{figure}

\FloatBarrier

\section{Structure of the proofs}\label{sec:proof_struc}
\def\theequation{5.\arabic{equation}}
\setcounter{equation}{0}

Throughout the proofs we will assume that the processes $b_s,\sigma_s$ and $\Gamma_s$ are bounded on $[0,T]$. By Conditions \ref{cond_cons_J} and \ref{cond_cons_CoJ} they are all locally bounded. A localization procedure then shows that the results for bounded processes can be carried over to the case of locally bounded processes (see e.g.\ Section 4.4.1 in \cite{JacPro12}).

Further we introduce the decomposition $X_t=X_0+B(q)_t+C_t+M(q)_t+N(q)_t$ of the It\^o semimartingale \eqref{ItoSemimart} or \eqref{ItoSemimart_2d} with
\begin{align*}
B(q)_t&=\int_0^t \big( b_s -\int(\delta(s,z)\mathds{1}_{\{\|\delta(s,z)\|\leq 1\}}-\delta(s,z)\mathds{1}_{\{\gamma(z)\leq 1/q\}})\lambda(dz)\big)ds,
\\C_t &= \int_0^t \sigma_s dW_s,
\\M(q)_t&=\int_0^t \int \delta(s,z) \mathds{1}_{\{\gamma(z)\leq 1/q\}}(\mu-\nu)(ds,dz),
\\N(q)_t&=\int_0^t \int \delta(s,z) \mathds{1}_{\{\gamma(z)>1/q\}}\mu(ds,dz).
\end{align*}
Here $q$ is a parameter which controls whether jumps are classified as small jumps or big jumps. \\

As we have already applied most of the techniques used in the upcoming proofs in \cite{MarVet17}, we keep the discussion of similar parts brief and add references to \cite{MarVet17} for more detailed arguments. Parts in the proofs which are new or specific to the tests introduced in this paper will be discussed in full detail.

\subsection{Proofs for Section \ref{Test_J}}\label{sec:proof_struc_J}
\begin{proof}[Proof of Theorem \ref{theo_cons_J}]
We obtain the following two representations for $\Phi_{k,T,n}^{(J)}$
\begin{align}
&\Phi_{k,T,n}^{(J)}= \frac{\sum_{l=0}^{k-1}\sum_{i\geq 1:t_{ki+l,n}\leq T} \left( \Delta_{ki+l,k,n} X \right)^4}{k V(g,\pi_n)_T}, \label{theo_cons_J_proof1}
\\&\Phi_{k,T,n}^{(J)}=\frac{ \frac{n}{k} \sum_{i \geq k:t_{i,n}\leq T} \left( \Delta_{i,k,n} X \right)^4}{ n V(g,\pi_n)_T}. \label{theo_cons_J_proof2}
\end{align}
Theorem 3.3.1 from \cite{JacPro12} yields
\begin{align*}
\sum_{i \geq 1:t_{ki+l}\leq T} \left( \Delta_{ki+l,k,n} X \right)^4 \overset{\mathbb{P}}{\longrightarrow} B_T^{(J)},~0\leq l \leq k-1,
\end{align*}
and $V(g,\pi_n)_T \overset{\mathbb{P}}{\longrightarrow} B_T^{(J)}$. Because of $B_T^{(J)} \neq 0$ on $\Omega_T^{(J)}$ plugging these into \eqref{theo_cons_J_proof1} yields the convergence to $1$ on $\Omega_T^{(J)}$.

Proposition A.2 in \cite{MarVet17} applied to the It\^o semimartingale $\widetilde{X}_t=(X_t,X_t)^*$ and the observation scheme defined by $\tilde{t}^{(l)}_{i,n}=t_{i,n}$, $l=1,2$, yields
\begin{align}\label{theo_cons_J_proof3}
n\sum_{i:t_{i,n}\leq T} \left( \Delta_{i,n} C \right)^4 \overset{\mathbb{P}}{\longrightarrow} C_{T}^{(J)}
\end{align} 
by Condition \ref{cond_cons_J} because of $\widetilde{\rho}_s=1$, $\widetilde{\sigma}^{(l)}_s=\sigma_s$, $l=1,2$, and $\widetilde{H}(t)=\widetilde{G}(t)=G(t)$. Similarly we obtain
\begin{align}\label{theo_cons_J_proof4}
\frac{n}{k}\sum_{i:t_{i,n}\leq T} \left( \Delta_{i,k,n} C \right)^4 \overset{\mathbb{P}}{\longrightarrow} kC_{k,T}^{(J)}.
\end{align} 
We then get the convergence on $\Omega_T^{(C)}$ from \eqref{theo_cons_J_proof2} using \eqref{theo_cons_J_proof3}, \eqref{theo_cons_J_proof4} and
\begin{align}\label{theo_cons_J_proof5}
\big|n V(g,[k'],\pi_n)_T-\frac{n}{k'} \sum_{i:t_{i}\leq T} \left( \Delta_{i,k',n} C \right)^4 \big|\mathds{1}_{\Omega_T^{(C)}} \overset{\mathbb{P}}{\longrightarrow} 0, \quad k'=1,k,
\end{align}
which will be proved in Section \ref{sec:proof_detail_J}. Note that because of $\int_0^T |\sigma_s| ds>0$ almost surely and because $G$ is strictly increasing we have $C_T^{(J)}>0$ almost surely.
\end{proof}

\begin{proof}[Proof of Theorem \ref{theo_clt_J}]
We prove
\begin{align}\label{proof_clt_J1}
\sqrt{n}\left(V(g,[k],\pi_n)_T-kV(g,\pi_n)_T \right) \overset{\mathcal{L}-s}{\longrightarrow} F_{k,T}^{(J)}
\end{align}
from which \eqref{theo_clt_J_1} easily follows. 

\textit{Step 1.} First we discretize $\sigma$ and restrict ourselves to the big jumps $\Delta N(q)_s$ of which there are only finitely many. We define $\sigma(r),C(r)$ for $r \in \mathbb{N}$ by $\sigma(r)_t=\sigma_{j 2^{-r}}$ if $t \in [(j-1)2^{-r},j2^{-r})$, $C(r)_t=\int_0^t \sigma(r)_s dW_s$. Further let $(S_{q,p})_p$ be an enumeration of the jump times of $N(q)$. By $\lceil x \rceil=\min \{n \in \mathbb{Z}: n \geq x\}$ we denozr for $x \in \mathbb{R}$ the smallest integer which is greater or equal to $x$. Using the notion of the jump times we modify the discretized processes $\sigma(r),C(r)$ further via
\begin{align*}
\tilde{\sigma}(r,q)_s=\begin{cases} \sigma_{S_{q,p}} & \mbox{if } s \in [S_{q,p},\lceil S_{q,p}/2^r\rceil/2^{r}) \\ 
\sigma(r)_s & \mbox{otherwise } \end{cases}, \quad \widetilde{C}(r,q)_t=\int_0^t \tilde{\sigma}(r,q)_s ds.
\end{align*}
Using this notation we then define
\begin{gather*}
R(n,q,r)
=4\sqrt{n}\sum_{i:t_{i,n} \leq T} (\Delta_{i,n} N(q))^3 \sum_{j=1}^{k-1} (k-j)(\Delta_{i+j,n}\widetilde{C}(r,q)+\Delta_{i-j,n}\widetilde{C}(r,q)). 
\end{gather*}
and show
\begin{align}\label{proof_clt_J1b}
 \lim_{q \rightarrow \infty} \limsup_{r \rightarrow \infty}\limsup_{n \rightarrow \infty} \mathbb{P}(|\sqrt{n}\left(V(g,[k],\pi_n)_T-kV(g,\pi_n)_T \right)-R(n,q,r)|> \varepsilon) \rightarrow 0
\end{align}
for all $\varepsilon>0$. The proof of \eqref{proof_clt_J1b} is given in Section \ref{sec:proof_detail_J}.

\textit{Step 2.} 
Next we prove
\begin{multline}\label{proof_clt_J10}
R(n,q,r)\overset{\mathcal{L}-s}{\longrightarrow} F_{k,T}^{(J)}(q,r)
 :=4\sum_{S_{q,p} \leq T} (\Delta N(q)_{S_{q,p}})^3 
 \\\times\big((\tilde{\sigma}(r,q)_{S_{q,p}-})^2 \xi_{k,-}(S_{{q,p}})+(\tilde{\sigma}(r,q)_{S_{q,p}})^2\xi_{k,+}(S_{{q,p}})\big)^{1/2}U_{S_{q,p}}
\end{multline}
for all $q>0$ and $r \in \mathbb{N}$.

To this end note that on the set $\Omega(n,q,r)$ where two different jumps of $N(q)$ are further apart than $|\pi_n|_T$ and any jump time of $N(q)$ is further away from $j 2^{-r}$ than $k|\pi_n|_T$ for any $j \in \{1,\ldots, \lfloor T 2^{r} \rfloor \}$ it holds
\begin{multline*}
R(n,q,r) \mathds{1}_{\Omega(n,q,r)} \nonumber
 =4\sum_{S_{q,p} \leq T } (\Delta N(q)_{S_{q,p}})^3 \sqrt{n} \sum_{j=1}^{k-1} (k-j)
 \\ \times(\tilde{\sigma}(r,q)_{S_{q,p}-}\Delta_{i_n(S_{q,p})-j,n}W+\tilde{\sigma}(r,q)_{S_{q,p}}\Delta_{i_n(S_{q,p})+j,n}W)\mathds{1}_{\Omega(n,q,r)}.
\end{multline*}
As in the proof of Proposition A.3 of \cite{MarVet17} it can be shown that Condition \ref{cond_clt_J}(ii) yields the $\mathcal{X}$-stable convergence
\begin{multline*}
\Big(\big(\sqrt{n}\sum_{j=1}^{k-1} (k-j)\Delta_{i_n(S_{q,p})-j,n}W,\sqrt{n}\sum_{j=1}^{k-1} (k-j)\Delta_{i_n(S_{q,p})+j,n}W\big)_{S_{q,p} \leq T} \Big)
\\ \overset{\mathcal{L}-s}{\longrightarrow}
\Big(\big(\sqrt{\xi_{k,-}(S_{q,p})}U_{S_{q,p},-},\sqrt{\xi_{k,+}(S_{q,p})}U_{S_{q,p},+}\big)_{S_{q,p} \leq T} \Big)
\end{multline*}
for standard normally distributed random variables $(U_{s,-},U_{s,+})$ which are independent of $\mathcal{F}$ and of the $\xi_{k,-}(s),\xi_{k,+}(s)$. Using this stable convergence, Proposition 2.2 in \cite{PodVet10} and the continuous mapping theorem we then obtain
\begin{multline*}
4\sum_{S_{q,p} \in P(q,T)} (\Delta N(q)_{S_{q,p}})^3\sqrt{n}\sum_{j=1}^{k-1}(k-j)
 \\ \times (\tilde{\sigma}(r,q)_{S_{q,p}-}\Delta_{i_n(S_{q,p})-j,n}W+\tilde{\sigma}(r,q){S_{q,p}}\Delta_{i_n(S_{q,p})+j,n}W)
 \overset{\mathcal{L}-s}{\longrightarrow}
 F_{k,T}^{(J)}(q,r).
\end{multline*}
Because of $\mathbb{P}(\Omega(n,q,r))\rightarrow 1$ as $n \rightarrow \infty$ for any $q,r$ this convergence yields \eqref{proof_clt_J10}.

\textit{Step 3.}
Finally we need to show
\begin{align}\label{proof_clt_J11}
 \lim_{q \rightarrow \infty} \limsup_{r \rightarrow \infty}
 \widetilde{\mathbb{P}}(|F_{k,T}^{(J)}-F_{k,T}^{(J)}(q,r)|> \varepsilon )=0,
\end{align}
for all $\varepsilon>0$. \eqref{proof_clt_J11} can be proven using that $\Gamma(\cdot,dy)$ has uniformly bounded first moments together with the boundedness of the jump sizes of $X$ respectively $N(q)$. 

Combining the results from Steps 1 to 3 we obtain \eqref{proof_clt_J1}.
\end{proof}

The structure of the proof of Theorem \ref{test_theo_J} and especially of \eqref{conv_Q_J} therein is identical to the structure of the proof of Theorem 4.2 and (A.27) in \cite{MarVet17}. 

\begin{proof}[Proof of Theorem \ref{test_theo_J}]
For proving \eqref{test_theo_level} we need to show
\begin{align*}
\widetilde{\mathbb{P}}\big(\sqrt{n} (V(g,[k],\pi_n)_T-k V(g,\pi_n)_T ) > \widehat{Q}_{k,T,n}(1-\alpha) \big| F^{(J)} \big) \rightarrow \alpha
\end{align*}
which follows from Theorem \ref{theo_clt_J} and 
\begin{align}\label{conv_Q_J}
\lim_{n \rightarrow \infty} \widetilde{\mathbb{P}}\big(\{|\widehat{Q}_{k,T,n}^{(J)}(\alpha) -Q_k^{(J)}(\alpha)|> \varepsilon \} \cap \Omega_T^{(J)}\big) \rightarrow 0
\end{align}
for all $\varepsilon>0$ and any $\alpha \in [0,1]$. Hence it remains to prove \eqref{conv_Q_J} for which we give a proof in Section \ref{sec:proof_detail_J}.

For proving \eqref{test_theo_power} we observe that $\Phi_{k,T,n}^{(J)}$ converges on $\Omega_T^{(C)}$ to a limit strictly greater than $1$ by Theorem \ref{theo_cons_J} and Condition \ref{cond_testproc_J}(ii). On the other hand we will show 
\begin{align}\label{test_power_klein_o_J}
\mathpzc{c}_{k,T,n}^{(J)} \mathds{1}_{\Omega_T^{(C)}}=o_{\widetilde{\mathbb{P}}}(1)
\end{align}
in Section \ref{sec:proof_detail_J}. Hence we obtain \eqref{test_theo_power}. 
\end{proof}

As a prerequisite for the proof of \eqref{conv_Q_J} we need the following Lemma which will be proven in Section \ref{sec:proof_detail_J}.

\begin{lemma}\label{lemma_conv_sigma}
Suppose Conditions \ref{cond_clt_J} and \ref{cond_testproc_J}(ii) are fulfilled and $S_p$ is a stopping time with $\Delta X_{S_p} \neq 0$ almost surely. Then it holds
\begin{align*}
\hat{\sigma}_n(S_p,-)\mathds{1}_{\{S_p < T\}} \overset{\mathbb{P}}{\longrightarrow} \sigma_{S_p-}\mathds{1}_{\{S_p < T\}},\quad \hat{\sigma}_n(S_p,+)\mathds{1}_{\{S_p < T\}} \overset{\mathbb{P}}{\longrightarrow} \sigma_{S_p}\mathds{1}_{\{S_p < T\}}.
\end{align*}
\end{lemma}

\begin{proposition}\label{lemma_test_J}
Suppose Condition \ref{cond_testproc_J} is fulfilled. Then it holds 
\begin{align}\label{test_theo_J_proof_A1}
\widetilde{\mathbb{P}} \big( \big\{\big| \frac{1}{M_n} \sum_{m=1}^{M_n} \mathds{1}_{\{\widehat{F}_{k,T,n,m}^{(J)} \leq \Upsilon\}} -\widetilde{\mathbb{P}}(F_{k,T}^{(J)}\leq \Upsilon|\mathcal{X})\big|>\varepsilon \big\} \cap \Omega_T^{(J)}\big) \rightarrow 0
\end{align}
for any $\mathcal{X}$-measurable random variable $\Upsilon$ and all $\varepsilon>0$.
\end{proposition}

The proof of Proposition \ref{lemma_test_J} requires Lemma \ref{lemma_conv_sigma} and will be given in Section \ref{sec:proof_detail_J}.

\begin{proof}[Proof of \eqref{conv_Q_J}]
We have
\begin{align*}
&\widetilde{\mathbb{P}}(\{\widehat{Q}_{k,T,n}^{(J)}(\alpha)>Q_k^{(J)}(\alpha)+\varepsilon\}\cap \Omega_T^{(J)})
\\&~~~=\widetilde{\mathbb{P}}\big(\big\{\frac{1}{M_n} \sum_{m=1}^{M_n} \mathds{1}_{\{\widehat{F}_{k,T,n,m}^{(J)}>Q_k^{(J)}(\alpha)+\varepsilon\}}>\frac{M_n-(\lfloor \alpha M_n \rfloor -1)}{M_n} \big\}\cap \Omega_T^{(J)}\big)
\\&~~~\leq\widetilde{\mathbb{P}}\big(\big\{\frac{1}{M_n} \sum_{m=1}^{M_n} \mathds{1}_{\{\widehat{F}_{k,T,n,m}^{(J)}>Q_k^{(J)}(\alpha)+\varepsilon\}}-\Upsilon(\alpha,\varepsilon)>(1-\alpha)-\Upsilon(\alpha,\varepsilon) \big\}\cap \Omega_T^{(J)}\big)
\end{align*}
with $\Upsilon(\alpha,\varepsilon)=\widetilde{\mathbb{P}}(F_{k,T}^{(J)}>Q_k^{(J)}(\alpha)+\varepsilon|\mathcal{X})$. Because the $\mathcal{X}$-conditional distribution of $F_{k,T}^{(J)}$ is continuous on $\Omega_T^{(J)}$, it holds $(1-\alpha)-\Upsilon(\alpha,\varepsilon)>0$ almost surely. \eqref{test_theo_J_proof_A1} then yields
\begin{align}\label{test_theo_J_proof4}
\widetilde{\mathbb{P}}(\{\widehat{Q}_{k,T,n}^{(J)}>Q_k^{(J)}(\alpha)+\varepsilon\}\cap \Omega_T^{(J)}) \rightarrow 0
\end{align}
because 
\begin{align*}
&\frac{1}{M_n} \sum_{m=1}^{M_n} \mathds{1}_{\{\widehat{F}_{k,T,n,m}^{(J)}>Q_k^{(J)}(\alpha)+\varepsilon\}}-\Upsilon(\alpha,\varepsilon)
\end{align*}
converges on $\Omega_T^{(J)}$ in probability to zero by \eqref{test_theo_J_proof_A1}. Analogously we derive
\begin{align*}
\widetilde{\mathbb{P}}(\{\widehat{Q}_{k,T,n}^{(J)}<Q_k^{(J)}(\alpha)-\varepsilon\}\cap \Omega_T^{(J)}) \rightarrow 0
\end{align*}
which together with \eqref{test_theo_J_proof4} yields \eqref{conv_Q_J}.
\end{proof}

\begin{proof}[Proof of Corollary \ref{test_cor_J}] 
From $A_{k,T,n}^{(J)} \overset{\mathbb{P}}{\longrightarrow} k^2C_{k,T}^{(J)}-kC_T^{(J)}$, which can be shown using techniques from the proof of (A.26) in \cite{MarVet17}, we deduce
\begin{align}\label{theo_cor_J_proof1}
\sqrt{n}\rho \frac{n^{-1}A_{k,T,n}^{(J)}}{kV(g,\pi_n)_T}\mathds{1}_{\Omega_T^{(J)}} =\rho n^{-1/2} \frac{A_{k,T,n}^{(J)}}{kV(g,\pi_n)_T}\mathds{1}_{\Omega_T^{(J)}} \overset{\mathbb{P}}{\longrightarrow} 0
\end{align}
on $\Omega_T^{(J)}$ and
\begin{align}\label{theo_cor_J_proof2}
\rho \frac{A_{k,T,n}^{(J)}}{knV(g,\pi_n)_T}\mathds{1}_{\Omega_T^{(C)}} \overset{\mathbb{P}}{\longrightarrow} \rho \frac{k^2 C_{k,T}^{(J)}-kC_T^{(J)}}{k C_T^{(J)}}\mathds{1}_{\Omega_T^{(C)}} 
\end{align}
on $\Omega_T^{(C)}$. 

\eqref{theo_cor_J_proof1} together with Theorem \ref{theo_clt_J} yields the $\mathcal{X}$-stable convergence
\begin{align*}
\sqrt{n}(\widetilde{\Phi}_{k,T,n}^{(J)}(\rho)-1)=\sqrt{n}(\Phi_{k,T,n}^{(J)}-\rho(kn)^{-1}A_{k,T,n}^{(J)}/V(g,\pi_n)_T-1) \overset{\mathcal{L}-s}{\longrightarrow}F_{k,T}^{(J)}
\end{align*}
on $\Omega_T^{(J)}$ and hence \eqref{test_cor_level_J} follows as in the proof of \eqref{test_theo_level}.

From \eqref{theo_cor_J_proof2} we derive
\begin{align*}
\Phi_{k,T,n}^{(J)}-\rho \frac{n^{-1}A_{k,T,n}^{(J)}}{kV(g,\pi_n)_T}\mathds{1}_{\Omega_T^{(J)}}-1 \overset{\mathbb{P}}{\longrightarrow} (1-\rho) \frac{k^2 C_{k,T}^{(J)}-kC_T^{(J)}}{k C_T^{(J)}}\mathds{1}_{\Omega_T^{(C)}}
\end{align*}
which is almost surely larger than $0$ by Condition \ref{cond_testproc_J}(ii) and hence \eqref{test_cor_power_J} follows as in the proof of \eqref{test_theo_power}.

\end{proof}

\subsection{Proofs for Section \ref{Test_CoJ}}\label{sec:proof:CoJ}

\begin{proof}[Proof of Theorem \ref{theo_cons_CoJ1}] From (2.2) in \cite{MarVet17} we obtain
\begin{align*}
&V(f,\pi_n)_T \overset{\mathbb{P}}{\longrightarrow} B_T^{(CoJ)},
\\&V(f,[k],\pi_n)_T=\sum_{l_1,l_2=0}^{k-1} \sum_{i,j \geq k:t_{ki+l_1,n}^{(1)}\wedge t_{kj+l_2,n}^{(2)}\leq T} f(\Delta_{ki+l_1,k,n}^{(1)}X^{(1)},\Delta_{kj+l_2,k,n}^{(2)}X^{(2)})
\\&~~~~~~~~~~~~~~~~~~~~~~~~~~~~~~~~~~~~~~~~\times \mathds{1}_{\{\mathcal{I}_{ki+l_1,k,n}\cap \mathcal{I}_{kj+l_2,k,n} \neq \emptyset\}}
\\&~~~~~~~~~~~~~~~~~\overset{\mathbb{P}}{\longrightarrow} \sum_{l_1,l_2=0}^{k-1} B_T^{(CoJ)} = k^2 B_T^{(CoJ)},
\end{align*}
which yields \eqref{theo_cons_CoJ1_conv}. 
\end{proof}

\begin{proof}[Proof of Theorem \ref{theo_cons_CoJ2}] It holds
\begin{align*}
\Phi_{k,T,n}^{(CoJ)}=\frac{n/k^2 V(f,[k],\pi_n)_T}{n V(f,\pi_n)_T}
\end{align*} 
and the proof of
\begin{align*}
(n/k^2 V(f,[k],\pi_n)_T, n V(f,\pi_n)_T) \overset{\mathcal{L}-s}{\longrightarrow}
(D_{k,T}^{(CoJ)}+k C_{k,T}^{(CoJ)}, D_{1,T}^{(CoJ)}+C_{1,T}^{(CoJ)} )
\end{align*}
is similar to the proof of Theorem 3.2 in \cite{MarVet17}. This yields \eqref{theo_cons_CoJ2_conv} because by Condition \ref{cond_clt_coJ} we have $\int_0^T|\sigma_s^{(1)}\sigma_s^{(2)}|ds>0$ which together with the fact that $H_1$ is strictly increasing guarantees $C_{1,T}^{(CoJ)}+D_{1,T}^{(CoJ)}\geq C_{1,T}^{(CoJ)}>0.$ \end{proof}

The proof of Theorem \ref{theo_clt_CoJ} has the same structure as the proof of Theorem \ref{theo_clt_J}.

\begin{proof}[Proof of Theorem \ref{theo_clt_CoJ}]
We prove
\begin{align}\label{proof_theo_clt_CoJ_step0}
\sqrt{n}\left(V(f,[2],\pi_n)_T-4 V(f,\pi_n)_T\right) \overset{\mathcal{L}-s}{\longrightarrow} F_{2,T}^{(CoJ)}
\end{align}
from which \eqref{theo_clt_CoJ1} easily follows. 

\textit{Step 1.} We introduce discretized versions $\tilde{\sigma}(q,r), \widetilde{C}(r,q)$ of $\sigma, C$ as in Step 1 in the proof of Theorem \ref{theo_clt_J} with the only difference that by $(S_{q,p})_p$ we denote an enumeration of the jump times of $N^{(1)}(q) N^{(2)}(q)$ and not of $N(q)$. Thereby we turn the focus to the common jumps. We then show
\begin{align}\label{proof_theo_clt_CoJ_step1}
\lim_{q \rightarrow \infty} \limsup_{r \rightarrow \infty} \limsup_{n \rightarrow \infty} \mathbb{P}(
|\sqrt{n}\left(V(f,[2],\pi_n)_T-4 V(f,\pi_n)_T\right)-\widetilde{R}(n,q,r)|>\varepsilon)=0
\end{align}
for all $\varepsilon>0$ where
\begin{align*}
&\widetilde{R}(n,q,r)=\widetilde{R}^{(1)}(n,q,r)+\widetilde{R}^{(2)}(n,q,r),
\\&\widetilde{R}^{(l)}(n,q,r)=\sqrt{n}\sum_{i,j\geq 2:t_{i,n}^{(l)}\wedge t_{j,n}^{(3-l)}\leq T} 
2 \big( \Delta_{i-1,1,n}^{(l)}N^{(l)}(q) \Delta_{i,1,n}^{(l)}\widetilde{C}^{(l)}(r,q)
\\&~~~~~~+\Delta_{i-1,1,n}^{(l)}\widetilde{C}^{(l)}(r,q) \Delta_{i,1,n}^{(l)}N^{(l)} (q)\big) 
 \big(\Delta_{j,2,n}^{(3-l)}N^{(3-l)}(q) \big)^2\mathds{1}_{\{\mathcal{I}_{i,2,n}^{(l)}\cap \mathcal{I}_{j,2,n}^{(3-l)}\neq \emptyset\}},~l=1,2.
\end{align*}
The proof for \eqref{proof_theo_clt_CoJ_step1} will be given in Section \ref{sec:proof_detail_CoJ}.

\textit{Step 2.} Next we will show
\begin{align}\label{proof_theo_clt_CoJ_step2}
\widetilde{R}(n,q,r) \overset{\mathcal{L}-s}{\longrightarrow} F_T^{(CoJ)}(q,r)
\end{align}
for any $q>0$ and $r \in \mathbb{N}$ where
\begin{small}\begin{align*}
&F_{T}^{(CoJ)}(q,r):=4\sum_{S_{q,p} \in P(q,T)} \Delta X^{(1)}_{S_{q,p}} \Delta X^{(2)}_{S_{q,p}}
\\&~\times\Big[\Delta X^{(2)}_{S_{q,p}}\Big( \tilde{\sigma}^{(1)}(r,q)_{S_{q,p}-}\sqrt{ 
\mathcal{L}({S_{q,p}})}U_{S_{q,p}}^{(1),-}+\tilde{\sigma}^{(1)}(r,q)_{S_{q,p}}
\sqrt{\mathcal{R}({S_{q,p}})} U_{S_{q,p}}^{(1),+}
\\&~~+\sqrt{(\tilde{\sigma}^{(1)}(r,q)_{S_{q,p}-})^2(\mathcal{L}^{(1)}-\mathcal{L})({S_{q,p}})
+(\tilde{\sigma}^{(1)}(r,q)_{S_{q,p}})^2(\mathcal{R}^{(1)}-\mathcal{R})({S_{q,p}}) } U_{S_{q,p}}^{(2)}\Big) 
\\&~+\Delta X^{(1)}_{S_{q,p}} \Big(\tilde{\sigma}^{(2)}(r,q)_{S_{q,p}-}\tilde{\rho}(r,q)_{S_{q,p}-}\sqrt{\mathcal{L}({S_{q,p}}) }U_{S_p}^{(1),-}+\tilde{\sigma}^{(2)}(r,q)_{S_{q,p}}\tilde{\rho}(r,q)_{S_{q,p}}\sqrt{\mathcal{R}({S_p}) }U_{S_{q,p}}^{(1),+}
\\&~~+\big((\tilde{\sigma}^{(2)}(r,q)_{S_{q,p}-})^2(1-(\tilde{\rho}(r,q)_{S_{q,p}-})^2)\mathcal{L}({S_{q,p}})
\\&~~~~~~~~~~~~~~~~~~~~~+(\tilde{\sigma}^{(2)}(r,q)_{S_{q,p}})^2(1-(\tilde{\rho}(r,q)_{S_{q,p}})^2)\mathcal{R}({S_{q,p}}) \big)^{1/2}U_{S_{q,p}}^{(3)}
\\&~~+\sqrt{(\tilde{\sigma}^{(2)}(r,q)_{S_{q,p}-})^2(\mathcal{L}^{(2)}-
\mathcal{L})({S_{q,p}})+(\tilde{\sigma}^{(2)}(r,q)_{S_{q,p}})^2(\mathcal{R}^{(2)}-
\mathcal{R})({S_{q,p}})}U_{S_{q,p}}^{(4)}\Big)\Big].
\end{align*}
\end{small}Here $\tilde{\sigma}^{(1)}(r,q),\tilde{\sigma}^{(2)}(r,q),\tilde{\rho}(r,q)$ are discretized versions of $\sigma^{(1)},\sigma^{(2)},\rho$ which are obtained in the same way as $\tilde{\sigma}(r,q),\widetilde{C}(r,q)$ above. The proof of \eqref{proof_theo_clt_CoJ_step2} will also be sketched in Section \ref{sec:proof_detail_CoJ}.

\textit{Step 3.} Finally we consider
\begin{align}\label{proof_theo_clt_CoJ_step3}
 \lim_{q \rightarrow \infty} \limsup_{r \rightarrow \infty}
 \widetilde{\mathbb{P}}(|F_{T}^{(CoJ)}-F_{T}^{(CoJ)}(q,r)|> \varepsilon )=0
\end{align}
for all $\varepsilon>0$ which can be proven using that the first moments of $\Gamma(\cdot,dy)$ are uniformly bounded together with the boundedness of the jump sizes of $X$ respectively $N(q)$. 

Combining \eqref{proof_theo_clt_CoJ_step1}, \eqref{proof_theo_clt_CoJ_step2} and \eqref{proof_theo_clt_CoJ_step3} yields \eqref{proof_theo_clt_CoJ_step0}.
\end{proof}

\begin{proof}[Proof of Theorem \ref{test_theo_CoJ}]
For proving \eqref{test_theo_level_CoJ} we will prove
\begin{gather}
\widetilde{\mathbb{P}}\big(\sqrt{n} \left|V(f,[2],\pi_n)_T-4 V(f,\pi_n)_T \right| > \widehat{Q}^{(CoJ)}_{T,n}(1-\alpha) \big| F^{(J)} \big) \rightarrow \alpha \nonumber
\end{gather}
which follows from Theorem \ref{theo_clt_CoJ} and 
\begin{align}\label{conv_Q_CoJ}
\lim_{n \rightarrow \infty} \widetilde{\mathbb{P}}\big(\big\{|\widehat{Q}_{T,n}^{(CoJ)}(\alpha) -Q^{(CoJ)}(\alpha)|> \varepsilon \big\}\cap \Omega_T^{(CoJ)}\big) =0
\end{align}
for all $\varepsilon>0$ and any $\alpha \in [0,1]$. Hence it remains to prove \eqref{conv_Q_CoJ} for which we give a proof in Section \ref{sec:proof_detail_CoJ}.

For proving \eqref{test_theo_power_CoJ} we observe that $\Phi^{(CoJ)}_{2,T,n}$ converges on the given $F$ to a random variable which is under Condition \ref{cond_testproc_CoJ}(iv) almost surely different from $1$ by Theorem \ref{theo_cons_CoJ2}. Note to this end that $C_{k,T}^{(CoJ)},C_{1,T}^{(CoJ)}$ are $\mathcal{F}$-conditional constant while the $\mathcal{F}$-conditional distribution of $D_{k,T}^{(CoJ)},D_{1,T}^{(CoJ)}$ is continuous. Hence $kC_{k,T}^{(CoJ)}\neq C_{1,T}^{(CoJ)}$ or $D_{k,T}^{(CoJ)}\neq D_{1,T}^{(CoJ)}$ almost surely imply 
$$kC_{k,T}^{(CoJ)}+D_{k,T}^{(CoJ)}\neq C_{1,T}^{(CoJ)}+D_{1,T}^{(CoJ)}~~ a.s. $$ Hence \eqref{test_theo_power_CoJ} follows from
\begin{align}\label{test_power_klein_o_CoJ}
\mathpzc{c}_{2,T,n}^{(CoJ)} \mathds{1}_{\Omega_T^{(nCoJ)}} =o_\mathbb{P}(1)
\end{align}
which will be proved in Section \ref{sec:proof_detail_CoJ}.
\end{proof}

As a prerequisite for the proof of \eqref{conv_Q_CoJ} we need the following Lemma which will be proven in Section \ref{sec:proof_detail_CoJ}.

\begin{lemma}\label{lemma_conv_sigma_CoJ}
Suppose Conditions \ref{cond_clt_coJ} and \ref{cond_testproc_CoJ}(ii) are fulfilled and $S_p$ is a stopping time with $\Delta X_{S_p} \neq 0$ almost surely. Then for $l=1,2$ it holds
\begin{align}
&\hat{\sigma}_n^{(l)}(S_p,-)\mathds{1}_{\{S_p < T\}} \overset{\mathbb{P}}{\longrightarrow} \sigma^{(l)}_{S_p-}\mathds{1}_{\{S_p < T\}},\quad \hat{\sigma}^{(l)}_n(S_p,+)\mathds{1}_{\{S_p < T\}} \overset{\mathbb{P}}{\longrightarrow} \sigma^{(l)}_{S_p}\mathds{1}_{\{S_p < T\}},\label{lemma_conv_sigma_sigma}
\\&\hat{\rho}_n(S_p,-)\mathds{1}_{\{S_p < T\}} \overset{\mathbb{P}}{\longrightarrow} \rho_{S_p-}\mathds{1}_{\{S_p < T\}},\quad \hat{\rho}_n(S_p,+)\mathds{1}_{\{S_p < T\}} \overset{\mathbb{P}}{\longrightarrow} \rho_{S_p}\mathds{1}_{\{S_p < T\}}.\label{lemma_conv_sigma_rho}
\end{align}
\end{lemma}

\begin{proposition}\label{lemma_test_CoJ}
Suppose Condition \ref{cond_testproc_CoJ} is fulfilled. Then it holds 
\begin{align}\label{lemma_test_CoJ1}
\widetilde{\mathbb{P}} \big( \big\{\big| \frac{1}{M_n} \sum_{m=1}^{M_n} \mathds{1}_{\{\widehat{F}_{T,n,m}^{(CoJ)} \leq \Upsilon\}} -\widetilde{\mathbb{P}}(F_{2,T}^{(CoJ)}\leq \Upsilon|\mathcal{X})\big|>\varepsilon \big\} \cap \Omega_T^{(J)}\big) \rightarrow 0
\end{align}
for any $\mathcal{X}$-measurable random variable $\Upsilon$ and all $\varepsilon>0$.
\end{proposition}

The proof of Proposition \ref{lemma_test_CoJ} requires Lemma \ref{lemma_conv_sigma_CoJ} and will be given in Section \ref{sec:proof_detail_CoJ}. The proof for \eqref{conv_Q_CoJ} based on Proposition \ref{lemma_test_CoJ} is then identical to the proof of \eqref{conv_Q_J} based on Proposition \ref{lemma_test_J} and is therefore omitted here.

\begin{proof}[Proof of Corollary \ref{test_cor_CoJ}]
Using arguments from the proof of Theorem \ref{theo_cons_CoJ2} and the proof of (A.26) in \cite{MarVet17} we derive
\begin{align}\label{theo_cor_proof1}
A_{T,n}^{(CoJ)} \overset{\mathcal{L}-s}{\longrightarrow} 4(2C_{2,T}^{(CoJ)}-C_{1,T}^{(CoJ)})+4(D_{2,T}^{(CoJ)}-D_{1,T}^{(CoJ)}).
\end{align}
Hence on $\Omega_T^{(CoJ)}$ it holds 
\begin{align*}
\sqrt{n}\frac{n^{-1} A_{T,n}^{(CoJ)}}{4V(f,\pi_n)_T} \mathds{1}_{\Omega_T^{(CoJ)}} \overset{\mathbb{P}}{\longrightarrow} 0
\end{align*}
and combining this with \eqref{theo_clt_CoJ1} yields the $\mathcal{X}$-stable convergence
\begin{align}\label{theo_cor_proof2}
\sqrt{n}\Big( \Phi_{2,T,n}^{(CoJ)}-\rho\frac{n^{-1} A_{T,n}^{(CoJ)}}{4V(f,\pi_n)_T} -1 \Big) \overset{\mathcal{L}-s}{\longrightarrow} \frac{F_{2,T}^{(CoJ)}}{4B_T^{(CoJ)}}
\end{align}
on $\Omega_T^{(CoJ)}$. Replacing \eqref{theo_clt_CoJ1} with \eqref{theo_cor_proof2} in the proof of \eqref{test_theo_level_CoJ} yields \eqref{test_cor_level_CoJ}.

Under the alternative we obtain using Theorem \ref{theo_cons_CoJ2} and \eqref{theo_cor_proof1}
\begin{multline*}
\big(\Phi_{2,T,n}^{(CoJ)}-\rho\frac{n^{-1}A_{T,n}^{(CoJ)}}{4V(f,\pi_n)_T}-1\big) \mathds{1}_{\Omega_T^{(nCoJ)}}
\\ \overset{\mathcal{L}-s}{\longrightarrow}
(1-\rho)\frac{(2C_{2,T}^{(CoJ)}-C_{1,T}^{(CoJ)})+(D_{2,T}^{(CoJ)}-D_{1,T}^{(CoJ)})}{C_{1,T}^{(CoJ)}+D_{1,T}^{(CoJ)}}\mathds{1}_{\Omega_T^{(nCoJ)}}
\end{multline*}
where the limit is almost surely different from zero by Condition \ref{cond_testproc_CoJ}(iv). We then obtain \eqref{test_cor_power_CoJ} as in the proof of Theorem \ref{test_theo_CoJ} because of $\mathpzc{c}_{2,T,n}^{(CoJ)}\mathds{1}_{\Omega_T^{(nCoJ)}}=o_\mathbb{P}(1)$; compare \eqref{test_power_klein_o_CoJ}.
\end{proof}

\bibliographystyle{chicago}
\bibliography{bibliography}

\begin{thebibliography}{}

\bibitem[\protect\citeauthoryear{A{\"{\i}}t-Sahalia and
  Jacod}{A{\"{\i}}t-Sahalia and Jacod}{2009}]{aitjac2009}
A{\"{\i}}t-Sahalia, Y. and J.~Jacod (2009).
\newblock Testing for jumps in a discretely observed process.
\newblock {\em Ann. Statist.\/}~{\em 37\/}(1), 184--222.

\bibitem[\protect\citeauthoryear{A\"it-Sahalia and Jacod}{A\"it-Sahalia and
  Jacod}{2014}]{AitJac14}
A\"it-Sahalia, Y. and J.~Jacod (2014).
\newblock {\em High-Frequency Financial Econometrics}.
\newblock Princeton University Press.
\newblock ISBN: 0-69116-143-3.

\bibitem[\protect\citeauthoryear{Bibinger and Vetter}{Bibinger and
  Vetter}{2015}]{BibVet15}
Bibinger, M. and M.~Vetter (2015).
\newblock Estimating the quadratic covariation of an asynchronously observed
  semimartingale with jumps.
\newblock {\em Annals of the Institute of Statistical Mathematics\/}~{\em 67},
  707--743.

\bibitem[\protect\citeauthoryear{Billingsley}{Billingsley}{1999}]{Bil99}
Billingsley, P. (1999).
\newblock {\em Convergence of probability measures\/} (2 ed.).
\newblock Wiley series in probability and statistics. Probability and
  statistics section. Wiley.

\bibitem[\protect\citeauthoryear{Brockwell and Davis}{Brockwell and
  Davis}{1991}]{BroDav91}
Brockwell, P.~J. and R.~A. Davis (1991).
\newblock {\em Time Series: Theory and Methods\/} (2 ed.).
\newblock Springer Series in Statistics. Springer.
\newblock ISBN: 1-441-90319-4.

\bibitem[\protect\citeauthoryear{Hayashi and Yoshida}{Hayashi and
  Yoshida}{2005}]{hayyos05}
Hayashi, T. and N.~Yoshida (2005).
\newblock On covariance estimation of non-synchronously observed diffusion
  processes.
\newblock {\em Bernoulli\/}~{\em 11\/}(2), 359--379.

\bibitem[\protect\citeauthoryear{Hayashi and Yoshida}{Hayashi and
  Yoshida}{2008}]{HayYos08}
Hayashi, T. and N.~Yoshida (2008).
\newblock Asymptotic normality of a covariance estimator for non-synchronously
  observed processes.
\newblock {\em Annals of the Institute of Statistical Mathematics\/}~{\em
  60\/}(2), 367--406.

\bibitem[\protect\citeauthoryear{Huang and Tauchen}{Huang and
  Tauchen}{2006}]{HuaTau06}
Huang, X. and G.~Tauchen (2006).
\newblock The relative contribution of jumps to total price variance.
\newblock {\em J. Financial Econometrics\/}~{\em 4}, 456--499.

\bibitem[\protect\citeauthoryear{Jacod and Protter}{Jacod and
  Protter}{2012}]{JacPro12}
Jacod, J. and P.~Protter (2012).
\newblock {\em Discretization of Processes}.
\newblock Springer.
\newblock ISBN: 3-64224-126-3.

\bibitem[\protect\citeauthoryear{Jacod and Shiryaev}{Jacod and
  Shiryaev}{2002}]{JacShi02}
Jacod, J. and A.~Shiryaev (2002).
\newblock {\em Limit Theorems for Stochastic Processes\/} (2 ed.).
\newblock Springer.
\newblock ISBN: 3-540-43932-3.

\bibitem[\protect\citeauthoryear{Jacod and Todorov}{Jacod and
  Todorov}{2009}]{JacTod09}
Jacod, J. and V.~Todorov (2009).
\newblock Testing for common arrivals of jumps for discretely observed
  multidimensional processes.
\newblock {\em The Annals of Statistics\/}~{\em 37\/}(1), 1792--1838.

\bibitem[\protect\citeauthoryear{Martin and Vetter}{Martin and
  Vetter}{2017}]{MarVet17}
Martin, O. and M.~Vetter (2017).
\newblock Testing for simultaneous jumps in case of asynchronous observations.
\newblock {\em To appear in: Bernoulli\/}.

\bibitem[\protect\citeauthoryear{Podolskij and Vetter}{Podolskij and
  Vetter}{2010}]{PodVet10}
Podolskij, M. and M.~Vetter (2010).
\newblock Understanding limit theorems for semimartingales: a short survey.
\newblock {\em Statistica Neerlandica\/}~{\em 64}, 329--351.

\bibitem[\protect\citeauthoryear{Todorov and Tauchen}{Todorov and
  Tauchen}{2011}]{TodTau11}
Todorov, V. and G.~Tauchen (2011).
\newblock Volatility jumps.
\newblock {\em Journal of Business and Economic Statistics\/}~{\em 29},
  356--371.

\end{thebibliography}

\newpage

\setcounter{page}{1}

\section{Online-Appendix: Proofs for the arguments omitted in Section \ref{sec:proof_struc}}\label{sec:proof_detail}
\def\theequation{6.\arabic{equation}}
\setcounter{equation}{0}

We will make repeatedly use of the following estimates. Throughout the upcoming proofs $K$ and $K_q$ will denote generic constants, the latter dependent on $q$, to simplify notation. 
\begin{lemma}\label{elem_ineq}
If Condition \ref{cond_cons_J} or Condition \ref{cond_cons_CoJ} is fulfilled and the processes $b_t,\sigma_t, \Gamma_t$ are bounded there exist constants $K_p, K_{p'},K_{p,q},\widetilde{K}_{p,q},e_q \geq 0$ such that
\begin{align}
&\|B(q)_{s+t}-B(q)_s\|^p \leq K_{p,q} t^p, \label{elem_ineq_B}
\\ &\mathbb{E} \big[ \|C_{s+t}-C_{s}\|^p|\mathcal{F}_s\big] \leq K_p t^{p/2}, \label{elem_ineq_C}
\\ &\mathbb{E} \big[ \|M(q)_{s+t}-M(q)_{s}\|^{p'}|\mathcal{F}_s\big] \leq K_{p'} t e_q, \label{elem_ineq_M}
\\ &\mathbb{E} \big[ \|N(q)_{s+t}-N(q)_{s}\|^p|\mathcal{F}_s\big] \leq \widetilde{K}_{p,q} t+K_{p,q}t^p, \label{elem_ineq_N}
\\ &\mathbb{E} \big[ \|X_{s+t}-X_{s}\|^{p'}|\mathcal{F}_s\big] \leq K_{p'} t, \label{elem_ineq_X}
\end{align}
for all $s , t \geq 0$ with $s+t \leq T$ and all $q >0$, $p \geq 1$, $p' \geq 2$. Here, $e_q$ can be chosen such that $e_q \rightarrow 0$ for $q \rightarrow \infty$. For $p \geq 2$ the constant $\widetilde{K}_{p,q}$ may be chosen independently of $q$.
\end{lemma}

\begin{proof}[Proof of Lemma \ref{elem_ineq}]
The inequalities \eqref{elem_ineq_B}--\eqref{elem_ineq_X} follow from Condition \ref{cond_cons_J} respectively \ref{cond_cons_CoJ} and inequalities (2.1.33), (2.1.34), (2.1.37), (2.1.41) in \cite{JacPro12}.
\end{proof}

The following elementary statement will be used frequently in the upcoming proofs.

\begin{lemma}\label{conv_cond_expec}
Let $(Y_n)_{n \in \mathbb{N}}$ be a sequence of real-valued integrable random variables on $(\Omega,\mathcal{F},\mathbb{P})$ and $\mathcal{G} \subset \mathcal{F}$ a sub-$\sigma$-algebra. Then it holds
\begin{align*}
\mathbb{E}\big[|Y_n|\big|\mathcal{G}\big] \overset{\mathbb{P}}{\longrightarrow} 0 \Rightarrow Y_n \overset{\mathbb{P}}{\longrightarrow} 0
\end{align*}
\end{lemma}

\begin{proof} For $\varepsilon <1$ it holds by Markov's inequality and Jensen's inequality
\begin{gather*}
\mathbb{P}(|Y_n|>\varepsilon)=\mathbb{P}(|Y_n|\wedge 1>\varepsilon) \leq \frac{\mathbb{E}[|Y_n|\wedge 1] }{\varepsilon} 
=\frac{\mathbb{E}\big[\mathbb{E}\big[|Y_n|\wedge 1\big|\mathcal{G}\big]\big] }{\varepsilon}
 \leq \frac{\mathbb{E}\big[\mathbb{E}\big[|Y_n|\big|\mathcal{G}\big]\wedge 1\big] }{\varepsilon}.
\end{gather*}
This estimate yields the claim as a sequence of random variables $(\widetilde{Y}_n)_{n \in \mathbb{N}}$ converges in probability to zero if and only if
$\mathbb{E}[|\widetilde{Y}_n| \wedge c] \rightarrow 0$ for any $c>0$.
\end{proof}

\subsection{Proofs for Section \ref{Test_J}}\label{sec:proof_detail_J}
\begin{proof}[Proof of \eqref{theo_cons_J_proof5}] We will give the proof for $k'=1$ only as the proof for general $k'=k$ is identical.
It holds $X_t \mathds{1}_{\Omega_T^{(C)}}=(B_t+C_t)\mathds{1}_{\Omega_T^{(C)}}$ with
\begin{align}\label{definition_B}
B_t=\int_0^t\big(b_s-\int_{\mathbb{R}^2}\delta(s,z)\mathds{1}_{\{\|\delta(s,z)\|\leq 1\}}\lambda(dz)\big)ds
\end{align}
since $\int_0^T \int \delta(s,z)\mu(ds,dz) \equiv 0$ on $\Omega_T^{(C)}$.

On $\Omega_T^{(C)}$ we obtain using \eqref{elem_ineq_B} and \eqref{elem_ineq_C}
\begin{align*}
&\mathbb{E} \big[\big|n V(g, \pi_n)_T-n\sum_{i:t_{i,n}\leq T} \left( \Delta_{i,n} C \right)^4 \big|\mathds{1}_{\Omega_T^{(C)}} \big| \mathcal{S} \big] 
\\&~~~~~~\leq nK \sum_{i:t_{i,n}\leq T} \Big( \mathbb{E}\big[\left( \Delta_{i,n} B \right)^4 \big|\mathcal{S}\big]+\mathbb{E}\big[\left| \Delta_{i,n} B \right|^3 |\Delta_{i,n} C| \big|\mathcal{S}\big] 
\\&~~~~~~~~~~~~~~~~~~~~~+\mathbb{E}\big[\left( \Delta_{i,n} B \right)^2 \left( \Delta_{i,n} C \right)^2\big|\mathcal{S}\big] +\mathbb{E}\big[|\Delta_{i,n} B| \left| \Delta_{i,n} C \right|^3\big|\mathcal{S}\big] \Big)
\\&~~~~~~ \leq nK \sum_{i:t_{i,n}\leq T} \big(K\big|\mathcal{I}_{i,n}\big|^4+K \big|\mathcal{I}_{i,n}\big|^{7/2}+K \big|\mathcal{I}_{i,n}\big|^{3}+K \big|\mathcal{I}_{i,n}\big|^{5/2} \big)
\\&~~~~~~ \leq n K \max\{T^{3/2},1\}(|\pi_n|_T)^{1/2}\sum_{i:t_{i,n}\leq T} \big|\mathcal{I}_{i,n}\big|^2
\end{align*}
which vanishes as $n \rightarrow \infty$ by Condition \ref{cond_cons_J}(i) and (ii). \eqref{theo_cons_J_proof5} then follows from Lemma \ref{conv_cond_expec}.
\end{proof}

\begin{proof}[Proof of \eqref{proof_clt_J1b}] To simplify notation we set $\Delta_{i,k',n}X=0$, $\mathcal{I}_{i,k',n}=\emptyset$, $k'=1,k$, whenever $t_{i,n} > T$ in this proof.\\
\textit{Step 1.}
We first prove
\begin{align}\label{proof_clt_J1b_step1}
 \limsup_{n \rightarrow \infty} \mathbb{P}(|\sqrt{n}\left(V(g,[k],\pi_n)_T-kV(g,\pi_n)_T \right)-R(n)|> \varepsilon) \rightarrow 0
\end{align}
where
\begin{align*}
R(n)=\sqrt{n} \sum_{i:t_{i,n}\leq T}4 \left(\Delta_{i,n} X \right)^3\sum_{j=1}^{k-1}(k-j) \big(
 \Delta_{i+j,n} X + \Delta_{i-j,n} X \big).
\end{align*}

Using the identity
\begin{multline}\label{proof_clt_J2}
(x_1+\ldots+x_k)^4=\sum_{i=1}^k (x_i)^4+4 \sum_{i=1}^k\sum_{j \neq i} (x_i)^3 x_j+6 \sum_{i=1}^k\sum_{j > i} (x_i)^2 (x_j)^2
\\+ 12\sum_{i=1}^k\sum_{j \neq i}\sum_{j' >j} (x_i)^2 x_j x_{j'}+24\sum_{i=1}^k \sum_{i'> i}\sum_{j >i'}\sum_{j' >j} x_i x_{i'} x_j  x_{j'}
\end{multline}
which is a specific case of the multinomial theorem, we derive
\begin{align}\label{proof_clt_J3}
&\sqrt{n}\left(V(g,[k],\pi_n)_T-k V(g,\pi_n)_T \right) \nonumber
\\&~~=\sqrt{n} \sum_{i:t_{i,n}\leq T}\sum_{j=1}^{k-1}(k-j) \big(
4 \left(\Delta_{i,n} X \right)^3 \left( \Delta_{i+j,n} X \right)
+4 \left(\Delta_{i,n} X \right)^3 \left( \Delta_{i-j,n} X \right)
\big) \nonumber
\\&~~~~+O_{\mathbb{P}}\Big(\sqrt{n} \sum_{i:t_{i,n}\leq T}\sum_{j=1}^{k-1}K_{k,j} \left(\Delta_{i,n} X \right)^2 \left( \Delta_{i+j,n} X \right)^2 \Big)
+ O_{\mathbb{P}}\left(\sqrt{n}|\pi_n|_T\right)
\end{align}
where we used the inequalities $|x_i x_i'|\leq (x_i)^2+(x_{i'})^2$, $|x_j x_{j'}|\leq (x_j)^2+(x_{j'})^2$ to include terms with powers $(2,1,1,0)$ and $(1,1,1,1)$ from \eqref{proof_clt_J2} into the first summand of the last line of \eqref{proof_clt_J3}. The last term in \eqref{proof_clt_J3} is due to boundary effects at $T$.

By iterated expectations and inequality \eqref{elem_ineq_X} we get
\begin{align*}
\sqrt{n}\mathbb{E} \Big[\sum_{i:t_{i,n}\leq T}\sum_{j=1}^{k-1}K_{k,j} \left(\Delta_{i,n} X \right)^2 \left( \Delta_{i+j,n} X \right)^2 \Big| \mathcal{S}\Big]
 \leq \sqrt{n} K T |\pi_n|_T.
\end{align*}
Hence those terms and the $O_{\mathbb{P}}(\sqrt{n}|\pi_n|_T)$-term vanish for $n \rightarrow \infty$ by Condition \ref{cond_clt_J}(i). This yields \eqref{proof_clt_J1b_step1}.

\textit{Step 2.} Next we prove
\begin{align}\label{proof_clt_J4}
\lim_{q \rightarrow \infty} \limsup_{n \rightarrow \infty} \mathbb{P}(|R(n)-R(n,q)|> \varepsilon) \rightarrow 0
\end{align}
where
\begin{align*}
&R(n,q)=\sqrt{n} \sum_{i:t_{i,n}\leq T}4 \left(\Delta_{i,n} N(q) \right)^3\sum_{j=1}^{k-1}(k-j) \big(
 \Delta_{i+j,n} C + \Delta_{i-j,n} C \big).
\end{align*}

Therefore we first consider
\begin{align}\label{proof_clt_J5}
\lim_{q \rightarrow \infty} \limsup_{n \rightarrow \infty} \mathbb{P}(|R(n)-R'(n,q))|> \varepsilon) \rightarrow 0
\end{align}
for all $\varepsilon>0$ with
\begin{align*}
&R'(n,q)=\sqrt{n} \sum_{i:t_{i,n}\leq T}4 \left(\Delta_{i,n} N(q) \right)^3\sum_{j=1}^{k-1}(k-j) \big(
 \Delta_{i+j,n} X + \Delta_{i-j,n} X \big)
\end{align*}
Using 
\begin{align*}
|a^3-b^3|=|a-b||a^2+ab+b^2| \leq \frac{3}{2}|a-b|(a^2+b^2)
\end{align*}
we derive
\begin{align}
&|R(n)-R'(n,q)| \nonumber 
\\ &~~~~ \leq \sqrt{n} K \sum_{i:t_{i,n}\leq T} \big|\Delta_{i,n} X-\Delta_{i,n} N(q) \big| \big( \left(\Delta_{i,n} X \right)^2+\left(\Delta_{i,n} N(q) \right)^2\big) \nonumber
 \\&~~~~~~~~~~~~~~~~~~~~~~~~~~~~~~~~~~~~~~\times\sum_{j=1}^{k-1}|
 \Delta_{i+j,n} X + \Delta_{i-j,n} X |. \label{proof_clt_J6}
\end{align}
The $\mathcal{S}$-conditional expectation of \eqref{proof_clt_J6} is after using iterated expectations, the H\"older inequality with $p_1=3$, $p_2=3/2$ on $ \Delta_{i,n}(B(q)+C)((\Delta_{i,n}X)^2+(\Delta_{i,n}N(q))^2)$, the Cauchy-Schwarz inequality on $ \Delta_{i,n}M(q)((\Delta_{i,n}X)^2+(\Delta_{i,n}N(q))^2)$ and Lemma \ref{elem_ineq} bounded by
\begin{align}\label{proof_clt_J7}
&\sqrt{n} K \sum_{i:t_{i,n}\leq T} \big[ (K_q |\mathcal{I}_{i,n}|^3+|\mathcal{I}_{i,n}|^{3/2})^{1/3}(|\mathcal{I}_{i,n}|+K|\mathcal{I}_{i,n}|+K_q |\mathcal{I}_{i,n}|^{3})^{2/3} \nonumber
\\&~~~~~~~~~~~~~~~ + (e_q|\mathcal{I}_{i,n}|)^{1/2}(|\mathcal{I}_{i,n}|+K|\mathcal{I}_{i,n}|+K_q |\mathcal{I}_{i,n}|^4)^{1/2}
\big] \sum_{j=-(k-1),j \neq 0}^{k-1} |\mathcal{I}_{i+j,n}|^{1/2} \nonumber
\\&~~ \leq K(K_q(|\pi_n|_T)^{1/6} +(e_q+K_q(|\pi_n|_T)^3 )^{1/2}) \sqrt{n}\sum_{i:t_{i,n}\leq T} |\mathcal{I}_{i,n}| \sum_{j=-(k-1),j \neq 0}^{k-1} |\mathcal{I}_{i+j,n}|^{1/2} \nonumber
\\&~~ \leq K(K_q(|\pi_n|_T)^{1/6} +(e_q)^{1/2}) \sqrt{n}\sum_{i:t_{i,n}\leq T} |\mathcal{I}_{i,k,n}||\mathcal{I}_{i,k,n}|^{1/2} \nonumber
\\&~~ \leq K(K_q(|\pi_n|_T)^{1/6} +(e_q)^{1/2}) \Big(nT\sum_{i:t_{i,n}\leq T} |\mathcal{I}_{i,k,n}|^{2} \Big)^{1/2} 
\end{align}
where we used $ |\mathcal{I}_{i,n}||\mathcal{I}_{i+j,n}|^{1/2}\leq |\mathcal{I}_{\max\{i,i+j\},k,n}|^{3/2}$ for $|j|\leq k-1$ and the Cauchy-Schwarz inequality for sums to obtain the last two inequalities. The last bound vanishes as $n \rightarrow \infty$, $q \rightarrow \infty$ by Condition \ref{cond_cons_J}. Hence by Lemma \ref{conv_cond_expec} we have proved \eqref{proof_clt_J5}.

Further we prove
\begin{align}\label{proof_clt_J8}
\lim_{q \rightarrow \infty} \limsup_{n \rightarrow \infty} \mathbb{P}(|R(n,q)-R'(n,q))|> \varepsilon) \rightarrow 0
\end{align}
for all $\varepsilon>0$. Denote by $\Omega(n,q)$ the set where two jumps of $N(q)$ are further apart than $k|\pi_n|_T$. By iterated expectations and Lemma \ref{elem_ineq} it holds
\begin{align*}
&\mathbb{E}[|R(n,q)-R'(n,q))|\mathds{1}_{\Omega(n,q)}|\mathcal{S}]
\\ &~~~~ \leq K \sqrt{n} \sum_{i:t_{i,n}\leq T} \mathbb{E} \big[
|\Delta_{i,n}N(q)|^3 \sum_{j=-(k-1),j \neq 0}^{k-1} |\Delta_{i+j,n}(B(q)+M(q))| \big| \mathcal{S}
\big] 
\\&~~~~ \leq K \sqrt{n} \sum_{i:t_{i,n}\leq T} (|\mathcal{I}_{i,n}|+K_q |\mathcal{I}_{i,n}|^3) \sum_{j=-(k-1),j \neq 0}^{k-1} (K_q|\mathcal{I}_{i,n}|^2+e_q |\mathcal{I}_{i,n}|)^{1/2}
\\&~~~~ \leq K (1+K_q(|\pi_n|_T)^2)(K_q|\pi_n|_T+e_q)^{1/2}\Big(nT\sum_{i:t_{i,n}\leq T} |\mathcal{I}_{i,k,n}|^{2} \Big)^{1/2}
\end{align*}
where the last inequality follows as in \eqref{proof_clt_J7}. The last bound vanishes as $n \rightarrow \infty$, $q \rightarrow \infty$ by Condition \ref{cond_cons_J}. Hence by Lemma \ref{conv_cond_expec} we have proved \eqref{proof_clt_J8} because of $\mathbb{P}(\Omega(n,q))\rightarrow 1$ as $n \rightarrow \infty$ for any $q>0$ which together with \eqref{proof_clt_J5} yields \eqref{proof_clt_J4}.

\textit{Step 3.} Finally we consider
\begin{align}\label{proof_clt_J1b_step3}
 \lim_{q \rightarrow \infty} \limsup_{r \rightarrow \infty}\limsup_{n \rightarrow \infty} \mathbb{P}(|R(n,q)-R(n,q,r)|> \varepsilon) \rightarrow 0.
\end{align}
Using iterated expectations, inequality \eqref{elem_ineq_N} and (2.1.34) from \cite{JacPro12} we obtain
\begin{small}
\begin{align*}
&\mathbb{E}[|R(n,q)-R(n,q,r)||\mathcal{S}]
\\&~ \leq (K+K_q(|\pi_n|_T)^2)\mathbb{E}\big[\sqrt{n} \sum_{i:t_{i,n}\leq T}|\mathcal{I}_{i,n}|\sum_{j=-(k-1),j\neq 0}^{k-1} (\int_{t_{i+j-1,n}}^{t_{i+j,n}} \|\sigma_s-\tilde{\sigma}(r,q)_s\|^2ds)^{1/2}\big|\mathcal{S}\big]
\\&~ \leq (K+K_q(|\pi_n|_T)^2)
\\&~~~~~~~ \times \mathbb{E}\big[\sqrt{n} \sum_{i:t_{i,n}\leq T}|\mathcal{I}_{i,n}|(2k-2)^{1/2}\big(\sum_{j=-(k-1),j\neq 0}^{k-1} \int_{t_{i+j-1,n}}^{t_{i+j,n}} \|\sigma_s-\tilde{\sigma}(r,q)_s\|^2ds\big)^{1/2}\big|\mathcal{S}\big].
\end{align*}
\end{small}This quantity is using the Cauchy-Schwarz inequality for sums further bounded by
\begin{align*}
&(K+K_q(|\pi_n|_T)^2) \sqrt{G_n(T)} \mathbb{E}\big[\big( (2k-2)^{2}\sum_{i:t_{i,n}\leq T}\int_{t_{i-1,n}}^{t_{i,n}} \|\sigma_s-\tilde{\sigma}(r,q)_s\|^2ds \big)^{1/2}\big| \mathcal{S} \big]
\\&~ = (K+K_q(|\pi_n|_T)^2) \sqrt{G_n(T)} \mathbb{E}\big[\big( \int_{0}^{T} \|\sigma_s-\tilde{\sigma}(r,q)_s\|^2 ds \big)^{1/2}\big| \mathcal{S} \big]
\\&~~~~~~~~~~~~~~~~~~~ \overset{\mathbb{P}}{\longrightarrow} K\sqrt{G(T)} \mathbb{E}\big[\big( \int_{0}^{T} \|\sigma_s-\tilde{\sigma}(r,q)_s\|^2 ds \big)^{1/2}\big| \mathcal{S} \big]
\end{align*} 
where the convergence as $n \rightarrow \infty$ follows from Condition \ref{cond_cons_J}. Here the limit vanishes as $r \rightarrow \infty$ for any $q>0$ since the expectation vanishes as $r \rightarrow \infty$ by dominated convergence because $\sigma,\tilde{\sigma}(r,q)$ are bounded by assumption and because $(\tilde{\sigma}(r,q))_{r \in \mathbb{N}}$ is a sequence of right continuous elementary processes approximating $\sigma$. Hence we get \eqref{proof_clt_J1b_step3} by Lemma \ref{conv_cond_expec}. 
 
Combining \eqref{proof_clt_J1b_step1}, \eqref{proof_clt_J4} and \eqref{proof_clt_J1b_step3} we obtain \eqref{proof_clt_J1b}.
\end{proof}

\begin{proof}[Proof of \eqref{test_power_klein_o_J}]
Note that it holds
\begin{align}\label{test_theo_cons_proof1}
\mathpzc{c}_{k,T,n}^{(J)}=\frac{\sqrt{n}\widehat{Q}_{k,T,n}^{(J)}(1-\alpha)}{n kV(g,\pi_n)_T}
\end{align}
where the denominator in \eqref{test_theo_cons_proof1} converges to $k C_T^{(J)}>0$ on $\Omega_T^{(C)}$ as shown in the proof of Theorem \ref{theo_cons_J}. Hence it remains to show that the numerator $\sqrt{n}\widehat{Q}_{n,T}^{(J)}(1-\alpha)$ is $o_{\widetilde{\mathbb{P}}}(1)$ on $\Omega_T^{(C)}$. We have
\begin{align}\label{test_theo_cons_proof2}
\sqrt{n}\widehat{Q}_{n,T}^{(J)}(1-\alpha) \leq \frac{\sqrt{n}}{\lfloor \alpha M_n \rfloor} \sum_{m=1}^{M_n} \big|\widehat{F}_{k,T,n,m}^{(J)}\big|.
\end{align}
Further we get using 
\begin{align*}
\mathbb{P}(|\Delta_{i,n}(B+C)|^p> \beta^p |\mathcal{I}_{i,n}|^{p\varpi}|\mathcal{S})
\leq \frac{K |\mathcal{I}_{i,n}|^p+K_p|\mathcal{I}_{i,n}|^{p/2}}{\beta^p |\mathcal{I}_{i,n}|^{p\varpi}}
\leq K_{p} |\mathcal{I}_{i,n}|^{p(1/2-\varpi)}
\end{align*}
for $p \geq 1$ where the process $(B_t)_{t \geq 0}$ is defined as in \eqref{definition_B}, the inequality $\sqrt{a+b} \leq \sqrt{a}+\sqrt{b}$ for $a,b \geq 0$, iterated expectations and twice the Cauchy-Schwarz inequality
\begin{align*}
&\mathbb{E}\big[\sqrt{n}|\widehat{F}_{k,T,n,m}^{(J)}|\mathds{1}_{\Omega^{(C)}_T}\big| \mathcal{S} \big] 
\\&\leq \sqrt{n}\sum_{i:t_{i,n} \leq T} \mathbb{E}[|\Delta_{i,n}(B+C)|^3
\mathds{1}_{\{|\Delta_{i,n}(B+C)| > \beta |\mathcal{I}_{i,n}|^\varpi\}}
\\&~~~~~~~~~~~ \times \big(\hat{\sigma}_n(t_{i,n},-)\sqrt{\mathbb{E}[\hat{\xi}_{k,n,m,-}(t_{i,n})|\mathcal{F}]}+\hat{\sigma}_n(t_{i,n},+)\sqrt{\mathbb{E}[\hat{\xi}_{k,n,m,+}(t_{i,n})|\mathcal{F}]} \big)|\mathcal{S}] 
\\&\leq \sqrt{n}\sum_{i:t_{i,n} \leq T} \big(\mathbb{E}[(\Delta_{i,n}(B+C))^{12} |\mathcal{S}]
\mathbb{P}(|\Delta_{i,n}(B+C)|^p > \beta^p |\mathcal{I}_{i,n}|^{p\varpi}|\mathcal{S}) \big)^{1/4} 
\\&~~~~~~~~~~~ \times \big(\mathbb{E}[2(\hat{\sigma}_n(t_{i,n},-))^2+2(\hat{\sigma}_n(t_{i,n},+))^2 |\mathcal{S}] \big)^{1/2} K \sqrt{n|\pi_n|_T}
\\& \leq K_{p} n \sqrt{|\pi_n|_T}\sum_{i:t_{i,n} \leq T} (K_q|\mathcal{I}_{i,n}|^{12}+|\mathcal{I}_{i,n}|^6)^{1/4}|\mathcal{I}_{i,n}|^{p(1-2\varpi)/4} (4b_n/ b_n )^{1/2}
\\& \leq K_{p} \sqrt{n} (|\pi_n|_T)^{1/2+p(1-2\varpi)/4} \sqrt{n} \sum_{i:t_{i,n} \leq T} |\mathcal{I}_{i,n}|^{3/2}.
\end{align*}
If we pick $p$ such that $p(1-2 \varpi)/4\geq 1/2$ the final bound vanishes as $n \rightarrow \infty$ by Condition \ref{cond_clt_J}(i) and \ref{cond_cons_J}(ii); compare \eqref{proof_clt_J7}. Hence by Lemma \ref{conv_cond_expec} we obtain $\sqrt{n}\widehat{Q}_{n,T}^{(J)}(1-\alpha)=o_{\widetilde{\mathbb{P}}}(1)$.
\end{proof}

\begin{proof}[Proof of Lemma \ref{lemma_conv_sigma}]
First we prove the claim for a constant stopping time $s \in(0,T)$. We will only give the proof for $\hat{\sigma}_n(s,+)$ as the proof for $\hat{\sigma}_n(s,-)$ is identical. Without loss of generality we assume $s+b_n \leq T$.

Using the elementary inequality 
\[
|x^2-y^2| \leq \frac{1+\varepsilon}{\varepsilon} |x-y|^2+\varepsilon y^2
\]
which holds for all $x,y \in \mathbb{R}$ and $\varepsilon > 0$ (compare Step 2 in the proof of Theorem 9.3.2 of \cite{JacPro12}) we obtain
\begin{align*}
&\big|(\hat{\sigma}^{(l)}_n(s,+))^2-\frac{1}{b_n}\sum_{i:\mathcal{I}_{i,n} \subset [s,s+b_n]} (\Delta_{i,n} C)^2\big|
\\&~\leq \frac{1}{b_n}\sum_{i:\mathcal{I}_{i,n} \subset [s,s+b_n]} \big( K_\varepsilon (\Delta_{i,n}(B(q)+M(q)+N(q)))^2 + \varepsilon(\Delta_{i,n} C )^2 \big)
\\&~\leq\frac{1}{b_n}\sum_{i:\mathcal{I}_{i,n} \subset [s,s+b_n]} \big(3 K_\varepsilon((\Delta_{i,n}B(q) )^2+(\Delta_{i,n}M(q) )^2+(\Delta_{i,n}N(q))^2)+ \varepsilon(\Delta_{i,n} C )^2\big).
\end{align*}
Here we have $\sum_{i:\mathcal{I}_{i,n} \subset [s,s+b_n]} (\Delta_{i,n} N(q) )^2=0$ with probability approaching $1$ as $n \rightarrow \infty$ because the probability for the event that there exists a jump of $N(q)$ in $[s,s+b_n]$ vanishes as $b_n \rightarrow 0$. Further the $\mathcal{S}$-conditional expectation of the remaining terms is using \eqref{elem_ineq_B}--\eqref{elem_ineq_M}
bounded by
\begin{align*}
K\big(K_\varepsilon e_q +K_{\varepsilon,q}|\pi_n|_T+\varepsilon  \big)
\end{align*}
which vanishes as first $n \rightarrow \infty$, then $q \rightarrow \infty$ and finally $\varepsilon \rightarrow 0$. 

Hence it remains to show
\begin{align}\label{cons_spot_vola_endo_proof1}
\frac{1}{b_n}\sum_{i:\mathcal{I}_{i,n} \subset [s,s+b_n]} (\Delta_{i,n} C)^2 \overset{\mathbb{P}}{\longrightarrow} \sigma_s^2.
\end{align}
Therefore we apply Lemma 2.2.12 from \cite{JacPro12} with
\begin{align}\label{cons_spot_vola_endo_proof2b}
\zeta_i^n=\frac{1}{b_n}(\Delta_{i_n(s)+i,n} C)^2-\frac{1}{b_n} (\sigma_s)^2|\mathcal{I}_{i_n(s)+i,n}|, \quad \mathcal{G}_{i}^n=\sigma(\mathcal{F}_{t_{i_n(s)+i,n}},\mathcal{S}).
\end{align}
We obtain using the conditional It\^o isometry and the fact that $\sigma$ is bounded
\begin{align}\label{cons_spot_vola_endo_proof3}
&\mathbb{E} \big[\sum_{i:\mathcal{I}_{i_{n}(s)+i,n} \subset [s,s+b_n]} |\mathbb{E}[\zeta_i^n|\mathcal{G}_{i-1}^n]| \big| \mathcal{S}\big]
\\&~~=\frac{1}{b_n}\sum_{i:\mathcal{I}_{i,n} \subset [s,s+b_n]} \mathbb{E} \big[\mathbb{E}[\int_{t_{i-1,n}}^{t_{i,n}} |(\sigma_u)^2-(\sigma_s)^2|du|\mathcal{F}_{t_{i-1,n}},\mathcal{S}]\big| \mathcal{S}\big] +O_\mathbb{P}(|\pi_n|_T/b_n) \nonumber
\\&~~\leq \frac{1}{b_n}\sum_{i:\mathcal{I}_{i,n} \subset [s,s+b_n]} \mathbb{E} \big[\mathbb{E}[\varepsilon |\mathcal{I}_{i,n}|+K|\mathcal{I}_{i,n}| \mathds{1}_{\{\sup_{u \in \mathcal{I}_{i,n}}|(\sigma_u)^2-(\sigma_s)^2|>\varepsilon\}}|\mathcal{F}_{t_{i-1,n}},\mathcal{S}] \big| \mathcal{S}\big] \nonumber
\\&~~~~~~~+O_\mathbb{P}(|\pi_n|_T/b_n) \nonumber
\\ &~~ \leq (\varepsilon+K\mathbb{E}[\mathds{1}_{\{\sup_{u \in [s,s+b_n]}|(\sigma_u)^2-(\sigma_s)^2|>\varepsilon\}}| \mathcal{S}])+O_\mathbb{P}(|\pi_n|_T/b_n) \nonumber
\\ &~~ \leq \varepsilon +K\mathbb{P}(\sup_{u \in [s,s+b_n]}|(\sigma_u)^2-(\sigma_s)^2|>\varepsilon|\mathcal{S})+O_\mathbb{P}(|\pi_n|_T/b_n)
\end{align}
where the first term vanishes as $\varepsilon \rightarrow 0$, the second term vanishes as $n \rightarrow \infty$ for all $\varepsilon>0$ because $\sigma$ is right-continuous and the third term vanishes as $n \rightarrow \infty$ by Condition \ref{cond_testproc_J}. Further we get 
\begin{align}\label{cons_spot_vola_endo_proof4}
\sum_{i:\mathcal{I}_{i_{n}(s)+i,n} \subset [s,s+b_n]} \mathbb{E}[(\zeta_i^n)^2|\mathcal{G}_{i-1}^n] 
 \leq \frac{4|\pi_n|_T b_n}{b_n^2}K \overset{\mathbb{P}}{\longrightarrow} 0
\end{align}
from $(a-b)^2\leq 2a^2+2b^2$, inequality \eqref{elem_ineq_C}, the boundedness of $\sigma$ and $|\pi_n|_T/b_n \overset{\mathbb{P}}{\longrightarrow} 0$. Hence we obtain $\sum_i \mathbb{E}[\zeta_i^n|\mathcal{G}_{i-1}^n] \overset{\mathbb{P}}{\longrightarrow} 0$ and $\sum_i \mathbb{E}[(\zeta_i^n)^2|\mathcal{G}_{i-1}^n] \overset{\mathbb{P}}{\longrightarrow} 0$ from \eqref{cons_spot_vola_endo_proof3} and \eqref{cons_spot_vola_endo_proof4} which yields
\begin{align*}
\sum_{i:\mathcal{I}_{i_{n}(s)+i,n} \subset [s,s+b_n]} \zeta_i^n
\overset{\mathbb{P}}{\longrightarrow} 0
\end{align*}
by Lemma 2.2.12 from \cite{JacPro12} which implies \eqref{cons_spot_vola_endo_proof1}.

The proof for the convergence of $\hat{\sigma}_n(S_p,+)$ for a stopping time $S_p$ is identical to the proof for $\hat{\sigma}_n(s,+)$ above. For $\hat{\sigma}_n(S_p,-)$ however the situation is different. There we use the sum over all $i \in \mathbb{N}$ with $t_{i,n} \in [S_p-b_n,S_p)$ which can be written as $$t_{i_n(S_p-b_n)+i-1,n}, \quad i=1,\ldots , i_n(S_p)-i_n(S_p-b_n)+1.$$ Here, in general $t_{i_n(S_p-b_n)+i-1,n}$ is no stopping time and we cannot use conditional expectations as above in that case. This problem can be circumvented by observing that jump times $S_p$ are asymptotically independent of $(C_u)_{u \in [S_p-\delta,S_p+\delta]}$ for $\delta \rightarrow 0$. To formalize this idea we define new filtrations $(\mathcal{F}^\varepsilon_t)_{t \geq 0}$, $\varepsilon>0$, such that jump times $S_p$ with $\mu(S_p,(-\infty,\varepsilon]\cup [\varepsilon,\infty))=1$ are $\mathcal{F}_0^\varepsilon$-measurable. Then $X$ is also an It\^o semimartingale with respect to the filtration $(\mathcal{F}^\varepsilon_t)_{t \geq 0}$ and we have $C^\varepsilon=C$ for the continuous martingale part $C^\varepsilon$ of $X$ under the filtration $(\mathcal{F}^\varepsilon_t)_{t \geq 0}$. The convergence of $\hat{\sigma}_n(S_p,-)$ for any jump time $S_p$ with $\mu(S_p,(-\infty,\varepsilon]\cup [\varepsilon,\infty))=1$ then follows because $S_p$ is $\mathcal{F}_0^\varepsilon$-measurable and therefore practically deterministic with respect to the new filtration. The statement for general jump times $S_p$ is obtained by letting $\varepsilon \rightarrow 0$. Consult Step 6 in the proof of Theorem 9.3.2 from \cite{JacPro12} for more details.
\end{proof}

\begin{proof}[Proof of Proposition \ref{lemma_test_J}]
Denote by $S_p$, $p=1,\ldots,P$, the jump times of the $P$ largest jumps of $X$ in $[0,T]$ and introduce the notation
\begin{align*}
&R_k(P,n,m)=\sum_{p=1}^P (\Delta_{i_n(S_p),n}X)^3 \mathds{1}_{\{|\Delta_{i_n(S_p),n}X|> \beta |\mathcal{I}_{i_n(S_p),n}|^\varpi\}} 
\\&~~~\times\big((\hat{\sigma}_n(t_{i_n(S_p),n},-))^2\hat{\xi}_{k,n,m,-}(t_{i_n(S_p),n})
~+(\hat{\sigma}_n(t_{i_n(S_p),n},+))^2\hat{\xi}_{k,n,m,+}(t_{i_n(S_p),n})\big)^{1/2}
\\&~~~\times U_{n,i_n(S_p),m},
\\ &R_k(P)=\sum_{p=1}^P (\Delta X_{S_p})^3\big((\sigma_{S_p-})^2\xi_{k,-}(S_p)+(\sigma_{S_p})^2 \xi_{k,+}(S_p)\big)^{1/2} U_{S_p}.
\end{align*}

\textit{Step 1.} From a result similar to Lemma A.6 in \cite{MarVet17}, which follows from Condition \ref{cond_testproc_J}(i), and Lemma \ref{lemma_conv_sigma} we derive
\begin{align}\label{test_theo_J_proof_A2}
\widetilde{\mathbb{P}} \big( \big\{\big| \frac{1}{M_n} \sum_{m=1}^{M_n} \mathds{1}_{\{R_k(P,n,m) \leq \Upsilon\}} -\widetilde{\mathbb{P}}(R_k(P)\leq \Upsilon|\mathcal{X})\big|>\varepsilon \big\} \cap \Omega_T^{(J)}\big) \rightarrow 0
\end{align}
as $n \rightarrow \infty$ for any fixed $P$. 

\textit{Step 2.} Next we show
\begin{align}\label{test_theo_J_proof_A5}
\lim_{P \rightarrow \infty} \limsup_{n \rightarrow \infty}  \widetilde{\mathbb{P}}\big(\big\{\big|\frac{1}{M_n} \sum_{m=1}^{M_n}(\mathds{1}_{\{R_k(P,n,m)\leq \Upsilon\}}-\mathds{1}_{\{\widehat{F}_{k,T,n,m}^{(J)}\leq \Upsilon\}} )\big| >\varepsilon\big\}\cap \Omega_T^{(J)} \big)=0
\end{align}
for all $\varepsilon >0$ which follows from
\begin{align}\label{test_theo_J_proof_A4}
\lim_{P \rightarrow \infty} \limsup_{n \rightarrow \infty} \frac{1}{M_n} \sum_{m=1}^{M_n} \widetilde{\mathbb{P}}(|R_k(P,n,m)-\widehat{F}_{k,T,n,m}^{(J)}|>\varepsilon)=0
\end{align}
for all $\varepsilon>0$; compare Step 4 in the proof of Proposition A.7 from \cite{MarVet17}. 

On the set $\Omega(q,P,n)$ on which the jumps of $N(q)$ are among the $P$ largest jumps and two different jumps of $N(q)$ are further apart than $|\pi_n|_T$ it holds
\begin{small}\begin{align}\label{test_theo_J_proof1}
&|R_k(P,n,m)-\widehat{F}_{k,T,n,m}^{(J)}|\mathds{1}_{\Omega(q,P,n)} \nonumber
\\&~~\leq \sum_{i:t_{i,n} \leq T} |\Delta_{i,n} (X-N(q)) |^3 \mathds{1}_{\{|\Delta_{i,n} X |> \beta |\mathcal{I}_{i,n}|^{\varpi} \}} 
\Big(\frac{2}{b_n} \sum_{j \neq i:\mathcal{I}_{j,n} \subset [t_{i,n}-b_n,t_{i,n}+b_n]} \left(\Delta_{j,n} X \right)^2  \Big)^{1/2} \nonumber
\\&~~~~~~~~~~~~~~~~~~\times\big( \hat{\xi}_{k,n,m,-}(t_{i,n})+\hat{\xi}_{k,n,m,+}(t_{i,n})\big)^{1/2} |U_{n,i,m} | \nonumber
\\&~~\leq K\sum_{i:t_{i,n} \leq T} \big(\left|\Delta_{i,n} (B(q)+ C) \right|^3+\left|\Delta_{i,n} M(q) \right|^3\big)\Big(\frac{1}{b_n} \sum_{j \neq i:\mathcal{I}_{j,n} \subset [t_{i,n}-b_n,t_{i,n}+b_n]} \left(\Delta_{j,n} X \right)^2 \Big)^{1/2} \nonumber
\\&~~~~~~~~~~~~~~~~~~\times
\big( \hat{\xi}_{k,n,m,-}(t_{i,n})+\hat{\xi}_{k,n,m,+}(t_{i,n})\big)^{1/2} | U_{n,i,m}|.
\end{align}
\end{small}For the continuous parts in \eqref{test_theo_J_proof1} we get using the Cauchy-Schwarz inequality, Lemma \ref{elem_ineq} and \mbox{$\hat{\xi}_{k,n,m,-}(t_{i,n}),\hat{\xi}_{k,n,m,+}(t_{i,n}) \leq n K |\pi_n|_T$} 
\begin{align}\label{test_theo_J_proof3}
&\mathbb{E} \Big[\sum_{i:t_{i,n} \leq T} |\Delta_{i,n} (B(q)+ C) |^3
\Big(\frac{1}{b_n} \sum_{j \neq i:\mathcal{I}_{j,n} \subset [t_{i,n}-b_n,t_{i,n}+b_n]} (\Delta_{j,n} X )^2 \Big)^{1/2} \nonumber
\\&~~~~~~~~~~~~~~~~~~~~~~~~ \times \big( \hat{\xi}_{k,n,m,-}(t_{i,n})+\hat{\xi}_{k,n,m,+}(t_{i,n})\big)^{1/2} |U_{n,i,m}| \Big| \mathcal{S} \Big] \nonumber
\\ &~~ \leq K \sum_{i:t_{i,n} \leq T} \left(\left( K_q |\mathcal{I}_{i,n}|^6 + |\mathcal{I}_{i,n}|^3 \right) \frac{2b_n}{ b_n} n|\pi_n|_T \right)^{1/2} \nonumber
\\&~~= K \sqrt{n} |\pi_n|_T \left(K_q(|\pi_n|_T)^3+1\right)^{1/2} T 
\end{align}
which vanishes as $n \rightarrow \infty$ by Condition \ref{cond_clt_J}(i).

Denote $$\hat{\xi}_{k,n,m}(i)=\hat{\xi}_{k,n,m,-}(t_{i,n})+\hat{\xi}_{k,n,m,+}(t_{i,n}).$$ We get for the remaining terms containing $|\Delta_{i,n} M(q)|^3$ in \eqref{test_theo_J_proof1} using the Jensen inequality and inequalities \eqref{elem_ineq_M}, \eqref{elem_ineq_X}
\begin{small}\begin{align}\label{test_theo_J_proof2}
&\mathbb{E} \Big[\sum_{i:t_{i,n} \leq T} \left|\Delta_{i,n} M(q) \right|^3
\Big(\frac{1}{b_n} \sum_{j \neq i:\mathcal{I}_{j,n} \subset [t_{i,n}-b_n,t_{i,n}+b_n]} (\Delta_{j,n} X )^2 \Big)^{1/2} \nonumber
\\&~~~~~~~~~~~~~~~~~~~~~~~~ \times \big( \hat{\xi}_{k,n,m,-}(t_{i,n})+\hat{\xi}_{k,n,m,+}(t_{i,n})\big)^{1/2} |U_{n,i,m}| \Big| \mathcal{S} \Big] \nonumber
\\ &~\leq K\sum_{i:t_{i,n} \leq T} 
 \Big(
\mathbb{E}\Big[\mathbb{E}[|\Delta_{i,n} M(q) |^3|\mathcal{S},\mathcal{F}_{t_{i-1,n}} ]\Big(\frac{1}{b_n} \sum_{j :\mathcal{I}_{j,n} \subset [t_{i,n}-b_n,t_{i,n})} (\Delta_{j,n} X )^2\Big)^{1/2}\Big|\mathcal{S} \Big] \nonumber
\\&~~+ \mathbb{E}\Big[|\Delta_{i,n} M(q) |^3 \mathbb{E}\Big[\Big(\frac{1}{b_n} \sum_{j :\mathcal{I}_{j,n} \subset [t_{i,n},t_{i,n}+b_n]} \left(\Delta_{j,n} X \right)^2\Big)^{1/2}\Big|\mathcal{S},\mathcal{F}_{t_{i,n}} \Big]\Big|\mathcal{S}\Big] \Big)
 \mathbb{E} \big[2\sqrt{\hat{\xi}_{k,n,m}(i)} \big| \mathcal{S} \big] \nonumber
\\ &~ \leq K e_q \sum_{i:t_{i,n} \leq T} 
\Big(|\mathcal{I}_{i,n}|\Big(\mathbb{E}\Big[\frac{1}{b_n} \sum_{j :\mathcal{I}_{j,n} \subset [t_{i,n}-b_n,t_{i,n})} (\Delta_{j,n} X )^2\Big|\mathcal{S} \Big]\Big)^{1/2} \nonumber
\\&~~+ \mathbb{E}\Big[|\Delta_{i,n} M(q) |^3 \Big(\mathbb{E}\Big[\frac{1}{b_n} \sum_{j :\mathcal{I}_{j,n} \subset [t_{i,n},t_{i,n}+b_n]} (\Delta_{j,n} X )^2\Big|\mathcal{S},\mathcal{F}_{t_{i,n}} \Big]\Big)^{1/2} \Big| \mathcal{S} \Big]\Big) \big(\mathbb{E} \big[\hat{\xi}_{k,n,m}(i) \big| \mathcal{S} \big]\big)^{1/2} \nonumber
 \\&~ \leq K e_q \sum_{i:t_{i,n} \leq T} |\mathcal{I}_{i,n}|
 \Big(K\frac{b_n}{b_n} \Big)^{1/2} \nonumber
\\&~~~~~\times \Big(
\sum_{l=-L_n}^{L_n}|\mathcal{I}_{i+l,n}|\Big( \sum_{\zeta = -L_n}^{L_n} |\mathcal{I}_{i+\zeta,n}| \Big)^{-1} \Big(n \sum_{j=1}^{k-1}(k-j)^2(|\mathcal{I}_{i+l+j,n}|+|\mathcal{I}_{i+l-j,n}|) \Big)^{1/2}
\Big) \nonumber
\\&~ \leq K e_q \sqrt{n} \sum_{i:t_{i,n} \leq T} |\mathcal{I}_{i,n}|
\sum_{l=-L_n}^{L_n}|\mathcal{I}_{i+l,n}| \Big( \sum_{\zeta = -L_n}^{L_n} |\mathcal{I}_{i+\zeta,n}| \Big)^{-1} \sum_{j=-(k-1)}^{k-1}|\mathcal{I}_{i+l+j,n}|^{1/2} \nonumber
\\&~ = K e_q \sqrt{n} \sum_{i:t_{i,n} \leq T} |\mathcal{I}_{i,n}| \sum_{j=-(k-1)}^{k-1}|\mathcal{I}_{i+j,n}|^{1/2} 
\sum_{l=-L_n}^{L_n} |\mathcal{I}_{i+l,n}| 
\Big( \sum_{\zeta = -L_n}^{L_n} |\mathcal{I}_{i+l+\zeta,n}| \Big)^{-1} \nonumber
\\&~~~~~+O_\mathbb{P}(\sqrt{n}|\pi_n|_T) 
\end{align}
\end{small}where we changed the index $i \rightarrow i+l$ to derive the last equality.
Then \eqref{test_theo_J_proof2} vanishes as in \eqref{proof_clt_J7} because of
\begin{multline}\label{index_shift_trick}
\sum_{l=-L_n}^{L_n} |\mathcal{I}_{i+l,n}| 
\big( \sum_{\zeta = -L_n}^{L_n} |\mathcal{I}_{i+l+\zeta,n}| \big)^{-1} 
 \\ ~\leq 
\sum_{l=-L_n}^{0}|\mathcal{I}_{i+l,n}| 
\big( \sum_{\zeta = -L_n}^{L_n} |\mathcal{I}_{i+l+\zeta,n}| \big)^{-1} +\sum_{l=0}^{L_n} |\mathcal{I}_{i+l,n}| 
\big( \sum_{\zeta = -L_n}^{L_n} |\mathcal{I}_{i+l+\zeta,n}| \big)^{-1} \leq 2
.
\end{multline}
Hence by Lemma \ref{conv_cond_expec} we get \eqref{test_theo_J_proof_A4} because of $\mathbb{P}(\Omega(q,P,n))\rightarrow \infty$ as $P,n \rightarrow \infty$ for any $q>0$.

\textit{Step 3.} Finally we show
\begin{align}\label{test_theo_J_proof_A3}
\widetilde{\mathbb{P}}(R_k(P) \leq \Upsilon|\mathcal{X}) \mathds{1}_{\Omega_T^{(J)}} \overset{\mathbb{P}}{\longrightarrow} \widetilde{\mathbb{P}}(F_{k,T}^{(J)} \leq \Upsilon|\mathcal{X}) \mathds{1}_{\Omega_T^{(J)}}
\end{align}
as $P \rightarrow \infty$. We have $R_k(P) \overset{\mathbb{P}}{\longrightarrow} F_{k,T}^{(J)}$ as $P \rightarrow \infty$ and it can be shown that this convergence yields \eqref{test_theo_J_proof_A3}; compare Step 3 in the proof of Proposition A.7 in \cite{MarVet17}.

Combining \eqref{test_theo_J_proof_A2}, \eqref{test_theo_J_proof_A5} and \eqref{test_theo_J_proof_A3} yields \eqref{test_theo_J_proof_A1}.
\end{proof}

\subsection{Proofs for Section \ref{Test_CoJ}}\label{sec:proof_detail_CoJ}

\begin{proof}[Proof of \eqref{proof_theo_clt_CoJ_step1}]
\textit{Step 1.} We first prove 
\begin{align}\label{theo_clt_CoJ_proof_step1_step1}
\lim_{n \rightarrow \infty} \mathbb{P}(|\sqrt{n}\left(V(f,[2],\pi_n)_T-4 V(f,\pi_n)_T\right)-R(n)|> \varepsilon)=0
\end{align}
for all $\varepsilon>0$ where $R(n)=R^{(1)}(n)+R^{(2)}(n)$ and
\begin{align*}
R^{(l)}(n)=&\sqrt{n}\sum_{i,j\geq 2:t_{i,n}^{(l)}\wedge t_{j,n}^{(3-l)}\leq T}
2 \Delta_{i-1,1,n}^{(l)}X^{(l)}\Delta_{i,1,n}^{(l)}X^{(l)} 
\\&~~~~~~~~~~~~~~~~\times\big(\Delta_{j,2,n}^{(3-l)}X^{(3-l)} \big)^2\mathds{1}_{\{\mathcal{I}_{i,2,n}^{(l)}\cap \mathcal{I}_{j,2,n}^{(3-l)}\neq \emptyset\}}, \quad l=1,2.
\end{align*}
Therefore note that it holds
\begin{align}\label{theo_clt_CoJ_proof1}
&\sqrt{n}\left(V(f,[2],\pi_n)_T-4 V(f,\pi_n)_T\right) \nonumber 
\\&~=\sqrt{n}\sum_{i,j\geq 2:t_{i,n}^{(1)}\wedge t_{j,n}^{(2)}\leq T} 
\Big[\big(\Delta_{i-1,1,n}^{(1)}X^{(1)}+\Delta_{i,1,n}^{(1)}X^{(1)}  \big)^2
\big(\Delta_{j-1,1,n}^{(2)}X^{(2)}+\Delta_{j,1,n}^{(2)}X^{(2)}  \big)^2 \nonumber 
\\&~~~~-\sum_{l_1,l_2=0,1}\big( \Delta_{i-l_1,1,n}^{(1)}X^{(1)}  \big)^2 \big( \Delta_{j-l_2,1,n}^{(2)}X^{(2)} \big)^2
 \mathds{1}_{\{\mathcal{I}_{i-l_1,1,n}^{(1)}\cap \mathcal{I}_{j-l_2,1,n}^{(2)}\neq \emptyset\}} \Big] \mathds{1}_{\{\mathcal{I}_{i,2,n}^{(1)}\cap \mathcal{I}_{j,2,n}^{(2)}\neq \emptyset\}} \nonumber 
 \\&~~~~+O_\mathbb{P}(\sqrt{n}|\pi_n|_T) \nonumber
 \\&~=\sqrt{n}\sum_{i,j\geq 2:t_{i,n}^{(1)}\wedge t_{j,n}^{(2)}\leq T}
\Big[ 2 \Delta_{i-1,1,n}^{(1)}X^{(1)}\Delta_{i,1,n}^{(1)}X^{(1)} 
\big(\Delta_{j-1,1,n}^{(2)}X^{(2)}+\Delta_{j,1,n}^{(2)}X^{(2)}  \big)^2 \nonumber 
\\&~~~~+2 \big(\Delta_{i-1,1,n}^{(1)}X^{(1)}+\Delta_{i,1,n}^{(1)}X^{(1)} \big)^2 
\Delta_{j-1,1,n}^{(2)}X^{(2)}\Delta_{j,1,n}^{(2)}X^{(2)} \nonumber 
\\&~~~~+\sum_{l_1,l_2=0,1}\big( \Delta_{i-l_1,1,n}^{(1)}X^{(1)}  \big)^2 \big( \Delta_{j-l_2,1,n}^{(2)}X^{(2)} \big)^2
 \mathds{1}_{\{\mathcal{I}_{i-l_1,1,n}^{(1)}\cap \mathcal{I}_{j-l_2,1,n}^{(2)}= \emptyset\}} \Big]\mathds{1}_{\{\mathcal{I}_{i,2,n}^{(1)}\cap \mathcal{I}_{j,2,n}^{(2)}\neq \emptyset\}} \nonumber
 \\&~~~~+O_\mathbb{P}(\sqrt{n}|\pi_n|_T),
\end{align}
where the $O_\mathbb{P}(\sqrt{n}|\pi_n|_T)$-terms are due to boundary effects. Because of the indicator $ \mathds{1}_{\{\mathcal{I}_{i-l_1,1,n}^{(1)}\cap \mathcal{I}_{j-l_2,1,n}^{(2)}= \emptyset\}}$ we may use iterated expectations and inequality \eqref{elem_ineq_X} to bound the $\mathcal{S}$-conditional expectation of the sum over the $( \Delta_{i-l_1,1,n}^{(1)}X^{(1)}  )^2 ( \Delta_{j-l_2,1,n}^{(2)}X^{(2)} )^2$ in \eqref{theo_clt_CoJ_proof1} by
\begin{align*}
4K \sqrt{n} \sum_{i,j\geq 2:t_{i,n}^{(1)}\wedge t_{j,n}^{(2)}\leq T} |\mathcal{I}_{i,2,n}^{(1)}||\mathcal{I}_{j,2,n}^{(2)}| \mathds{1}_{\{\mathcal{I}_{i,2,n}^{(1)}\cap \mathcal{I}_{j,2,n}^{(2)}\neq \emptyset\}}
=\frac{32KH_{2,n}(T)}{\sqrt{n}}
\end{align*}
which converges to zero in probability as $n \rightarrow \infty$ because of Condition \ref{cond_cons_CoJ2}(ii). Hence we obtain \eqref{theo_clt_CoJ_proof_step1_step1} by Lemma \ref{conv_cond_expec}.

\textit{Step 2.} Next we will show
\begin{align}\label{theo_clt_CoJ_proof2}
\lim_{q \rightarrow \infty} \limsup_{n \rightarrow \infty} \mathbb{P}(|R^{(l)}(n)-R^{(l)}(n,q)|> \varepsilon) \rightarrow 0
\end{align}
for all $\varepsilon >0$ and $l=1,2$ with
\begin{gather*}
R^{(l)}(n,q)=\sqrt{n}\sum_{i,j\geq 2:t_{i,n}^{(l)}\wedge t_{j,n}^{(3-l)}\leq T}
2 \Delta_{i-1,1,n}^{(l)}X^{(l)}\Delta_{i,1,n}^{(l)}X^{(l)} 
\big(\Delta_{j,2,n}^{(3-l)}N^{(3-l)}(q) \big)^2\mathds{1}_{\{\mathcal{I}_{i,2,n}^{(l)}\cap \mathcal{I}_{j,2,n}^{(3-l)}\neq \emptyset\}},
\end{gather*}
$l=1,2$. Denote by $\Delta_{(j,i),(k_2,k_1),n}^{(3-l,l)} X$ the increment of $X$ over the interval $\mathcal{I}_{i,k_1,n}^{(l)} \cap \mathcal{I}_{j,k_2,n}^{(3-l)}$ and by $\Delta_{(j,i),(k_1,k_2),n}^{(3-l\setminus l)} X$ the increment of $X$ over $\mathcal{I}_{j,k_1,n}^{(3-l)} \setminus \mathcal{I}_{i,k_2,n}^{(l)}$ (which might be the sum of the increments over two separate intervals). 

Using the elementary inequality
\begin{align*}
|a^2-b^2| = (\rho|a+b|)(\rho^{-1}|a-b|) \leq \rho^2(a+b)^2+\rho^{-2}(a-b)^2
\end{align*}
which holds for any $a,b \in \mathbb{R}$, $\rho>0$ for $a=\Delta_{j,2,n}^{(2)}X^{(2)}$ and $b=\Delta_{j,2,n}^{(2)}N^{(2)}(q)$ yields
\begin{small}
\begin{align}\label{theo_clt_CoJ_proof3}
&\mathbb{E}[|R^{(l)}(n)-R^{(l)}(n,q)|| \mathcal{S} ] \nonumber
\\ &~\leq 2\sqrt{n}\sum_{i,j\geq 2:t_{i,n}^{(l)}\wedge t_{j,n}^{(3-l)}\leq T}
\Big(\rho^2 \mathbb{E}\big[|\Delta_{i-1,1,n}^{(l)}X^{(l)}\Delta_{i,1,n}^{(l)}X^{(l)}| 
\big(\Delta_{j,2,n}^{(3-l)}(X^{(3-l)}+N^{(3-l)}(q)) \big)^2 \big| \mathcal{S} \big] \nonumber
\\&~~~+\rho^{-2}\mathbb{E}\big[|\Delta_{i-1,1,n}^{(l)}X^{(l)}\Delta_{i,1,n}^{(l)}X^{(l)}| 
\big(\Delta_{j,2,n}^{(3-l)}(X^{(3-l)}-N^{(3-l)}(q)) \big)^2 \big| \mathcal{S} \big] \Big)
\mathds{1}_{\{\mathcal{I}_{i,2,n}^{(l)}\cap \mathcal{I}_{j,2,n}^{(3-l)}\neq \emptyset\}}.
\end{align}
\end{small}
Using
\begin{align*}
& |\Delta_{\iota,1,n}^{(l)} X^{(l)}|\leq|\Delta_{(j,\iota),(2,1),n}^{(3-l,l)} X^{(l)}|+|\Delta_{(\iota,j),(1,2),n}^{(l\setminus (3-l))} X^{(l)}|, ~\iota=i-1,i,
\\&\big(\Delta_{j,2,n}^{(3-l)} Y \big)^2 \leq 3 \big(\Delta_{(j,i),(2,1),n}^{(3-l,l)} Y \big)^2 +3 \big(\Delta_{(j,i-1),(2,1),n}^{(3-l,l)} Y  \big)^2 +3 \big(\Delta_{(j,i),(2,2),n}^{(3-l\setminus l)} Y  \big)^2,
\end{align*}
for $Y=X^{(3-l)}+N^{(3-l)}(q)$ to treat increments over overlapping and non-overlapping intervals differently we obtain that the first summand in \eqref{theo_clt_CoJ_proof3} is bounded by
\begin{small}
\begin{align}\label{theo_clt_CoJ_proof4}
&6\rho^2\sqrt{n}\sum_{i,j\geq 2:t_{i,n}^{(l)}\wedge t_{j,n}^{(3-l)}\leq T} \bigg[
\mathbb{E}\big[|\Delta_{i-1,1,n}^{(l)}X^{(l)}\Delta_{i,1,n}^{(l)}X^{(l)}|
\big(\Delta_{(j,i),(2,2),n}^{(3-l\setminus l)} (X^{(3-l)}+N^{(3-l)}(q)) \big)^2 \big| \mathcal{S} \big]\nonumber
\\& ~~~~+\sum_{\iota=i-1,i} \Big(\mathbb{E}\big[|\Delta_{2i-1-\iota,1,n}^{(l)}X^{(l)}\Delta_{(j,\iota),(2,1),n}^{(3-l,l)} X^{(l)}|
\big(\Delta_{(j,\iota),(2,1),n}^{(3-l,l)} (X^{(3-l)}+N^{(3-l)}(q))  \big)^2 \big| \mathcal{S} \big]\nonumber
\\& ~~~~~~~~~~~~ +\mathbb{E}\big[|\Delta_{2i-1-\iota,1,n}^{(l)}X^{(l)}\Delta_{(\iota,j),(1,2),n}^{(l\setminus 3-l)} X^{(l)}|
\big(\Delta_{(j,\iota),(2,1),n}^{(3-l,l)} (X^{(3-l)}+N^{(3-l)}(q))  \big)^2 \big| \mathcal{S} \big]
 \Big) \bigg]\nonumber
 \\ &~~\times\mathds{1}_{\{\mathcal{I}_{i,2,n}^{(l)}\cap \mathcal{I}_{j,2,n}^{(3-l)}\neq \emptyset\}}\nonumber
\\& \leq K \rho^2 \sqrt{n} \sum_{i,j\geq 2:t_{i,n}^{(l)}\wedge t_{j,n}^{(3-l)}\leq T} \bigg[(|\mathcal{I}_{i-1,1,n}^{(l)}||\mathcal{I}_{i,1,n}^{(l)}|)^{1/2} (|\mathcal{I}_{j,2,n}^{(3-l)}|+K_q |\mathcal{I}_{j,2,n}^{(3-l)}|^2)\nonumber
\\&~~~~+\sum_{\iota=i-1,i} \Big( |\mathcal{I}_{2i-1-\iota,1,n}^{(l)}|^{1/2}  |\mathcal{I}_{\iota,1,n}^{(l)}\cap \mathcal{I}_{j,2,n}^{(3-l)}|^{1/2}(|\mathcal{I}_{\iota,1,n}^{(l)}\cap \mathcal{I}_{j,2,n}^{(3-l)}|+K_q |\mathcal{I}_{\iota,1,n}^{(l)}\cap \mathcal{I}_{j,2,n}^{(3-l)}|^4)^{1/2}\nonumber
\\&~~~~~~~~+ (|\mathcal{I}_{i-1,1,n}^{(l)}||\mathcal{I}_{i,1,n}^{(l)}|)^{1/2}
(|\mathcal{I}_{\iota,1,n}^{(l)}\cap \mathcal{I}_{j,2,n}^{(3-l)}|+K_q|\mathcal{I}_{\iota,1,n}^{(l)}\cap \mathcal{I}_{j,2,n}^{(3-l)}|^2)\nonumber
\big) \bigg]\mathds{1}_{\{\mathcal{I}_{i,2,n}^{(l)}\cap \mathcal{I}_{j,2,n}^{(3-l)}\neq \emptyset\}}\nonumber
\\& \leq K \rho^2 
(1+K_q|\pi_n|_T)H_{2,n}(T)/\sqrt{n} \nonumber
\\&~~+ K \rho^2 (1+K_q(|\pi_n|_T)^3)^{1/2} \sqrt{n} \sum_{i,j\geq 2:t_{i,n}^{(l)}\wedge t_{j,n}^{(3-l)}\leq T} |\mathcal{I}_{i,2,n}^{(l)}|^{1/2} |\mathcal{I}_{i,2,n}^{(l)} \cap \mathcal{I}_{j,2,n}^{(3-l)}|\mathds{1}_{\{\mathcal{I}_{i,2,n}^{(l)}\cap \mathcal{I}_{j,2,n}^{(3-l)}\neq \emptyset\}}
\end{align}
\end{small}where we used iterated expectations, the Cauchy-Schwarz inequality for the second summand and Lemma \ref{elem_ineq} repeatedly. The first term in the last bound vanishes as $n \rightarrow \infty$ because of Condition \ref{cond_cons_CoJ} and \ref{cond_cons_CoJ2}(ii) while the second term vanishes after letting $n \rightarrow \infty$, $q \rightarrow \infty$ and then $\rho \rightarrow 0$ by Condition \ref{cond_cons_CoJ2}(i) and (ii) because of
\begin{align}\label{bound_prob_square}
&\sqrt{n}\sum_{i,j\geq 2:t_{i,n}^{(l)}\wedge t_{j,n}^{(3-l)}\leq T} |\mathcal{I}_{i,2,n}^{(l)}|^{1/2} |\mathcal{I}_{i,2,n}^{(l)} \cap \mathcal{I}_{j,2,n}^{(3-l)}|\mathds{1}_{\{\mathcal{I}_{i,2,n}^{(l)}\cap \mathcal{I}_{j,2,n}^{(3-l)}\neq \emptyset\}} \nonumber \\
&~~~~~~\leq \sqrt{n} \sum_{i\geq 2:t_{i,n}^{(l)}\leq T} |\mathcal{I}_{i,2,n}^{(l)}|^{1/2}|\mathcal{I}_{i,2,n}^{(l)}| \nonumber
 \leq \big(n T \sum_{i\geq 2:t_{i,n}^{(l)}\leq T} |\mathcal{I}_{i,2,n}^{(l)}|^2 \big)^{1/2} \nonumber
\\&~~~~~~\leq (2^3T H_{2,n}(T))^{1/2}+O_\mathbb{P} (\sqrt{n}|\pi_n|_T)
\end{align}
where the inequality in the second to last line follows from the Cauchy-Schwarz inequality for sums.

Analogously we can bound the second summand in \eqref{theo_clt_CoJ_proof3} by
\begin{align*}
&K \rho^{-2}
(K_q|\pi_n|_T+1+e_q)H_{2,n}(T)/\sqrt{n} 
\\&~~+ K \rho^{-2} (K_q(|\pi_n|_T)^3+|\pi_n|_T+e_q)^{1/2} \big(n T \sum_{i\geq 2:t_{i,n}^{(l)}\leq T} |\mathcal{I}_{i,2,n}^{(l)}|^2 \big)^{1/2}+O_\mathbb{P} (\sqrt{n}|\pi_n|_T)
\end{align*}
where the first term vanishes as $n \rightarrow \infty$ and the second term vanishes as $n \rightarrow \infty$ and then $q \rightarrow \infty$ for any $\rho>0$. Hence using Lemma \ref{conv_cond_expec} we have proved \eqref{theo_clt_CoJ_proof2}.

\textit{Step 3.} Further we will show
\begin{align}\label{theo_clt_CoJ_proof5}
\lim_{q \rightarrow \infty} \limsup_{n \rightarrow \infty} \mathbb{P}(|R^{(l)}(n,q)-\widetilde{R}^{(l)}(n,q)|> \varepsilon) \rightarrow 0
\end{align}
for all $\varepsilon>0$ and $l=1,2$ with
\begin{multline*}
\widetilde{R}^{(l)}(n,q)=\sqrt{n}\sum_{i,j\geq 2:t_{i,n}^{(l)}\wedge t_{j,n}^{(3-l)}\leq T} 
2 \big( \Delta_{i-1,1,n}^{(l)}N^{(l)}(q) \Delta_{i,1,n}^{(l)}C^{(l)}
\\+\Delta_{i-1,1,n}^{(l)}C^{(l)} \Delta_{i,1,n}^{(l)}N^{(l)} (q)\big) 
 \big(\Delta_{j,2,n}^{(3-l)}N^{(3-l)}(q)  \big)^2\mathds{1}_{\{\mathcal{I}_{i,2,n}^{(l)}\cap \mathcal{I}_{j,2,n}^{(3-l)}\neq \emptyset\}},~l=1,2.
\end{multline*}
Using 
\begin{multline*}
\mathbb{E}[\Delta_{(j,\iota),(2,1),n}^{(3-l,l)}(X^{(l)}-N^{(l)}(q))
\big(\Delta_{(j,\iota),(2,1),n}^{(3-l,l)}N^{(3-l)}(q)  \big)^2 |\mathcal{S}]
\\ \leq (K_q|\mathcal{I}_{\iota,1,n}^{(l)}\cap \mathcal{I}_{j,2,n}^{(3-l)}|^{1/6} +(e_q)^{1/2})|\mathcal{I}_{\iota,1,n}^{(l)}\cap \mathcal{I}_{j,2,n}^{(3-l)}|
\end{multline*}
which can be derived as in \eqref{proof_clt_J7} we get analogously to \eqref{theo_clt_CoJ_proof4}
\begin{align}\label{theo_clt_CoJ_proof6}
&\mathbb{E}\big[\sqrt{n}\sum_{i,j\geq 2:t_{i,n}^{(l)}\wedge t_{j,n}^{(3-l)}\leq T} 
2 \Delta_{i-1,1,n}^{(l)}(X^{(l)}-N^{(l)}(q)) \Delta_{i,1,n}^{(l)}(X^{(l)}-N^{(l)}(q)) \nonumber
\\&~~~~~~~~~~~~~~~~~~~~~\times \big(\Delta_{j,2,n}^{(3-l)}N^{(3-l)}(q)  \big)^2
  \mathds{1}_{\{\mathcal{I}_{i,2,n}^{(l)}\cap \mathcal{I}_{j,2,n}^{(3-l)}\neq \emptyset\}}
\big| \mathcal{S} \big] \nonumber
\\ &~\leq  K\sqrt{n}\sum_{i,j\geq 2:t_{i,n}^{(l)}\wedge t_{j,n}^{(3-l)}\leq T} 
\Big[((K_q|\mathcal{I}_{i-1,1,n}^{(l)}|+1+e_q)|\mathcal{I}_{i-1,1,n}^{(l)}|)^{1/2} \nonumber
\\&~~~~~~~~~~~~~~~~~~~~~\times((K_q|\mathcal{I}_{i,1,n}^{(l)}|+1+e_q)|\mathcal{I}_{i,1,n}^{(l)}|)^{1/2}
(1+K_q |\mathcal{I}_{j,2,n}^{(3-l)}|)|\mathcal{I}_{j,2,n}^{(3-l)}|
  \nonumber
 \\&~~~~+ \sum_{\iota=i-1,i}  \Big( ((K_q|\mathcal{I}_{2i-1-\iota,1,n}^{(l)}|+1+e_q)|\mathcal{I}_{2i-1-\iota,1,n}^{(l)}|)^{1/2} \nonumber
 \\&~~~~~~~~~~~~~~~~~~~~~\times (K_q|\mathcal{I}_{\iota,1,n}^{(l)}\cap \mathcal{I}_{j,2,n}^{(3-l)}|^{1/6} +(e_q)^{1/2})|\mathcal{I}_{\iota,1,n}^{(l)}\cap \mathcal{I}_{j,2,n}^{(3-l)}| \nonumber
\\& ~~~~+ ((K_q|\mathcal{I}_{2i-1-\iota,1,n}^{(l)}|+1+e_q)|\mathcal{I}_{2i-1-\iota,1,n}^{(l)}|)^{1/2}((K_q|\mathcal{I}_{\iota,1,n}^{(l)}|+1+e_q)|\mathcal{I}_{\iota,1,n}^{(l)}|)^{1/2} \nonumber
\\&~~~~~~~~~~~~~~~~~~~~~\times(1+K_q |\mathcal{I}_{\iota,1,n}^{(l)}\cap \mathcal{I}_{j,2,n}^{(3-l)}| )|\mathcal{I}_{\iota,1,n}^{(l)}\cap \mathcal{I}_{j,2,n}^{(3-l)}| \Big)
 \Big]\mathds{1}_{\{\mathcal{I}_{i,2,n}^{(l)}\cap \mathcal{I}_{j,2,n}^{(3-l)}\neq \emptyset\}}\nonumber
 \\ &~\leq K (K_q |\pi_n|_T+1+e_q)(1+K_q|\pi_n|_T)H_{2,n}(T)/\sqrt{n}   \nonumber
 \\&~~~~+ (K_q |\pi_n|_T+1+e_q)^{1/2}(K_q(|\pi_n|_T)^{1/6} +(e_q)^{1/2}) \big( nT\sum_{i \geq 2 :t_{i,n}^{(l)} \leq T} | \mathcal{I}_{i,2,n}^{(l)}|^2 \big)^{1/2}
\end{align}
which vanishes as $n \rightarrow \infty$ and then $q \rightarrow \infty$. Furthermore we get also analogously to \eqref{theo_clt_CoJ_proof4} 
\begin{align}\label{theo_clt_CoJ_proof7}
&\mathbb{E}\big[\sqrt{n}\sum_{i,j\geq 2:t_{i,n}^{(l)}\wedge t_{j,n}^{(3-l)}\leq T} 
2 \Delta_{i-1,1,n}^{(l)}(B^{(l)}(q)+M^{(l)}(q)) \Delta_{i,1,n}^{(l)}(N^{(l)}(q)) \nonumber
\\&~~~~~~~~~~~~~~~~\times \big(\Delta_{j,2,n}^{(3-l)}N^{(3-l)}(q)  \big)^2
  \mathds{1}_{\{\mathcal{I}_{i,2,n}^{(l)}\cap \mathcal{I}_{j,2,n}^{(3-l)}\neq \emptyset\}}
\big| \mathcal{S} \big] \nonumber
\\ &\leq  K\sqrt{n}\sum_{i,j\geq 2:t_{i,n}^{(l)}\wedge t_{j,n}^{(3-l)}\leq T} 
\Big[((K_q|\mathcal{I}_{i-1,1,n}^{(l)}|+e_q)|\mathcal{I}_{i-1,1,n}^{(l)}|)^{1/2} ((1+K_q|\mathcal{I}_{i,1,n}^{(l)}|)|\mathcal{I}_{i,1,n}^{(l)}|)^{1/2}\nonumber
\\&~~~~~~~~~~~~~~~~~~~~~\times 
(1+K_q |\mathcal{I}_{j,2,n}^{(3-l)}|)|\mathcal{I}_{j,2,n}^{(3-l)}|
  \nonumber
\\&~~~~+((K_q|\mathcal{I}_{i-1,1,n}^{(l)}|+e_q)|\mathcal{I}_{i-1,1,n}^{(l)}|)^{1/2} ((1+K_q|\mathcal{I}_{i,1,n}^{(l)}\cap \mathcal{I}_{j,2,n}^{(3-l)}|)|\mathcal{I}_{i,1,n}^{(l)}\cap \mathcal{I}_{j,2,n}^{(3-l)}|)^{1/2}\nonumber
 \\&~~~~~~~~~~~~~~~~~~~~~\times ((1+K_q|\mathcal{I}_{i,1,n}^{(l)}\cap \mathcal{I}_{j,2,n}^{(3-l)}|^3)|\mathcal{I}_{i,1,n}^{(l)}\cap \mathcal{I}_{j,2,n}^{(3-l)}|)^{1/2} \nonumber
 \\&~~~~+((K_q|\mathcal{I}_{i-1,1,n}^{(l)}\cap \mathcal{I}_{j,2,n}^{(3-l)}|+e_q)|\mathcal{I}_{i-1,1,n}^{(l)}\cap \mathcal{I}_{j,2,n}^{(3-l)}|)^{1/2} ((1+K_q|\mathcal{I}_{i,1,n}^{(l)}|)|\mathcal{I}_{i,1,n}^{(l)}|)^{1/2}\nonumber
 \\&~~~~~~~~~~~~~~~~~~~~~\times ((1+K_q|\mathcal{I}_{i-1,1,n}^{(l)}\cap \mathcal{I}_{j,2,n}^{(3-l)}|^3)|\mathcal{I}_{i-1,1,n}^{(l)}\cap \mathcal{I}_{j,2,n}^{(3-l)}|)^{1/2} \nonumber
\\& ~~~~+ ((K_q|\mathcal{I}_{i-1,1,n}^{(l)}|+e_q)|\mathcal{I}_{i-1,1,n}^{(l)}|)^{1/2}((1+K_q|\mathcal{I}_{i,1,n}^{(l)}|)|\mathcal{I}_{i,1,n}^{(l)}|)^{1/2} \nonumber
\\&~~~~~~~~~~~~~~~~~~~~~\times(1+K_q |\mathcal{I}_{i,1,n}^{(l)}\cap \mathcal{I}_{j,2,n}^{(3-l)}| )|\mathcal{I}_{i,1,n}^{(l)}\cap \mathcal{I}_{j,2,n}^{(3-l)}| \nonumber
\\& ~~~~+ ((K_q|\mathcal{I}_{i-1,1,n}^{(l)}|+e_q)|\mathcal{I}_{i-1,1,n}^{(l)}|)^{1/2}((1+K_q|\mathcal{I}_{i,1,n}^{(l)}|)|\mathcal{I}_{i,1,n}^{(l)}|)^{1/2} \nonumber
\\&~~~~~~~~~~~~~~~~~~~~~\times(1+K_q |\mathcal{I}_{i-1,1,n}^{(l)}\cap \mathcal{I}_{j,2,n}^{(3-l)}| )|\mathcal{I}_{i-1,1,n}^{(l)}\cap \mathcal{I}_{j,2,n}^{(3-l)}| \nonumber
 \Big]\mathds{1}_{\{\mathcal{I}_{i,2,n}^{(l)}\cap \mathcal{I}_{j,2,n}^{(3-l)}\neq \emptyset\}}\nonumber
 \\ &~ \leq K (K_q |\pi_n|_T+e_q)^{1/2} (1+K_q|\pi_n|_T)^{3/2}H_{2,n}(T)/\sqrt{n}  \nonumber 
 \\&~~~~+(K_q |\pi_n|_T+e_q)^{1/2}(1+K_q|\pi_n|_T)\big( nT\sum_{i \geq 2 :t_{i,n}^{(l)} \leq T} | \mathcal{I}_{i,2,n}^{(l)}|^2 \big)^{1/2}
\end{align}
which vanishes as $n \rightarrow \infty$ and then $q \rightarrow \infty$. The same obviously also holds if we switch the roles of $i-1$ and $i$. Hence using Lemma \ref{conv_cond_expec} the estimates \eqref{theo_clt_CoJ_proof6} and \eqref{theo_clt_CoJ_proof7} yield \eqref{theo_clt_CoJ_proof5} because $\Delta_{i-1,1,n}^{(l)}N(q)\Delta_{i,1,n}^{(l)}N(q)=0$ eventually for all $i$ as there are only finitely many big jumps.

\textit{Step 4.} Finally it remains to show
\begin{align}\label{theo_clt_CoJ_proof_step1_step4}
\lim_{q \rightarrow \infty} \limsup_{r \rightarrow \infty} \limsup_{n \rightarrow \infty} \mathbb{P}(|\widetilde{R}(n,q)-\widetilde{R}(n,q,r)|>\varepsilon)=0
\end{align}
for all $\varepsilon>0$ with $\widetilde{R}(n,q)=\widetilde{R}^{(1)}(n,q)+\widetilde{R}^{(2)}(n,q)$. The proof for \eqref{theo_clt_CoJ_proof_step1_step4} is identical to the proof of \eqref{proof_clt_J1b_step3} because we have 
\begin{align*}
&|\Delta_{i-\iota,1,n}^{(l)}N^{(l)}(q)|(\Delta_{j,2,n}^{(3-l)}N^{(3-l)}(q) )^2 \mathds{1}_{\{\mathcal{I}_{i,2,n}^{(l)} \cap \mathcal{I}_{j,2,n}^{(3-l)}\neq \emptyset\}}
\\&~~= |\Delta_{i-\iota,1,n}^{(l)}N^{(l)}(q)|(\Delta_{i-\iota,1,n}^{(3-l)}N^{(3-l)}(q) )^2
\leq K \|\Delta_{i-\iota,1,n}^{(l)}N(q) \|^3
\end{align*}
for $\iota=0,1$ and $l=1,2$ on the set $\Omega(n,q)$ where two different jump times of $N(q)$ are further apart than $4|\pi_n|_T$ and because of $\mathbb{P}(\Omega(n,q))\rightarrow 1$ for any $q>0$.

Combining \eqref{theo_clt_CoJ_proof_step1_step1}, \eqref{theo_clt_CoJ_proof2}, \eqref{theo_clt_CoJ_proof5} and \eqref{theo_clt_CoJ_proof_step1_step4} then yields \eqref{proof_theo_clt_CoJ_step1}.
\end{proof}

\begin{proof}[Proof of \eqref{proof_theo_clt_CoJ_step2}]
On the set $\Omega(n,q,r)$ where any jump times $S_{q,p}$ in $[0,T]$ are further apart than $4|\pi_n|_T$ from each other and from all $j 2^{-r}$ with $1 \leq j \leq \lfloor T 2^r \rfloor $ we have
\begin{small}\begin{align*}
&\widetilde{R}(n,q,r)\mathds{1}_{\Omega(n,q,r)}
= \mathds{1}_{\Omega(n,q,r)}4\sqrt{n}\sum_{S_{q,p} \leq T} 
\Delta N^{(2)}(q)_{S_{q,p}} \Delta N^{(1)}(q)_{S_{q,p}} \nonumber
\\&~~\times \Big(\Delta N^{(2)}(q)_{S_{p}}\big(\tilde{\sigma}^{(1)}_{S_{q,p}-}(r,q)\Delta_{i_n^{(1)}(S_{q,p})-1,n}^{(1)}W^{(1)}+\tilde{\sigma}^{(1)}_{S_{q,p}}(r)\Delta_{i_n^{(1)}(S_{q,p})+1,n}^{(1)}W^{(1)}\big)
\\&~~~~+\Delta N^{(1)}(q)_{S_{q,p}}\Big[\tilde{\sigma}^{(2)}_{S_{q,p}-}(r,q)
\\&~~~~~~~~~~~\times \big(\tilde{\rho}_{S_{q,p}-}(r,q)\Delta_{i_n^{(2)}(S_{q,p})-1,n}^{(2)}W^{(1)}+\sqrt{1-\tilde{\rho}_{S_{q,p}-}(r,q)^2}\Delta_{i_n^{(2)}(S_{q,p})-1,n}^{(2)}W^{(2)}\big) 
\\&~~~~~~+\tilde{\sigma}^{(2)}_{S_{q,p}}(r,q)\big(\tilde{\rho}_{S_{q,p}}(r,q)\Delta_{i_n^{(2)}(S_{q,p})+1,n}^{(2)}W^{(1)}+\sqrt{1-(\tilde{\rho}_{S_{q,p}}(r,q))^2}\Delta_{i_n^{(2)}(S_{q,p})+1,n}^{(2)}W^{(2)}\big)
\Big]\Big) 
\end{align*}
\end{small}where the factor $4$ stems from the fact that any jump of $N^{(l)}(q)$ is observed in two consecutive intervals $\mathcal{I}_{i-1,2,n}^{(l)}$, $\mathcal{I}_{i,2,n}^{(l)}$, $l=1,2$. 

Similarly as in the proof of Proposition A.3 in \cite{MarVet17} it can be shown that Condition \ref{cond_clt_coJ} yields the $\mathcal{X}$-stable convergence
\begin{align}\label{theo_clt_CoJ_proof8b}
&\big(\big(\Delta_{i_n^{(1)}(S_{q,p})-1,n}^{(1)}W^{(1)},\Delta_{i_n^{(1)}(S_{q,p})+1,n}^{(1)}W^{(1)},\Delta_{i_n^{(2)}(S_{q,p})-1,n}^{(2)}W^{(1)},\Delta_{i_n^{(2)}(S_{q,p})+1,n}^{(2)}W^{(1)}, \nonumber
\\&~~~\Delta_{i_n^{(2)}(S_{q,p})-1,n}^{(2)}W^{(2)},\Delta_{i_n^{(2)}(S_{q,p})+1,n}^{(2)}W^{(2)} \big)_{S_{q,p} \leq T} \big)  \nonumber
\\ &~~~~~\overset{\mathcal{L}-s}{\longrightarrow} \Big(\Big(\sqrt{\mathcal{L}(S_{q,p})} U_{S_{q,p},-}^{(1)}+\sqrt{(\mathcal{L}^{(1)}-\mathcal{L})(S_{q,p})}U_{S_{q,p},-}^{(2)}, \nonumber
\\&~~~~~~~~~~~~~~~\sqrt{\mathcal{R}(S_{q,p})} U_{S_{q,p},+}^{(1)}+\sqrt{(\mathcal{R}^{(1)}-\mathcal{R})(S_{q,p})}U_{S_{q,p},+}^{(2)}, \nonumber
\\&~~~~~~~~~~~~~~~\sqrt{\mathcal{L}(S_{q,p})} U_{S_{q,p},-}^{(1)}+\sqrt{(\mathcal{L}^{(2)}-\mathcal{L})(S_{q,p})}U_{S_{q,p},-}^{(3)}, \nonumber
\\&~~~~~~~~~~~~~~~\sqrt{\mathcal{R}(S_{q,p})} U_{S_{q,p},+}^{(1)}+\sqrt{(\mathcal{R}^{(2)}-\mathcal{R})(S_{q,p})}U_{S_{q,p},+}^{(3)}, \nonumber
\\&~~~~~~~~~~~~~~~\sqrt{\mathcal{L}^{(2)}(S_{q,p})}U_{S_{q,p},-}^{(4)},\sqrt{\mathcal{R}^{(2)}(S_{q,p})}U_{S_{q,p},+}^{(4)}\Big)_{S_{q,p} \leq T} \Big)
\end{align}
where $U_{s,-}^{(i)},U_{s,+}^{(i)}$ for $i=1,\ldots,4$ are standard normally distributed random variables which are independent of $\mathcal{F}$ and of the random vectors $(\mathcal{L}^1,\mathcal{R}^1,\mathcal{L}^2,\mathcal{R}^2,\mathcal{L},\mathcal{R})(s)$. \eqref{theo_clt_CoJ_proof8b} together with Proposition 2.2 in \cite{PodVet10} and the continuous mapping theorem then yields
\begin{align*}
&4\sqrt{n}\sum_{S_{q,p} \in P(q,T)} 
\Delta N^{(2)}(q)_{S_{q,p}} \Delta N^{(1)}(q)_{S_{q,p}} \nonumber
\\&~\times \Big(\Delta N^{(2)}(q)_{S_{p}}\Big[\tilde{\sigma}^{(1)}_{S_{q,p}-}(r,q)\Delta_{i_n^{(1)}(S_{q,p})-1,n}^{(1)}W^{(1)}+\tilde{\sigma}^{(1)}_{S_{q,p}}(r,q)\Delta_{i_n^{(1)}(S_{q,p})+1,n}^{(1)}W^{(1)}\Big]
\\&~~~+\Delta N^{(1)}(q)_{S_{q,p}}
 \Big[\tilde{\sigma}^{(2)}_{S_{q,p}-}(r,q)\big(\tilde{\rho}_{S_{q,p}-}(r,q)\Delta_{i_n^{(2)}(S_{q,p})-1,n}^{(2)}W^{(1)}
\\&~~~~~~~~~~~~~~~~~~~~~~~~~~~~~~~~~~+\sqrt{1-(\tilde{\rho}_{S_{q,p}-}(r,q))^2}\Delta_{i_n^{(2)}(S_{q,p})-1,n}^{(2)}W^{(2)}\big) 
\\&~~~~~~+\tilde{\sigma}^{(2)}_{S_{q,p}}(r,q)\big(\tilde{\rho}_{S_{q,p}}\Delta_{i_n^{(2)}(S_{q,p})+1,n}^{(2)}W^{(1)}+\sqrt{1-\tilde{\rho}_{S_{q,p}}(r,q))^2}\Delta_{i_n^{(2)}(S_{q,p})+1,n}^{(2)}W^{(2)}\big)
\Big]\Big) 
\\& ~~\overset{\mathcal{L}-s}{\longrightarrow}F_{T}^{(CoJ)}(q,r).
\end{align*}
Because of $\mathbb{P}(\Omega(n,q,r))\rightarrow 1$ as $n \rightarrow \infty$ for any $q,r$ this yields \eqref{proof_theo_clt_CoJ_step2}.
\end{proof}

\begin{proof}[Proof of \eqref{test_power_klein_o_CoJ}]
From the proof of Theorem \ref{theo_clt_CoJ} we know that $n V(f,\pi_n)_T$ converges on $\Omega_T^{(nCoJ)}$ stably in law to a non-negative random variable. Comparing \eqref{test_theo_cons_proof2} in the proof of Theorem \ref{test_theo_J} it then suffices to show
\begin{align}\label{test_theo_cons_CoJ_proof2}
\sqrt{n} \widehat{F}_{T,n,m}^{(CoJ)}\mathds{1}_{\Omega_T^{(nCoJ)}}=o_\mathbb{P}(1)
\end{align}
uniformly in $m$ for proving \eqref{test_power_klein_o_CoJ}. In order to achieve this goal, observe that it holds
\begin{align}
&\mathbb{E}[\sqrt{n} |\widehat{F}_{T,n,m}^{(CoJ)}|\mathds{1}_{\Omega_T^{(nCoJ)}} |\mathcal{F} ] \nonumber
\\&~~\leq K n \sum_{l=1,2}\sum_{i,j:t_{i,n}^{(l)} \wedge t_{j,n}^{(3-l)} \leq T} |\Delta_{i,n}^{(l)} X^{(l)}|(\Delta_{j,n}^{(3-l)} X^{(3-l)})^2 |\hat{\sigma}^{(l)}(t_{i,n}^{(l)},-)+\hat{\sigma}^{(l)}(t_{i,n}^{(l)},+)| \nonumber
\\&~~~~~~~~~~~\times\sqrt{|\mathcal{I}_{i-1,n}^{(l)}|+|\mathcal{I}_{i+1,n}^{(l)}|} \mathds{1}_{\{|\Delta_{i,n}^{(l)} X^{(l)} |> \beta |\mathcal{I}_{i,n}^{(l)}|^{\varpi} \}}\mathds{1}_{\{|\Delta_{j,n}^{(3-l)} X^{(3-l)} |> \beta |\mathcal{I}_{j,n}^{(3-l)}|^{\varpi} \}} \nonumber
\\&~~~~~~~~~~~\times\mathds{1}_{\{\mathcal{I}_{i,n}^{(l)} \cap\mathcal{I}_{j,n}^{(3-l)}\neq \emptyset\}} \mathds{1}_{\Omega_T^{(nCoJ)}}\nonumber
\\& ~~ \leq Kn \frac{1}{\varepsilon}\sum_{l=1,2}\sum_{i,j:t_{i,n}^{(l)} \wedge t_{j,n}^{(3-l)} \leq T} (\Delta_{i,n}^{(l)} X^{(l)})^2(\Delta_{j,n}^{(3-l)} X^{(3-l)})^2\mathds{1}_{\{|\Delta_{i,n}^{(l)} X^{(l)} |> \beta |\mathcal{I}_{i,n}^{(l)}|^{\varpi} \}} \nonumber
\\&~~~~~~~~~~~\times
\mathds{1}_{\{|\Delta_{j,n}^{(3-l)} X^{(3-l)} |> \beta |\mathcal{I}_{j,n}^{(3-l)}|^{\varpi} \}} \mathds{1}_{\{\mathcal{I}_{i,n}^{(l)} \cap\mathcal{I}_{j,n}^{(3-l)}\neq \emptyset\}}\mathds{1}_{\Omega_T^{(nCoJ)}} \label{test_theo_cons_CoJ_proof3}
\\&~~~~+ Kn \varepsilon \sum_{l=1,2}\sum_{i,j:t_{i,n}^{(l)} \wedge t_{j,n}^{(3-l)} \leq T} (\Delta_{j,n}^{(3-l)} X^{(3-l)})^2 (\hat{\sigma}^{(l)}(t_{i,n}^{(l)},-)+\hat{\sigma}^{(l)}(t_{i,n}^{(l)},+))^2\nonumber
\\&~~~~~~~~~~~\times(|\mathcal{I}_{i-1,n}^{(l)}|+|\mathcal{I}_{i+1,n}^{(l)}|) \mathds{1}_{\{\mathcal{I}_{i,n}^{(l)} \cap\mathcal{I}_{j,n}^{(3-l)}\neq \emptyset\}}\mathds{1}_{\Omega_T^{(nCoJ)}}
\label{test_theo_cons_CoJ_proof4}
\end{align}
where we used the elementary inequality $|ab| \leq a^2/\varepsilon+ \varepsilon b^2$ which holds for any $a,b \in \mathbb{R}$, $\varepsilon>0$.

First we consider \eqref{test_theo_cons_CoJ_proof3}. Note that by Proposition A.5 from \cite{MarVet17} it suffices on the subset where no common jumps are present to consider \eqref{test_theo_cons_CoJ_proof3} where $(\Delta_{i,n}^{(l)} X^{(l)})^2(\Delta_{j,n}^{(3-l)} X^{(3-l)})^2$ is replaced with $$((\Delta_{i,n}^{(l)} C^{(l)})^2+(\Delta_{i,n}^{(l)} N(q)^{(l)})^2)((\Delta_{j,n}^{(3-l)} C^{(3-l)})^2+(\Delta_{j,n}^{(3-l)} N(q)^{(3-l)})^2).$$
Because we are on $\Omega_T^{(nCoJ)}$ the sum with $(\Delta_{i,n}^{(l)} N(q)^{(l)})^2(\Delta_{j,n}^{(3-l)} N(q)^{(3-l)})^2$ vanishes as $n \rightarrow \infty$. Further we obtain using the Cauchy-Schwarz inequality, inequalities (A.31) from \cite{MarVet17} and \eqref{elem_ineq_C}
\begin{align*}
&\mathbb{E}[n \sum_{i,j:t_{i,n}^{(1)} \wedge t_{j,n}^{(2)}\leq T}(\Delta_{i,n}^{(1)} C^{(1)})^2(\Delta_{j,n}^{(2)} C^{(2)})^2\mathds{1}_{\{|\Delta_{i,n}^{(1)} X^{(1)} |> \beta |\mathcal{I}_{i,n}^{(1)}|^{\varpi} \}} \nonumber
\\&~~~~~~~~~~~\times
\mathds{1}_{\{|\Delta_{j,n}^{(2)} X^{(2)} |> \beta |\mathcal{I}_{j,n}^{(2)}|^{\varpi} \}} \mathds{1}_{\{\mathcal{I}_{i,n}^{(1)} \cap\mathcal{I}_{j,n}^{(2)}\neq \emptyset\}} |\mathcal{S}]
\\&~~\leq n \sum_{i,j:t_{i,n}^{(1)} \wedge t_{j,n}^{(2)}\leq T}\big(\mathbb{E}[(\Delta_{i,n}^{(1)} C^{(1)})^8|\mathcal{S}]\mathbb{E}[(\Delta_{j,n}^{(2)} C^{(2)})^8|\mathcal{S}] \mathbb{P}(|\Delta_{i,n}^{(1)} X^{(1)} |> \beta |\mathcal{I}_{i,n}^{(1)}|^{\varpi}|\mathcal{S}) 
\\&~~~~~~~~~~~\times\mathbb{P}(|\Delta_{j,n}^{(2)} X^{(2)} |> \beta |\mathcal{I}_{j,n}^{(2)}|^{\varpi}|\mathcal{S})\big)^{1/4} \mathds{1}_{\{\mathcal{I}_{i,n}^{(1)} \cap\mathcal{I}_{j,n}^{(2)}\neq \emptyset\}}
\\&~~ \leq K(|\pi_n|_T)^{(1/2-\varpi)/2} H_{1,n}(T)
\end{align*}
which vanishes as $n \rightarrow \infty$ by Condition \ref{cond_cons_CoJ2}. Hence for showing that \eqref{test_theo_cons_CoJ_proof3} vanishes as $n\rightarrow \infty$ it remains to discuss
\begin{align}\label{test_theo_cons_CoJ_proof5}
&Y^{(l)}(n,q)=n \sum_{l=1,2}\sum_{i,j:t_{i,n}^{(l)} \wedge t_{j,n}^{(3-l)} \leq T} (\Delta_{i,n}^{(l)} C^{(l)})^2(\Delta_{j,n}^{(3-l)} N(q)^{(3-l)})^2\mathds{1}_{\{|\Delta_{i,n}^{(l)} X^{(l)} |> \beta |\mathcal{I}_{i,n}^{(l)}|^{\varpi} \}} \nonumber
\\&~~~~~~~~~~~~~~~~~~~~\times
\mathds{1}_{\{|\Delta_{j,n}^{(3-l)} X^{(3-l)} |> \beta |\mathcal{I}_{j,n}^{(3-l)}|^{\varpi} \}} \mathds{1}_{\{\mathcal{I}_{i,n}^{(l)} \cap\mathcal{I}_{j,n}^{(3-l)}\neq \emptyset\}}\mathds{1}_{\Omega_T^{(nCoJ)}},\quad l=1,2.
\end{align}
Let $S_{q,p}^{(l)}$ denote an enumeration of the jump times of $N(q)^{(l)}$. Then $Y^{(l)}(n,q)$ is asymptotically equal to
\begin{align*}
&\sum_{l=1,2}\sum_{S_{q,p}^{(3-l)}\leq T} (\Delta N(q)^{(3-l)}_{S_{q,p}^{(3-l)}})^2
\\&~~~\times n\sum_{i:t_{i,n}^{(l)}\leq T}(\Delta_{i,n}^{(l)} C^{(l)})^2\mathds{1}_{\{|\Delta_{i,n}^{(l)} X^{(l)} |> \beta |\mathcal{I}_{i,n}^{(l)}|^{\varpi} \}} \mathds{1}_{\{\mathcal{I}_{i,n}^{(l)} \cap\mathcal{I}_{i_n^{(3-l)}(S_{q,p}^{(3-l)}),n}^{(3-l)}\neq \emptyset\}}\mathds{1}_{\Omega_T^{(nCoJ)}}.
\end{align*}
In the proof of Proposition A.3 in \cite{MarVet17} it is shown that 
\begin{align*}
n\sum_{i:t_{i,n}^{(l)}\leq T}(\Delta_{i,n}^{(l)} C^{(l)})^2\mathds{1}_{\{\mathcal{I}_{i,n}^{(l)} \cap\mathcal{I}_{i_n^{(3-l)}(S_{q,p}^{(3-l)}),n}^{(3-l)}\neq \emptyset\}} \overset{\mathcal{L}-s}{\longrightarrow} R_1^{(l)}(S_{q,p}^{(3-l)})
\end{align*}
for any jump time $S_{q,p}^{(3-l)}$ of $N(q)^{(3-l)}$. Further we obtain from (A.31) in \cite{MarVet17} and inequality \eqref{elem_ineq_C}
\begin{align*}
&\mathbb{E}[n\sum_{i:t_{i,n}^{(l)}\leq T}(\Delta_{i,n}^{(l)} C^{(l)})^2\mathds{1}_{\{|\Delta_{i,n}^{(l)} X^{(l)} |> \beta |\mathcal{I}_{i,n}^{(l)}|^{\varpi} \}} \mathds{1}_{\{\mathcal{I}_{i,n}^{(l)} \cap\mathcal{I}_{i_n^{(3-l)}(S_{q,p}^{(3-l)}),n}^{(3-l)}\neq \emptyset\}} |\mathcal{S} ]
\\&~~\leq n\sum_{i:t_{i,n}^{(l)}\leq T} (\mathbb{E}[(\Delta_{i,n}^{(l)} C^{(l)})^4|\mathcal{S}] \mathbb{P}(|\Delta_{i,n}^{(l)} X^{(l)} |> \beta |\mathcal{I}_{i,n}^{(l)}|^{\varpi} |\mathcal{S}))^{1/2}\mathds{1}_{\{\mathcal{I}_{i,n}^{(l)} \cap\mathcal{I}_{i_n^{(3-l)}(S_{q,p}^{(3-l)}),n}^{(3-l)}\neq \emptyset\}}
\\&~~ \leq K(|\pi_n|_T)^{1/2-\varpi}n\sum_{i:t_{i,n}^{(l)}\leq T} |\mathcal{I}_{i,n}^{(l)}|\mathds{1}_{\{\mathcal{I}_{i,n}^{(l)} \cap\mathcal{I}_{i_n^{(3-l)}(S_{q,p}^{(3-l)}),n}^{(3-l)}\neq \emptyset\}}.
\end{align*}
Here the expression in the last line vanishes as $n \rightarrow \infty$ because of
\begin{align*}
&\mathbb{E}[n\sum_{i:t_{i,n}^{(l)}\leq T} |\mathcal{I}_{i,n}^{(l)}|\mathds{1}_{\{\mathcal{I}_{i,n}^{(l)} \cap\mathcal{I}_{i_n^{(3-l)}(s),n}^{(3-l)}\neq \emptyset\}}]
\\&~~=\mathbb{E}[\eta_{1,n,-}^{(l)}(s)+(\Delta_{i_n^{(l)}(s),n}^{(l)}\overline{W}^{(l)})^2 +\eta_{1,n,+}^{(l)}(s)]
\\&~~=\mathbb{E}[\eta_{1,n,-}^{(l)}(s)+(s-t_{i_n^{(l)}(s)-1,n}^{(l)})+(t_{i_n^{(l)}(s),n}^{(l)}-s) +\eta_{1,n,+}^{(l)}(s)]
\end{align*}
where the right hand side converges as $n \rightarrow \infty$ by Condition \ref{cond_cons_CoJ2}. This estimate yields
\begin{align*}
&\lim_{n \rightarrow \infty} \mathbb{P}\big((|\pi_n|_T)^{-(1/2-\varpi)/2} n\sum_{i:t_{i,n}^{(l)}\leq T} |\mathcal{I}_{i,n}^{(l)}|\mathds{1}_{\{|\Delta_{i,n}^{(l)} X^{(l)} |> \beta |\mathcal{I}_{i,n}^{(l)}|^{\varpi} \}}\mathds{1}_{\{\mathcal{I}_{i,n}^{(l)} \cap\mathcal{I}_{i_n^{(3-l)}(S_{q,p}^{(3-l)}),n}^{(3-l)}\neq \emptyset\}} \\&~~~~~~~~~~~~~~~~~~~~~~~~~~~~~~~~~~~~~~~~~~~~~~~~~~~~~~~~~~~~~~~~~~~~~~~~~~~~~ \leq R_1^{(l)}(S_{q,p}^{(3-l)})\big)=1,
\end{align*} 
for any jump time $S_{q,p}^{(3-l)}$ from which we conclude
\begin{align*}
\lim_{n \rightarrow \infty} \mathbb{P}((|\pi_n|_T)^{-(1/2-\varpi)/2}Y^{(l)}(n,q) \leq D_{1,T}^{(CoJ)})=1, \quad l=1,2.
\end{align*}
Hence we finally get that \eqref{test_theo_cons_CoJ_proof5} and also \eqref{test_theo_cons_CoJ_proof3} vanish as $n \rightarrow \infty$.

Next we consider \eqref{test_theo_cons_CoJ_proof4}. Using Condition \ref{cond_cons_CoJ} and the assumption that $\Gamma_t$ is bounded we obtain that there exists a $q'>0$ with $X=B(q')+C+M(q')$. Plugging this decomposition into \eqref{test_theo_cons_CoJ_proof4} and using Lemma A.4 from \cite{MarVet17} we obtain that the $\mathcal{S}$-conditional expectation of \eqref{test_theo_cons_CoJ_proof4} is bounded by
\begin{align*}
&Kn \varepsilon \sum_{l=1,2} \sum_{i,j:t_{i,n}^{(l)} \wedge t_{j,n}^{(3-l)} \leq T} |\mathcal{I}_{j,n}^{(3-l)}| \Big(\frac{1}{ b_n}\sum_{\iota:\mathcal{I}_{\iota,n}^{(l)} \subset [t_{i,n}^{(l)}-b_n,t_{i,n}^{(l)}+b_n]} |\mathcal{I}_{\iota,n}^{(l)}|\Big) (|\mathcal{I}_{i-1,n}^{(l)}|+|\mathcal{I}_{i+1,n}^{(l)}|)
\\&~~~~~~~~~~~\times\mathds{1}_{\{\mathcal{I}_{i,n}^{(l)} \cap\mathcal{I}_{j,n}^{(3-l)}\neq \emptyset\}}
\\&~~\leq K \varepsilon H_{2,n}(T)+O_\mathbb{P}(n(|\pi_n|_T)^2).
\end{align*}
This expression converges to $K \varepsilon H_2(T)$ as $n \rightarrow \infty$. Hence alltogether we have shown
\begin{align*}
\lim_{\varepsilon \rightarrow 0} \limsup_{n \rightarrow \infty} \mathbb{P}(\mathbb{E}[\sqrt{n} |\widehat{F}_{T,n,m}^{(CoJ)}|\mathds{1}_{\Omega_T^{(nCoJ)}} |\mathcal{F} ] >\delta)=0
\end{align*}
for any $\delta >0$ which by Lemma \ref{conv_cond_expec} yields \eqref{test_theo_cons_CoJ_proof2}.
\end{proof}

The following Lemma is needed in the proof of Lemma \eqref{lemma_conv_sigma_CoJ}. The techniques used in the proof are taken from the proof of Theorem 3.1 in \cite{hayyos05}.

\begin{lemma}\label{lemma_hayyos_L2_mart}
Let $Y^{(1)},Y^{(2)}$ be martingales which are adapted to the filtration $(\mathcal{F}_t)_{t \geq 0}$. Further let $(s,t]$ be a deterministic interval and the observation scheme be exogenous. If the inequalities
\begin{align}
&\mathbb{E}[(Y^{(1)}_{t'}-Y^{(1)}_{s'})^2|\mathcal{F}_{s'}]\leq K_{Y^{(1)}} (t-s), \quad \mathbb{E}[(Y^{(2)}_{t'}-Y^{(2)}_{s'})^2|\mathcal{F}_{s'}]\leq K_{Y^{(2)}} (t-s),~~~~~ \label{lemma_hayyos_L2_mart_cond1}
\\& \mathbb{E}\big[(Y^{(1)}_{t'}-Y^{(1)}_{s'})^2(Y^{(2)}_{t'}-Y^{(2)}_{s'})^2\big|\mathcal{F}_{s'}\big] \leq K_{Y^{(1)},Y^{(2)}}(t'-s')^2, \label{lemma_hayyos_L2_mart_cond2}
\end{align}
hold, then for all $s\leq s'\leq t' \leq t$ we have 
\begin{align*}
\mathbb{E} [ (\sum_{i,j:\mathcal{I}_{i,n}^{(1)} \cup \mathcal{I}_{j,n}^{(2)} \subset (s,t]} \Delta_{i,n}^{(1)}Y^{(1)}\Delta_{j,n}^{(2)}Y^{(2)}\mathds{1}_{\{\mathcal{I}_{i,n}^{(1)} \cap \mathcal{I}_{j,n}^{(2)} \neq \emptyset\}}  )^2|\sigma(\mathcal{F}_s,\mathcal{S})] \leq \widetilde{K}_{Y^{(1)},Y^{(2)}}(t-s)^2
\end{align*}
where $\widetilde{K}_{Y^{(1)},Y^{(2)}}=14K_{Y^{(1)}}K_{Y^{(2)}}+3K_{Y^{(1)},Y^{(2)}}$.
\end{lemma}

\begin{proof} To shorten notation set $\mathds{1}_{i,j}=\mathds{1}_{\{\mathcal{I}_{i,n}^{(1)} \cap \mathcal{I}_{j,n}^{(2)} \neq \emptyset\}}$. Note that it holds
\begin{align}
&\big(\sum_{i,j:\mathcal{I}_{i,n}^{(1)} \cup \mathcal{I}_{j,n}^{(2)} \subset (s,t]} \Delta_{i,n}^{(1)}Y^{(1)}\Delta_{j,n}^{(2)}Y^{(2)}\mathds{1}_{i,j}  \big)^2 \nonumber
\\&~~~~=\sum_{i,j,i',j':\mathcal{I}_{i,n}^{(1)} \cup \mathcal{I}_{j,n}^{(2)} \cup \mathcal{I}_{i',n}^{(1)} \cup \mathcal{I}_{j',n}^{(2)} \subset (s,t]} \Delta_{i,n}^{(1)}Y^{(1)}\Delta_{j,n}^{(2)}Y^{(2)}\Delta_{i',n}^{(1)}Y^{(1)}\Delta_{j',n}^{(2)}Y^{(2)}\mathds{1}_{i,j}\mathds{1}_{i',j'}\nonumber
\\&~~~~=\sum_{i=i',j=j'} (\Delta_{i,n}^{(1)}Y^{(1)})^2(\Delta_{j,n}^{(2)}Y^{(2)})^2\mathds{1}_{i,j} \label{lemma_hayyos_L2_mart_term1}
\\&~~~~~~~~~~+\sum_{i=i',j\neq j'} (\Delta_{i,n}^{(1)}Y^{(1)})^2\Delta_{j,n}^{(2)}Y^{(2)}\Delta_{j',n}^{(2)}Y^{(2)}\mathds{1}_{i,j}\mathds{1}_{i,j'} \label{lemma_hayyos_L2_mart_term2}
\\&~~~~~~~~~~+\sum_{i\neq i',j=j'} \Delta_{i,n}^{(1)}Y^{(1)}\Delta_{i',n}^{(1)}Y^{(1)}(\Delta_{j,n}^{(2)}Y^{(2)})^2\mathds{1}_{i,j}\mathds{1}_{i',j} \label{lemma_hayyos_L2_mart_term3}
\\&~~~~~~~~~~+\sum_{i\neq i',j\neq j'} \Delta_{i,n}^{(1)}Y^{(1)}\Delta_{i',n}^{(1)}Y^{(1)}\Delta_{j,n}^{(2)}Y^{(2)}\Delta_{j',n}^{(2)}Y^{(2)}\mathds{1}_{i,j}\mathds{1}_{i',j'} \label{lemma_hayyos_L2_mart_term4}
\end{align}
and we discuss the terms \eqref{lemma_hayyos_L2_mart_term1}--\eqref{lemma_hayyos_L2_mart_term4} separately.

For discussing \eqref{lemma_hayyos_L2_mart_term1} we denote by $\Delta_{(i,j),n}^{(1,2)}Y^{(l)}$ the increment of $Y^{(l)}$, $l=1,2$, over the interval $\mathcal{I}_{i,n}^{(1)} \cap \mathcal{I}_{j,n}^{(2)}$ and by $\Delta_{(i,\{j\}),n}^{(l \setminus 3-l)}Y^{(l)}$ the increment of $Y^{(l)}$ over the interval $\mathcal{I}_{i,n}^{(l)} \setminus \mathcal{I}_{j,n}^{(3-l)} $ (which might also be the sum of increments over two distinct intervals). Using this notation, iterated expectations and the inequalities \eqref{lemma_hayyos_L2_mart_cond1}, \eqref{lemma_hayyos_L2_mart_cond2} we obtain that the $\sigma(\mathcal{F}_s,\mathcal{S})$-conditional expectation of \eqref{lemma_hayyos_L2_mart_term1} is equal to
\begin{align*}
&\sum_{i,j} \mathbb{E}[ (\Delta_{(i,j),n}^{(1,2)}Y^{(1)}+\Delta_{(i,\{j\}),n}^{(1 \setminus 2)}Y^{(1)})^2 (\Delta_{(i,j),n}^{(1,2)}Y^{(2)}+\Delta_{(j,\{i\}),n}^{(2 \setminus 1)}Y^{(2)})^2 |\sigma(\mathcal{F}_s,\mathcal{S})] \mathds{1}_{i,j}
\\&~= \sum_{i,j} \mathbb{E}[ ((\Delta_{(i,j),n}^{(1,2)}Y^{(1)})^2+(\Delta_{(i,\{j\}),n}^{(1 \setminus 2)}Y^{(1)})^2)(\Delta_{(j,\{i\}),n}^{(2 \setminus 1)}Y^{(2)})^2
\\&~~~~~~~~~~~~~~~~~~~~~~~~~~~~~~~~~~~+(\Delta_{(i,\{j\}),n}^{(1 \setminus 2)}Y^{(1)})^2(\Delta_{(i,j),n}^{(1,2)}Y^{(2)})^2 |\sigma(\mathcal{F}_s,\mathcal{S})] \mathds{1}_{i,j}
\\&~~~~+\sum_{i,j}\mathbb{E}[ (\Delta_{(i,j),n}^{(1,2)}Y^{(1)})^2(\Delta_{(i,j),n}^{(1,2)}Y^{(2)})^2 |\sigma(\mathcal{F}_s,\mathcal{S})]
\\&~\leq \sum_{i,j} K_{Y^{(1)}}K_{Y^{(2)}}(|\mathcal{I}_{i,n}^{(1)}| |\mathcal{I}_{j,n}^{(2)}\setminus \mathcal{I}_{i,n}^{(1)}|+|\mathcal{I}_{i,n}^{(1)}\setminus \mathcal{I}_{j,n}^{(2)}||\mathcal{I}_{i,n}^{(1)}\cap \mathcal{I}_{j,n}^{(2)}|)\mathds{1}_{i,j}
\\&~~~~+\sum_{i,j}K_{Y^{(1)},Y^{(2)}}|\mathcal{I}_{i,n}^{(1)}\cap \mathcal{I}_{j,n}^{(2)}|^2
\\&~ =(2K_{Y^{(1)}}K_{Y^{(2)}} + K_{Y^{(1)},Y^{(2)}})\sum_{i,j:\mathcal{I}_{i,n}^{(1)} \cup \mathcal{I}_{j,n}^{(2)} \subset (s,t]} |\mathcal{I}_{i,n}^{(1)}||\mathcal{I}_{j,n}^{(2)}|\mathds{1}_{i,j} 
\\&~\leq 3(2K_{Y^{(1)}}K_{Y^{(2)}} + K_{Y^{(1)},Y^{(2)}})  (|\pi_n|_t \wedge(t-s))(t-s).
\end{align*}

For treating \eqref{lemma_hayyos_L2_mart_term2} we additionally denote by $\Delta_{(i,\{j,j'\}),n}^{(1 \setminus 2)}Y^{(1)}$ the increment of $Y^{(1)}$ over the interval $\mathcal{I}_{i,n}^{(1)} \setminus( \mathcal{I}_{j,n}^{(2)} \cup \mathcal{I}_{j',n}^{(2)})$ (which might also be the sum of increments over up to three distinct intervals). We then obtain by using iterated expectations, the fact that $Y^{(1)}, Y^{(2)}$ are martingales, the Cauchy-Schwarz inequality and \eqref{lemma_hayyos_L2_mart_cond1} that the $\sigma(\mathcal{F}_s,\mathcal{S})$-conditional expectation of \eqref{lemma_hayyos_L2_mart_term2} is bounded by
\begin{align*}
&\sum_{i,j \neq j'} \big|\mathbb{E}[(\Delta_{(i,j),n}^{(1,2)}Y^{(1)}+\Delta_{(i,j'),n}^{(1,2)}Y^{(1)}+\Delta_{(i,\{j,j'\}),n}^{(1 \setminus 2)}Y^{(1)})^2
\\&~~~~~~~~~~~~~~~~~~~~~~~~~~~~~~~~~~~~~~~~~~\times\Delta_{(i,j),n}^{(1,2)}Y^{(2)}\Delta_{(i,j'),n}^{(1,2)}Y^{(2)} |\sigma(\mathcal{F}_s,\mathcal{S})]\big|\mathds{1}_{i,j}\mathds{1}_{i,j'}
\\&~~=2\sum_{i,j \neq j'} \big|\mathbb{E}[\Delta_{(i,j),n}^{(1,2)}Y^{(1)}\Delta_{(i,j),n}^{(1,2)}Y^{(2)}\Delta_{(i,j'),n}^{(1,2)}Y^{(1)}\Delta_{(i,j'),n}^{(1,2)}Y^{(2)} |\sigma(\mathcal{F}_s,\mathcal{S})]\big|
\\&~~\leq 2\sum_{i,j \neq j'} K_{Y^{(1)}}K_{Y^{(2)}}|\mathcal{I}_{i,n}^{(1)} \cap \mathcal{I}_{j,n}^{(2)}||\mathcal{I}_{i,n}^{(1)} \cap \mathcal{I}_{j',n}^{(2)}|
\\&~~\leq K_{Y^{(1)}}K_{Y^{(2)}} \sum_{i:\mathcal{I}_{i,n}^{(1)} \subset (s,t]} |\mathcal{I}_{i,n}^{(1)}|^2 \leq 6 K_{Y^{(1)}}K_{Y^{(2)}} (|\pi_n|_t \wedge (t-s)) (t-s).
\end{align*}

By symmetry \eqref{lemma_hayyos_L2_mart_term3} can be treated similarly as \eqref{lemma_hayyos_L2_mart_term2}. Hence it remains to discuss \eqref{lemma_hayyos_L2_mart_term4}. Using the notation introduced above we obtain that the $\sigma(\mathcal{F}_s,\mathcal{S})$-conditional expectation of \eqref{lemma_hayyos_L2_mart_term3} is equal to
\begin{align*}
&\sum_{i \neq i',j \neq j'}\mathbb{E}[
(\Delta_{(i,j),n}^{(1,2)}Y^{(1)}+\Delta_{(i,j'),n}^{(1,2)}Y^{(1)}+\Delta_{(i,\{j,j'\}),n}^{(1 \setminus 2)}Y^{(1)})
\\&~~~~~~~~\times(\Delta_{(i',j),n}^{(1,2)}Y^{(1)}+\Delta_{(i',j'),n}^{(1,2)}Y^{(1)}+\Delta_{(i',\{j,j'\}),n}^{(1 \setminus 2)}Y^{(1)})
\\&~~~~~~~~\times(\Delta_{(i,j),n}^{(1,2)}Y^{(2)}+\Delta_{(i',j),n}^{(1,2)}Y^{(2)}+\Delta_{(j,\{i,i'\}),n}^{(2 \setminus 1)}Y^{(2)})
\\&~~~~~~~~\times(\Delta_{(i,j'),n}^{(1,2)}Y^{(2)}+\Delta_{(i',j'),n}^{(1,2)}Y^{(2)}+\Delta_{(j',\{i,i'\}),n}^{(2 \setminus 1)}Y^{(2)})|\sigma(\mathcal{F}_s,\mathcal{S})] 
\mathds{1}_{i,j}\mathds{1}_{i',j'}
\\&~~=\sum_{i \neq i',j \neq j'}\mathbb{E}[
(\Delta_{(i,j),n}^{(1,2)}Y^{(1)}+\Delta_{(i,j'),n}^{(1,2)}Y^{(1)})(\Delta_{(i',j),n}^{(1,2)}Y^{(1)}+\Delta_{(i',j'),n}^{(1,2)}Y^{(1)})
\\&~~~~~~~~\times(\Delta_{(i,j),n}^{(1,2)}Y^{(2)}+\Delta_{(i',j),n}^{(1,2)}Y^{(2)})(\Delta_{(i,j'),n}^{(1,2)}Y^{(2)}+\Delta_{(i',j'),n}^{(1,2)}Y^{(2)})|\sigma(\mathcal{F}_s,\mathcal{S})] 
\mathds{1}_{i,j}\mathds{1}_{i',j'}.
\end{align*}
Here terms including factors like $\Delta_{(i,\{j,j'\}),n}^{(1 \setminus 2)}Y^{(1)}$ vanish using iterated expectations because $Y^{(1)}$ is a martingale and $\mathcal{I}_{i,n}^{(1)} \setminus (\mathcal{I}_{j,n}^{(2)} \cup \mathcal{I}_{j',n}^{(2)})$ overlaps with none of the other intervals for $i \neq i'$. Expanding the product above and observing that the expectation vanishes for all products where one of $i,i',j,j'$ occurs exactly once yields that the $\sigma(\mathcal{F}_s,\mathcal{S})$-conditional expectation of \eqref{lemma_hayyos_L2_mart_term4} is equal to
\begin{align*}
&\sum_{i \neq i',j \neq j'}\mathbb{E}[
\Delta_{(i,j),n}^{(1,2)}Y^{(1)}\Delta_{(i',j'),n}^{(1,2)}Y^{(1)}\Delta_{(i,j),n}^{(1,2)}Y^{(2)}\Delta_{(i',j'),n}^{(1,2)}Y^{(2)}
\\&~~~~~~~~+\Delta_{(i,j),n}^{(1,2)}Y^{(1)}\Delta_{(i',j'),n}^{(1,2)}Y^{(1)}\Delta_{(i',j),n}^{(1,2)}Y^{(2)}\Delta_{(i,j'),n}^{(1,2)}Y^{(2)}
\\&~~~~~~~~+\Delta_{(i,j'),n}^{(1,2)}Y^{(1)}\Delta_{(i',j),n}^{(1,2)}Y^{(1)}\Delta_{(i,j),n}^{(1,2)}Y^{(2)}\Delta_{(i',j'),n}^{(1,2)}Y^{(2)}
\\&~~~~~~~~+\Delta_{(i,j'),n}^{(1,2)}Y^{(1)}\Delta_{(i',j),n}^{(1,2)}Y^{(1)}\Delta_{(i',j),n}^{(1,2)}Y^{(2)}\Delta_{(i,j'),n}^{(1,2)}Y^{(2)}
|\sigma(\mathcal{F}_s,\mathcal{S})] 
\mathds{1}_{i,j}\mathds{1}_{i',j'}
\\&~~=\sum_{i \neq i',j \neq j'}\mathbb{E}[
\Delta_{(i,j),n}^{(1,2)}Y^{(1)}\Delta_{(i',j'),n}^{(1,2)}Y^{(1)}\Delta_{(i,j),n}^{(1,2)}Y^{(2)}\Delta_{(i',j'),n}^{(1,2)}Y^{(2)}
\\&~~~~~~~~+\Delta_{(i,j'),n}^{(1,2)}Y^{(1)}\Delta_{(i',j),n}^{(1,2)}Y^{(1)}\Delta_{(i',j),n}^{(1,2)}Y^{(2)}\Delta_{(i,j'),n}^{(1,2)}Y^{(2)}
|\sigma(\mathcal{F}_s,\mathcal{S})] 
\mathds{1}_{i,j}\mathds{1}_{i',j'}
\end{align*}
where the second identity holds as one of the intervals $\mathcal{I}_{i,n}^{(1)}\cap \mathcal{I}_{j,n}^{(2)}$, $\mathcal{I}_{i,n}^{(1)}\cap \mathcal{I}_{j',n}^{(2)}$, ${\mathcal{I}_{i',n}^{(1)}\cap \mathcal{I}_{j,n}^{(2)}}$, $\mathcal{I}_{i',n}^{(1)}\cap \mathcal{I}_{j',n}^{(2)}$ always has to be empty for $i=i'$, $j \neq j'$. We then further obtain using iterated expectations, the conditional Cauchy-Schwarz inequality and \eqref{lemma_hayyos_L2_mart_cond2} that the $\sigma(\mathcal{F}_s,\mathcal{S})$-conditional expectation of \eqref{lemma_hayyos_L2_mart_term4} is bounded by
\begin{align*}
& K_{Y^{(1)}}K_{Y^{(2)}} \sum_{i \neq i',j \neq j'} (|\mathcal{I}_{i,n}^{(1)} \cap \mathcal{I}_{j,n}^{(2)}| |\mathcal{I}_{i',n}^{(1)} \cap \mathcal{I}_{j',n}^{(2)}|
+|\mathcal{I}_{i,n}^{(1)} \cap \mathcal{I}_{j',n}^{(2)}| |\mathcal{I}_{i',n}^{(1)} \cap \mathcal{I}_{j,n}^{(2)}|)\mathds{1}_{i,j}\mathds{1}_{i',j'}
\\&~~\leq 2K_{Y^{(1)}}K_{Y^{(2)}}\sum_{i:\mathcal{I}_{i,n}^{(1)} \subset (s,t]} |\mathcal{I}_{i,n}^{(1)}|\sum_{i':\mathcal{I}_{i',n}^{(1)} \subset (s,t]} |\mathcal{I}_{i',n}^{(1)}| \leq 2K_{Y^{(1)}}K_{Y^{(2)}} (t-s)^2.
\end{align*}
We conclude the proof by combining the bounds for \eqref{lemma_hayyos_L2_mart_term1}--\eqref{lemma_hayyos_L2_mart_term4} and using Lemma \ref{conv_cond_expec}. 
\end{proof}

\begin{proof}[Proof of Lemma \ref{lemma_conv_sigma_CoJ}]
First note that \eqref{lemma_conv_sigma_sigma} follows from Lemma \ref{lemma_conv_sigma}. Further \eqref{lemma_conv_sigma_rho} follows from \eqref{lemma_conv_sigma_sigma} and
\begin{equation}\label{lemma_conv_sigma_kappa}
\begin{aligned}
&\hat{\kappa}_n(S_p,-)\mathds{1}_{\{S_p < T\}} \overset{\mathbb{P}}{\longrightarrow} \rho_{S_p-}\sigma^{(1)}_{S_p-}\sigma^{(2)}_{S_p-}\mathds{1}_{\{S_p < T\}}, 
\\& \hat{\kappa}_n(S_p,+)\mathds{1}_{\{S_p < T\}} \overset{\mathbb{P}}{\longrightarrow} \rho_{S_p}\sigma^{(1)}_{S_p}\sigma^{(2)}_{S_p}\mathds{1}_{\{S_p < T\}}
\end{aligned}
\end{equation}
because of $\sigma^{(1)}_s \sigma^{(2)}_s \neq 0$ almost surely for almost all $s \in [0,T]$ by Condition \ref{cond_clt_coJ}. Hence it remains to prove \eqref{lemma_conv_sigma_kappa}.

We will only give the proof for $\hat{\kappa}_n(s,+)$ as the proof for $\hat{\kappa}_n(s,-)$ is identical. Further we will give a formal proof only for constant stopping times $s \in (0,T)$ and assume without loss of generality $s+b_n \leq T$ as we did for Lemma \ref{lemma_conv_sigma}. The extension to jump times $S_p$ then works analogous as for the estimators for $\sigma_{S_p-}$, $\sigma_{S_p}$; compare the proof of Theorem 9.3.2 in \cite{JacPro12}. 

We first prove
\begin{align}\label{cons_spot_kappa_exo_proof1}
\hat{\kappa}^{(l)}(s,+)-\frac{1}{b_n}\sum_{i,j:\mathcal{I}_{i,n}^{(1)} \cup \mathcal{I}_{j,n}^{(2)} \subset [s,s+b_n]} \Delta_{i,n}^{(1)} C^{(1)}\Delta_{j,n}^{(2)} C^{(2)}\mathds{1}_{\{\mathcal{I}_{i,n}^{(1)} \cap \mathcal{I}_{j,n}^{(2)} \neq \emptyset\}} \overset{\mathbb{P}}{\longrightarrow} 0.
\end{align}
Therefore consider
\begin{align}
&\hat{\kappa}^{(l)}(s,+)-\frac{1}{b_n}\sum_{i,j:\mathcal{I}_{i,n}^{(1)} \cup \mathcal{I}_{j,n}^{(2)} \subset [s,s+b_n]} \Delta_{i,n}^{(l)} C^{(1)}\Delta_{j,n}^{(l)} C^{(2)}\mathds{1}_{\{\mathcal{I}_{i,n}^{(1)} \cap \mathcal{I}_{j,n}^{(2)} \neq \emptyset\}} \nonumber
\\&~~=\frac{1}{b_n}\sum_{i,j:\mathcal{I}_{i,n}^{(1)} \cup \mathcal{I}_{j,n}^{(2)} \subset [s,s+b_n]} \big(\Delta_{i,n}^{(1)} B^{(1)}(q)\Delta_{j,n}^{(2)} X^{(2)}+\Delta_{i,n}^{(1)}X^{(1)}\Delta_{j,n}^{(2)} B^{(2)}(q) \big)\nonumber
\\&~~~~~~~~~~~~~~~~~~~~~~~~~~~~~~~- \Delta_{i,n}^{(1)}B^{(1)}(q)\Delta_{j,n}^{(2)} B^{(2)}(q) \big)\mathds{1}_{\{\mathcal{I}_{i,n}^{(1)} \cap \mathcal{I}_{j,n}^{(2)} \neq \emptyset\}} ~~\label{cons_spot_kappa_exo_proof2}
\\&~~~~+\frac{1}{b_n}\sum_{i,j:\mathcal{I}_{i,n}^{(1)} \cup \mathcal{I}_{j,n}^{(2)} \subset [s,s+b_n]} \big(\Delta_{i,n}^{(1)} N^{(1)}(q)\Delta_{j,n}^{(2)} (X^{(2)}-B^{(2)}(q))\nonumber
\\&~~~~~~~~~~~~~~~~~~~~~~~~~~~~~~~+ \Delta_{i,n}^{(1)}(C^{(1)}+M^{(1)}(q))\Delta_{j,n}^{(2)} N^{(2)}(q) \big)\mathds{1}_{\{\mathcal{I}_{i,n}^{(1)} \cap \mathcal{I}_{j,n}^{(2)} \neq \emptyset\}}~~\label{cons_spot_kappa_exo_proof3}
\\&~~~~+\sum_{l=1,2}\frac{1}{b_n}\sum_{i,j:\mathcal{I}_{i,n}^{(l)} \cup \mathcal{I}_{j,n}^{(3-l)} \subset [s,s+b_n]} \Delta_{i,n}^{(l)} M^{(l)}(q)\Delta_{j,n}^{(3-l)} C^{(3-l)}\mathds{1}_{\{\mathcal{I}_{i,n}^{(l)} \cap \mathcal{I}_{j,n}^{(3-l)} \neq \emptyset\}} \label{cons_spot_kappa_exo_proof4}
\\&~~~~+\frac{1}{b_n}\sum_{i,j:\mathcal{I}_{i,n}^{(1)} \cup \mathcal{I}_{j,n}^{(2)} \subset [s,s+b_n]} \Delta_{i,n}^{(1)} M^{(1)}(q)\Delta_{j,n}^{(2)} M^{(2)}(q)\mathds{1}_{\{\mathcal{I}_{i,n}^{(1)} \cap \mathcal{I}_{j,n}^{(2)} \neq \emptyset\}} \label{cons_spot_kappa_exo_proof5}
\end{align}
where all terms \eqref{cons_spot_kappa_exo_proof2}--\eqref{cons_spot_kappa_exo_proof5} vanish for first $n \rightarrow \infty$ and then $q \rightarrow \infty$ as we will show in the following.

The $\mathcal{S}$-conditional expectation of the absolute value of \eqref{cons_spot_kappa_exo_proof2} is using \eqref{elem_ineq_B} and \eqref{elem_ineq_X} bounded by
\begin{align*}
&\sum_{l=1,2} \frac{1}{b_n}\sum_{i:\mathcal{I}_{i,n}^{(l)} \subset [s,s+b_n]} K_q|\mathcal{I}_{i,n}^{(l)}|
\\ &~~~~~~~~~~\times 
\mathbb{E}\big[(X^{(3-l)}_{t^{(3-l)}_{i_n^{(3-l)}(t_{i,n}^{(l)}) \wedge (i_n^{(3-l)}(s+b_n)-1),n }}-X^{(3-l)}_{t^{(3-l)}_{i_n^{(3-l)}(t_{i-1,n}^{(l)})-1) \vee i_n^{(3-l)}(s),n }})^2\big|\mathcal{S}\big]^{1/2}
\\&~~~~~+ \frac{1}{b_n}\sum_{i,j:\mathcal{I}_{i,n}^{(1)} \cup \mathcal{I}_{j,n}^{(2)} \subset [s,s+b_n]}K_q |\mathcal{I}^{(1)}_{i,n}| |\mathcal{I}^{(2)}_{j,n}| \mathds{1}_{\{\mathcal{I}_{i,n}^{(1)} \cap \mathcal{I}_{j,n}^{(2)} \neq \emptyset\}}
\\&~~ \leq (3|\pi_n|_T)^{1/2}\sum_{l=1,2} \frac{1}{b_n}\sum_{i:\mathcal{I}_{i,n}^{(l)} \subset [s,s+b_n]} K_q|\mathcal{I}_{i,n}^{(l)}| + K_q |\pi_n|_T
\\&~~\leq K_q (|\pi_n|_T)^{1/2}
\end{align*}
which vanishes as $n \rightarrow \infty$. \eqref{cons_spot_kappa_exo_proof3} vanishes as $n \rightarrow \infty$ because of $$\lim_{n \rightarrow \infty} \mathbb{P}(\exists t \in (s,s+b_n]:\Delta N(q)_t \neq 0)= 0.$$

Using \eqref{elem_ineq_C}, \eqref{elem_ineq_M} and 
\begin{align*}
\mathbb{E}[(M(q)^{(l)}_{s+t}-M(q)^{(l)}_{s})^2 (C^{(3-l)}_{s+t}-C^{(3-l)}_{s})^2 |\mathcal{F}_s]\leq K e_q t, \quad s,t \geq 0 ,~l=1,2,
\end{align*}
which is obtained from Lemma A.4 in \cite{MarVet17} we may apply Lemma \ref{lemma_hayyos_L2_mart} for \eqref{cons_spot_kappa_exo_proof4} which yields
\begin{align*}
\mathbb{E}[\eqref{cons_spot_kappa_exo_proof4}^2|\mathcal{S}]
\leq K e_q. 
\end{align*}
The right hand side vanishes as $q \rightarrow \infty$. Next consider the following decomposition of \eqref{cons_spot_kappa_exo_proof5} using the notation from the proof of Lemma \ref{lemma_hayyos_L2_mart}
\begin{align*}
&\eqref{cons_spot_kappa_exo_proof5}=\frac{1}{b_n}\sum_{i,j:\mathcal{I}_{i,n}^{(1)} \cup \mathcal{I}_{j,n}^{(2)} \subset [s,s+b_n]}\Delta_{(i,j),n}^{(1,2)} M^{(1)}(q)\Delta_{(i,j),n}^{(1,2)} M^{(2)}(q) 
\\&~~~~~~+\sum_{l=1,2}\frac{1}{b_n}\sum_{i,j:\mathcal{I}_{i,n}^{(l)} \cup \mathcal{I}_{j,n}^{(3-l)} \subset [s,s+b_n]} \Delta_{(i,\{j\}),n}^{(l\setminus 3-l)} M^{(l)}(q)\Delta_{(i,j),n}^{(l,3-l)} M^{(3-l)}(q)\mathds{1}_{\{\mathcal{I}_{i,n}^{(l)} \cap \mathcal{I}_{j,n}^{(3-l)} \neq \emptyset\}}
\\&~~~~~~+ \frac{1}{b_n}\sum_{i,j:\mathcal{I}_{i,n}^{(1)} \cup \mathcal{I}_{j,n}^{(2)}\subset [s,s+b_n]} \Delta_{(i,\{j\}),n}^{(1\setminus 2)} M^{(1)}(q)\Delta_{(j,\{i\}),n}^{(2\setminus 1)} M^{(2)}(q)\mathds{1}_{\{\mathcal{I}_{i,n}^{(1)} \cap \mathcal{I}_{j,n}^{(2)} \neq \emptyset\}}
\end{align*}
where the $\mathcal{S}$-conditional expectation of the first sum can be bounded by $e_q$ using \eqref{elem_ineq_M} because of $|\Delta_{(i,j),n}^{(1,2)} M^{(1)}(q)\Delta_{(i,j),n}^{(1,2)} M^{(2)}(q)|\leq K\| \Delta_{(i,j),n}^{(1,2)} M(q)\|^2$. Further the expectation of the square of the second and third sum can be bounded as in the proof of Lemma \ref{lemma_hayyos_L2_mart} by $(e_q)^2$ because the condition \eqref{lemma_hayyos_L2_mart_cond2} is only needed for the treatment of the sum involving terms of the form $(\Delta_{(i,j),n}^{(1,2)} M^{(1)}(q))^2(\Delta_{(i,j),n}^{(1,2)} M^{(2)}(q))^2$ which we here treat separately. Hence we have proven \eqref{cons_spot_kappa_exo_proof1}.

Then \eqref{cons_spot_kappa_exo_proof1} yields that for proving $\hat{\kappa}_n(s,+) \overset{\mathbb{P}}{\longrightarrow} \rho_{s}\sigma^{(1)}_{s}\sigma^{(2)}_{s}$ it suffices to verify
\begin{align}\label{cons_spot_kappa_exo_proofb1}
\frac{1}{b_n}\sum_{i,j:\mathcal{I}_{i,n}^{(1)} \cup \mathcal{I}_{j,n}^{(2)}\subset [s,s+b_n]} \Delta_{i,n}^{(1)} C^{(1)}\Delta_{j,n}^{(2)} C^{(2)}\mathds{1}_{\{\mathcal{I}_{i,n}^{(1)} \cap \mathcal{I}_{j,n}^{(2)} \neq \emptyset\}} \overset{\mathbb{P}}{\longrightarrow} \rho_s \sigma_s^{(1)} \sigma_s^{(2)}.
\end{align}
For proving \eqref{cons_spot_kappa_exo_proofb1} define
\begin{gather*}
\zeta_i^n=\frac{1}{b_n}\sum_{j,k:\mathcal{I}_{j,n}^{(1)} \cup\mathcal{I}_{k,n}^{(2)}\subset (s+(i-1)b_n/r_n,s+ib_n/r_n]} \Delta_{j,n}^{(1)} C^{(1)}\Delta_{k,n}^{(2)} C^{(2)}\mathds{1}_{\{\mathcal{I}_{j,n}^{(1)} \cap \mathcal{I}_{k,n}^{(2)} \neq \emptyset\}}- \frac{1}{r_n}\rho_s \sigma_s^{(1)} \sigma_s^{(2)}
\end{gather*}
and $\mathcal{G}_{i}^n=\sigma(\mathcal{F}_{s+ib_n/r_n},\mathcal{S})$ for some deterministic sequence $(r_n)_{n \in \mathbb{N}} \subset \mathbb{N}$ with $r_n \rightarrow \infty$ and $r_n|\pi_n|_T/b_n \overset{\mathbb{P}}{\longrightarrow}0$. Such a sequence $(r_n)_{n \in \mathbb{N}}$ always exists because of $|\pi_n|_T/b_n \overset{\mathbb{P}}{\longrightarrow}0$. Denote $$L(n,i)=\{(j,k):\mathcal{I}_{j,n}^{(1)} \cup\mathcal{I}_{k,n}^{(2)}\subset (s+(i-1)b_n/r_n,s+ib_n/r_n]\}.$$ We then obtain using iterated expectations, the fact that $C^{(1)},C^{(2)}$ are martingales and a form of the conditional It\^o isometry for two different integrals
\begin{align*}
&\sum_{i=1}^{r_n} \mathbb{E}[\zeta_i^n|\mathcal{G}_{i-1}^n] 
\\&~~=\frac{1}{b_n}\sum_{i=1}^{r_n}
\sum_{(j,k) \in L(n,i)} \mathbb{E}[ \Delta_{(j,k),n}^{(1,2)} C^{(1)}\Delta_{(j,k),n}^{(1,2)} C^{(2)}|\mathcal{G}_{i-1}^n]- \rho_s \sigma_s^{(1)} \sigma_s^{(2)}
\\&~~=\frac{1}{b_n}\sum_{i=1}^{r_n}
 \mathbb{E}[ \int_{s+(i-1)b_n/r_n)}^{s+ib_n/r_n}(\rho_u\sigma^{(1)}_u\sigma^{(2)}_u- \rho_s \sigma_s^{(1)} \sigma_s^{(2)})du||\mathcal{G}_{i-1}^n]+O_\mathbb{P}(r_n|\pi_n|_T/b_n).
\end{align*}
As argued for \eqref{cons_spot_vola_endo_proof3} this term vanishes in probability as $n \rightarrow \infty$ because $\rho,\sigma^{(1)},\sigma^{(2)}$ are right-continuous and bounded and because of the condition on $r_n$. Note that we obtain from the Cauchy-Schwarz inequality and inequality \eqref{elem_ineq_C} that
\begin{align*}
&\mathbb{E}[(C^{(1)}_t-C^{(1)}_s)^2(C^{(2)}_t-C^{(2)}_s)^2|\sigma(\mathcal{F}_s,\mathcal{S})]
\\~~~~&\leq (\mathbb{E}[(C^{(1)}_t-C^{(1)}_s)^4|\sigma(\mathcal{F}_s,\mathcal{S})]\mathbb{E}[(C^{(2)}_t-C^{(2)}_s)^4|\sigma(\mathcal{F}_s,\mathcal{S})])^{1/2} \leq K (t-s)^2
\end{align*}
holds for any $t\geq s\geq 0$. Hence we obtain using $(a-b)^2 \leq 2a^2+2b^2$ and Lemma \ref{lemma_hayyos_L2_mart}
\begin{align*}
&\sum_{i=1}^{r_n} \mathbb{E}[(\zeta_i^n)^2|\mathcal{G}_{i-1}^n] 
\\&~~\leq \frac{2}{b_n^2}\sum_{i=1}^{r_n} \mathbb{E} \big[\big(\sum_{(j,k) \in L(n,i)} \Delta_{j,n}^{(l)} C^{(1)}\Delta_{k,n}^{(l)} C^{(2)}\mathds{1}_{\{\mathcal{I}_{j,n}^{(1)} \cap \mathcal{I}_{k,n}^{(2)} \neq \emptyset\}} \big)^2 \big| \mathcal{G}_{i-1}^n \big]+r_n\frac{K}{r_n^2}
\\&~~\leq \frac{2}{b_n^2}\sum_{i=1}^{r_n} K \frac{b_n^2}{r_n^2}+\frac{K}{r_n}
= \frac{K}{r_n}
\end{align*}
which vanishes as $r_n \rightarrow 0$ for $n \rightarrow \infty$. Lemma 2.2.12 from \cite{JacPro12} then yields
\begin{multline*}
\sum_{i=1}^{r_n} \zeta_i^n = \frac{1}{b_n}\sum_{i,j:\mathcal{I}_{i,n}^{(1)} \cup\mathcal{I}_{j,n}^{(2)}\subset (s,s+b_n]} \Delta_{i,n}^{(1)} C^{(1)}\Delta_{i,n}^{(2)} C^{(2)}\mathds{1}_{\{\mathcal{I}_{i,n}^{(1)} \cap \mathcal{I}_{j,n}^{(2)} \neq \emptyset\}}- \rho_{s}\sigma^{(1)}_{s}\sigma^{(2)}_{s}
\\+O_\mathbb{P}(r_n |\pi_n|_T/b_n) \overset{\mathbb{P}}{\longrightarrow} 0
\end{multline*}
which is equivalent to \eqref{cons_spot_kappa_exo_proof1} and therefore implies $\hat{\kappa}_n(s,+) \overset{\mathbb{P}}{\longrightarrow} \rho_{s}\sigma^{(1)}_{s}\sigma^{(2)}_{s}$.
\end{proof}

\begin{proof}[Proof of Proposition \ref{lemma_test_CoJ}] 
Let $S_p$, $p=1,\ldots,P$, denote the jump times of the $P$ biggest common jumps of $X$ on $[0,T]$. Similarly as in the proof of Proposition \ref{lemma_test_J} we define
\begin{align*}
&R(P,n,m)=4\sum_{p=1}^P \Delta_{i_p,n}^{(1)} X^{(1)} \Delta_{j_p,n}^{(2)} X^{(2)}\mathds{1}_{\{|\Delta_{i_p,n}^{(1)} X^{(1)} |> \beta |\mathcal{I}_{i_p,n}^{(1)}|^{\varpi} \}}\mathds{1}_{\{|\Delta_{j_p,n}^{(2)} X^{(2)} |> \beta |\mathcal{I}_{j_p,n}^{(2)}|^{\varpi} \}}
\\&~\times\Big[\Delta_{j_p,n}^{(2)} X^{(2)}\Big( \hat{\sigma}^{(1)}(t_{i_p,n}^{(1)},-) 
\sqrt{\widehat{\mathcal{L}}_{n,m}({\tau_{i_p,j_p,n}})}U_{n,(i_p,j_p),m}^{(1,2),-}\\&~~~~~~~~~~~~~~~~~~~~~~~~~~~~~~+\hat{\sigma}^{(1)}(t_{i_p,n}^{(1)},+)
\sqrt{\widehat{\mathcal{R}}_{n,m}({\tau_{i_p,j_p,n}})} U_{n,(i_p,j_p),m}^{(1,2),+}
\\&~~~~~~~~~~~~~~~~~~~~+\Big((\hat{\sigma}^{(1)}(t_{i_p,n}^{(1)},-))^2 (\widehat{\mathcal{L}}_{n,m}^{(1)}-\widehat{\mathcal{L}}_{n,m})(\tau_{i_p,j_p,n})
\\&~~~~~~~~~~~~~~~~~~~~~~~~~~~~~~+(\hat{\sigma}^{(1)}(t_{i_p,n}^{(1)},+))^2 (\widehat{\mathcal{R}}_{n,m}^{(1)}-\widehat{\mathcal{R}}_{n,m})({\tau_{i_p,j_p,n}}) \Big)^{1/2} U_{n,i_p,m}^{(1)}\Big)
\\&~~~~~~~~+\Delta_{i_p,n}^{(1)} X^{(1)} \Big(\hat{\sigma}^{(2)}(t_{j_p,n}^{(2)},-)\hat{\rho}(\tau_{i_p,j_p,n},-)\sqrt{\widehat{\mathcal{L}}_{n,m}({\tau_{i_p,j_p,n}}) }U_{n,(i_p,j_p),m}^{(1,2),-}
\\&~~~~~~~~~~~~~~~~~~~~~~~~~~~~~~+\hat{\sigma}^{(2)}(t_{j_p,n}^{(2)},+)\hat{\rho}(\tau_{i_p,j_p,n},+)\sqrt{\widehat{\mathcal{R}}_{n,m}(\tau_{i_p,j_p,n}) }U_{n,(i_p,j_p),m}^{(1,2),+}
\\&~~~~~~~~~~~~~~~~~~~~+\Big((\hat{\sigma}^{(2)}(t_{j_p,n}^{(2)},-))^2(1-(\hat{\rho}(\tau_{i_p,j_p,n},-))^2)\widehat{\mathcal{L}}_{n,m}(\tau_{i_p,j_p,n})
\\&~~~~~~~~~~~~~~~~~~~~~~~~~~~~~~+(\hat{\sigma}^{(2)}(t_{j_p,n}^{(2)},+))^2(1-(\hat{\rho}(\tau_{i_p,j_p,n},+))^2)\widehat{\mathcal{R}}_{n,m}({\tau_{i_p,j_p,n}}) \Big)^{1/2}U_{n,j_p,m}^{(2)}
\\&~~~~~~~~~~~~~~~~~~~~+\Big((\hat{\sigma}^{(2)}(t_{j_p,n}^{(2)},-))^2 (\widehat{\mathcal{L}}_{n,m}^{(2)}-
\widehat{\mathcal{L}}_{n,m})(\tau_{i_p,j_p,n})
\\&~~~~~~~~~~~~~~~~~~~~~~~~~~~~~~+(\hat{\sigma}^{(2)}(t_{j_p,n}^{(2)},+))^2 \widehat{\mathcal{R}}_{n,m}^{(2)}-
\widehat{\mathcal{R}}_{n,m})(\tau_{i_p,j_p,n})\Big)^{1/2}U_{n,j_p,m}^{(3)}\Big)\Big]
\end{align*}
where $(i_p,j_p)=(i_n^{(1)}(S_p),i_n^{(2)}(S_p))$ and further set
\begin{align*}
&R(P)=4\sum_{p=1}^P \Delta X^{(1)}_{S_p} \Delta X^{(2)}_{S_p}
\\&~~~~~~~~~ \times\Big[\Delta X^{(2)}_{S_p}\Big( \sigma_{S_p-}^{(1)}\sqrt{ 
\mathcal{L}({S_p})}U_{S_p}^{(1),-}+\sigma_{S_p}^{(1)}
\sqrt{\mathcal{R}({S_p})} U_{S_p}^{(1),+}
\\&~~~~~~~~~~~~~~~~~~~~~~~+\sqrt{(\sigma_{S_p-}^{(1)})^2(\mathcal{L}^{(1)}-\mathcal{L})({S_p})
+(\sigma_{S_p}^{(1)})^2(\mathcal{R}^{(1)}-\mathcal{R})({S_p}) } U_{S_p}^{(2)}\Big) 
\\&~~~~~~~~~~~+\Delta X^{(1)}_{S_p} \Big(\sigma_{S_p-}^{(2)}\rho_{S_p-}\sqrt{\mathcal{L}({S_p}) }U_{S_p}^{(1),-}+\sigma_{S_p}^{(2)}\rho_{S_p}\sqrt{\mathcal{R}({S_p}) }U_{S_p}^{(1),+}
\\&~~~~~~~~~~~~~~~~~~~~~~~+\sqrt{(\sigma_{S_p-}^{(2)})^2(1-(\rho_{S_p-})^2)\mathcal{L}({S_p})+(\sigma_{S_p}^{(2)})^2(1-(\rho_{S_p})^2)\mathcal{R}({S_p}) }U_{S_p}^{(3)}
\\&~~~~~~~~~~~~~~~~~~~~~~~+\sqrt{(\sigma_{S_p-}^{(2)})^2(\mathcal{L}^{(2)}-
\mathcal{L})({S_p})+(\sigma_{S_p}^{(2)})^2(\mathcal{R}^{(2)}-
\mathcal{R})({S_p})}U_{S_p}^{(4)}\Big)\Big].
\end{align*}

Using this notation we obtain
\begin{align}\label{CoJ_Test_Proof00}
\widetilde{\mathbb{P}} \big( \big\{\big| \frac{1}{M_n} \sum_{m=1}^{M_n} \mathds{1}_{\{R(P,n,m) \leq \Upsilon\}} -\widetilde{\mathbb{P}}(R(P)\leq \Upsilon|\mathcal{X})\big|>\varepsilon \big\} \cap \Omega_T^{(CoJ)}\big) \rightarrow 0
\end{align}
and
\begin{align*}
\widetilde{\mathbb{P}}(R(P) \leq \Upsilon |\mathcal{X}) \mathds{1}_{\Omega_T^{(CoJ)}} \overset{\mathbb{P}}{\longrightarrow} \widetilde{\mathbb{P}}(F_{2,T}^{(CoJ)} \leq \Upsilon |\mathcal{X})\mathds{1}_{\Omega_T^{(CoJ)}}
\end{align*}
as in Step 1 and Step 3 in the proof of Proposition \ref{lemma_test_J}. Here, Lemma \ref{lemma_conv_sigma_CoJ} is needed in the proof of \eqref{CoJ_Test_Proof00}.

Hence it remains to prove the equivalent to \eqref{test_theo_J_proof_A4} in Step 2 of the proof of Proposition \ref{lemma_test_J} which is
\begin{align*}
\lim_{P \rightarrow \infty} \limsup_{n \rightarrow \infty}  \widetilde{\mathbb{P}}\big(\big\{\big|\frac{1}{M_n} \sum_{m=1}^{M_n}(\mathds{1}_{\{R(P,n,m)\leq \Upsilon\}}-\mathds{1}_{\{\widehat{F}_{T,n,m}^{(CoJ)}\leq \Upsilon\}} )\big| >\varepsilon\big\}\cap \Omega_T^{(CoJ)} \big)=0,
\end{align*}
implied by
\begin{align}\label{CoJ_Test_Proof0}
\lim_{P \rightarrow \infty} \limsup_{n \rightarrow \infty} \frac{1}{M_n} \sum_{m=1}^{M_n} \widetilde{\mathbb{P}}(|R(P,n,m)-\widehat{F}_{T,n,m}^{(CoJ)}|>\varepsilon)=0
\end{align}
for all $\varepsilon>0$.

On the set $\Omega(q,P,n)$ on which the common jumps of $N(q)$ are among the $P$ largest common jumps and on which two different jumps of $N(q)$ are further apart than $2|\pi_n|_T$ it holds
\begin{align}\label{CoJ_Test_Proof1}
&\mathbb{E}\big[|R(P,n,m)-\widehat{F}_{T,n,m}^{(CoJ)}|\mathds{1}_{\Omega(q,P,n)} \big|\mathcal{F} \big] \nonumber
\\& \leq K\sum_{l=1,2} \sum_{i,j:t_{i,n}^{(l)} \wedge t_{j,n}^{(3-l)} \leq T} |\Delta_{i,n}^{(l)}(X^{(l)}-N^{(l)}(q))||\Delta_{j,n}^{(3-l)}X^{(3-l)}| \nonumber
\\&~~~~~~~~ \times \Big(|\Delta_{j,n}^{(3-l)}X^{(3-l)}| \hat{\sigma}^{(l)}(n,i)
\mathbb{E}\big[\sqrt{(\widehat{\mathcal{L}}_{n,m}^{(l)}+\widehat{\mathcal{R}}_{n,m}^{(l)})(t_{i,n}^{(l)})} \big| \mathcal{S} \big] \nonumber
\\&~~~~~~~~~~~~~+|\Delta_{i,n}^{(l)}(X^{(l)}-N^{(l)}(q))|\hat{\sigma}^{(3-l)}(n,j)
\mathbb{E}\big[\sqrt{(\widehat{\mathcal{L}}_{n,m}^{(3-l)}+\widehat{\mathcal{R}}_{n,m}^{(3-l)})(t_{j,n}^{(3-l)})} \big| \mathcal{S} \big]
 \Big) \nonumber
 \\& ~~~~~~~~\times \mathds{1}_{\{|\Delta_{i,n}^{(l)} X^{(l)} |> \beta |\mathcal{I}_{i,n}^{(l)}|^{\varpi}\wedge|\Delta_{j,n}^{(3-l)} X^{(3-l)} |> \beta |\mathcal{I}_{j,n}^{(3-l)}|^{\varpi} \}}\mathds{1}_{\{\mathcal{I}_{i,n}^{(l)}\cap \mathcal{I}_{j,n}^{(3-l)}\neq \emptyset \}}
\end{align}
with
\begin{align*}
\hat{\sigma}^{(l)}(n,i)=\big((\hat{\sigma}_n^{(l)}(t_{i,n}^{(l)},-))^2+(\hat{\sigma}_n^{(l)}(t_{i,n}^{(l)},+))^2\big)^{1/2}.
\end{align*}
Because of $\mathbb{P}(\Omega(q,P,n))\rightarrow \infty$ as $n ,P \rightarrow \infty$ for all $q>0$ it suffices to show that \eqref{CoJ_Test_Proof1} vanishes as first $n \rightarrow \infty$ and then $q \rightarrow \infty$ for proving \eqref{CoJ_Test_Proof0}. 

Reconsidering the notation introduced in Step 2 in the proof of \eqref{proof_theo_clt_CoJ_step1} we define the following terms
\begin{align*}
&Y_{(i,j),n}^{(l)}=\Big(\frac{1}{b_n} \sum_{k \neq i:\mathcal{I}_{k,n}^{(l)} \subset [t_{i,n}^{(l)}-b_n,t_{i,n}^{(l)}+b_n]} (\Delta_{(k,j),(1,1),n}^{(l,3-l)} X^{(l)} )^2  \Big)^{1/2},
\\& \widetilde{Y}_{(i,j),n}^{(l)}=\Big(\frac{1}{b_n} \sum_{k \neq i:\mathcal{I}_{k,n}^{(l)} \subset [t_{i,n}^{(l)}-b_n,t_{i,n}^{(l)}+b_n]} (\Delta_{(k,j),(1,1),n}^{(l\setminus 3-l)}, X^{(l)} )^2 \Big)^{1/2},
\end{align*}
$l=1,2$. Then the Minkowski inequality yields
\begin{align}\label{CoJ_Test_Proof2}
\begin{split}
 \hat{\sigma}^{(l)}(n,i) &\leq Y_{(i,j),n}^{(l)} +\widetilde{Y}_{(i,j),n}^{(l)}, 
 \\ \hat{\sigma}^{(3-l)}(n,j) &\leq Y_{(j,i),n}^{(3-l)} +\widetilde{Y}_{(j,i),n}^{(3-l)}
\end{split}
\end{align}
which allows to treat the increments in the estimation of $\sigma^{(l)},\sigma^{(3-l)}$ over intervals which do overlap with $\mathcal{I}_{j,n}^{(3-l)},\mathcal{I}_{i,n}^{(l)}$ and those which do not, separately.

We apply the inequalities \eqref{CoJ_Test_Proof2} to get an upper bound for the $\mathcal{S}$-conditional expectation of \eqref{CoJ_Test_Proof1}. First we get using the Cauchy-Schwarz inequality, iterated expectations and Lemma \ref{elem_ineq} the following bound for \eqref{CoJ_Test_Proof1} where we replaced the estimators for $\sigma^{(l)},\sigma^{(3-l)}$ with the first summands from \eqref{CoJ_Test_Proof2} 
\begin{align}\label{CoJ_Test_Proof3}
&K\sum_{i,j:t_{i,n}^{(l)} \wedge t_{j,n}^{(3-l)} \leq T} \mathbb{E}\big[|\Delta_{i,n}^{(l)}(X^{(l)}-N^{(l)}(q))| |\Delta_{j,n}^{(3-l)} X^{(3-l)}|^2 Y_{(i,j),n}^{(l)} 
\nonumber
\\& ~~~~~~+ |\Delta_{j,n}^{(3-l)} X^{(3-l)}||\Delta_{i,n}^{(l)}(X^{(l)}-N^{(l)}(q))|^2
Y_{(j,i),n}^{(3-l)}\big| \mathcal{S}\big] \sqrt{2 n |\pi_n|_T}
 \mathds{1}_{\{\mathcal{I}_{i,n}^{(l)}\cap \mathcal{I}_{j,n}^{(3-l)}\neq \emptyset \}}\nonumber
\\& ~\leq K  \sum_{i,j:t_{i,n}^{(l)} \wedge t_{j,n}^{(3-l)} \leq T} \big[\big((K_q|\mathcal{I}_{i,n}^{(l)}|^2+|\mathcal{I}_{i,n}^{(l)}|+e_q|\mathcal{I}_{i,n}^{(l)}|)(|\mathcal{I}_{j,n}^{(3-l)}|/b_n)\big)^{1/2} \nonumber
\big(K|\mathcal{I}_{j,n}^{(3-l)}| 
\big)^{1/2}
\\& ~~~~~~+\big((K|\mathcal{I}_{j,n}^{(3-l)}|)(|\mathcal{I}_{i,n}^{(l)}|/b_n)\big)^{1/2} \big(K_q|\mathcal{I}_{i,n}^{(l)}|^4+|\mathcal{I}_{i,n}^{(l)}|^2+e_q|\mathcal{I}_{i,n}^{(l)}|\big)^{1/2}\big] \nonumber
\\& ~~~~~~~~~~~~~~~~~~\times\sqrt{n |\pi_n|_T} \mathds{1}_{\{\mathcal{I}_{i,n}^{(l)}\cap \mathcal{I}_{j,n}^{(3-l)}\neq \emptyset \}}\nonumber
\\&~ \leq K(K_q(|\pi_n|_T+(|\pi_n|_T)^3)+1+e_q)^{1/2}(|\pi_n|_T/b_n)^{1/2} \nonumber
\\&~~~~~~\times \sqrt{n} \sum_{i,j:t_{i,n}^{(l)} \wedge t_{j,n}^{(3-l)} \leq T} \big(|\mathcal{I}_{i,n}^{(l)}|^{1/2} |\mathcal{I}_{j,n}^{(3-l)}|+|\mathcal{I}_{i,n}^{(l)}| |\mathcal{I}_{j,n}^{(3-l)}|^{1/2}\big)
 \mathds{1}_{\{\mathcal{I}_{i,n}^{(l)}\cap \mathcal{I}_{j,n}^{(3-l)}\neq \emptyset \}}
\end{align}
for $l=1,2$ which vanishes as $n \rightarrow \infty$ for any $q >0$ because of Condition \ref{cond_testproc_CoJ}(ii) and ${|\pi_n|_T/b_n\overset{\mathbb{P}}{\longrightarrow} 0}$.

Further we obtain by treating increments over $\mathcal{I}_{i,n}^{(l)}\cap \mathcal{I}_{j,n}^{(3-l)}$ and increments over non-overlapping parts of $\mathcal{I}_{i,n}^{(l)}, \mathcal{I}_{j,n}^{(3-l)}$ differently the following bound for the $\mathcal{S}$-conditional expectation of \eqref{CoJ_Test_Proof1} where we replaced the estimators for $\sigma^{(l)},\sigma^{(3-l)}$ with the second summands from \eqref{CoJ_Test_Proof2} 
\begin{align}\label{CoJ_Test_Proof4}
&K  \sum_{i,j:t_{i,n}^{(l)} \wedge t_{j,n}^{(3-l)} \leq T} \big(\big[
\mathbb{E}[|\Delta_{(i,j),(1,1),n}^{(l\setminus 3-l)}(X^{(l)}-N^{(l)}(q))| |\Delta_{j,n}^{(3-l)} X^{(3-l)}|^2 \widetilde{Y}_{(i,j),n}^{(l)}\nonumber
\\&~~~~~~~~~~~~~~~+ |\Delta_{(i,j),(1,1),n}^{(l,3-l)}(X^{(l)}-N^{(l)}(q))| |\Delta_{(j,i),(1,1),n}^{(3-l\setminus l)} X^{(3-l)}|^2 \widetilde{Y}_{(j,i),n}^{(l)}
| \mathcal{S}] \nonumber
\\&~~+ \mathbb{E} [ |\Delta_{(j,i),(1,1),n}^{(3-l \setminus l)} X^{(3-l)}||\Delta_{i,n}^{(l)}(X^{(l)}-N^{(l)}(q))|^2 \widetilde{Y}_{(i,j),n}^{(3-l)}\nonumber
\\&~~~~~~~~~~~~~~~+|\Delta_{(j,i),(1,1),n}^{(3-l,l )} X^{(3-l)}||\Delta_{(i,j),(1,1),n}^{(l\setminus 3-l)}(X^{(l)}-N^{(l)}(q))|^2\widetilde{Y}_{(j,i),n}^{(3-l)} |\mathcal{S}]  \big] 
\sqrt{n |\pi_n|_T}
\nonumber
\\& ~~+\mathbb{E} [ |\Delta_{(i,j),(1,1),n}^{(l,3-l)} (X^{(l)}-N(q)^{(l)})||\Delta_{(i,j),(1,1),n}^{(l,3-l)}X^{(3-l)}|^2 \widetilde{Y}_{(i,j),n}^{(l)}| \mathcal{S} ] \nonumber
\\&~~~~~~~~~~~~~~~~~~~~~~~~~~~~~~~~~~ \times\mathbb{E}[\sqrt{(\widehat{\mathcal{L}}_{n,m}^{(l)} \nonumber
+\widehat{\mathcal{R}}_{n,m}^{(l)})(t_{i,n}^{(l)}\wedge t_{j,n}^{(3-l)})}| \mathcal{S} ]
\nonumber
\\& ~~+ \mathbb{E} [|\Delta_{(i,j),(1,1),n}^{(l,3-l)} X^{(3-l)}||\Delta_{(i,j),(1,1),n}^{(l,3-l)}(X^{(l)}-N^{(l)}(q)|^2
\big)\widetilde{Y}_{(j,i),n}^{(3-l)}| \mathcal{S}] \nonumber
\\& ~~~~~~~~~~~~~~~~~~~~~~~~~~~~~~~~~~ \times\mathbb{E}[\sqrt{(\widehat{\mathcal{L}}_{n,m}^{(3-l)}
+\widehat{\mathcal{R}}_{n,m}^{(3-l)})(t_{i,n}^{(l)}\wedge t_{j,n}^{(3-l)})}| \mathcal{S} ]\big)
  \mathds{1}_{\{\mathcal{I}_{i,n}^{(l)}\cap \mathcal{I}_{j,n}^{(3-l)}\neq \emptyset \}} \nonumber
\\& \leq K  \sum_{i,j:t_{i,n}^{(l)} \wedge t_{j,n}^{(3-l)} \leq T}
\big(\big[(K_q|\mathcal{I}_{i,n}^{(l)}|+1+e_q)|\mathcal{I}_{i,n}^{(l)}|)^{1/2} 
|\mathcal{I}_{j,n}^{(3-l)}| \nonumber
\\&~~~~~~+(K_q|\mathcal{I}_{i,n}^{(3-l)}|+1+e_q)|\mathcal{I}_{i,n}^{(3-l)}|)
|\mathcal{I}_{j,n}^{(l)}|^{1/2}\big] \sqrt{n |\pi_n|_T}
\mathds{1}_{\{\mathcal{I}_{i,n}^{(1)}\cap \mathcal{I}_{j,n}^{(2)}\neq \emptyset \}}\nonumber
\\&~~~+\big[((K_q|\mathcal{I}_{i,n}^{(l)} \cap \mathcal{I}_{j,n}^{(3-l)}|^2+|\mathcal{I}_{i,n}^{(l)} \cap \mathcal{I}_{j,n}^{(3-l)}|^{1/2}+e_q)|\mathcal{I}_{i,n}^{(l)} \cap \mathcal{I}_{j,n}^{(3-l)}|)^{1/3} |\mathcal{I}_{i,n}^{(l)} \cap \mathcal{I}_{j,n}^{(3-l)}|^{2/3} \nonumber
\\& ~~~~~~~~~~~~~~~~~~~~\times\mathbb{E}[\sqrt{(\widehat{\mathcal{L}}_{n,m}^{(l)}
+\widehat{\mathcal{R}}_{n,m}^{(l)})(t_{i,n}^{(l)}\wedge t_{j,n}^{(3-l)})}| \mathcal{S} ] \nonumber
\\&~~~~~~+ |\mathcal{I}_{i,n}^{(l)} \cap \mathcal{I}_{j,n}^{(3-l)}|^{1/2}((K_q|\mathcal{I}_{i,n}^{(l)} \cap \mathcal{I}_{j,n}^{(3-l)}|^3+|\mathcal{I}_{i,n}^{(l)} \cap \mathcal{I}_{j,n}^{(3-l)}|+e_q)|\mathcal{I}_{i,n}^{(l)} \cap \mathcal{I}_{j,n}^{(3-l)}|)^{1/2}\nonumber
\\& ~~~~~~~~~~~~~~~~~~~~\times\mathbb{E}[\sqrt{(\widehat{\mathcal{L}}_{n,m}^{(3-l)}
+\widehat{\mathcal{R}}_{n,m}^{(3-l)})(t_{i,n}^{(l)}\wedge t_{j,n}^{(3-l)})}| \mathcal{S} ] \big]
\big).
\end{align}
Here, we used iterated expectations and
\begin{align*}
\mathbb{E}[\widetilde{Y}_{(i,j),n}^{(l)}|\mathcal{S}] \leq \Big( \frac{1}{b_n} \sum_{k \neq i:\mathcal{I}_{k,n}^{(l)}\subset[t_{i,n}^{(l)}-b_n,t_{i,n}^{(l)}+b_n]} K|\mathcal{I}_{k,n}^{(l)}|\Big)^{1/2} \leq K
\end{align*}
along with the similar bound for $\mathbb{E}[\widetilde{Y}_{(j,i),n}^{(3-l)}|\mathcal{S}]$.

The sum over the expression in the first set of square brackets in \eqref{CoJ_Test_Proof4} vanishes similarly as in \eqref{CoJ_Test_Proof3} while the sum over the expression in the second set of square brackets is bounded by
\begin{align}\label{CoJ_Test_Proof5}
&K (K_q(|\pi_n|_T)^{1/2}+e_q)^{1/3}\sum_{i,j:t_{i,n}^{(l)} \wedge t_{j,n}^{(3-l)} \leq T} |\mathcal{I}_{i,n}^{(l)} \cap \mathcal{I}_{j,n}^{(3-l)}| 
 \nonumber
\\&~~~~~~~~\times \sum_{l_1,l_2=-L_n}^{L_n} |\mathcal{I}_{i+l_1,n}^{(l)} \cap \mathcal{I}_{j+l_2,n}^{(3-l)}|\Big(\sum_{k_1,k_2=-L_n}^{L_n} |\mathcal{I}_{i+k_1,n}^{(l)} \cap \mathcal{I}_{j+k_2,n}^{(3-l)}| \Big)^{-1} \nonumber
\\&~~~~~~~~\times((n|\mathcal{I}_{i+l_1-1,n}^{(l)}|+n|\mathcal{I}_{i+l_1+1,n}^{(l)}|)^{1/2}+(n|\mathcal{I}_{j+l_2-1,n}^{(3-l)}|+n|\mathcal{I}_{j+l_2+1,n}^{(3-l)}|)^{1/2}) \nonumber
\\&~=K (K_q(|\pi_n|_T)^{1/2}+e_q)\sum_{i,j:t_{i,n}^{(l)} \wedge t_{j,n}^{(3-l)} \leq T} |\mathcal{I}_{i,n}^{(l)} \cap \mathcal{I}_{j,n}^{(3-l)}| \nonumber
\\&~~~~~~~~\times((n|\mathcal{I}_{i-1,n}^{(l)}|+n|\mathcal{I}_{i+1,n}^{(l)}|)^{1/2}+(n|\mathcal{I}_{j-1,n}^{(3-l)}|+n|\mathcal{I}_{j+1,n}^{(3-l)}|)^{1/2}) \nonumber
\\&~~~~~~~~\times \sum_{l_1,l_2=-L_n}^{L_n} |\mathcal{I}_{i+l_1,n}^{(l)} \cap \mathcal{I}_{j+l_2,n}^{(3-l)}|
\Big(\sum_{k_1,k_2=-L_n}^{L_n} |\mathcal{I}_{i+l_1+k_1,n}^{(l)} \cap \mathcal{I}_{j+l_2+k_2,n}^{(3-l)}| \Big)^{-1}\nonumber
\\&~~~~~~+O_\mathbb{P}(\sqrt{n}|\pi_n|_T) 
\end{align}
where the sum in the last line is less or equal than $4$ which can be shown similarly to \eqref{index_shift_trick}. The $O_\mathbb{P}(\sqrt{n}|\pi_n|_T) $-term is due to boundary effects. Hence \eqref{CoJ_Test_Proof5} is bounded by
\begin{gather*}
~~K (K_q(|\pi_n|_T)^{1/2}+e_q)\sqrt{n}\Big(\sum_{i\geq 2:t_{i,n}^{(l)}\leq T} |\mathcal{I}_{i,2,n}^{(l)}|^{3/2} +\sum_{j\geq 2:t_{j,n}^{(3-l)}\leq T} |\mathcal{I}_{j,2,n}^{(3-l)}|^{3/2} \Big)+O_\mathbb{P}(\sqrt{n}|\pi_n|_T) ~~
\end{gather*}
which vanishes as $n,q \rightarrow \infty$ by Condition \ref{cond_cons_CoJ2}(ii); compare also \eqref{bound_prob_square}.
\end{proof}

\subsection{Proofs regarding the Poisson setting}\label{sec:proof_detail_poiss}

\begin{lemma}\label{obs_sch1}
Consider an observation scheme given by $t_{0,n}^{(l)}=0$ and $$t_{i,n}^{(l)}=t_{i-1,n}^{(l)}+\alpha^{(l)}(n)^{-1} E_{i,n}^{(l)},~i \geq1,$$ where the $E^{(l)}_{i,n}$, $i,n \in \mathbb{N}$, are i.i.d.\ random variables with values in $\mathbb{R}_{\geq 0}$ and $\alpha^{(l)}(n)$ are positive functions with $\alpha^{(l)}(n) \rightarrow \infty$ as $n \rightarrow \infty$. 

If $E_{1,1}^{(l)} \in L^2(\Omega)$ it holds
\begin{align*}
\alpha^{(l)}(n)^{-1}\sum_{i: t_{i,n}^{(l)} \leq T} g\big(\alpha^{(l)}(n)\big|\mathcal{I}_{i,n}^{(l)}\big|\big) \overset{\mathbb{P}}{\longrightarrow} \frac{\mathbb{E}[g(E^{(l)}_{1,1})]}{\mathbb{E}[E^{(l)}_{1,1}]} T
\end{align*}
for all functions $g:\mathbb{R}_{\geq 0} \rightarrow \mathbb{R}$ with $g(E_{1,1}^{(l)}) \in L^1(\Omega)$.
\end{lemma}

\begin{proof} This proof is based on the proof of Proposition 1 in \cite{HayYos08}. Set $m^{(l)}_1=\mathbb{E}[E^{(l)}_{1,1}]$, $m^{(l)}_g=\mathbb{E}[g(E^{(l)}_{1,1})]$ and
\begin{align*}
\lambda^{(l)}(n)=\lceil \alpha^{(l)}(n)T/m_1^{(l)}\rceil.
\end{align*}
Further define
\begin{align*}
N^{(l)}_T(n)=\sum_{i\in \mathbb{N} } \mathds{1}_{\{t_{i,n}^{(l)}\leq T\}}.
\end{align*}

\textit{Step 1.} We prove
\begin{align}\label{obs_sch1_proof1}
\mathbb{P}( |N^{(l)}_T(n)-\lambda^{(l)}(n)|>\alpha^{(l)}(n)^{1/2+ \varepsilon})\rightarrow 0
\end{align}
as $n \rightarrow \infty$ for $\varepsilon \in (0,1/2)$; compare Lemma 9 in \cite{HayYos08}. Denote $\nu^{(l)}=\Var(E_{1,1}^{(l)})$. The case $\nu^{(l)}=0$ is simple, so let us assume $\nu^{(l)}>0$ in the following. It holds
\begin{align*}
&N^{(l)}_T(n)-\lambda^{(l)}(n)>\alpha^{(l)}(n)^{1/2+ \varepsilon}
\Leftrightarrow \sum_{i=1}^{\lambda^{(l)}(n)+\alpha^{(l)}(n)^{1/2+ \varepsilon}} |\mathcal{I}_{i,n}^{(l)}|< T
\\& \Leftrightarrow \frac{\sum_{i=1}^{\lambda^{(l)}(n)+\alpha^{(l)}(n)^{1/2+ \varepsilon}} (E_{i,n}^{(l)}-m_1^{(l)})}{\sqrt{(\lambda^{(l)}(n)+\alpha^{(l)}(n)^{1/2+ \varepsilon})\nu^{(l)}}} < \frac{\alpha^{(l)}(n)T-(\lambda^{(l)}(n)+\alpha^{(l)}(n)^{1/2+ \varepsilon})m_1}{\sqrt{(\lambda^{(l)}(n)+\alpha^{(l)}(n)^{1/2+ \varepsilon})\nu^{(l)}}} 
\end{align*}
where the left hand side is approximately standard normal distributed by the classical central limit theorem while the right hand side is approximately
\begin{align*}
\frac{-\alpha^{(l)}(n)^{1/2+ \varepsilon}m_1}{\sqrt{(\lambda^{(l)}(n)+\alpha^{(l)}(n)^{1/2+ \varepsilon})\nu^{(l)}}} =O\big(\frac{-\alpha^{(l)}(n)^{1/2+ \varepsilon}}{\alpha^{(l)}(n)^{1/2 \vee (1/2+ \varepsilon)/2}} \Big) \leq O(-\alpha^{(l)}(n)^{\varepsilon/2})
\end{align*}
which tends to $- \infty$ as $n \rightarrow \infty$. Hence 
\begin{align}\label{obs_sch1_proof3}
\mathbb{P}( N^{(l)}_T(n)-\lambda^{(l)}(n)>\alpha^{(l)}(n)^{1/2+ \varepsilon})\rightarrow 0.
\end{align}
Analogously it follows
\begin{align*}
\mathbb{P}( N^{(l)}_T(n)-\lambda^{(l)}(n) < -\alpha^{(l)}(n)^{1/2+ \varepsilon})\rightarrow 0
\end{align*}
which together with \eqref{obs_sch1_proof3} yields \eqref{obs_sch1_proof1}.

\textit{Step 2.} By the weak law of large numbers it holds
\begin{align*}
\alpha^{(l)}(n)^{-1}\sum_{i=1}^{\lambda^{(l)}(n)} g\big(\alpha^{(l)}(n)\big|\mathcal{I}_{i,n}^{(l)}\big|\big)=\frac{\lambda^{(l)}(n)}{\alpha^{(l)}(n)} \frac{1}{\lambda^{(l)}(n)} \sum_{i=1}^{\lambda^{(l)}(n)} g\big(E_{i,n}^{(l)}\big) \overset{\mathbb{P}}{\longrightarrow}\frac{T}{m_1^{(l)}} m_g^{(l)}.
\end{align*}
Hence it remains to prove
\begin{align}\label{obs_sch1_proof2}
\alpha^{(l)}(n)^{-1}\sum_{i=\lambda^{(l)}(n)\wedge N^{(l)}_T(n)+1}^{\lambda^{(l)}(n)\vee N^{(l)}_T(n)} g\big(\alpha^{(l)}(n)\big|\mathcal{I}_{i,n}^{(l)}\big|\big) \overset{\mathbb{P}}{\longrightarrow} 0
\end{align}
as $n \rightarrow \infty$. Because of \eqref{obs_sch1_proof1} it suffices to prove the convergence \eqref{obs_sch1_proof2} restricted to the set 
$$\Omega^{(l)}(n,\varepsilon)=\{ |N^{(l)}_T(n)-\lambda^{(l)}(n)|\leq \alpha^{(l)}(n)^{1/2+ \varepsilon}\}.$$
On this set \eqref{obs_sch1_proof2} is bounded by
\begin{align*}
2\alpha^{(l)}(n)^{(1/2+\varepsilon)-1} \Big(\frac{1}{2 \alpha^{(l)}(n)^{1/2+ \varepsilon}}\sum_{i=\lambda^{(l)}(n)-\lfloor\alpha^{(l)}(n)^{1/2+ \varepsilon}\rfloor+1}^{\lambda^{(l)}(n)+\lceil\alpha^{(l)}(n)^{1/2+ \varepsilon}\rceil} \big|g\big(E^{(l)}_{i,n}\big)\big| \Big) \mathds{1}_{\Omega^{(l)}(n,\varepsilon)}
\end{align*}
which converges in probability to zero for $\varepsilon< 1/2$ as $n \rightarrow \infty$ because the term in parantheses converges by the weak law of large numbers in probability to $m^{(l)}_g$.
\end{proof}

\begin{proof}[Proof of \eqref{lln_gamma}]
Because of $|\mathcal{I}_{i,n}| \sim Exp(n \lambda)$ it holds $|\mathcal{I}_{i,k,n}| \sim \Gamma(k,n \lambda)$. Lemma \ref{obs_sch1} with $\alpha^{(l)}(n)=n/k$ and $E_{1,1}^{(l)} \sim \zeta/k$ with $\zeta \sim \Gamma(k, \lambda)$ yields
\begin{align*}
&G_{n,k}(t)=\frac{1}{k}\sum_{l=0}^{k-1}\frac{n}{k}\sum_{i:t_{ki+l,n}  \leq t} \left|\mathcal{I}_{ki+l,k,n} \right|^2
=\frac{1}{k}\sum_{l=0}^{k-1}\frac{k}{n}\sum_{i:t_{ki+l,n} \leq t} \big(\frac{n}{k}\left|\mathcal{I}_{ki+l,k,n} \right|\big)^2
\\&~~~~~~~~~~\overset{\mathbb{P}}{\longrightarrow}\frac{1}{k}\sum_{l=0}^{k-1} \frac{\mathbb{E}[(\zeta/k)^2]}{\mathbb{E}[\zeta/k]}t=\frac{k(k+1)/(k\lambda)^2}{k/(k\lambda)}t=\frac{k+1}{k\lambda}t.
\end{align*}
\end{proof}

\begin{lemma}\label{G_k_H_k_cons}
Under the Poisson sampling introduced in Example \ref{example_poisson_2d} we have the pointwise convergences
\begin{align*}
&\widetilde{G}_{k,n}(t)\overset{\mathbb{P}}{\longrightarrow} \widetilde{G}_k(t),
\\&H_{k,n}(t)\overset{\mathbb{P}}{\longrightarrow} H_k(t),
\end{align*}
for functions $\widetilde{G}_k,H_k$ which are linear in $t$ with positive slope.
\end{lemma}

\begin{proof}
We will proof the claim for $H_{k,n}$ only, since the proof for $\widetilde{G}_{k,n}$ is similar. First observe
\begin{align*}
H_{k,n}(t)&=\frac{n}{k^3} \sum_{i,j:t_{i,n}^{(1)} \wedge t_{j,n}^{(2)} \leq t} |\mathcal{I}_{i,k,n}^{(1)}|| \mathcal{I}_{j,k,n}^{(2)}| \mathds{1}_{\{\mathcal{I}_{i,k,n}^{(1)} \cap \mathcal{I}_{j,k,n}^{(2)} \neq \emptyset\}}
\\ & \overset{\mathcal{L}}{=}\frac{1}{k^3}\frac{1}{n} \sum_{i,j:t_{i,1}^{(1)} \wedge t_{j,1}^{(2)} \leq nt} |\mathcal{I}_{i,k,1}^{(1)}|| \mathcal{I}_{j,k,1}^{(2)}| \mathds{1}_{\{\mathcal{I}_{i,k,1}^{(1)} \cap \mathcal{I}_{j,k,1}^{(2)} \neq \emptyset\}}
\\ & =\frac{1}{k^3} \frac{\lfloor nt \rfloor}{n} \frac{1}{\lfloor nt \rfloor} \sum_{m=1}^{\lfloor nt \rfloor} Y_m+O_\mathbb{P}(n^{-1})
\end{align*}
where $\overset{\mathcal{L}}{=}$ denotes equality in law and
\begin{align*}
Y_m=\sum_{i,j:t_{i-k,1}^{(1)} \wedge t_{j-k,1}^{(2)} \in (m-1,m]} |\mathcal{I}_{i,k,1}^{(1)}|| \mathcal{I}_{j,k,1}^{(2)}| \mathds{1}_{\{\mathcal{I}_{i,k,1}^{(1)} \cap \mathcal{I}_{j,k,1}^{(2)} \neq \emptyset\}},~~ m \in \mathbb{N}.
\end{align*}
Because the Poisson process has stationary increments, the sequence $Y_m$, $m \in \mathbb{N}$, is a stationary square integrable time series. Because further $Y_{m_1}$, $Y_{m_2}$ become asymptotically independent as $|m_1-m_2| \rightarrow \infty$, compare (A.51) in \cite{MarVet17}, it is possible to conclude $\Cov(Y_{m_1},Y_{m_2}) \rightarrow 0$ as $|m_1-m_2| \rightarrow \infty$. We then obtain by a law of large numbers for stationary processes, compare Theorem 7.1.1 in \cite{BroDav91},
\begin{align*}
H_{k,n}(t)\overset{\mathbb{P}}{\longrightarrow}\frac{t}{k^3} \mathbb{E}[Y_1].
\end{align*}
This yields the claim as clearly $\mathbb{E}[Y_1]>0$.
\end{proof}

\subsubsection{Proof of Condition \ref{cond_clt_J}(ii), \ref{cond_cons_CoJ2}(iii) and \ref{cond_clt_coJ} in the Poisson setting}\label{sec:proof_poiss_clt}

The convergences claimed in Condition \ref{cond_clt_J}(ii), \ref{cond_cons_CoJ2}(iii) and \ref{cond_clt_coJ} are all special cases of the general statement made in the following Lemma and can be easily verified by finding suitable functions $f^{(1)},f^{(2)},f$.

\begin{lemma}\label{lemma_poiss_cond_clt}
Let $d \in \mathbb{N}$ and $Z_n^{(1)}(s),Z_n^{(2)}(s),Z_n^{(3)}(s)$ be $\mathbb{R}^d$-valued random variables, which can be written as
\begin{align*}
&Z^{(l)}_n(s)=f^{(l)}\big(n(s-t^{(l)}_{i_n^{(l)}(s)-1,n}),n(t^{(l)}_{i_n^{(l)}(s),n}-s),(\sqrt{n}\Delta_{i_n^{(l)}(s)+j,n}\overline{W}^{(l)})_{j \in [k-1]},
\\&~~~~~(\sqrt{n}\Delta_{i_n^{(l)}(s)+j,k,n}^{(l)}\overline{W}^{(l)}\mathds{1}_{\{\mathcal{I}_{i_n^{(3-l)}(s)+i,k,n}^{(3-l)} \cap \mathcal{I}_{i_n^{(l)}(s)+j,k,n}^{(l)} \neq \emptyset\}} )_{i \in \{0,\ldots,k-1\}, j \in \mathbb{Z}}\big),\quad l=1,2,
\\&Z_n^{(3)}(s)=f\big(\big[(n|\mathcal{I}_{i_n^{(l)}(s)+i,n}^{(l)}|)_{i\in [k]}\big]_{l=1,2},
(n|\mathcal{I}_{i_n^{(1)}(s)+i,n}^{(1)} \cap \mathcal{I}_{i_n^{(2)}(s)+j,n}^{(2)}|)_{i,j \in [k]} \big)
\end{align*}
with $[k]=\{-k,\ldots,k\}$ for measurable functions $f^{(1)},f^{(2)},f$ and a fixed $k \in \mathbb{N}$. 

Then under the Poisson sampling introduced in Example \ref{example_poisson_2d} the integral
\begin{align}\label{lemma_clt1}
&\int_{[0,T]^{P_1+P_2+P_3}} g(x_1,\ldots,x_{P_1},x'_1,\ldots,x'_{P_2},x''_1,\ldots,x''_{P_3}) \mathbb{E} \Big[\prod_{p=1}^{P_1} h_p^{(1)}(Z_n^{(1)}(x_p)) \nonumber
\\& ~~~~~~\times \prod_{p=1}^{P_2} h_p^{(2)}(Z_n^{(2)}(x'_p))\prod_{p=1}^{P_3} h_p^{(3)}(Z_n^{(3)}(x''_p)) \Big]
d x_1 \ldots d x_{P_1} d x'_1 \ldots d x'_{P_2} d x''_1 \ldots d x''_{P_3}
\end{align}
converges for $n \rightarrow \infty$ to
\begin{align}\label{lemma_clt2}
&\int_{[0,T]^{P_1+P_2+P_3}} g(x_1,\ldots,x_{P_1},x'_1,\ldots,x'_{P_2},x''_1,\ldots,x''_{P_3})\prod_{p=1}^{P_1} \int h_p^{(1)}(y)\Gamma^{(1)}(dy) \nonumber
\\ &~~~~\times \prod_{p=1}^{P_2}\int h_p^{(2)}(y)\Gamma^{(2)}(dy)\prod_{p=1}^{P_3} \int h_p^{(3)}(y)\Gamma^{(3)}(dy) 
d x_1 \ldots d x_{P_1} d x'_1 \ldots d x'_{P_2} d x''_1 \ldots d x''_{P_3}
\end{align}
for all bounded continuous functions $g:\mathbb{R}^{P_1+P_2+P_3} \rightarrow \mathbb{R}$, $h_p^{(l)} : \mathbb{R}^d \rightarrow \mathbb{R}$ and all $P_1,P_2,P_3 \in \mathbb{N}$.
\end{lemma}

Note that $\Gamma^{(1)},\Gamma^{(2)},\Gamma^{(3)}$ are probability measures on $\mathbb{R}^d$ which do not depend on $x_p,x_p',x_p''$ as in the general case in Conditions \ref{cond_clt_J}(ii), \ref{cond_cons_CoJ2}(iii) and \ref{cond_clt_coJ} because the Poisson process has stationary increments.

\begin{proof} First we show that any two random variables $Z_n^{(l_1)}(x_1)$ and $Z_n^{(l_2)}(x_2)$ with $l_1,l_2 \in \{1,2,3\}$, $x_1,x_2 \in [0,T]$, become asymptotically independent for $x_1 \neq x_2$ which yields that the expectation in \eqref{lemma_clt1} factorizes in the limit. Let $x_1 <x_2$ and define
\begin{align*}
&\Omega_n^-(l_1,x_1)= \big\{t^{(l_1)}_{i_n^{(l_1)}(t^{(3-l_1)}_{i^{(3-l_1)}_n(x_1)+k})+k,n}
\leq \frac{x_1+x_2}{2}
 \big\},
 \\&\Omega_n^+(l_2,x_2)= \big\{ \frac{x_1+x_2}{2}
\leq t^{(l_2)}_{i_n^{(l_2)}(t^{(3-l_2)}_{i^{(3-l_2)}_n(x_2)-k-1})-k-1,n} 
 \big\},
 \\&\Omega_n(l_1,l_2,x_1,x_2)=\Omega_n^-(l_1,x_1) \cap \Omega_n^+(l_2,x_2).
\end{align*}
Here, $\Omega_n(l_1,l_2,x_1,x_2)$ describes the subset of $\Omega$ on which the set of intervals used for the construction of $Z_n^{(l_1)}(x_1)$ and the set of intervals used for the construction of $Z_n^{(l_2)}(x_2)$ are separated by $(x_1+x_2)/2$. Then there exist measurable functions $g_1,g_2$ such that
\begin{align*}
&Z_n^{(l_1)}(x_1) \mathds{1}_{\Omega_n^-(l_1,x_1)}=g_1((N_n(t))_{t \in [0,(x_1+x_2)/2]},(\overline{W}^{(l_1)}_t)_{t \in [0,(x_1+x_2)/2]}) \mathds{1}_{\Omega_n^-(l_1,x_1)} , 
\\&Z_n^{(l_2)}(x_2) \mathds{1}_{\Omega_n^+(l_2,x_2)}=g_2((N_n(t)-N_n((x_1+x_2)/2))_{t \in [(x_1+x_2)/2,\infty)},
\\&~~~~~~~~~~~~~~~~~~~~~~~~~~~~~~~~~~~~~~~~~~~~~(\overline{W}^{(l_2)}_t-\overline{W}^{(l_2)}_{(x_1+x_2)/2})_{t \in [(x_1+x_2)/2,\infty)}) \mathds{1}_{\Omega_n^+(l_2,x_2)}
\end{align*}
where $N_n(t)=(N_n^{(1)}(t),N_n^{(2)}(t))^*$, $N_n^{(l)}(t)=\sum_{i \in \mathbb{N}} \mathds{1}_{\{t_{i,n}^{(l)} \leq t\}}$, $l=1,2$, denotes the Poisson processes which create the stopping times. These identities yield that the random variables $Z_n^{(l_1)}(x_1) \mathds{1}_{\Omega_n^-(l_1,x_1)}$, $\mathds{1}_{\Omega_n^-(l_1,x_1)}$ are independent from the random variables $Z_n^{(l_2)}(x_2) \mathds{1}_{\Omega_n^+(l_2,x_2)}$, $\mathds{1}_{\Omega_n^+(l_2,x_2)}$ because the processes $\overline{W}$ and $N_n(t)$, have independent increments. Hence we get
\begin{align*}
&\mathbb{E}[h^{(l_1)}_{p_1}(Z_n^{(l_1)}(x_1))h^{(l_2)}_{p_2}(Z_n^{(l_2)}(x_2))\mathds{1}_{\Omega_n^-(l_1,x_1)}\mathds{1}_{\Omega_n^+(l_2,x_2)} ]
\\&~~=\mathbb{E}[h^{(l_1)}_{p_1}(Z_n^{(l_1)}(x_1))\mathds{1}_{\Omega_n^-(l_1,x_1)}]\mathbb{E}[h^{(l_2)}_{p_2}(Z_n^{(l_2)}(x_2))\mathds{1}_{\Omega_n^+(l_2,x_2)} ]
\\&~~=\frac{\mathbb{E}[h^{(l_1)}_{p_1}(Z_n^{(l_1)}(x_1))\mathds{1}_{\Omega_n^-(l_1,x_1)}\mathds{1}_{\Omega_n^+(l_2,x_2)}]}{\mathbb{E}[\mathds{1}_{\Omega_n^+(l_2,x_2)}]}\frac{\mathbb{E}[h^{(l_2)}_{p_2}(Z_n^{(l_2)}(x_2))\mathds{1}_{\Omega_n^+(l_2,x_2)}\mathds{1}_{\Omega_n^-(l_1,x_1)} ]}{\mathbb{E}[\mathds{1}_{\Omega_n^-(l_1,x_1)}]}
\\&~~=\frac{\mathbb{E}[h^{(l_1)}_{p_1}(Z_n^{(l_1)}(x_1))\mathds{1}_{\Omega_n(l_1,l_2,x_1,x_2)}]\mathbb{E}[h^{(l_2)}_{p_2}(Z_n^{(l_2)}(x_2))\mathds{1}_{\Omega_n(l_1,l_2,x_1,x_2)}]}{\mathbb{E}[\mathds{1}_{\Omega_n(l_1,l_2,x_1,x_2)}]}
\end{align*}
which is equivalent to
\begin{multline}\label{poisson_clt_proof1}
\mathbb{E}[h^{(l_1)}_{p_1}(Z_n^{(l_1)}(x_1))h^{(l_2)}_{p_2}(Z_n^{(l_2)}(x_2))|\Omega_n(l_1,l_2,x_1,x_2)]
\\ =\mathbb{E}[h^{(l_1)}_{p_1}(Z_n^{(l_1)}(x_1))|\Omega_n(l_1,l_2,x_1,x_2)]\mathbb{E}[h^{(l_2)}_{p_2}(Z_n^{(l_2)}(x_2))|\Omega_n(l_1,l_2,x_1,x_2)].
\end{multline}
Further we obtain as in (A.51) of \cite{MarVet17}
\begin{align}\label{poisson_clt_proof2}
\mathbb{P}(\Omega_n(l_1,l_2,x_1,x_2)) \geq 1-\frac{K/n}{|x_1-x_2|}.
\end{align}
Denote $P=P_1+P_2+P_3$, $A_T(\varepsilon)=\{(x_1,\ldots, x_k) \in [0,T]^k |\exists i,j : |x_i-x_j| \leq \varepsilon\}$ and $B_T(\sigma)=\{(x_1,\ldots x_k) \in [0,T]^k:x_{\sigma(1)}\leq x_{\sigma(2)} \leq \ldots \leq x_{\sigma(k)}\}$ where $\sigma$ denotes a permutation of $\{1,\ldots,k\}$. Using \eqref{poisson_clt_proof1}, \eqref{poisson_clt_proof2}, the inequality
\begin{align*}
|\mathbb{E}[X]-\mathbb{E}[X|A]| \leq 2K\frac{1-\mathbb{P}(A)}{\mathbb{P}(A)}
\end{align*}
which holds for any bounded random variable $X \leq K$ and any set $A$ with $\mathbb{P}(A)>0$ together with the boundedness of $g, h_p^{(l)}$, $l=1,2,3$, yields
\begin{align*}
&\Big|\eqref{lemma_clt1} - \int_{[0,T]^{P_1+P_2+P_3}} g(x_1,\ldots,x_{P_1},x'_1,\ldots,x'_{P_2},x''_1,\ldots,x''_{P_3})   \prod_{p=1}^{P_1} \mathbb{E} \big[h_p^{(1)}(Z_n^{(1)}(x_p))\big] \nonumber
\\& ~~~\times \prod_{p=1}^{P_2}\mathbb{E} \big[ h_p^{(2)}(Z_n^{(2)}(x'_p))\big]\prod_{p=1}^{P_3} \mathbb{E} \big[h_p^{(3)}(Z_n^{(3)}(x''_p)) \big]
d x_1 \ldots d x_{P_1} d x'_1 \ldots d x'_{P_2} d x''_1 \ldots d x''_{P_3} \Big|
\\&~ \leq K \lebesgue^{P \otimes } (A_T(\varepsilon))
\\&~~+\sum_{\sigma \in S_P}\Big|\eqref{lemma_clt1} - \int_{B_T(\sigma)\setminus A_T(\varepsilon)} g(x_1,\ldots,x_P)   \mathbb{E} \big[\prod_{p=1}^{P_1} h_p^{(1)}(Z_n^{(1)}(x_p))\prod_{p=P_1+1}^{P_1+P_2} h_p^{(2)}(Z_n^{(2)}(x_p)) \nonumber
\\& ~~~~~\times \prod_{p=P_1+P_2+1}^{P} h_p^{(3)}(Z_n^{(3)}(x_p)) |\bigcap_{i=2}^P \Omega(l_{\sigma({i-1})},l_{\sigma(i)},x_{\sigma({i-1})},x_{\sigma({i})})\big) \big]
d x_1  \ldots d x_{P} \Big|
\\&~~+\sum_{\sigma \in S_P}\Big|\int_{B_T(\sigma)\setminus A_T(\varepsilon)} g(x_1,\ldots,x_{P})   \prod_{l=1}^3
\\&~~~~~\times  \prod_{p=\sum_{m <l}P_m+ 1}^{\sum_{m \leq l}P_m} \mathbb{E} \big[h_p^{(l)}(Z_n^{(l)}(x_p))|\bigcap_{i=2}^P \Omega(l_{\sigma({i-1})},l_{\sigma(i)},x_{\sigma({i-1})},x_{\sigma({i})})\big)\big] d x_1  \ldots d x_{P}
\\&~~~~~~~~- \int_{B_T(\sigma)\setminus A_T(\varepsilon)} g(x_1,\ldots,x_{P})   \prod_{l=1}^3 \prod_{p=\sum_{m <l}P_m+ 1}^{\sum_{m \leq l}P_m} \mathbb{E} \big[h_p^{(l)}(Z_n^{(l)}(x_p))\big] d x_1  \ldots d x_{P} \Big|
\\ &~ \leq K \lebesgue^{P \otimes } (A_T(\varepsilon))
\\&~~~~+\sum_{\sigma \in S_P} K \int_{B_T(\sigma) \setminus A_T(\varepsilon)} \frac{1-\mathbb{P}\big(\bigcap_{i=2}^P \Omega(l_{\sigma({i-1})},l_{\sigma(i)},x_{\sigma({i-1})},x_{\sigma({i})})\big)}{\mathbb{P}\big(\bigcap_{i=2}^P \Omega(l_{\sigma({i-1})},l_{\sigma(i)},x_{\sigma({i-1})},x_{\sigma({i})})\big)}dx_1 \ldots d x_P
\\&~~~~+\sum_{\sigma \in S_P} K \int_{B_T(\sigma) \setminus A_T(\varepsilon)}
\sum_{j=1}^P \frac{1-\mathbb{P}\big(\bigcap_{i=2}^P \Omega(l_{\sigma({i-1})},l_{\sigma(i)},x_{\sigma({i-1})},x_{\sigma({i})})\big)}{\mathbb{P}\big(\bigcap_{i=2}^P \Omega(l_{\sigma({i-1})},l_{\sigma(i)},x_{\sigma({i-1})},x_{\sigma({i})})\big)}
dx_1 \ldots d x_P
\\ &~ \leq K \lebesgue^{P \otimes } (A_T(\varepsilon))
\\&~~~~+\sum_{\sigma \in S_P} K  \int_{B_T(\sigma) \setminus A_T(\varepsilon)} \frac{(P+1)\sum_{i=2}^k \mathbb{P}( \Omega(l_{\sigma({i-1})},l_{\sigma(i)},x_{\sigma({i-1})},x_{\sigma({i})})^c)}{1-\sum_{i=2}^P \mathbb{P}( \Omega(l_{\sigma({i-1})},l_{\sigma(i)},x_{\sigma({i-1})},x_{\sigma({i})}))^c
}dx_1 \ldots d x_P
\\&~ \leq K\lebesgue^{P \otimes } (A_T(\varepsilon)) +\binom{P}{2}K(P+1) \frac{PK/n}{\varepsilon} \Big(1- P\frac{K/n}{\varepsilon} \Big)^{-1}
\end{align*}
for all $\varepsilon>0$ where $\lebesgue^{P \otimes }$ denotes the Lebesgue measure in $\mathbb{R}^P$. This term vanishes as $n \rightarrow \infty$ and then $\varepsilon \rightarrow 0$. Next observe that it holds 
\begin{align*}
&\Big|\eqref{lemma_clt2} - \int_{[0,T]^{P_1+P_2+P_3}} g(x_1,\ldots,x_{P_1},x'_1,\ldots,x'_{P_2},x''_1,\ldots,x''_{P_3})   \prod_{p=1}^{P_1} \mathbb{E} \big[h_p^{(1)}(Z_n^{(1)}(x_p))\big] \nonumber
\\& ~~~\times \prod_{p=1}^{P_2}\mathbb{E} \big[ h_p^{(2)}(Z_n^{(2)}(x'_p))\big]\prod_{p=1}^{P_3} \mathbb{E} \big[h_p^{(3)}(Z_n^{(3)}(x''_p)) \big]
d x_1 \ldots d x_{P_1} d x'_1 \ldots d x'_{P_2} d x''_1 \ldots d x''_{P_3} \Big|
\\ &~\leq K \varepsilon^P +  K T^P \sum_{l=1}^3 \sum_{p=1}^{P_l}\sup_{x \in [\varepsilon,T]} \Big|\int h_p^{(l)}(y) \Gamma^{(l)}(dy)-\mathbb{E}[h_p^{(l)}(Z_n^{(l)}(x))] \Big|.
\end{align*}
Hence it remains to show that there exists some measure $\Gamma^{(l)}$ on $\mathbb{R}^d$ with
\begin{align}\label{poisson_clt_proof3}
\sup_{x \in [\varepsilon,T]}\Big|\int h_p^{(l)}(y) \Gamma^{(l)}(dy)-\mathbb{E}[h_p^{(l)}(Z_n^{(l)}(x))] \Big| \rightarrow 0
\end{align}
as $n \rightarrow \infty$ for all $\varepsilon >0$.

For proving \eqref{poisson_clt_proof3} observe that it holds $Z_n^{(l)}(x)\overset{\mathcal{L}}{=} Z_1^{(l)}(nx)$. Since the Poisson processes generating the observation times in $\pi_1$ and the Brownian motion $\overline{W}^{(l)}$ are independent stationary processes the law of $Z_1^{(l)}(nx)$ depends on $n$ only through the fact that all occuring intervals are bounded to the left by $0$. Hence if $\Omega(n,m,x)$ denotes the set on which all intervals needed for the construction of $Z^{(l)}_1(mx)$ are within $[mx-nx,\infty)$, it follows that all members of the sequence $(Z^{(l)}_1(mx)\mathds{1}_{\Omega(n,m,x)}))_{m \geq n}$ have the same law. This yields because of $\mathbb{P}(\Omega(n,m,x))=\mathbb{P}(\Omega(n,n,x))$, $m \geq n$, and
\begin{align*}
\lim_{n \rightarrow \infty} \mathbb{P}(\Omega(n,n,x))=1,~x>0,
\end{align*}
that the sequence $(Z^{(l)}_1(nx))_{n \in \mathbb{N}}$ converges in law for $x >0$. Hence the sequence $(Z^{(l)}_n(x))_{n \in \mathbb{N}}$ converges in law for $x>0$. By the stationarity of the processes the law of the limit, which we denote by $\Gamma^{(l)}$, does not depend on $x$. Finally we obtain the uniform convergence in \eqref{poisson_clt_proof3} because of $\mathbb{P}(\Omega(n,n,x))\leq \mathbb{P}(\Omega(n,n,\varepsilon))$ for $x \geq \varepsilon$. 
\end{proof}

In the proof of Lemma \ref{lemma_poiss_cond_clt} we obtained $$Z_n^{(l)}(s)\mathds{1}_{\Omega^{(l)}_n(s)}\overset{\mathcal{L}}{=}Z^{(l)} \mathds{1}_{\widetilde{\Omega}^{(l)}_n(s)}$$ for a random variable $Z^{(l)}\sim \Gamma^{(l)} $ and sequences $(\Omega^{(l)}_n(s))_n$, $(\widetilde{\Omega}^{(l)}_n(s))_n$ with ${\Omega^{(l)}_n(s) \uparrow \Omega}$, $\widetilde{\Omega}^{(l)}_n(s) \uparrow \widetilde{\Omega}$. Hence if the limit law would admit a part which is singular to the Lebesgue measure, this would also be true for the law of the $Z_n^{(l)}(s)$ with $n$ sufficiently large. In all three cases considered in this paper it can be shown that $Z_n^{(l)}(s)$ has no atom. Hence in Condition \ref{cond_clt_J}(ii) and Condition \ref{cond_clt_coJ} we obtain $\Gamma(x,\{0\})=0$. 

Further by Theorem 6.7 from \cite{Bil99} there exist random variables $(\widetilde{Z}^{(l)}, (\widetilde{Z}_n^{(l)})_{n \in \mathbb{N}})$ defined on a probability space $(\Omega',\mathcal{F}',\mathbb{P}')$ with
\begin{align*}
(\widetilde{Z}^{(l)}, (\widetilde{Z}_n^{(l)})_{n \in \mathbb{N}})\overset{\mathcal{L}}{=}(Z^{(l)}, (Z_n^{(l)})_{n \in \mathbb{N}})
\end{align*}
and $\widetilde{Z}^{(l)}_n \rightarrow \widetilde{Z}$ almost surely.
An application of Fatou's lemma then yields
\begin{gather*}
\mathbb{E}[\|Z^{(l)}\|^p]=\mathbb{E}[\|\widetilde{Z}^{(l)}\|^p]
\leq \liminf _{n \in \mathbb{N}} \mathbb{E}[\|\widetilde{Z}^{(l)}_n(s)\|^p]
\leq \sup _{n \in \mathbb{N}} \mathbb{E}[\|\widetilde{Z}^{(l)}_n(s)\|^p]
=\sup _{n \in \mathbb{N}} \mathbb{E}[\|Z^{(l)}_n(s)\|^p]
\end{gather*}
for $p \geq 0$. For the cases in Condition \ref{cond_clt_J}(ii), \ref{cond_cons_CoJ2}(iii) and \ref{cond_clt_coJ} it can be shown that the supremum is finite for all $p \geq 0$ and hence $\Gamma^{(l)}$ has finite moments. This yields the requirement of uniformly bounded first and second moments in Condition \ref{cond_clt_J}(ii), \ref{cond_cons_CoJ2}(iii) and \ref{cond_clt_coJ}.

\subsubsection{Proof of Condition \ref{cond_testproc_J}(i) and \ref{cond_testproc_CoJ}(i) in the Poisson setting}\label{sec:proof_poiss_test}

Note that it holds
\begin{align*}
(\hat{\xi}_{k,n,1,-}(s),\hat{\xi}_{k,n,1,+}(s))\overset{\mathcal{L}_\mathcal{S}}{=}(\xi_{k,n,-} (U_{n}(s)),\xi_{k,n,+} (U_{n}(s))) 
\end{align*}
with $U_{n}(s) \sim \mathcal{U}[t_{i_n(s)-L_n-1},t_{i_n(s)+L_n}]$ where $\mathcal{L}_\mathcal{S}$ denotes equality of the $\mathcal{S}$-conditional distributions. Hence the proof of Condition \ref{cond_testproc_J}(i) in the Poisson setting is identical to the proof of (4.3) in \cite{MarVet17}. The same holds true for Condition \ref{cond_testproc_CoJ}(i) as $\widehat{Z}_{n,1}(s)\overset{\mathcal{L}_\mathcal{S}}{=}Z_{n}(V_{n}(s))$ with
\begin{align*}
V_{n}(s)\sim \mathcal{U}[t^{(1)}_{i_n^{(1)}(s)-L_n-1,n}\vee t^{(2)}_{i^{(2)}_n(s)-L_n-1,n}, t^{(1)}_{i_n^{(1)}(s)+L_n,n}\wedge t^{(2)}_{i^{(2)}_n(s)+L_n,n}].
\end{align*}

\newpage

\FloatBarrier

\end{document}